\documentclass[10pt, oneside]{amsart}
\usepackage{graphicx,epsfig,overpic}
\usepackage{bbding,amssymb,amsfonts,amsmath,amsthm,dsfont,wasysym,pifont,stmaryrd,inputenc,ifthen,comment}
\usepackage{epstopdf,yfonts}
\usepackage{diagrams}
\usepackage{enumerate}
\usepackage[shortlabels]{enumitem}
\usepackage[backref=page]{hyperref}

\renewcommand*{\backref}[1]{}
\renewcommand*{\backrefalt}[4]{%
    \ifcase #1 (Not cited.)%
    \or        (Cited on page~#2.)%
    \else      (Cited on pages~#2.)%
    \fi}

\usepackage{wrapfig}

\usepackage{microtype}

\usepackage{mathrsfs}
\usepackage{wasysym}
\usepackage{tikz,tikz-cd,tikzsymbols}
\usetikzlibrary{matrix}
\usetikzlibrary{matrix,shapes,backgrounds}

\newcommand{\showcomments}{yes}
 \renewcommand{\showcomments}{no}

\newsavebox{\commentbox}
{\ifthenelse{\equal{\showcomments}{yes}}%
{\footnotemark
        \begin{lrbox}{\commentbox}
        \begin{minipage}[t]{1in}\raggedright\sffamily\tiny
        \footnotemark[\arabic{footnote}]}
{\begin{lrbox}{\commentbox}}}%
{\ifthenelse{\equal{\showcomments}{yes}}%
{\end{minipage}\end{lrbox}\marginpar{\usebox{\commentbox}}}
{\end{lrbox}}}

\numberwithin{equation}{section}

\newcounter{cl}

\newtheorem{theorem}[equation]{Theorem}
\newtheorem*{theorem'}{Main Theorem}

\newtheorem{question}[equation]{Question}

\newtheorem{prop}[equation]{Proposition}

\newtheorem{claim}[equation]{Claim}

\newtheorem{cor}[equation]{Corollary}
\newtheorem{lemma}[equation]{Lemma}

\newtheorem{thmi}{Theorem}

\newtheorem*{namedtheorem}{\theoremname}
\newcommand{\theoremname}{testing}
\newenvironment{named}[1]{\renewcommand{\theoremname}{#1}\begin{namedtheorem}}{\end{namedtheorem}}

\theoremstyle{definition}
\newtheorem{definition}[equation]{Definition}

\newtheorem{remark}[equation]{Remark}
\newtheorem{convention}[equation]{Convention}
\newtheorem{example}[equation]{Example}

\renewcommand{\~}[1]{\overline{#1}}

\renewcommand{\geq}{\geqslant}
\renewcommand{\leq}{\leqslant}

\newcommand{\Aut}{\text{Aut}}

\newcommand{\C}{\mathbb{C}}

\renewcommand{\Cup}[2]{\underset{#1}{\overset{#2}{\cup} }}

\newcommand{\eye}{\sphericalangle}

\newcommand{\frakH}{\mathfrak H}

\newcommand{\I}{\mathcal{I}}
\newcommand{\J}{\mathcal{J}}

\renewcommand{\int}{\varint}

\newcommand{\M}{\mathcal{M}}

\newcommand{\N}{\mathbb{N}}
\newcommand{\nerve}{\mathcal N}

\newcommand{\onto}{\twoheadrightarrow}

\newcommand{\R}{\mathbb{R}}
\renewcommand{\Re}{\mathfrak{R}}

\newcommand{\SL}{\textnormal{SL}}
\newcommand{\stab}{\mathrm{stab}}

\newcommand{\supp}{\mathrm{supp}}

\newcommand{\U}{\mathcal{U}}
\newcommand{\V}{\mathcal{V}}

\newcommand{\W}{\mathcal{W}}
\newcommand{\Z}{\mathbb{Z}}
\newcommand{\UR}{\operatorname{R^U}}
\newcommand{\RU}{\operatorname{U^R}}
\newcommand{\URtri}{\operatorname{R^U_\triangle}}
\newcommand{\RUtri}{\operatorname{U^R_\triangle}}
\newcommand{\RT}{\operatorname{T^R}}
\newcommand{\ST}{\operatorname{T^S}}
\newcommand{\SR}{\operatorname{R^S}}
\newcommand{\join}{{\star}}

\newcommand{\id}{{\operatorname{id}}}
\newcommand{\hull}{\operatorname{CH}} 
\newcommand{\insep}[1]{\ensuremath{\langle\!\langle#1\rangle\!\rangle}} 
\renewcommand{\emptyset}{\varnothing}

\renewcommand{\diameter}{\operatorname{diam}}

\newcommand{\tits}{\partial_T}
\newcommand{\guralnik}{\mathfrak R}
\newcommand{\absimp}{\mathfrak R_\triangle}
\newcommand{\visual}{\partial_\eye}
\newcommand{\dist}{\mathrm{d}}
\newcommand{\dimension}{\operatorname{dim}}
\newcommand{\naturals}{\mathbb N}

\newcommand{\visi}{\mathrm{Vis}}
\newcommand{\simp}{\partial_\triangle}
\newcommand{\roller}{{\partial_R}}
\newcommand{\reals}{\mathbb R}
\newcommand{\maxvis}{\mathrm{MaxVis}}
\newcommand{\vispart}{\mathfrak R^{\mathrm{\eye}}_\triangle}
\newcommand{\thickvispart}{{\mathfrak S}}
\newcommand{\integers}{\mathbb Z}

\newcommand{\simpubs}{\simp^{\textup{UBS}}}

\renewcommand{\setminus}{\smallsetminus}
\newcommand{\bdy}{{\partial}}
\newcommand{\from}{\colon}
\newcommand{\circum}{\chi}
\newcommand{\gate}{\pi}
\newcommand{\Ram}{\operatorname{Ram}}
\newcommand{\Ghom}{$G \!\!\sim$}
\newcommand{\Authom}{$\Aut(X) \!\!\sim$}

\title{Homotopy equivalent boundaries of cube complexes}
\author{Talia Fern\'os}
\address{Mathematics and Statistics Dept., University of North Carolina at Greensboro, Greensboro, NC 27412, USA}
\email{t\_fernos@uncg.edu}
\author{David Futer}
\address{Dept. of Mathematics, Temple University, Philadelphia, PA 19122, USA}
\email{dfuter@temple.edu}
\author{Mark Hagen}
\address{School of Mathematics, University of Bristol, Bristol BS8 1UG, UK}
\email{markfhagen@posteo.net}

\makeatletter
\@namedef{subjclassname@2020}{\textup{2020} Mathematics Subject Classification}
\makeatother

\begin{document}
\date{\today}
\subjclass[2020]{51F99; 20F65}
\keywords{CAT(0) cube complex, Tits boundary, Roller boundary, nerve theorem}

\begin{abstract}
A finite-dimensional CAT(0) cube complex $X$ is equipped with several well-studied boundaries. These include the 
\emph{Tits boundary} $\tits X$ (which depends on the CAT(0) metric), the \emph{Roller boundary} $\roller X$ (which 
depends only on the combinatorial structure), and the \emph{simplicial boundary} $\simp X$ (which also depends only on the combinatorial structure). We use a partial order on a certain quotient of 
$\roller X$ to define a simplicial Roller boundary  $\absimp X$. Then, we show that $\tits X$, $\simp X$, and $\absimp X$ are all homotopy equivalent, $\Aut(X)$--equivariantly up to 
homotopy. As an application, we deduce that the perturbations of the CAT(0) metric introduced by Qing do not affect the equivariant homotopy type of the Tits boundary. Along the way, we develop a 
self-contained exposition providing a dictionary among different perspectives on cube complexes.
\end{abstract}

\maketitle
\tableofcontents

\section{Introduction}\label{Sec:Intro}

CAT(0) cube complexes, which exist in many guises in discrete mathematics (see e.g. \cite{BC}), were introduced into 
group theory by Gromov \cite{Gromov} and have since taken on a central role in that field.  As 
combinatorial objects, CAT(0) cube complexes are ubiquitous due to their flexible, functorial construction from set-theoretic data \cite{Sageev_95,Roller, Chatterji-Niblo,Nica}.  This has led to an 
industry of \emph{cubulating} groups, i.e. 
constructing group actions on CAT(0) cube complexes in order to transfer  information from the highly organized cubical structure to the group.  Probably the best known application of this method is the resolution of the virtual 
Haken and virtual fibering conjectures in 3--manifold theory \cite{Agol:VH,Wise:QCH}.

The utility of CAT(0) cube complexes comes from the fact that they simultaneously exhibit several types of structures.  They have an organized combinatorial structure coming from their \emph{hyperplanes/half-spaces}; this is closely 
related to the very tractable geometry of their $1$--skeleta, which are \emph{median graphs} \cite{Chepoi}.  On the other 
hand, endowing the complex with the piecewise-Euclidean metric in which cubes are Euclidean unit cubes, one gets a CAT(0) 
space.  So, in studying CAT(0) cube complexes, and groups acting on them, one has a wide variety of tools.

This paper is about the interplay between the CAT(0) and combinatorial structures, at the level of \emph{boundaries}.    
We construe the term ``boundary'' broadly: we include not just  spaces arising as frontiers of injections into compact spaces, 
 but also other spaces encoding some sort of behavior at infinity or large-scale asymptotic structure, such as the Tits boundary.

\subsection{A plethora of boundaries} 
CAT(0) cube complexes have several natural boundaries, each encoding different information, and each defined in terms 
of either the CAT(0) metric structure or the combinatorial structure coming from the hyperplanes.  Fixing a 
finite-dimensional CAT(0) cube complex $X$, one has the following list of boundaries:
\begin{enumerate}
     \item The \emph{visual boundary} $\visual X$. 
      The visual boundary can be defined for any CAT(0) space; see 
\cite{BridsonHaefliger}.  Points are asymptotic equivalence 
classes of CAT(0) geodesic rays. The visual topology is defined in such a way 
that, roughly speaking, rays that fellow-travel for a long time are close.  When $X$ is locally finite, $X\cup\visual X$ is 
a compactification of $X$.  While it is a very useful object, we do not study the visual boundary in this paper.

\medskip

     \item The \emph{Tits boundary} $\tits X$.  As with $\visual X$, the points in $\tits X$ correspond to asymptotic equivalence classes of CAT(0) 
geodesic rays, but the topology is finer than the visual topology.  Specifically, we equip $\tits X$ with the \emph{Tits 
metric}: by taking a supremum of angles between rays
 one obtains the \emph{angle metric}, and the 
induced length metric is the Tits metric, which is CAT(1); see \cite[Theorem II.9.13]{BridsonHaefliger}.  The Tits boundary encodes much of the 
geometry of a CAT(0) space; for example, spherical join decompositions of the Tits boundary correspond to product 
decompositions of the space \cite[Theorem II.9.24]{BridsonHaefliger}.  Although $\tits X$ is in general not compact even when $X$ is 
proper, it is  of interest for  other reasons, such as encoding ``partial flat regions'' in $X$. For example, endpoints 
of axes of rank-one isometries are isolated points, while flats in $X$ yield spheres in $\tits X$.  We recall the definition 
of $\tits X$ in Definition~\ref{Def:tits_metric}.

\medskip

     \item The \emph{simplicial boundary} $\simp X$, from \cite{Hagen,Hagen:corr}, is an analogue 
of the Tits boundary depending on the hyperplane structure, rather than on the CAT(0) metric.  The idea is that certain sets 
of hyperplanes --- termed \emph{unidirectional boundary sets} (hereafter, \emph{UBSes}) --- identify ``ways of approaching infinity in $X$.''  Containment of UBSes (modulo finite differences) gives a partial order on 
UBSes, and this order gives rise to a simplicial complex $\simp X$ (see Definition~\ref{Def:simplicial_boundary}).  Here are three helpful examples. First, the simplicial boundary of a tree (or, more generally, of a $\delta$--hyperbolic cube complex) is a discrete set 
of $0$--simplices. Second, the simplicial boundary of the standard square tiling of $[0,\infty)^2$ is a $1$--simplex, whose $0$--simplices correspond to the sub-UBSes consisting of the vertical hyperplanes 
and the horizontal hyperplanes. Third,  the \emph{staircase} obtained from this square tiling by considering only the cubes below some increasing, unbounded function also has simplicial boundary a 
$1$--simplex. (See Figure~\ref{Fig:Staircases}.) The maximal simplices of $\simp X$ encode such ``generalized orthants'' (namely, convex hulls of $\ell^1$--geodesic rays) in $X$ in roughly the same way that the Tits boundary encodes ``partial flats.''

The 
simplicial boundary has been used to study quasi-isometry invariants like \emph{divergence} \cite{Hagen} and 
\emph{thickness} \cite{BehrstockHagen} for groups acting on cube complexes, and has been generalized in the context of 
\emph{median spaces} \cite{Fioravanti} as a tool for proving a Tits alternative for groups acting on such spaces.

\medskip

     \item The \emph{Roller boundary} $\roller X$ gives another way of compactifying $X$, using the half-space structure associated to the hyperplanes.
  Each hyperplane $\hat h$ of $X$ has two complementary components, called \emph{half-spaces}, $h$ and $h^*$, which induce a two-sided partition of 
$X^{(0)}$.  Each vertex $x\in X$ is completely determined by specifying the collection of half-spaces that contain it.  This gives an injective map $X^{(0)}\to 2^{\mathfrak H}$, where 
$\mathfrak H$ denotes the set of half-spaces.  The closure of the image of $X^{(0)}$ in the Tychonoff topology is the \emph{Roller compactification} $\overline X$, and $\roller X=\overline X \setminus X^{(0)}$ is the 
\emph{Roller boundary} of $X$. (See Definition~\ref{Def:Roller}.)

Of the boundaries of $X$ that are defined in terms of the cubical structure only, the Roller boundary is perhaps the most well-studied. It has been used to prove a variety of results about CAT(0) 
cube complexes and groups acting on them.

For example, the Roller boundary plays a key role in the proof that irreducible lattices in $\SL_2\C\times\SL_2\C$ always 
have global fixed points in any action on a CAT(0) cube complex, while reducible actions always admit proper cocompact 
actions on CAT(0) cube complexes \cite{CFI}. 

Nevo and Sageev identified the Roller boundary as a model for the Furstenberg--Poisson boundary of a cocompact lattice in 
$\Aut(X)$  \cite{NevoSageev}.  Fern\'os generalized this result  
 to groups acting properly and non-elementarily on finite dimensional $X$, and also proved a Tits alternative 
\cite{Fernos:Poisson}.   The same paper identifies an important subset of the Roller boundary: the \emph{regular points}, 
which 
correspond to ``hyperbolic'' directions in various senses.  

In particular,  the set of regular points with the induced 
topology is $\Aut(X)$--equivariantly homeomorphic with the boundary of the (hyperbolic) \emph{contact graph} 
\cite{FLM:CentralLimit}.  The set of regular points is used in~\cite{FLM:RandomWalks} to find rank-one isometries under more 
general conditions than in previous 
work~\cite{CapraceSageev}.  The method in \cite{FLM:RandomWalks} is to use convergence of random walks to regular points to 
deduce the existence of regular elements without assuming that the ambient group is a lattice.

The set of regular points also features in the proof of marked length spectrum rigidity~\cite{BeyrerFioravanti,beyrer2018rm}, 
and was independently identified by Kar and Sageev, who used it to study property $P_{\mathrm{naive}}$ and hence 
$C^*$--simplicity for cubulated 
groups \cite{KarSageev:Pnaive}.  Finally, the Roller boundary has recently been generalized in the context of median spaces 
\cite{Fioravanti}.

  \medskip

     \item The \emph{simplicial Roller boundary} $\absimp X$ is constructed as follows.
     The Roller boundary $\roller X$ carries a natural equivalence relation, first introduced 
 by Guralnik \cite{Guralnik}, where two points 
are equivalent if they differ on finitely many half-spaces.  The equivalence classes, called \emph{Roller classes}, play a fundamental role in this paper.  Part of the reason for this is that the set of Roller 
classes carries a natural partial order, which admits several equivalent useful characterizations (see Lemma~\ref{POSET Roller}).  The \emph{simplicial Roller boundary} $\absimp X$ is the simplicial 
complex realizing this partial order (see Definition~\ref{Def:SimplicialRoller}).  

In this paper, we build on earlier work relating Roller classes to CAT(0) geodesic rays.  Specifically, Guralnik in~\cite{Guralnik} recognized that each equivalence class of CAT(0) geodesic rays  determines a 
Roller class.  Conversely, in~\cite{FLM:RandomWalks}, it is shown that each Roller class determines a subset of the Tits boundary which admits a canonical circumcenter.  These observations 
are crucial for our arguments: Definition~\ref{Def:Qcircum} depends on the latter and Definition~\ref{Def:Psi} is based on the former.

We also note that $\absimp X$ is isomorphic to the 
 \emph{combinatorial boundary} defined and studied by Genevois \cite{Genevois:ContractingIsometries}. In that work, he explains that the combinatorial boundary is isomorphic to the 
face-poset of a naturally-defined subcomplex of $\simp X$ \cite[Proposition A.1]{Genevois:ContractingIsometries} and is isomorphic to $\absimp X$ \cite[Proposition 
A.5]{Genevois:ContractingIsometries}.  (More precisely, Genevois' work yields the map $\RUtri$ from Corollary~\ref{Cor:UBStoRollerSimp} below, although \cite{Genevois:ContractingIsometries} does not 
explicitly mention a homotopy equivalence.)  
\end{enumerate}

For each of these boundaries, the action of the group $\Aut(X)$ of cubical automorphisms of $X$ extends to an action on the 
boundary preserving its structure. The actions on $\visual X$ and $\roller X$ are by homeomorphisms, the action on $\tits 
X$ is by isometries of the Tits metric, and the actions on $\simp X$ and $\absimp X$ are by simplicial automorphisms.

The definitions of the simplicial boundary and the simplicial Roller boundary are conceptually similar. Part of the work 
in this paper is making that similarity precise, by establishing a correspondence between Roller classes and equivalence 
classes of UBSes. This line of work culminates in Corollary~\ref{Cor:UBStoRollerSimp}, which gives explicit maps between $\absimp X$ and a 
natural subcomplex of $\simp X$. These maps are $\Aut(X)$--equivariant homotopy equivalences, although in general they are 
not simplicial isomorphisms.  Then, in Proposition~\ref{prop:simplicial_boundary}, we upgrade this to a homotopy equivalence 
between $\absimp X$ and the whole of $\simp X$.

One glimpse of a relationship between $\simp X$ and $\tits X$ comes from UBSes associated to geodesic rays. Given a geodesic ray $\alpha \to X$, in either the CAT(0) metric or the combinatorial metric, the set of hyperplanes crossing $\alpha$ (denoted $\W(\alpha)$) is always a UBS. A UBS $\U$ is called \emph{$\ell^1$--visible} (resp. \emph{$\ell^2$--visible}) if it is has finite symmetric difference with $\W(\alpha)$ for a combinatorial (resp. CAT(0)) geodesic $\alpha$. Figure~\ref{Fig:Staircases} shows examples of both $\ell^1$--invisible and $\ell^2$--invisible UBSes. Despite those caveats, $\ell^2$--visible UBSes provide a way to map points of $\tits X$ to classes in $\simp X$; a similar construction also provides a map to $\absimp X$. 

\begin{figure}
\includegraphics[width=2.5in]{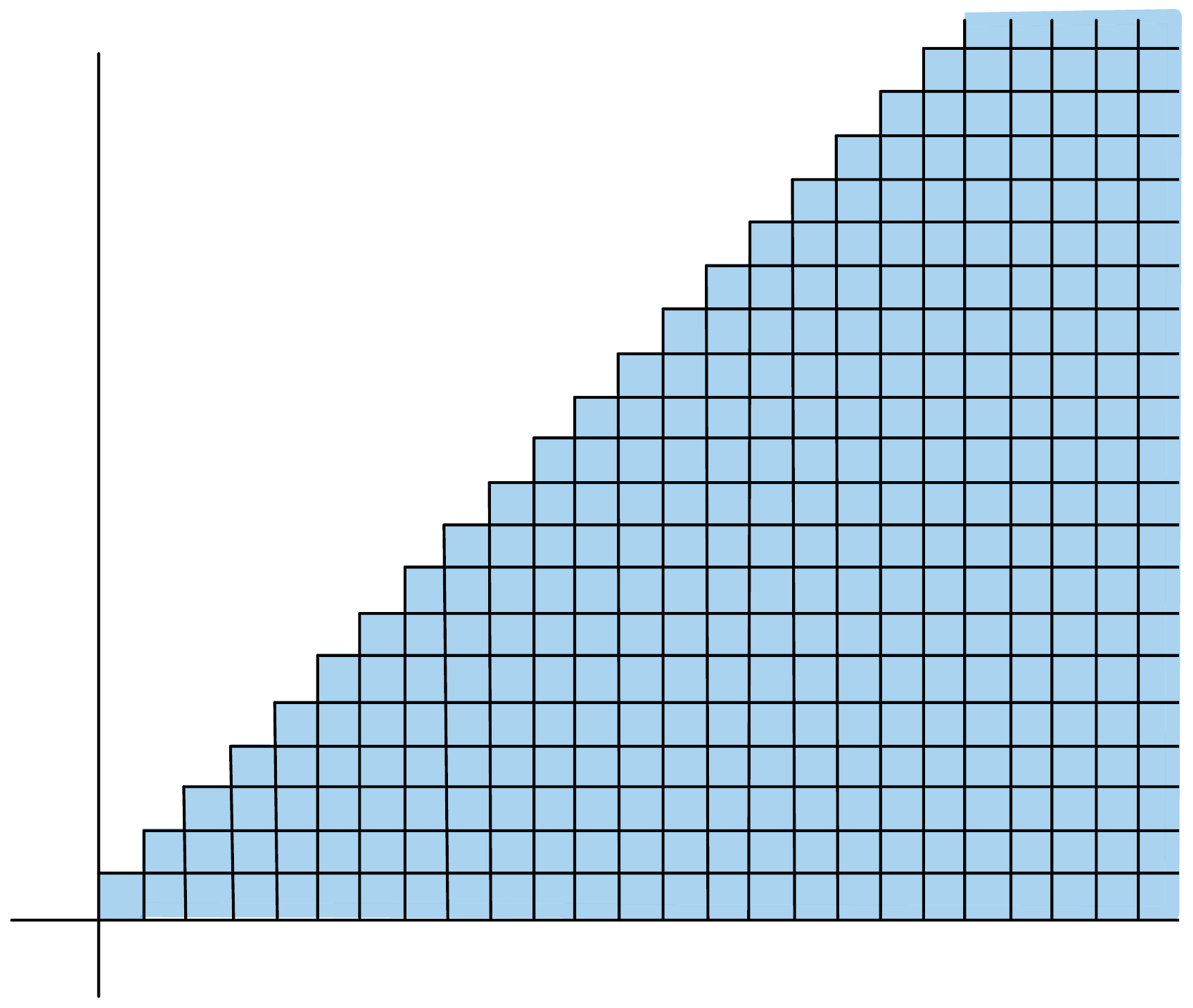}
\hspace{0.25in}
\includegraphics[width=2.5in]{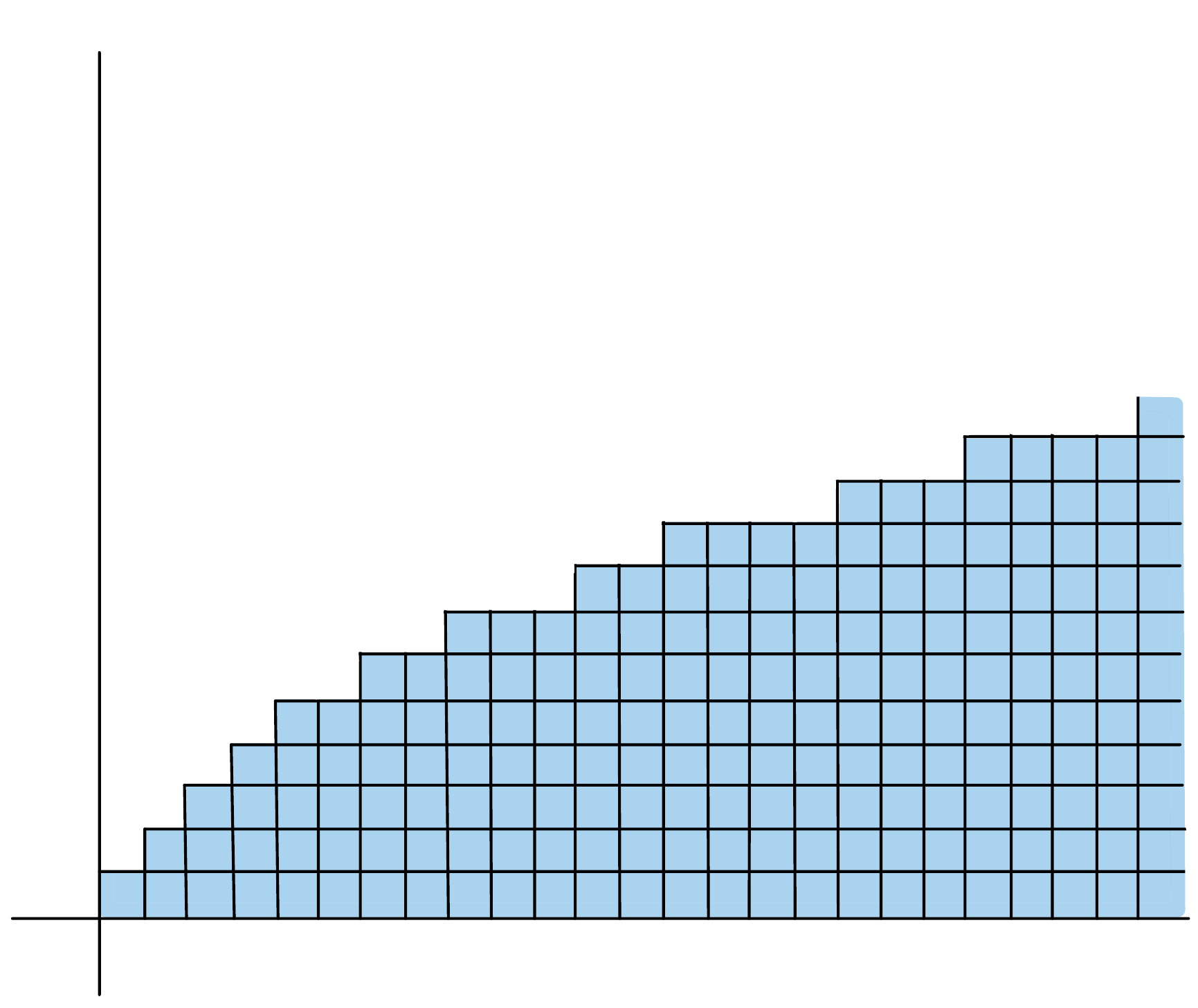}
\caption{Left: a linear staircase. Right: a sublinear staircase. For both staircases, the simplicial boundary is a $1$--simplex, whose vertices correspond to the UBS $\V$ of vertical hyperplanes and the UBS $\mathcal H$ of horizontal hyperplanes. In both staircases, $\mathcal H$ is not visible (in either the $\ell^1$ or the $\ell^2$ metric), because any geodesic ray crossing $\mathcal H$ must also cross $\mathcal V$. In the sublinear staircase, $\mathcal V \sqcup \mathcal H$ is also not $\ell^2$--visible.}
\label{Fig:Staircases}
\end{figure}

\subsection{Main result}  Our main theorem relates $\tits X$, $\simp X$, and  $\absimp X$ via maps that are $\Aut(X)$--equivariant up to homotopy.  
It helps to introduce the following terminology.  

\begin{definition}\label{Def:G-homotopy}
Suppose 
that $A,B$ are topological spaces with a group $G$ acting on both spaces by homeomorphisms.  A map $f \from A \to B$ is called a
a \emph{\Ghom homotopy equivalence} 
if $f$ is a homotopy equivalence, and furthermore  $g\circ f$ and $f\circ g$ are 
homotopic for every $g\in G$.  If such an $f$ exists, we say that $A,B$ are \emph{\Ghom homotopy equivalent}.
\end{definition}

It is immediate to check that the property of being \Ghom homotopy equivalent is an equivalence relation. 
Indeed, the composition of two \Ghom homotopy equivalences is a \Ghom homotopy equivalence. Furthermore, any homotopy inverse of a \Ghom homotopy equivalence is itself a
\Ghom homotopy equivalence.

\renewcommand*{\thethmi}{\Alph{thmi}}

\begin{thmi}\label{Thm:HE-triangle}
Let $X$ be a finite-dimensional CAT(0) cube complex.  Then 
we have the following commutative diagram of \Authom homotopy equivalences between boundaries of $X$:
\begin{diagram}
 \simp X &&\rTo^{\ST}  
 &     & \tits X  \\
&\rdTo_{\SR} & &\ruTo_{\RT}&\\
        && \Re_\triangle X&&
\end{diagram}
In particular, the spaces $\simp X,\absimp X,\tits X$ are all \Authom homotopy equivalent, where $\tits X$ is equipped with the metric topology and $\simp X,\absimp X$ are equipped with either the 
metric or the weak topology.
\end{thmi}

In the above diagram, the letter S stands for \emph{Simplicial}, the letter R for \emph{Roller-simplicial}, and the letter T 
for \emph{Tits}.  The map $\ST$ is an \Authom homotopy equivalence from the simplicial boundary to the Tits 
boundary, and similarly for the other two.  The map $\RT$ is provided by Corollary~\ref{Cor:RollerSimpTitsHomotopy}, and the 
map $\SR$ comes from Proposition~\ref{prop:simplicial_boundary}.  The map $\ST$ is just the composition $\ST=\RT\circ \SR$.

Some intuition behind the theorem can be gleaned from the following simplified situation. A cube complex $X$ is called \emph{fully visible} if every UBS is $\ell^1$--visible. 
 When $X$ is fully visible, UBSes correspond to combinatorial 
orthants in $X$ (i.e. products of combinatorial geodesic rays) \cite{Hagen}, and Theorem~\ref{Thm:HE-triangle} can be understood as relating 
combinatorial orthants to CAT(0) orthants. See~\cite[Section 3]{Hagen}, where a proof is sketched that the simplicial and 
Tits 
boundaries are homotopy equivalent under the restrictive hypothesis of full visibility.  However, 
 full visibility is a very strong and somewhat mysterious hypothesis: it is not known to hold even if $X$ 
admits a proper cocompact group action \cite{HagenSusse}. Thus 
Theorem~\ref{Thm:HE-triangle} is more satisfying (and difficult) because it does not assume this restrictive hypothesis. In order to prove the theorem, we have 
to understand combinatorial convex hulls of CAT(0) geodesic rays, which is a much more delicate affair when one is not 
simply assuming that these hulls can be taken to contain Euclidean orthants.

We observe that Theorem \ref{Thm:HE-triangle} does not have a direct analogue in which the Tits boundary is replaced by the visual boundary.  For example, if $X$ admits a proper, cocompact, essential action by a group $G$, and $X$ is irreducible, then $G$ contains a rank-one isometry of $X$ \cite{CapraceSageev}, and therefore $\simp X$ has two distinct isolated points. More generally, even without a cocompact group action, $\simp X$ has isolated points corresponding to regular points in $\roller X$, by \cite[Proposition 7.4]{Fernos:Poisson} and \cite[Corollary 3.20]{Hagen}.  On the other hand, if $X$ is one-ended, then the visual boundary of $X$ is connected.

Theorem~\ref{Thm:HE-triangle} also motivates the introduction of several new technical tools for relating the geometry of CAT(0) geodesic rays to the combinatorics of hyperplanes. We believe that these tools, explained in Section~\ref{Sec:ToolsIngredients}, are of independent interest.

  \subsection{Invariance of CAT(0) boundaries for cubulated groups}\label{subsec:perturb_intro}
A group $G$ is called \emph{CAT(0)} if there is a proper CAT(0) space $X$ on which $G$ acts 
geometrically (properly and cocompactly).  Multiple CAT(0) spaces might admit a geometric $G$--action and witness that $G$ is a CAT(0) group.  Since we are often 
interested in invariants of the group itself, it is natural to look for features of the geometry of $X$ that  depend 
only on $G$.  For instance, in the context of a Gromov hyperbolic group $G$, recall that all hyperbolic groups with a geometric $G$--action have Gromov boundaries that are $G$--equivariantly homeomorphic. So one might wonder whether a similar result might hold for CAT(0) groups that are not hyperbolic.

A famous result of Croke and Kleiner  \cite{CrokeKleiner,CrokeKleiner:Equiv}, on 
which we elaborate more below, shows that this is not the case.  Explicitly, they considered the right-angled Artin group
$$G=\langle a,b,c,d : [a,b],[b,c],[c,d]\rangle,$$
whose Cayley 2-complex $X_{\pi/2}$ is a CAT(0) square complex. The subscript $\pi/2$ emphasizes the angles at the corners of the $2$--cells. Viewing the $2$--cells as Euclidean squares with 
side length $1$, we obtain a CAT(0) space where $G$ acts geometrically.  Croke 
and Kleiner deformed the squares of $X$ into rhombi that have angles $\alpha$ and $\pi - \alpha$, producing a perturbed 
CAT(0) metric $X_\alpha$,
and showed that this perturbation changes the homeomorphism type of the visual boundary $\visual X_\alpha$. Subsequently, Wilson
\cite{Wilson} showed that the visual boundaries of $X_\alpha$ and $X_\beta$ are homeomorphic only if the angles satisfy 
$\alpha=\beta$, hence these perturbations produce  uncountably many homeomorphism types of boundaries.  Further examples are 
provided by Mooney \cite{Mooney2008}, who exhibited CAT(0) knot groups $G$ 
admitting uncountably many geometric actions on CAT(0) spaces, all with different visual boundaries.  Hosaka \cite{Hosaka} 
has given some conditions on $G,X,Y$ implying that there is an equivariant homeomorphism $\visual X\to\visual Y$ 
continuously extending the quasi-isometry $X\to Y$ coming from orbit maps, but the preceding examples show that any such 
condition will be hard to satisfy.

So, one has to look for weaker results or less refined invariants, or ask slightly different questions.  This has stimulated 
a great deal of work in various directions.  

In the context of $2$--dimensional CAT(0) complexes, there are more positive results obtained by replacing the visual 
boundary $\visual X$ by the Tits boundary $\tits X$, and passing to a natural subspace.  Specifically, the \emph{core} of $\tits X$ is the union of 
all embedded circles in $\tits X$. 
 Xie \cite{Xie} showed that if $G$ acts 
geometrically on CAT(0) $2$--complexes $X$ and $Y$, then the Tits boundaries $\tits X$ and $\tits Y$ have homeomorphic cores.
(This is weaker than Xie's actual statement; compare \cite[Theorem 6.1]{Xie}.)   In the Croke--Kleiner example, the core is a $G$--invariant connected 
component, while the rest of the Tits boundary consists of uncountably many isolated points and arcs.  

There is a closely related result due to Qing, which inspired us to consider \emph{cuboid complexes} in the present paper.  
Another way to perturb the CAT(0) metric on a CAT(0) cube complex $X$ is to leave the angles alone, but to vary the lengths 
of the edges in such a way that 
edges intersecting a common  hyperplane are given equal length. In this perturbation, each cube becomes a Euclidean box 
called a \emph{cuboid}.  The resulting path metric is still CAT(0), at least when there are uniform upper and lower bounds 
on 
the edge-lengths. In particular, this happens when some group $G$ acts on $X$ with 
finitely many orbits of hyperplanes, and the edge-lengths are assigned $G$--equivariantly.   We explain the details in 
Section~\ref{subSec:cuboid_defs},  
where we define a \emph{$G$--admissible hyperplane rescaling} in Definition~\ref{Def:hyp_rescaling} and  show in 
Lemma~\ref{Lem:cat_0_rescaling} that the resulting path-metric space $(X,\dist^\rho_X)$ 
is CAT(0) and $G$ continues to act by isometries.  Cuboid complexes have been studied by various authors, including Beyrer 
and Fioravanti \cite{BeyrerFioravanti}.

Qing \cite{Qing:thesis} studied the visual boundary of the CAT(0) cuboid metric under perturbation of the edge-lengths. She showed that 
the visual boundaries of the CAT(0) cuboid complexes obtained from the  Croke--Kleiner complex $X_{\pi/2}$ are all 
homeomorphic. More germane to the present paper, Qing also showed that the Tits boundaries of the rescaled cuboid complexes 
are homeomorphic, but that no  equivariant homeomorphism exists.  (Roughly, the cores of any two such Tits boundaries are 
homeomorphic by the result of Xie mentioned above. Via a cardinality argument, Qing extends the 
homeomorphism over the whole boundary by sending isolated points/arcs to isolated points/arcs.  But this extension cannot be 
done equivariantly because equivariance forces some isolated points to be sent to isolated arcs and vice versa.)  We will 
return to cuboid complexes shortly.

There are other results about homeomorphism type as a CAT(0) group invariant for certain classes of CAT(0) groups.  For 
example, Bowers and Ruane \cite{BowersRuane} showed that if $G$ is a product of a hyperbolic CAT(0) group and 
$\integers^D$, then the visual boundaries of the CAT(0) spaces $X$ and $Y$ with a geometric $G$--action are equivariantly homeomorphic. However, this equivariant homeomorphism need not come from a continuous extension of a quasi-isometry $X \to Y$.  Later, Ruane  \cite{Ruane} extended this result to the case where $G$ is 
a CAT(0) direct product of two non-elementary hyperbolic groups. 

Around the same time, Buyalo studied groups of the form $G=\pi_1(S)\times\mathbb Z$, where $S$ is a hyperbolic surface \cite[Section 14]{Buyalo}. He constructed two distinct geometric $G$--actions on $\mathbb H^2\times\reals$ and showed that while the two copies of $\mathbb H^2\times\reals$ are $G$--equivariantly quasi-isometric, the Tits boundaries do not admit any $G$--equivariant quasi-isometry. 

Since we are interested in $G$--equivariant results, it makes sense to seek more general positive results by replacing homeomorphism with a coarser equivalence relation.  This 
has been a successful idea: Bestvina \cite{Bestvina1996} proved that torsion-free CAT(0) groups have a well-defined visual 
boundary up to shape-equivalence. 
This was generalized by Ontaneda \cite{Ontaneda}, whose result removes the ``torsion-free'' 
hypothesis.

Homotopy equivalence is a finer equivalence relation than shape equivalence.  Might there be some result guaranteeing that the 
homotopy type of the Tits boundary of $X$ is to some extent independent of the choice of the CAT(0) space $X$, even if we 
allow only homotopy equivalences between boundaries that respect the $G$--action in the appropriate sense? One precise version of this question appears below as Question~\ref{Ques:GhomotopyTits}.

Motivated by Qing's results about perturbing the metric on a CAT(0) cube complex by changing edge lengths, we consider the 
situation where $G$ acts on the finite-dimensional CAT(0) cube complex $X$, and study the Tits boundaries of the CAT(0) spaces 
that result from $G$--equivariantly replacing cubes by cuboids.  Now, the hyperplane combinatorics, and hence the simplicial 
and Roller boundaries, are unaffected by this change.  And, for the unperturbed $X$, Theorem~\ref{Thm:HE-triangle} tells us 
that the homotopy type of the Tits boundary is really a feature of the hyperplane combinatorics.  So, in order to conclude 
that the homotopy type of the Tits boundary is unaffected by such perturbations of the CAT(0) metric, we just need to know 
that Theorem~\ref{Thm:HE-triangle} holds for cuboid complexes as well as cube complexes.  
Indeed, we show:

\begin{thmi}\label{Thm:MainCuboids}
 Let $X$ be a finite-dimensional CAT(0) cube complex and let $G$ be a group acting by automorphisms on $X$. Let $(X,\dist^\rho_X)$ be the CAT(0) cuboid complex obtained 
from a $G$--admissible hyperplane rescaling of $X$.  Then the original Tits boundary $\tits X$, the perturbed Tits boundary $\tits(X,\dist^\rho_X)$,  the simplicial boundary $\simp X$, and the simplicial Roller boundary $\absimp X$ are 
all \Ghom homotopy equivalent. 
\end{thmi}

The proof of this result only requires minor modifications to the proof of Theorem~\ref{Thm:HE-triangle},
essentially because convexity of 
half-spaces and the 
metric product structure 
of hyperplane carriers persist in the CAT(0) cuboid 
metric $\dist^\rho_X$.
These small modifications  are described in 
Sections~\ref{Sec:TitsCuboid}, \ref{Sec:CuboidCover},  and \ref{Subsec:FinalProofCuboids}. 

\begin{remark}\label{rem:borel-intro}
The proofs of Theorem~\ref{Thm:HE-triangle} and \ref{Thm:MainCuboids} yield a slightly stronger conclusion than \Ghom 
homotopy equivalence.  We say that $G$--spaces $A,B$ are \emph{Borel \Ghom homotopy equivalent} if there is a map $f:A\to B$ 
that factors as a composition of finitely many homotopy equivalences, each of which is either $G$--equivariant or a homotopy 
inverse of a $G$--equivariant map.   (This definition is somewhat reminiscent of the notion of cell-like equivalence. Compare Guilbault and Mooney~\cite{GuilbaultMooney2012}, specifically the discussion Bestvina's Cell-like Equivalence Question on page 120.)

In the proof of both theorems, the homotopy equivalence from $\tits X$ (or 
$\tits(X,\dist^\rho_X)$) to $\absimp X$ is constructed as an explicit composition of maps, each of which is either a 
$G$--equivariant homotopy equivalence, a deformation retraction homotopy inverse to a $G$--equivariant inclusion map (see 
Proposition~\ref{prop:tits_visible_homotopic}), or a homotopy equivalence coming from one of the nerve theorems; the latter 
are compositions of $G$--equivariant maps or homotopy inverses of $G$--equivariant maps by Remark~\ref{rem:borel}.  The 
homotopy equivalence $\absimp X\to \simp X$ similarly proceeds by composing $G$--equivariant maps and their homotopy 
inverses, in view of the same remark about the nerve theorems.

From the definitions, Borel \Ghom homotopy equivalence is a stronger notion than \Ghom homotopy equivalence.  Furthermore, 
letting $EG$ be a classifying space for $G$, Theorem~\ref{Thm:HE-triangle} and the preceding observations imply that the 
homotopy quotients $EG\times_G\tits X$ and $EG\times_G\absimp X$ are homotopy equivalent and therefore the $G$--equivariant 
(Borel) cohomology of $\tits X$ is isomorphic to that of $\absimp X$ and $\simp X$.  
\end{remark}

Theorem~\ref{Thm:MainCuboids} implies, in particular, that if $X$ admits a geometric action by a group $G$, then the \Ghom 
homotopy type of 
the Tits boundary is unaffected by replacing the standard CAT(0) metric (where all edge lengths are $1$) with a 
cuboid metric obtained by rescaling edges $G$--equivariantly. This is perhaps evidence in favor of the possibility that the 
results of Bestvina and Ontaneda about shape equivalence of visual boundaries \cite{Bestvina1996, Ontaneda} can be strengthened to results about \Ghom 
homotopy equivalence if one instead uses the Tits boundary:

\begin{question}\label{Ques:GhomotopyTits}
For which CAT(0) groups $G$ is it the case that any two CAT(0) spaces $X,Y$ on which $G$ acts geometrically have \Ghom 
homotopy equivalent Tits boundaries?

If $G$ acts geometrically on CAT(0) cube complexes $X,Y$, are the simplicial boundaries of $X,Y$ \Ghom 
homotopy equivalent?
\end{question}

\subsection{Quasiflats}\label{subsubsec:quasiflats}
A remarkable theorem of Huang \cite[Theorem 1.1]{Huang:quasiflats} says that if $X$ is a $d$--dimensional CAT(0) cube complex, then any quasi-isometric embedding $\reals^d\to X$ has image that is 
Hausdorff-close to a finite union of \emph{cubical orthants}, i.e. a convex subcomplex that splits as the product of rays.  This statement is important for the study of quasi-isometric rigidity in cubical groups, and it is natural to want to strengthen the 
statement to cover quasi-isometric embeddings $\reals^d\to X$, where $d$ is the largest dimension for which such a map exists, but $\dimension X$ is allowed to be larger than $d$ (although still 
finite).  Natural examples where this strengthened form of Huang's theorem is useful include right-angled Coxeter groups.

We believe that Theorem~\ref{Thm:HE-triangle} could be used as an ingredient in proving such a result.  Very roughly, the homotopy equivalence $\tits X\to\simp X$ can be used to 
produce, given a singular $d$--cycle $z$ in $\tits X$, a simplicial $d$--cycle $z'$ in $\simp X$ represented by a finite collection of standard orthants in $X$.  By the latter, we mean there is a 
finite collection of $d$--simplices in $\simp X$ whose union $S$ carries $z'$.  The proof of Theorem~\ref{Thm:HE-triangle},
specifically Proposition~\ref{Prop:diameter} and the nerve arguments in Sections~\ref{Sec:covering_tits_boundary} and~\ref{Sec:visible_part},
should provide enough metric control on the map $S\to \tits X$ 
to invoke a result of Kleiner--Lang~\cite{KleinerLang} to deduce the strengthened quasiflats theorem.  
Given that the 
strengthened quasiflats theorem has recently been established by 
Bowditch~\cite[Theorem 1.1]{Bowditch} and independently Huang--Kleiner--Stadler~\cite[Theorem 1.7]{HKS}, we have decided not 
to pursue the matter in this paper.

\subsection{Ingredients of our proof}\label{Sec:ToolsIngredients}

As mentioned above, our primary goal is to elucidate the relationships among three boundaries: the simplicial boundary $\simp X$, the simplicial Roller boundary $\absimp X$, and the Tits boundary $\tits X$. 

The first two boundaries that we study are combinatorial in nature. The primary difference is that the central objects in $\simp X$ are sequences of hyperplanes, whereas the central objects in $\absimp X$ are sequences of half-spaces. After developing a number of lemmas that translate between the two contexts, we prove the following result:

\begin{named}{Corollary~\ref{Cor:UBStoRollerSimp}}
The barycentric subdivision of $\simp X$ contains a canonical, $\Aut(X)$--invariant subcomplex $\simpubs X$. There are $\Aut(X)$--equivariant 
 simplicial maps $\URtri \from \simpubs  X\to \absimp X$ and $\RUtri \from  \absimp X \to \simpubs  X$, with the following properties:
\begin{enumerate}[$(1)$ ]
\item $\URtri \from \simpubs  X\to \absimp X$ is  surjective.
\item $\RUtri \from  \absimp X \to \simpubs  X$  is an injective section of $\URtri$.  
\item $\URtri$ is a homotopy 
equivalence with homotopy inverse  $\URtri$.
\end{enumerate}
\end{named}

 In fact, all of $\simp X$ deformation retracts to $\simpubs X$ in an 
 $\Aut(X)$--equivariant 
 way. See Remark~\ref{rem:alternate_strategy}.

In contrast to the the two combinatorial boundaries, the Tits boundary $\tits X$ is  inherently linked to the geometry of the CAT(0) metric on $X$.
To relate the Tits boundary to the other two boundaries, we need to combinatorialize it: that is, we need to cover $\tits X$ by a certain collection of open sets, and then study the nerve of the 
corresponding covering. Similarly, we cover each of $\simp X$ and $\absimp X$ by simplicial subcomplexes, and study the resulting nerves. We will use two different versions of the nerve theorem to 
show that a topological space (such as one of our boundaries) is homotopy equivalent to the nerve of a covering. The Open Nerve Theorem \ref{Thm:EquivariantOpenNerve} deals with open coverings and is 
originally due to Borsuk \cite{Borsuk}. The Simplicial Nerve Theorem \ref{Thm:EquivariantSimplicialNerve} deals with coverings by simplicial complexes and is originally due to Bj\"orner 
\cite{Bjorner:book}. In fact, since Theorem~\ref{Thm:HE-triangle} is a $G$--equivariant statement, we need equivariant versions of both theorems, which have not previously appeared in the literature 
to our knowledge. Consequently, Section~\ref{Sec:simplicial_complex} contains self-contained proofs of both theorems.

A central object in our construction of nerves is the Tits boundary realization of a Roller class, or point in  the Roller boundary (see Definitions~\ref{Def:Qy}, \ref{Def:Q'y} \ref{Def:Qcircum} 
). In \cite{FLM:RandomWalks}, a Roller class $v$, yields a convex, visually compact subset of the Tits boundary $ \tits X$. Employing the work of Caprace--Lytchak \cite{CapraceLytchak} and 
Balser--Lytchak \cite{BalserLytchak}, this set has radius at most $\pi/2$ and
a canonical circumcenter $\circum(v)$. A first approach might be to use these compacta to provide the $0$--skeleton of a nerve for $\tits X$ that will be homotopy equivalent to $\absimp X$. However, 
several issues arise. First, these associated convex closed sets must be made smaller so that their overlaps can be controlled. This is achieved by considering points in the Roller boundary versus 
their classes.

Secondly, the issue of \emph{$\ell^2$--visibility}, or rather invisibility must be addressed. A Roller class, (respectively a UBS) is $\ell^2$--visible if it is the intersection (respectively union) of the deep half-spaces (respectively hyperplanes) naturally associated to a CAT(0) geodesic ray $\alpha$.  As Figure~\ref{Fig:Staircases} shows, some Roller classes are not $\ell^2$--visible, and invisible classes cause headaches when trying to connect this data to the Tits boundary. In particular, if $v$ is an invisible Roller class, then we will have $Q(v) = Q(w)$ for any Roller class $w < v$.
To confront this challenge, we have to find a single  class which is maximal among the visible Roller classes represented by rays with endpoints in $Q(v)$. Proving the existence of such a maximal class requires finding a single geodesic ray that is diagonal, in the sense that its convex hull is the union of the convex hulls of two geodesic rays.
In particular, we need the following statement, which is of independent interest:

 \begin{named}{Proposition~\ref{Prop:build_CAT(0)}}
Let $\alpha$, $\beta$ be CAT(0) geodesic rays with $\alpha(0) = \beta(0) \in X^{(0)}$.
Suppose that $\W(\alpha)\cup\W(\beta)$ is commensurate 
with a UBS.  Then $a=\alpha(\infty)$ and $b=\beta(\infty)$ are joined by a unique geodesic 
$g$ in $\tits X$. Furthermore, any interior point $c$ of $g$ is represented by a CAT(0) geodesic ray $\gamma$ 
such that $\W(\gamma)  = \W(\alpha)\cup\W(\beta)$.
\end{named}

Next, 
we consider maps between $\tits X$ and $\guralnik X$ that  relate these boundaries.
 The first of these maps, called $\psi \from \tits X \to \guralnik X$, 
 is fairly easy to describe. A Tits point $a \in \tits X$ is represented by a CAT(0) geodesic ray $\alpha$ and  the intersection of half-spaces which are ``deep" yields a (principal) Roller class $\psi(a)$.  
 See Definition~\ref{Def:Psi} and Lemma~\ref{Lem:PsiAlternate}. The reverse map  $\varphi \from \guralnik X \to \tits X$ is somewhat more delicate: given a Roller class $v$, we start with the circumcenter $\circum(v) \in Q(v)$ and then perturb $\circum(v)$ to a nearby point $\varphi(v)$ that has slightly better properties. 
 See Definition~\ref{Def:phi_map} and Proposition~\ref{prop:phi_facts} for details, and note that the perturbation uses Proposition~\ref{Prop:build_CAT(0)} in a crucial way. The upshot is that $\varphi$ is a section of $\psi$ on exactly the $\ell^2$--visible classes: we have $v = \psi(\varphi(v))$ if and only if $v$ is $\ell^2$--visible (Lemma~\ref{Lem:L2Visibility}).

We can now construct nerves and prove results about them. We begin by defining the set $\maxvis(X)$ of all visible Roller classes that are maximal among all $\ell^2$--visible classes. Then, we construct a simplicial complex $\nerve_T$ whose vertex set is $\maxvis(X)$, with simplices corresponding to collections of Roller classes $v_i$ whose Tits boundary realizations $Q(v_i)$ all intersect. We check that this is indeed the nerve of a cover of $\tits X$ (Corollary~\ref{Cor:MtNerve}). Then, we prove:

\begin{named}{Theorem~\ref{Thm:first_HE}}
There is an \Authom homotopy equivalence from  the simplicial complex $\nerve_T$ to $\tits X$. \end{named}

The proof of Theorem~\ref{Thm:first_HE} requires an open thickening. The sets $Q(v)$ are closed (in fact, compact), and we do not have a version of the Nerve Theorem for closed covers. Thus we thicken up each compact set $Q(v)$ to an open set $U(v)$, in such a way that the intersection pattern of the open cover $\{U(v) :v\in\maxvis(X) \}$ is the same as that of the closed cover $\{Q(v):v\in\maxvis(X)\}$. The thickening procedure involves some delicate CAT(0) geometry; see Proposition~\ref{Prop:ThickNeighborhoodU}. As a result, we obtain an open cover of $\tits X$ whose nerve is isomorphic to the nerve of the closed cover, namely $\nerve_T$. Now, the Equivariant Open Nerve Theorem~\ref{Thm:EquivariantOpenNerve} gives an \Authom homotopy equivalence $\nerve_T \to \tits X$.
 
There are two reasons why the \Authom homotopy equivalence in Theorem~\ref{Thm:first_HE} is not $\Aut(X)$--equivariant. First, the Equivariant Open Nerve Theorem~\ref{Thm:EquivariantOpenNerve} does not provide $\Aut(X)$--equivariance on the nose. Second, while the collection of Tits boundary realizations $Q(v_i)$ is $\Aut(X)$--equivariant, the perturbed circumcenter map $\varphi$ might not be. For these reasons, \Authom homotopy equivalence is the strongest form of invariance that we can guarantee.
 
 In  Section~\ref{Sec:visible_part}, we use all of the above results  to complete the proof of Theorem~\ref{Thm:HE-triangle}. We already have an open cover of $\tits X$ whose nerve is $\nerve_T$. 
Working in $\absimp X$, we focus attention on a subcomplex $\vispart X$ whose vertex set corresponds to the $\ell^2$--visible Roller classes, and then construct a cover of $\vispart X$ by finite 
simplicial complexes $\{ \Sigma_v : v \in \maxvis X\}$. It is not too hard to check that the nerve $\nerve_\triangle$ of this cover is isomorphic to $\nerve_T$, which implies that $\vispart X$ and 
$\tits X$  are \Authom homotopy equivalent (Proposition~\ref{prop:tits_visible_homotopic}). To complete the construction of an \Authom homotopy equivalence $\RT 
\from \absimp X \to \tits X$, we build an \Authom deformation retraction  $\absimp^\eye X \to \absimp X$; see Proposition~\ref{prop:vis_all_equivalent}. While the construction of this retraction is 
mostly combinatorial, it relies on CAT(0) geometry and the perturbed circumcenter map $\varphi$ at one crucial step.

Finally,  the homotopy equivalence $\SR \from \simp X \to \absimp X$ is morally very similar to the simplicial map $\URtri \from \simpubs  X\to \absimp X$ constructed in
 Corollary~\ref{Cor:UBStoRollerSimp}. To make the argument precise, we construct isomorphic nerves of simplicial covers of 
 $\simp X$ and $\absimp X$ and apply the Equivariant Simplicial Nerve Theorem~\ref{Thm:EquivariantSimplicialNerve} one final time.

\begin{table}
  \begin{center}
    \caption{Table of notation.}
    \label{Tab:Notation}
    \begin{tabular}{| c | c|c|} 
 \hline
\textbf{Symbol} & \textbf{Meaning} & \textbf{Where}\\
 \hline
 $X$ & cube complex & Section~\ref{subsec:CAT(0)_cube_complex_defn} \\
 $d_1$ & $\ell^1$, combinatorial metric & Section~\ref{subsec:CAT(0)_cube_complex_defn} \\
 $d_X$ & $\ell^2$, CAT(0) metric & Section~\ref{subsec:CAT(0)_cube_complex_defn} \\
 $d_X^\rho$ & rescaled (cuboid) CAT(0) metric & Definition~\ref{Def:hyp_rescaling} \\
 $\frakH = \frakH(X)$ & set of half-spaces of $X$ & Section~\ref{Sec:half-spaces}  \\
  $\frakH^+_x$ & set of half-spaces containing $x$ & Section~\ref{Sec:Roller}  \\
$\W = \W(X)$ & set of hyperplanes (walls) of $X$ & Section~\ref{Sec:half-spaces} \\
$\W(x,y) $ & hyperplanes separating $x$ from $y$ & Section~\ref{Sec:half-spaces} \\    
$\I(x,y) $ & vertex interval from $x$ to $y$ & Definition~\ref{Def:Convex} \\
$m(x,y,z)$ & median in $\~X$ & Equation~\eqref{Eqn:Median} \\ 
$\hull(S)$ & cubical convex hull of set $S$ & Definition~\ref{Def:ConvexInX}  \\
$\J(x,y)$ & cubical interval from $x$ to $y$ & Definition~\ref{Def:IntervalInX}, \ref{Def:Q'y} \\
$\gamma$ & geodesic ray in $X$, either $\ell^1$ or $\ell^2$ & Definition~\ref{Def:CombGeodesic} \\
$\gamma(\infty)$ & endpoint of $\gamma$ at infinity & Lemma~\ref{Lem:CombLimit} \\
\hline
$\~X$ & Roller compactification of $X$ & Definition~\ref{Def:Roller}  \\
$\roller X$ & Roller boundary, $\~X \setminus X$ & Definition~\ref{Def:Roller}  \\
$x \sim y$ & finite distance in $\~X$ & Definition~\ref{Def:RollerEquivalence} \\
$\guralnik X$ & Guralnik quotient, $\roller X/ \sim$ &  Definition~\ref{Def:RollerEquivalence} \\
$\absimp X$ & simplicial Roller boundary & Definition~\ref{Def:SimplicialRoller} \\
\hline
$\insep{S}$ & inseparable closure of a set of hyperplanes & Section~\ref{Subsec:MinimalDominant} \\
$\U, \V$ & unidirectional boundary sets (UBS) & Definition~\ref{Def:UBS} \\
$\U \sim \V$ & finite symmetric difference & Definition~\ref{Def:PreceqMinimal} \\
$\simp X$ & simplicial boundary of $X$ & Definition~\ref{Def:simplicial_boundary} \\
$\~Y_\U$ & umbra of a UBS in $\roller X$ & Definition~\ref{Def:Umbra} \\
$Y_\U$ & principal class of umbra & Lemma~\ref{Lem:UmbraClass} \\
$\U_Y$ & UBS representing a Roller class $Y$ & Definition~\ref{Def:UBSofRoller} \\
$\W(\gamma)$ & UBS associated to a geodesic $\gamma$ & Definition~\ref{Def:L1visible} \\
$\mathcal{UBS}(X)$ & equivalence classes of UBSes & Theorem~\ref{Thm:UBStoRoller} \\
\hline
$\tits X$ & Tits boundary of $X$ & Definition~\ref{Def:tits_metric} \\
$D_a$ & Deep set of half-spaces for $a \in \tits X$ & Definition~\ref{Def:DeepSet} \\
$\psi(a)$ & Roller class associated to $a \in \tits X$ & Definition~\ref{Def:Psi} \\
$Q(y)$ & Tits boundary realization of $y \in \roller X$ & Definition~\ref{Def:Qy}, \ref{Def:Q'y} \\
$Q(v)$ & Tits boundary realization of $v \in \guralnik X$ & Definition~\ref{Def:Qcircum} \\
$\circum(v)$ & circumcenter of $Q(v)$ & Definition~\ref{Def:Qcircum} \\
$\varphi(v)$ & pseudocenter, perturbed circumcenter & Definition~\ref{Def:phi_map} \\
$M_v$ & maximal Roller class in $\psi(Q(v))$ & Lemma~\ref{Lem:m_v} \\
$\visi(X)$ & $\ell^2$--visible Roller classes, $\psi(\tits X)$ & Definition~\ref{Def:L2visible} \\
$\maxvis(X)$ & maximal visible Roller classes & Definition~\ref{Def:MaxVis} \\
\hline
$U(v)$ & open neighborhood of $Q(v)$ &  Proposition~\ref{Prop:ThickNeighborhoodU} \\
$\nerve_T$ & nerve of cover of $\tits X$ by $Q(\maxvis(X))$ & Corollary~\ref{Cor:MtNerve} \\
$\mathcal{L}_T$ & nerve of cover of $\tits X$ by $U(\maxvis(X))$ & Theorem~\ref{Thm:first_HE} \\
$\vispart X$ & visible subcomplex of $\absimp X$ & Definition~\ref{Def:VisiPart} \\
$\Sigma_v$ & subcomplex of $\vispart X$ for $v \in \visi(X)$ & Definition~\ref{Def:VisiPart} \\
$\nerve_\triangle$ & nerve of simplicial cover of $\vispart X$ by $\Sigma_v$'s & Definition~\ref{Def:SimplicialNerveVisi} \\
\hline
\end{tabular}
  \end{center}
\end{table}

\subsection{Expository content}
Part of our goal in this paper is to provide exposition of varying aspects of cubical theory. 
 CAT(0) cube complexes are ubiquitous objects that have been well-studied in many different guises. Accordingly, there are several different 
viewpoints, and various important technical statements are stated and proved in a variety of different ways throughout the literature.  Therefore, we have endeavored to give a self-contained 
discussion of CAT(0) cube complexes combining some of these viewpoints.  

Also, throughout the paper, we make heavy use of the \emph{nerve theorem}, for both open and simplicial covers.  Results of 
this sort were originally proved by Borsuk~\cite{Borsuk}, and are now widely used in many slightly different forms.  The 
versions in the literature closest to what we need here are for open covers of paracompact spaces \cite{Hatcher} and for 
(possibly locally infinite) covers of simplicial complexes by subcomplexes \cite{Bjorner:homotopy,Bjorner:book}.  The 
recent paper~\cite{Ramras:nerves} contains generalizations of both statements of the nerve theorem.
Since we could not find a written account of these results incorporating an additional conclusion about 
\Ghom equivalences in the presence of a group action, we have given self-contained proofs 
(based on arguments in~\cite{Hatcher,Bjorner:book}) in Section~\ref{Sec:simplicial_complex}.  

\subsection{Section Breakdown}
Section~\ref{Sec:simplicial_complex} establishes some language about simplicial complexes and proves equivariant versions of two nerve theorems.  Section~\ref{Sec:Background} is devoted to background on CAT(0) cube complexes and the Roller boundary.  In 
Section~\ref{subSec:cuboid_defs}, we discuss CAT(0) cuboid complexes coming from an admissible rescaling of the hyperplanes.  Section~\ref{Sec:SimplicialRoller} introduces the simplicial Roller 
boundary. In Section~\ref{Simplicial Boundary}, we introduce UBSes and the simplicial boundary, and relate these to Roller classes and the simplicial Roller boundary.  In Section~\ref{Sec:Tits}, 
we introduce the Tits boundary and prove some technical results relating CAT(0) geodesic rays to Roller classes and UBSes.  We apply these in Section~\ref{Sec:TitsRollerConnections} to analyze the 
realizations of Roller classes in the Tits boundary.  These results are in turn used in Section~\ref{Sec:covering_tits_boundary} to prove that the Tits boundary is homotopy equivalent to a 
simplicial complex $\nerve_T$ arising as the nerve of the covering by open sets associated to certain Roller classes.  In Section~\ref{Sec:visible_part}, we realize this nerve as the nerve of a covering of the 
simplicial Roller boundary by subcomplexes associated to Roller classes, and deduce Theorems~\ref{Thm:HE-triangle} and~\ref{Thm:MainCuboids}.  

See Table~\ref{Tab:Notation} for a summary of the notation used in this paper.

\subsection*{Acknowledgments}
We are grateful to Craig Guilbault, Jingyin Huang, Dan Ramras, and Kim Ruane for some helpful 
discussions. In particular, we thank Ramras for a correction, for pointing out the reference \cite{Bjorner:homotopy}, and 
for the observation about Borel \Ghom homotopy equivalence (Remark~\ref{rem:borel-intro}). 
We thank the organizers of the conference ``Nonpositively curved groups on the Mediterranean'' in May of 2018,  where the three of us began collaborating as a unit.
Fern\'os was partially supported by NSF grant DMS--2005640.
Futer was partially supported by NSF grant DMS--1907708.
Hagen was partially supported by EPSRC New Investigator Award EP/R042187/1.

\section{Simplicial complexes and nerve theorems}\label{Sec:simplicial_complex}
Throughout the paper, we will make use of simplicial complexes and assume that the reader is familiar with these objects.  
For clarity, we recall the definition and describe two topologies on a simplicial complex.
Then, in Section~\ref{Sec:equivariance}, we prove two group-equivariant homotopy equivalence theorems about nerves of covers.

A \emph{$k$--simplex} is the set 
$$\sigma=\left\{a_0e_0+\cdots+a_ke_k \: \Bigg\vert \: \sum_{i=0}^ka_k=1\ \text{and}\ a_i\geq 0\ \text{for all}\ i\right\},$$
where $e_0,\ldots,e_k$ are the standard basis vectors in $\reals^{k+1}$.  A \emph{face} of $\sigma$ is a $j$--simplex 
obtained by restricting all but $j+1$ of the $a_i$ to $0$.  Note that $\sigma$ 
has a CW complex structure where the $0$--cells 
are the $0$--dimensional faces and, more generally, the $j$--cells are the $j$--dimensional faces.

A \emph{simplicial complex} is a CW complex $\nerve$ whose closed cells are simplices, such that
\begin{itemize}
     \item each (closed) simplex is embedded in $\nerve$, and
     \item if $\sigma,\tau$ are simplices, then $\sigma\cap \tau$ is either empty or a face of both $\sigma$ and 
$\tau$.  In particular, simplices with the same $0$--skeleton are equal.
\end{itemize}
We often refer to $0$--simplices of $\nerve$ as \emph{vertices} and $1$--simplices as \emph{edges}.   

As a CW complex, $\nerve$ is endowed with the weak topology:

\begin{definition}[Weak topology]\label{Def:WeakTopology}
    Let $\nerve$ be a simplicial complex. The \emph{weak topology} $\mathcal T_w$ on  $\nerve$ is characterized by the property that a set $C \subset \nerve$ is closed if and only if $C \cap \sigma$ 
is closed for every simplex $\sigma \subset \nerve$.  
\end{definition}

In some of the arguments in this section, it will be more convenient to work with the metric topology on the simplicial complex $\nerve$.

\begin{definition}[Metric topology]\label{Def:metric_topology}
 Let $\nerve$ be a simplicial complex.  Let $V$ be the real vector space consisting of functions $f \from \nerve^{(0)}\to\reals$ such that $f(v)\neq 0$ for finitely many $v\in\nerve^{(0)}$.  
 Then there is a 
canonical inclusion $\nerve^{(0)}  \hookrightarrow V$, where every vertex $v \in \nerve^{(0)}$ maps to the corresponding Dirac function $\delta_v \in V$.  This map extends affinely over 
simplices endowed with barycentric coordinates to give an inclusion $\nerve \hookrightarrow V$. 

Equip $V$ with the $\ell^2$ norm $$\|f\|^2=\sum_{v\in\nerve^{(0)}}f^2(v),$$ which is well-defined since every $f\in V$ is finitely supported.  The restriction of the resulting metric topology on $V$ 
to the subspace $\nerve$ is the \emph{metric topology on $\nerve$}, denoted $\mathcal T_m$.
\end{definition}

The metric topology is coarser than the weak topology, hence the identity map $\id_{\nerve} \from (\nerve,\mathcal T_w)\to(\nerve,\mathcal T_m)$ is always continuous. When a simplicial complex 
$\nerve$ is locally infinite, the inverse map $(\nerve,\mathcal T_m)\to(\nerve,\mathcal T_w)$ is not continuous.  However, Dowker proved  that $\id_{\nerve}$ is a homotopy equivalence even when 
$\nerve$ is locally infinite~\cite{Dowker}. 
 Since our interest is in the homotopy type of $\nerve$, Dowker's theorem will be very useful.

\begin{convention}\label{Conv:WeakTopology}
Unless stated otherwise, a simplicial complex $\nerve$ is presumed to have the weak topology $\mathcal T_w$. The metric topology will only be needed in Lemma~\ref{Lem:homotopic_maps} and Theorem~\ref{Thm:EquivariantOpenNerve}, and never afterward.
\end{convention}

Some of the simplicial complexes used later will  arise from partially ordered sets, as follows.

\begin{definition}[Simplicial realization]\label{Def:simplicial_realisation}
Given a partially ordered set $(P,\leq)$, there is a simplicial complex $S$ in which the $k$--simplices are 
the $(k+1)$--chains in $(P,\leq)$, and the face relation is determined by containment of chains.  We call $S$ the \emph{simplicial realization} of the partially ordered set 
$(P,\leq)$. 
\end{definition}

In several other places, we will work with simplicial complexes arising as \emph{nerves} of coverings of topological spaces:

\begin{definition}[Nerve]\label{Def:nerve}
 Let $Y$ be a topological space and let $\mathcal U=\{Y_i\}_{i\in I}$ be a covering of $Y$, i.e. a family of subsets with $Y=\bigcup_{i\in I}Y_i$.  The \emph{nerve} of $\mathcal U$ is the simplicial 
complex $\nerve$ with a vertex $v_i$ for each $Y_i$, and with 
an $n$--simplex spanned by $v_{i_0},\ldots,v_{i_n}$ whenever
$\bigcap_{j=0}^nY_{i_j}\neq\emptyset$.  
\end{definition}

Note that we have defined nerves for arbitrary coverings of arbitrary topological spaces.  In practice, we will restrict $Y$ and $\mathcal U$: either $Y$ is a paracompact space and 
$\mathcal U$ is an open covering, or $Y$ is itself a simplicial complex and $\mathcal U$ is a covering by subcomplexes.  These assumptions provide the settings for the nerve theorems, which relate the homotopy type of $Y$ to that of the nerve of $\mathcal U$.

\subsection{Equivariant nerve theorems}\label{Sec:equivariance}

In this section, we prove two flavors of nerve theorem that are needed in our proofs of homotopy equivalence. Theorem~\ref{Thm:EquivariantOpenNerve} is a group-equivariant version of the classical nerve theorem for open coverings, originally due to Borsuk \cite{Borsuk}. Similarly, Theorem~\ref{Thm:EquivariantSimplicialNerve} is an equivariant version of the nerve theorem for simplicial complexes, which is due to Bj\"orner \cite[Theorem 10.6]{Bjorner:book}.

We need the following standard fact:

\begin{lemma}[Recognizing homotopic maps]\label{Lem:homotopic_maps}
 Let $Y$ be a topological space and $\nerve$  a simplicial complex  endowed with the metric topology. 
 Let 
$f_0,f_1 \from Y\to (\nerve,\mathcal T_m)$ be continuous maps.  Suppose that, for all $y\in Y$, 
there exists a closed simplex $\sigma$ of $\nerve$ such that 
$f_0(y), f_1(y)\in\sigma$.  Then $f_0$ and $f_1$ are homotopic via a straight line homotopy.
\end{lemma}

\begin{proof}
By abuse of notation, we identify $\nerve$ with the embedding $\nerve \hookrightarrow V$ described in Definition~\ref{Def:metric_topology}. Then $f_0, f_1 \from Y \to V$ are continuous functions. For $y \in Y$ and $t \in [0,1]$, define the affine combination $f_t(y) = 
(1-t)f_0(y) + t f_1(y)$. Then 
$(y,t) \mapsto f_t(y)$ is a continuous mapping from $Y \times [0,1]$ to $V$. Furthermore, for every $y \in Y$, the image $f_t(y) = (1-t)f_0(y) + t f_1(y)$ belongs to the same simplex $\sigma \subset \nerve$ that contains $f_0(y)$ and $f_1(y)$. Thus we get a straight line homotopy $f_t \from Y \to \nerve$.
\end{proof}

\begin{theorem}[Equivariant open nerve theorem]\label{Thm:EquivariantOpenNerve}
Let $Y$ be a paracompact space, on which a group $G$ acts by homeomorphisms.  Let $ I$ be a set equipped with a left $G$--action, and let 
$\U=\{Y_i\}_{i\in I}$ be an open covering of $Y$ with the following properties:
\begin{itemize}
     \item $gY_i=Y_{gi}$ for all $g\in G$ and $i\in I$;
     \item for any finite $ F\subset I$, the intersection $\bigcap_{i\in F}Y_i$ is either empty or 
contractible.
\end{itemize}
Let $\nerve$ be the nerve of the covering $\U$.  

Then $G$ acts on $\nerve$ by simplicial automorphisms, and there is 
a \Ghom homotopy equivalence $f \from Y\to(\nerve,\mathcal T_m)$.
Consequently, there is also a 
 \Ghom homotopy equivalence $ Y\to (\nerve,\mathcal T_w)$.
\end{theorem}

\begin{proof}
If the $G$--action is trivial, this result is the classical nerve theorem, proved for instance in Hatcher~\cite[Proposition 4G.2, Corollary 4G.3]{Hatcher}.
We adapt Hatcher's line of argument, accounting for the $G$--action where necessary.  Until we say otherwise, at the very end of the proof, the nerve $\nerve$ will be equipped with the metric 
topology $\mathcal T_m$.

\smallskip
\textbf{Simplices comprising the nerve:}  Let $\mathcal F$ be the set of 
finite subsets $ F\subset I$ 
such that $\bigcap_{i\in F}Y_i\neq\emptyset$.  Then, by Definition~\ref{Def:nerve}, every finite set $F \in \mathcal F$ can be canonically identified with the set of vertices of a simplex $\sigma(F) \subset \nerve$.
The partial order on $\mathcal F$ given by set inclusion corresponds to the face relation on $\nerve$.
The action of $G$ on $\U$  induces an action on $\mathcal F$, hence a simplicial action on $\nerve$.

\smallskip
\textbf{The space $\triangle Y$ and its quotient maps:}  Let $\sigma= \sigma(F)$ be a simplex 
of $\nerve$, corresponding to a finite set $F \in \mathcal F$. 
Let $U_F = \bigcap_{i \in F} Y_i$, a contractible open set, and consider the product $U_F \times \sigma(F) \subset Y \times \nerve$. Then, define
$$ \triangle Y = \bigcup_{F \in \mathcal F} \big( U_F \times \sigma(F) \big) \: \subset \: Y \times \nerve.$$
 The diagonal $G$--action on $Y \times \nerve$ induces a $G$--action on $\triangle Y$.
The projections from $Y \times \nerve$ to its factors restrict to $G$--equivariant projection maps
$$
p \from \triangle Y \to Y, \qquad q \from \triangle Y \to \nerve. $$

\smallskip
\textbf{The homotopy equivalence $\triangle Y \to \nerve$:}  
 Given a point $x \in \nerve$ contained in the interior of a simplex $\sigma(F)$, the fiber  $q^{-1}(x) = U_F \times \{x\}$ is contractible by hypothesis.  Thus, by \cite[Proposition 4G.1]{Hatcher}, 
the projection $q \from \triangle Y \to \nerve$ is a homotopy equivalence.

\smallskip
\textbf{The fiber $p^{-1}(y)$:}
Fix $y\in Y$, and consider the fiber 
$p^{-1}(y)$. Let $(y,x) \in p^{-1}(y) \subset Y \times \nerve$. Let $\sigma(F)$ be a simplex of $\nerve$ containing $x = q(y,x)$. The vertices of $\sigma$ correspond to the elements of $F$, hence to subsets $Y_i$ for $i \in F$.
We endow $\sigma(F)$ with barycentric coordinates, so that
 $x = \sum_{i \in F} \beta_i Y_i$ where $\sum \beta_i = 1$. Thus
 $$p^{-1}(y) = \left\{ \Big(y, \, \sum_{i \in F} \beta_i Y_i \Big)  \: \Bigg\vert \:\: y \in U_F , F \in \mathcal F \right\} = \{y\} \times \Sigma_y
  $$
for a simplicial subcomplex $\Sigma_y \subset \nerve$. Since the projection $p \from \triangle Y \to Y$ is $G$--equivariant, we have
$$ \{gy\} \times g \Sigma_y \: = \: g \big(  \{y\} \times \Sigma_y \big) \: = \: g p^{-1}(y) = p^{-1}(gy) \: = \: \{gy\} \times  \Sigma_{gy},$$
and in particular $g \Sigma_y = \Sigma_{gy}$.

For any pair of points $x,x' \in \Sigma_y$, belonging to simplices $\sigma(F), \sigma(F')$ of $\nerve$, we have $y \in U_F \cap U_{F'} = U_{F \cup F'}$. Thus there is also a simplex $\sigma(F \cup F')$ containing both $\sigma(F)$ and $\sigma(F')$ as faces. It follows that $x,x'$ are connected by an affine line segment in the common simplex $\sigma(F \cup F')$. 
In other words, $\Sigma_y$ is convex.

\smallskip
\textbf{The section $s \from Y \to \triangle Y$:}  
Since $Y$ is paracompact, there is a partition of unity $\{\phi_\alpha\}_{\alpha \in A}$ subordinate to $\U = \{Y_i\}_{i \in I}$. That is,
 for each $\alpha \in A$, we have $\supp(\phi_\alpha)  \subset Y_{i(\alpha)}$ for some $i(\alpha) \in I$.
  Given $y\in Y$, let $A(y) = \{ \alpha \in A : y \in \supp(\phi_\alpha) \}$, a finite set.
Then, define
$$f(y) = \sum_{\alpha \in A(y)} \phi_\alpha (y) Y_{i (\alpha)} \: \in \: \Sigma_y
\qquad \text{and} \qquad
s(y) = \big( y, f(y) \big)
 \: \in \: 
p^{-1}(y) .$$ 
We will eventually show that $f \from Y \to \nerve$ is the homotopy equivalence claimed in the theorem statement. For now, we focus on $s$.

\smallskip
\textbf{Checking that $s$ is a homotopy inverse of $p$:}
Since $s(y) \in p^{-1}(y)$, we have $p \circ s = id_Y$ by construction.
We claim that the other composition $s \circ p$ is homotopic to $id_{\triangle Y}$. 
To prove this, consider the following pair of functions from  $\triangle Y \to \nerve$:
$$q_0(y,x) = q(y,x) = x, \qquad q_1(y,x) = f(y).$$
For every pair $(y, x) \in \triangle Y$, the image points $x = q_0(y,x)$ and $x' = q_1(y,x)  = f(y)$ are both contained in $\Sigma_y$. Above, we have checked both image points $x,x'$ are contained in a common simplex $\sigma(F \cup F') \subset \Sigma_y \subset \nerve$. Recall that we are working with the metric topology on $\nerve$. Thus, by Lemma~\ref{Lem:homotopic_maps}, $q_0$ is homotopic to $q_1$ via a straight line homotopy $q_t$.  For every $y$, this straight-line homotopy runs through $\Sigma_y$.

We can now define a homotopy $h_t \from \triangle Y \to Y \times \nerve$ as follows:
$$
h_t(y,x) = (y, q_t(y,x)).
$$
Observe that $p \circ h_t$ is continuous because $p \circ h_t(y,x) = y$, and $q \circ h_t$ is continuous because $q \circ h_t(y,x) = q_t(y,x)$. Thus, by the universal property of the product topology, $h_t$ is continuous.
Restricting $t$ to be $0$ or $1$ gives
$$
h_0(y,x) = (y, q_0(y,x)) = (y,x), \quad h_1(y,x) = (y, q_1(y,x)) = (y, f(y)) = s\circ p(y,x).
$$
In fact, for every $(y,x)$, the path $t \mapsto q_t(y,x)$ has image in $\Sigma_y$, hence
$h_t(y,x)$ has image in $\{y\} \times \Sigma_y \subset \triangle Y$. Thus $h_t$ is a homotopy from $s\circ p$ to $id_{\triangle Y}$, and $s$ is a homotopy inverse of $p$.

\smallskip
\textbf{A \Ghom homotopy equivalence, for both topologies on $\nerve$:}  Thus far, we have homotopy equivalences $s \from Y\to 
\triangle Y$ and $q \from \triangle Y\to\nerve$.  Hence the composition $f = q\circ s \from Y \to \nerve$ is a homotopy equivalence as well.  

It remains to check that $f$ is a \Ghom homotopy equivalence.  We saw above that $s$ is a homotopy inverse for $p$; in particular $p$ is an isomorphism in the homotopy category of spaces, with 
inverse $s$.  Moreover, each $g\in G$ acts as a homeomorphism, and in particular an isomorphism in the homotopy category, on both $Y$ and $\triangle Y$.  By construction, we have $g\circ p = p\circ 
g$.  So, letting $\simeq$ denote homotopy of maps, and using that $s\circ p\simeq Id_{\triangle Y}$ and $p\circ s\simeq Id_{Y}$, we have 
$$s\circ g\simeq s\circ g\circ(p\circ s) \simeq 
s\circ(p\circ g)\circ s\simeq g\circ s,$$
which is to say that $s$ is a \Ghom homotopy equivalence.  Since $q$ is a $G$--equivariant homotopy equivalence, it is a \Ghom homotopy equivalence, and hence $f = q\circ s \from Y \to \nerve$ is as 
well.  Since $\id_{\nerve} \from (\nerve,\mathcal 
T_w)\to(\nerve,\mathcal T_m)$ is a $G$--equivariant homotopy equivalence by~\cite{Dowker}, $f$ is thus a \Ghom homotopy equivalence for either topology on $\nerve$.
%
%
%
%
\end{proof}

We are now finished with having the metric topology on our nerves, 
and work only with the weak topology in the remainder of the paper.

The following result is stated (without the group action) as 
\cite[Theorem 10.6]{Bjorner:book}.  The proof given there assumes that $\V = \{\Sigma_i\}$ is a locally finite cover, and a 
proof in full generality is given in \cite[Lemma 1.1]{Bjorner:homotopy}.  (See \cite{Ramras:nerves} for a more general 
result.)  Since we need to adapt the proof slightly to account for the group action, we write down a proof using 
Theorem~\ref{Thm:EquivariantOpenNerve}, without any assumptions on local finiteness of the cover $\{\Sigma_i\}$. 

\begin{theorem}[Equivariant simplicial nerve theorem]\label{Thm:EquivariantSimplicialNerve}
Let $Y$ be a simplicial complex, and let $G$ be a group 
acting on $Y$ by simplicial automorphisms.   Let $ I$ be a set equipped with a left $G$--action, and let 
$\V =\{\Sigma_i\}_{i\in I}$ be a covering of $Y$ by subcomplexes, with the following properties:
\begin{itemize}
     \item $g \Sigma_i = \Sigma_{gi}$ for all $g\in G$ and $i\in I$;
     \item for any finite $ F\subset I$, the intersection $\bigcap_{i\in F} \Sigma_i$ is either empty or 
contractible.
\end{itemize}
Let $\nerve$ be the nerve of the covering $\V$. 
Then there is a 
\Ghom homotopy equivalence $f \from Y\to (\nerve,\mathcal T_w)$.
\end{theorem}

\begin{proof}
Let $Y'$ be the barycentric subdivision of $Y$. We think of $Y$ and $Y'$ as the same underlying topological space, with different simplicial structures. 

Every subcomplex $\Sigma \subset Y$ can be viewed as a subcomplex of $Y'$. Indeed, $\Sigma \subset Y$ is a \emph{full} subcomplex of $Y'$, in the following sense: for every simplex $\tau \subset Y'$ 
whose vertices belong to $\Sigma$, it follows that $\tau \subset \Sigma$. This holds because the vertices of $\tau$ correspond to simplices of $Y$; if these vertices belong to $\Sigma$, then so do the 
corresponding cells of $Y$, hence $\tau \subset \Sigma$.

\smallskip
\textbf{Complementary complexes and open neighborhoods:} 
For a subcomplex $\Sigma \subset Y$, 
define the complementary complex $T_\Sigma$ to be the union of all closed simplices of $Y'$ disjoint from $\Sigma$. Then $T_\Sigma$ is also a full subcomplex of $Y'$. Indeed, consider a simplex $\tau \subset Y'$ whose vertices are disjoint from $\Sigma$. Then $\tau \cap \Sigma$ cannot contain any faces of $\tau$, hence $\tau \cap \Sigma = \emptyset$. Since $T_\Sigma$ is a subcomplex of $Y'$, it is closed (in the weak topology).

Given a subcomplex $\Sigma \subset Y$, we define an open neighborhood $U_\Sigma = Y \setminus T_\Sigma$. Thus $U_\Sigma$ is the union of all the open simplices of $Y'$ whose closures intersect $\Sigma$.

For every subcomplex $\Sigma_i$, where $i \in I$, we write $U_i = U_{\Sigma_i}$.  Observe that $U_{g i}=gU_i$ for all $g\in G$ and $i \in I$. We will prove the theorem by replacing the subcomplex cover $\V = \{ \Sigma_i \}_{i \in I}$ by the open cover
$\U = \{ U_i \}_{i \in I}$.

\smallskip
\textbf{Same nerve:} Let $\Sigma_\alpha$ and $\Sigma_\beta$ be a pair of subcomplexes of $Y$, not necessarily belonging to $\V$. We claim that $\Sigma_\alpha \cap \Sigma_\beta = \emptyset$ if and only if $U_{\Sigma_\alpha} \cap U_{\Sigma_\beta} = \emptyset$. The ``if'' direction is obvious. For the ``only if'' direction, suppose that $\Sigma_\alpha \cap \Sigma_\beta = \emptyset$. Then $T_{\Sigma_\alpha} \!\cup T_{\Sigma_\beta} = T_{\Sigma_\alpha \cap \Sigma_\beta} = T_\emptyset = Y'$, hence $U_{\Sigma_\alpha} \!\cap U_{\Sigma_\beta} = \emptyset$, as desired.

Now, let $F \subset I$ be a finite set. By induction on $|F|$, combined with the argument of the above paragraph, we see that 
$\bigcap_{i\in F} \Sigma_i = \emptyset$ if and only if $\bigcap_{i\in F} U_{\Sigma_i} = \emptyset$. Hence the open cover $\U = \{ U_i \}_{i \in I}$ and the subcomplex cover $\V = \{ \Sigma_i \}_{i \in I}$ have the same nerve $\nerve$.

\smallskip
\textbf{Same homotopy type:} For every subcomplex $\Sigma \subset Y$, we claim that the open neighborhood $U_\Sigma$ deformation retracts to $\Sigma$. This is a standard fact about simplicial complexes, proved e.g.\ in Munkres \cite[Lemma 70.1]{Munkres:AlgebraicTopology}. The proof constructs a straight-line homotopy in every simplex $\sigma \subset Y'$ that belongs to neither $\Sigma$ nor $T_\Sigma$. The proof uses the fullness of $\Sigma$ and $T_\Sigma$ in $Y'$, but does not depend on any finiteness properties of these complexes.

Now, let $F \subset I$ be a finite set such that $\bigcap_{i\in F} \Sigma_i \neq \emptyset$. Define $\Sigma_F = \bigcap_{i\in F} \Sigma_i$, and recall that by hypothesis, $\Sigma_F$ is contractible. Since $U_{\Sigma_F} =  \bigcap_{i\in F} U_{\Sigma_i}$, and $U_{\Sigma_F}$ deformation retracts to $\Sigma_F$, it follows that $\bigcap_{i\in F} U_i = \bigcap_{i\in F} U_{\Sigma_i}$ is also contractible.

\smallskip
\textbf{Conclusion:} By a theorem of Bourgin \cite{Bourgin}, the simplicial complex $Y$ is paracompact (with the weak topology $\mathcal T_w$).
We have shown that the open cover $\U = \{ U_i \}_{i \in I}$ and the subcomplex cover $\V = \{ \Sigma_i \}_{i \in I}$ have the same nerve $\nerve$. We have checked that $U_{g i}=gU_i$ for all $g\in G$ and $i \in I$. Furthermore, for every finite set $F \subset I$, the intersection $\bigcap_{i\in F} U_i $ is either empty or contractible. Thus, by Theorem~\ref{Thm:EquivariantOpenNerve}, we have a \Ghom homotopy equivalence $f \from Y\to (\nerve,\mathcal T_w)$.
\end{proof}

\begin{remark}\label{rem:borel}
The proof of Theorem~\ref{Thm:EquivariantOpenNerve} yields a slightly stronger conclusion than \Ghom homotopy equivalence.  
Indeed, we produced $G$--equivariant homotopy equivalences $\triangle Y\to Y$ and $\triangle Y\to\nerve$, and \Ghom homotopy 
equivalences between $Y$ and $\nerve$ were then obtained by composing one with a homotopy inverse of the other.  The same is 
true for Theorem~\ref{Thm:EquivariantSimplicialNerve}, since it is proved by applying Theorem~\ref{Thm:EquivariantOpenNerve} 
to a $G$--invariant open cover.  Hence, under the hypotheses of either theorem, we have actually shown that $Y$ and $\nerve$ 
are Borel \Ghom homotopy equivalent (see Remark~\ref{rem:borel-intro}).     
\end{remark}

\section{Cube complexes and the Roller boundary}
\label{Sec:Background}

This section recalls some background about CAT(0) cube complexes and their Roller boundaries. Many facts about CAT(0) cube 
complexes, median graphs, wallspaces, and related structures occur in various 
 places in the literature, in many equivalent 
guises. We have endeavored to connect some of the perspectives, in part because we will need to use these connections in the sequel. We refer the reader to \cite{Sageev_95} and \cite{Wise:Riches2Raags} for more background.  

\subsection{CAT(0) cube complexes and metrics}\label{subsec:CAT(0)_cube_complex_defn}
A \emph{cube} is a copy of a Euclidean unit cube $[-\frac12,\frac12]^n$ for $0\leq n<\infty$.  A \emph{face} of 
$[-\frac12,\frac12]^n$ is a subspace obtained by restricting some of the coordinates to $\pm \frac12$.  A \emph{cube 
complex} is a CW complex whose cells are cubes and whose attaching maps restrict to isometries on faces.  A cube complex 
with embedded cubes is \emph{nonpositively curved} if, for all vertices $v$ and all edges $e_1,\ldots,e_k$ incident to $v$ such that 
$e_i,e_j$ span a $2$--cube for all $i\neq j$, we have that $e_1,\ldots,e_k$ span a unique $k$--cube. 
 If $X$ is 
nonpositively-curved and simply connected, then $X$ is a \emph{CAT(0)} cube complex.  The \emph{dimension} of $X$, denoted $\dimension X$, is the 
supremum of the dimensions of cubes of $X$.

Throughout this paper, $X$ will denote a finite-dimensional CAT(0) cube complex.   We let $\Aut(X)$ denote the group of 
cellular automorphisms of $X$.  We do not assume that $X$ is locally finite. 

The cube complex $X$ carries a metric $\dist_X$ such that $(X,\dist_X)$ is a CAT(0) space, the restriction of $\dist_X$ to 
each cube is the Euclidean metric on a unit cube, and each cube is geodesically convex (see~\cite{Gromov,BridsonHaefliger}). 
 We refer to $\dist_X$ as the \emph{$\ell^2$ metric} or the \emph{CAT(0) metric} on $X$.  
 
One can view $X^{(1)}$ as a graph whose vertices are the $0$--cubes and whose edges are the $1$--cubes.  (We 
often refer to $0$--cubes as vertices of $X$, and $1$--cubes as edges.)  We let $d_1$ denote the \emph{$\ell^1$ metric} on the 
vertex set $X^{(0)}$, which is the restriction of the usual graph metric on $X^{(1)}$. We refer to $d_1$ as the \emph{combinatorial metric} on $X^{(0)}$. A \emph{combinatorial geodesic} between $x,y \in X^{(0)}$ is an edge path in $X^{(1)}$ that realizes $d_1(x,y)$.

\subsection{Hyperplanes, half-spaces, crossing, and separation}\label{Sec:half-spaces}

Let $[- \frac{1}{2}, \frac{1}{2}]^n$ be an $n$-dimensional cube. A \emph{midcube} of $[- \frac{1}{2}, \frac{1}{2}]^n$ is the subset obtained by restricting one coordinate to be $0$.

A \emph{hyperplane} of a CAT(0) cube complex $X$ is a connected subset whose intersection with each cube is either a midcube 
of that cube, or empty. The open $1/2$--neighborhood of a hyperplane $\hat h$ in 
the metric $d_X$ is denoted $N(\hat h)$ and is called the 
 \emph{open 
hyperplane carrier} of $\hat h$. It is known that every open carrier $N(\hat h)$ is geodesically convex in $d_X$. Furthermore, $X \setminus N(\hat 
h)$ consists of two connected components, each of which is also convex
 \cite[Theorem 4.10]{Sageev_95}.
 
 A component of $X \setminus N(\hat 
h)$ is called a \emph{CAT(0) half-space}. 
 The intersection of one of these components with the vertex set $X^{(0)}$ is called a \emph{vertex half-space}. 
 The two vertex half-spaces associated to a hyperplane $\hat h$ are denoted 
$h,h^*$. Note that $\hat{h^*}= \hat h$.
 Given a vertex half-space $h$,  the corresponding CAT(0) half-space is the union of all cubes of $X$ whose vertices lie in $h$.
 We will use the unmodified term \emph{half-space} to mean either a CAT(0) half-space or a vertex half-space when the meaning is clear from context.

 The collection of all vertex half-spaces is denoted by $\frakH$, or $\frakH(X)$ if we wish to specify the space $X$.    
 If $h\in\frakH$, then $h^* = X^{(0)} \setminus h$ is exactly the complementary half-space associated to the same hyperplane 
$\hat h$.

The collection of all hyperplanes of $X$ is denoted $\W$ (for ``walls''). 
 Generally speaking, calligraphic letters denote collections of hyperplanes, while gothic letters denote collections of  
half-spaces.
There is a two-to-one map $\frakH \to \W$, namely $h \mapsto \hat h$. Given a subset $\mathcal{S}\subset \W$, an \emph{orientation} on $\mathcal S$ is a choice of a lift $\mathcal S \to \frakH$.

A  pair of half-spaces $h,k\in\frakH$ are called \emph{transverse} (denoted $h\pitchfork k$) if the four intersections $h\cap 
k$, $h\cap k^*$, $h^*\cap  k$ and $h^*\cap k^*$ are nonempty. In terms of the underlying hyperplanes $\hat h$ and $\hat k$, 
transversality is equivalent to the condition that $\hat h \neq \hat k$ and $\hat h\cap \hat k\neq\emptyset$. In this case, 
we also write $\hat h\pitchfork \hat k$.  We sometimes say that transverse hyperplanes \emph{cross}.  

Given subsets $A,B\subset X$, we say that $\hat h$ \emph{separates} $A,B$ if $A$ 
is contained in one CAT(0) half-space associated to $\hat h$, and $B$ is 
contained in the other.  We will usually be interested in situations where $A,B$ 
are subcomplexes, sets 
of vertices, or hyperplanes, in which case this notion is equivalent to another notion of separation:  namely that $A,B$ lie in distinct components of $X \setminus \hat h$.  Later in the paper, we will occasionally use the 
latter notion when we need to talk about points in $N(\hat h)$ being separated by $\hat h$.  (Elsewhere in the literature, 
e.g.~\cite{Sageev_95}, the half-spaces associated to $\hat h$ are defined to be the components of $X \setminus \hat h$, or sometimes 
their closures.  This small difference in viewpoint is usually irrelevant.)

If $e$ is an edge, i.e. a $1$--cube,  of $X$, then there is a unique hyperplane $\hat h$ separating the two vertices of $e$.  
We say that $\hat h$ is \emph{dual to} $e$, and vice versa.  Note that  $\hat h$ is the unique hyperplane intersecting $e$.  
More generally, if $k\geq 0$ and $c$ is a $k$--cube, then there are exactly $k$ hyperplanes intersecting $c$.  These 
hyperplanes pairwise cross, and their intersection contains the barycenter of $c$.  We say that $c$ is \emph{dual to} this 
family of hyperplanes.  Any set of $k$ pairwise-crossing hyperplanes is dual to at least one $k$--cube.  In particular, 
$\dimension X$ bounds the cardinality of any set of pairwise crossing hyperplanes in $X$.

For any points $x,y \in X$, define $\W(x,y)$ to be the set of hyperplanes separating $x$ from $y$.  If $x,y\in X^{(0)}$, 
then an edge path $\gamma$ from $x$ to $y$ is a geodesic in $X^{(1)}$ if and only if $\gamma$ never crosses the same hyperplane twice. Consequently,
 $$\dist_1(x,y)=|\W(x,y)|.$$
In view of this, the following standard lemma relates the metrics $\dist_X$ and $\dist_1$. See Caprace and 
Sageev~\cite[Lemma 2.2]{CapraceSageev} or Hagen~\cite[Lemma 3.6]{Hagen:FacingTuples} for proofs.

\begin{lemma}\label{Lem:WallQI}
There are constants $\lambda_0 \geq 1, \lambda_1 \geq 0$, depending on $\dim X$, such that the following holds. For any pair 
of points $x, y \in X$,
$$
\frac{1}{\lambda_0} d_X(x,y) - \lambda_1 \leq |\W(x,y)| \leq \lambda_0 d_X(x,y) + \lambda_1.
$$
\end{lemma}

For an arbitrary set $B \subset X$, let $\W(B)$ be the union of all $\W(x,y)$ for all $x,y \in B$.  In practice, we will 
often be interested in the following special cases.  First, if $B$ is a hyperplane, then $\W(B)$ coincides with the set of hyperplanes
$\hat h$ that cross $B$.  If $B$ is a convex subcomplex (see Definition~\ref{Def:ConvexInX}), then $\W(B)$ coincides with the set of $\hat h$ with $\hat h\cap 
B\neq\emptyset$.  If $B$ is a combinatorial geodesic in $X^{(1)}$, then, because of the above characterization of $\dist_1$, the set 
$\W(B)$ is the set of hyperplanes dual to edges of $B$, or equivalently the set of hyperplanes intersecting $B$, or 
equivalently the set of hyperplanes separating the endpoints of $B$.

\begin{remark}\label{Rem:L1MetricOnX}
One can extend the $\ell^1$ metric  on $X^{(0)}$ to all of $X$ as follows. On a single cube $c$, let $d_1 \vert_{c \times c} \to \R$ be the standard $\ell^1$ metric. This can be extended to a path-metric on all of $X$ by concatenating paths in individual cubes. Miesch~\cite{Miesch} showed that this procedure gives a geodesic path metric on $X$ that extends the graph metric on $X^{(1)}$. With this extended definition, $d_1$ becomes bilipschitz to $d_X$, with the Lipschitz constant $\lambda_0$ depending only on $\dim(X)$.
\end{remark}

\subsection{The Roller boundary $\roller X$}\label{Sec:Roller}
Next, we define the \emph{Roller boundary} and \emph{Roller compactification} of $X$.

Every vertex $v\in X^{(0)}$ defines the collection $\frakH_v^+= \{h\in\frakH: v\in h\}$ of half-spaces containing $v$.  Going 
in the opposite direction, the collection $\frakH_v^+$ uniquely determines $v$:
$$
\bigcap_{h \in \frakH_v^+}{} h = \{v\}.
$$

We endow the set $2^\frakH$ with the product topology, which is compact by Tychonoff's theorem. Recall that a basis for this 
topology consists of \emph{cylinder sets} defined by the property that finitely many coordinates (i.e. finitely many half-spaces)  take a specified value ($0$ 
or $1$). Cylinder sets are both open and closed.

There is a continuous one-to-one map $X^{(0)} \hookrightarrow 2^\frakH$ defined by $v\mapsto \frakH_v^+$. (In fact,  
$X^{(0)}$ is homeomorphic to its image if and only if $X$ is locally compact. We will not need this fact.)

\begin{definition}[Roller compactification, Roller boundary]\label{Def:Roller}
The \emph{Roller Compactification} of $X^{(0)}$, denoted by $\~X$ or $\~X(\frakH)$, 
is the closure of the image of $X^{(0)}$ in $2^\mathfrak H$. 
The \emph{Roller Boundary} of $X$ is  $\roller X = \~X \setminus X^{(0)}$.

Recall that $\Aut(X)$ is the group of cubical automorphisms of $X$. We observe that the action of $\Aut(X)$ extends to a 
continuous action on $\~X$ and therefore on $\roller X$.
\end{definition}

\begin{definition}[Extended half-spaces]
Let $h \in \frakH(X)$. The \emph{extension of $h$} to $\~X$ is defined to be the intersection in $2^\frakH$ between $\~X$ and 
the cylinder set corresponding to $h$. 
 It is straightforward to check that the extensions of $h$ and $h^*$ form a partition 
of  $\~X$. We think of the extensions of $h$ and $h^*$ as complementary (vertex) half-spaces in $\~X$. By a slight abuse of 
notation, we use the same symbol $h$ to refer to both a half-space in $X$ and its extension in $\~X$.

By the above discussion of the topology on $2^\mathfrak H$, the basic open sets in $\~X$ are intersections between $\~X$ and 
cylinder sets. Therefore, every basis set is a finite intersections of (extended) half-spaces.
\end{definition}

The duality between points and half-spaces in $X$, described above, extends to all of $\~X = X^{(0)} \cup \roller X$.
Let $y \in \~ X$, and let $\frakH_y^+$ be the set of extended half-spaces that contain $y$. Then $y$ and $\frakH_y^+$ determine 
one another:
\begin{equation}\label{Eqn:SRDuality}
\bigcap_{h \in \frakH_y^+}{} h = \{ y \}.
\end{equation}

Chatterji and Niblo \cite{Chatterji-Niblo}, and independently Nica \cite{Nica}, extended Sageev's construction 
\cite{Sageev_95} to prove the following.

\begin{theorem}\label{Thm:SageevDuality}
Let $\frakH' \subset \frakH(X)$ be an involution-invariant collection of half-spaces. Then $\frakH'$ determines a CAT(0) cube 
complex $X(\frakH')$ and the Roller compactification $\~X(\frakH')$.  In particular, $X(\frakH(X)) = X$.
\end{theorem}

If  $x \in X^{(0)}$, the corresponding set $\frakH_x^+ \in 2^\frakH$ satisfies the \emph{descending chain condition}: any 
decreasing sequence of elements of $\frakH_x^+$ is eventually constant. On the other hand, the collection of half-spaces $\frakH_y^+$ corresponding to a boundary point $y 
\in \roller X$ will fail the descending chain condition. Namely, for $y\in \~X$ we have that $y\in \roller X$ if and only if 
there exists a sequence $\{h_n\}_{n\in \N}\subset \frakH_y^+$ such that $h_{n+1} \subsetneq h_n$ for each $n\in \N$.

\subsection{Convexity, intervals, and medians}
Here, we discuss the related notions of convexity and intervals in $X$ and $\~X$. Since we are interested in two different metrics on $X$, there are several distinct definitions of convexity ($\ell^1$ geodesic convexity, cubical convexity, interval convexity) that turn out to be equivalent for cubical subcomplexes of $X$.

We then introduce a median operation on $\~X$ and its restriction to $X$.

\begin{definition}[Convexity in $X^{(0)}$, convex subcomplexes]\label{Def:ConvexInX}
A set $A \subset X^{(0)}$ is called \emph{vertex-convex} if it is the intersection of vertex half-spaces in 
$X$. For a set $S \subset X^{(0)}$, the \emph{vertex convex hull} of $S$ is the intersection of all vertex-convex sets containing $S$, or equivalently the  intersection of all vertex half-spaces containing $S$.

A subcomplex $C \subset X$ is called \emph{cubically convex} 
if it is the intersection of CAT(0) half-spaces. For a set $S \subset X$, the \emph{cubical convex hull}, denoted $\hull(S)$, is the intersection of all cubically convex sets containing $S$, or equivalently the  intersection of all CAT(0) half-spaces containing $S$. We remark that for subcomplexes of $X$, this definition coincides with geodesic convexity in the CAT(0) metric $d_X$; see \cite[Remark 2.10]{Haglund} and  \cite[Theorem 4.10]{Sageev_95}.
\end{definition}

Observe that if $C$ is a full subcomplex of $X$ (meaning, $C$ contains a cube $c$ whenever it contains the $0$--skeleton of $c$), then $C$ is cubically convex whenever $C^{(0)}$ is vertex--convex. For this reason, we can use the two notions of convexity interchangeably when referring to full subcomplexes. 

\begin{definition}[Intervals in $X$]\label{Def:IntervalInX}
Given $x,y \in X^{(0)}$, define the \emph{(vertex) interval} $\I(x,y)$ to be vertex convex hull of $\{x, y\}$. Define the \emph{cubical interval} $\J(x,y)$ to be the cubical convex hull of $\{x, y\}$. 
\end{definition}

Observe that $\J(x,y)$ is a full, convex subcomplex, and is the union of all the cubes whose vertex sets are contained in $\I(x,y)$. In addition, $\J^{(1)}(x,y)$ is the union of all of the combinatorial geodesics from $x$ to $y$. In fact, intervals can be used to characterize convexity as follows:

\begin{lemma}\label{Lem:CharacterizeConvex}
A set $A \subset X^{(0)}$ is vertex-convex if and only if the following holds: for all $x,y \in A$, the interval $\I(x,y) \subset A$.

A subcomplex $C \subset X$ is cubically convex if and only if the following holds: for all $x,y \in C^{(0)}$, the cubical interval $\J(x,y) \subset C$.
\end{lemma}

\begin{proof}
Suppose $C$ is a cubically convex subcomplex and $x,y \in C$. Since $\J(x,y)$ is the intersection of all convex subcomplexes containing $x,y$, and $C$ is one such set, we have $\J(x,y) \subset C$. 

Toward the converse, suppose that $C$ is interval-convex, meaning $\J(x,y) \subset C$ for every pair $x,y$. We wish to prove that $C$ is convex.

We claim the following: for every  hyperplane $\hat h$ that intersects $C$ and is dual to an edge $e$ with endpoints $x,y$ such that $x \in C$, the other endpoint $y$ belongs to $C$ as well. This can be shown as follows. Since $\hat h$ intersects $C$, there is an edge $e' \subset C$ that is dual to $\hat h$. Let $x', y'$ be the endpoints of $e'$, such that $x$ and $x'$ are on the same side of $\hat h$. Then $\W(x,y') = \W(x,x') \sqcup \hat h = \hat h \sqcup \W(y,y')$. In particular, there exists a combinatorial geodesic from $x$ to $y'$ whose initial edge is $e$. Since $C$ is interval-convex, this proves the claim.

We can now show that $C$ is cubically convex. Suppose $z \in X^{(0)} \setminus C$. Let $\gamma \subset X^{(0)}$ be a shortest combinatorial geodesic from $z$ to $C$, and let $x$ be the terminus of $\gamma$, and let $e$ be the last edge of $\gamma$. Since the next-to-last vertex of $\gamma$ is not in $C$, the above claim implies that the hyperplane $\hat h$ dual to $e$ is disjoint from $C$, and separates $C$ from $z$. Thus $C$ is contained in a CAT(0) half-space disjoint from $z$. Since $z$ was arbitrary, it follows that $C$ is cubically convex.

The proof for vertex-convex subsets is identical, up to replacing all the appropriate sets by their $0$--skeleta.
\end{proof}

Our next goal is to extend the notion of convexity to $\~X$. 
It turns out that in $\~X$, the correct definition of convexity begins with intervals. After developing some definitions and tools, we will eventually prove  a generalization of Lemma~\ref{Lem:CharacterizeConvex} to $\~X$, using a very different argument.

\begin{definition}[Intervals and convexity in $\~X$]\label{Def:Convex}
Given $x,y \in \~X$, define the \emph{interval} $\I(x,y)$ to be the intersection of all (extended, vertex) half-spaces that 
contain both $x$ and $y$.  Equivalently, 
$$\I(x,y) = \big\{ z \in \~X \: : \:  \frakH^+_z \supset ( \frakH^+_x\cap \frakH^+_y) \big\}.$$
A nonempty set $A \subset \~ X$ is called \emph{convex} if it contains $\I(x,y)$ for all $x, y 
\in A$.  
\end{definition}

As a very particular case, half-spaces and hence also their intersections are convex in $\~X$.

Observe that for $x,y \in X^{(0)}$, the interval $\I(x,y)$ will in fact be contained in $X^{(0)}$, and coincides with the previous definition of $\I(x,y)$. Consequently, $X^{(0)}$ is convex in $\~X$. By Lemma~\ref{Lem:CharacterizeConvex}, any vertex-convex subset of $X^{(0)}$ is also convex in $\~X$.

Since $\~X \subset 2^{\frakH(X)}$ is by definition a set of vertices, we do not define convex subcomplexes of $\~X$.

As in \cite{Roller}, the fact that $\~X$ is a \emph{median space} is captured by the following property: for every $x,y,z \in 
\~X$ there is a unique point $m = m(x,y,z) \in \~X$ such that 
\begin{equation}\label{Eqn:Median}
\{m\}= \I(x,y) \cap \I(y,z) \cap \I(x,z).
\end{equation}
This unique point $m = m(x,y,z)$ is called the {\em median} of $x$, $y$, 
and $z$. 
We will not need the general definition of an (extended) median metric space; 
the only property of medians that we will need is  that $m \from \~X\times\~X\times\~X\to\~X$ satisfies \eqref{Eqn:Median}. 
In terms of half-spaces, we have:
\begin{equation*}
\frakH^+_m = (\frakH^+_x \cap  \frakH^+_y) \cup (\frakH^+_y\cap  \frakH^+_z)\cup ( \frakH^+_x \cap  \frakH^+_z).
\end{equation*}

\begin{remark}[Median in $X^{(0)}$]
Since intervals between points in $X^{(0)}$ are contained in $X^{(0)}$, we see that $m(x,y,z)\in X^{(0)}$ if 
$x,y,z\in X^{(0)}$.  So, the median $m$ on $\~X$ restricts to a median $m$ on $X^{(0)}$.  The point $m(x,y,z)$ is the unique 
vertex $m$ for which any two of $x,y,z$ can be joined by a geodesic in $X^{(1)}$ passing through $m$. In other words, we also get an analogue of  \eqref{Eqn:Median} with $\I(\cdot,\cdot)$ replaced by $\J(\cdot, \cdot)$:
$$
\{m\}= \J(x,y) \cap \J(y,z) \cap \J(x,z)
\quad
\text{for}
\quad x,y,z\in X^{(0)}.
$$

 In fact, a graph with 
this property --- called a \emph{median graph} --- is always the $1$--skeleton of a uniquely determined CAT(0) cube 
complex~\cite{Chepoi}.  
\end{remark}

\subsection{Lifting decompositions and convexity}\label{Sec:DualityLifting}

A set of half-spaces $\mathfrak{s} \subset \frakH$ is called \emph{consistent} if the following two conditions hold: if $h\in 
\mathfrak{s}$ then $h^*\notin \mathfrak{s}$, and if $k \supset h \in \mathfrak{s}$ then $k\in \mathfrak{s}$. Given a 
subset $\frakH' \subset \frakH$, a \emph{lifting decomposition} for $\frakH'$ is a consistent set of half-spaces 
$\mathfrak{s}  \subset \frakH \setminus \frakH'$ such that $\frakH = \frakH'\sqcup \mathfrak s\sqcup \mathfrak s^*$. 
Lifting decompositions do not necessarily exist, and are not necessarily unique.
 
Lifting decompositions naturally occur in the following way. Consider a set $A \subset \~X$. Analogous to $\frakH_x^+$,  we 
define a set $\frakH_A^+ = \{h\in \frakH : A\subset h\}$ of half-spaces that contain $A$, and observe that $\frakH_A^+$ is 
consistent. Define $\frakH_A^- = (\frakH_A^+)^*$ to be the set of half-spaces disjoint from $A$, and finally a set  
$$
\frakH_A = \{h\in \frakH : A\cap h \neq \emptyset, \, A\cap h^*\neq \emptyset \}
$$
 of half-spaces that cut $A$. Then we get a lifting decomposition $\frakH = \frakH_A \sqcup \frakH_A^+ \sqcup \frakH_A^-$, 
where $\mathfrak{s} = \frakH_A^+$.

\begin{remark}\label{SubsetReverseInclusion}
 Observe that if $A\subset B$, then $\frakH_B^+ \subset \frakH_A^+$. 
\end{remark}

Given points $x,y \in \~X$, let $\frakH(x,y)= \frakH_x^+\triangle\, \frakH_y^+$ denote 
the set of half-spaces that separate $x$ from $y$. 
 More generally, given two disjoint convex sets $A,B\subset \~X$, let $\frakH(A,B)= \frakH_A^+\triangle\, \frakH_B^+$ denote 
the set of half-spaces that separate $A$ from $B$. 

In an analogous fashion, we generalize the definition of $\W(x,y)$ from Section~\ref{Sec:half-spaces} to convex subsets $A, 
B \subset\~X$. For subsets of this form, $\W(A,B) = \widehat{\frakH(A,B)}$ consists of all the hyperplanes that 
separate $A$ from $B$. 
We also extend the
 $\ell^1$ metric on $X^{(0)}$ to $\~X$, where 
$d_1(x,y) = | \W(x,y) | = \frac{1}{2} | \frakH_x^+ \triangle \frakH_y^+| $ is allowed to take the value $\infty$.

 The following result says that lifting decompositions are in one-to-one correspondence with Roller-closed subcomplexes of 
$\~X$.

\begin{prop}\label{LiftingDecomp} 
The following are true:
\begin{enumerate}[$(1)$ ]
  \item\label{LiftingEmbedding} Suppose that $\frakH'\subset \frakH(X)$. If there exists a lifting decomposition 
$\mathfrak{s}$ for $\frakH'$, then there is a $d_1$--isometric
 embedding $\~X(\frakH') \hookrightarrow \~X$ induced from the map $2^{\frakH'} \hookrightarrow 2^{\frakH(X)}$ where $U 
\mapsto U\sqcup \mathfrak{s}$. The image of this embedding is 
$$\bigcap_{h\in \mathfrak{s}}{} h\subset \~X.$$
\item\label{ConsistentEmbedding} Similarly, if $\mathfrak{s}\subset \frakH(X)$ is a consistent set of half-spaces, then, 
setting $\frakH_\mathfrak{s} = 
\frakH(X)\setminus(\mathfrak{s}\sqcup \mathfrak{s}^*)$ we get an isometric embedding 
$\~X(\frakH_\mathfrak{s}) \hookrightarrow \~X$, obtained as above, onto
$$\bigcap_{h\in \mathfrak{s}}{} h\subset \~X.$$
\item\label{DCClifting} The set $\mathfrak{s}$ satisfies the descending chain condition if and only if the image of 
$X(\frakH_\mathfrak{s})$ is in $X$. 
\end{enumerate}
Furthermore, if the set $\mathfrak{s}$ is $G$--invariant, for some group $G\leq\Aut(X)$, then, with the restricted action on 
the image, the above natural embeddings are  $G$--equivariant.
\end{prop}

This result is a slightly strengthened version of  \cite[Lemma 2.6]{CFI}. See also \cite[Proposition 2.11]{Fernos:Poisson} for a 
very similar statement.

We remark that  
Proposition~\ref{LiftingDecomp} includes the possibility that $\mathfrak{s}=\emptyset$ and hence $\frakH\sqcup\emptyset \sqcup \emptyset$ is a legitimate lifting decomposition of $\frakH$. In this case, recall that an intersection over an empty collection of sets is everything, i.e. $\bigcap_{h\in \varnothing}{} h= \~X$. 

\begin{proof}
Suppose that there is a lifting decomposition $\mathfrak{s}$ for  $\frakH'$. Then, since $\frak s \cap \frak s^* = \emptyset$ 
and $\frakH = \frakH'\sqcup \frak s \sqcup \frak s^*$, we deduce that $\frakH' = \frakH \setminus(\frak s\sqcup \frak s^*)$ 
is involution invariant. 

Next, we claim that $\frakH' \subset \frakH$ has a property called \emph{tightly nested}, meaning that  for every pair $h,k\in \frakH'$, 
and for every $\ell\in \frakH$ with $h\subset \ell \subset k$, we have $\ell \in \frakH'$. Assume, for a contradiction, that 
$h,k\in \frakH'$, and $\ell\in \frakH$ with $h\subset \ell \subset k$, but 
$\ell \notin \frakH'$. Then (up to replacing the three half-spaces with their complements) we may assume that $\ell \in 
\mathfrak{s}$. But then $k\in \mathfrak{s}$, a contradiction. 

With the verification that $\frakH'$ is involution invariant and tightly nested, our hypotheses imply those of  \cite[Lemma 
2.6]{CFI}. 
Now, conclusions \ref{LiftingEmbedding}--\ref{DCClifting} follow from that lemma.

Finally, the conclusion about $G$--invariance follows because the embeddings are canonically determined by the associated 
half-spaces.
\end{proof}

\begin{remark}\label{Rem:EmbeddingImage}
In Proposition~\ref{LiftingDecomp}, the  isometric embeddings of $\~X(\frakH')$ and $\~X(\frakH_\mathfrak s)$ restrict to 
isometric embeddings of $X(\frakH')$ and $X(\frakH_\mathfrak s)$, respectively. 
Although $\roller X$ is typically not closed, the images under the isometric embedding of  $\roller X(\frakH')$ and $\roller 
X(\frakH_\mathfrak s)$  always lie in $\roller X$. Indeed, recall from the discussion following Theorem 
\ref{Thm:SageevDuality} that a point belongs to the Roller boundary if and only if there is an infinite descending chain  of 
half-spaces containing the point in question. Hence, if $U$ has an infinite descending chain, then so does $U\sqcup \frak 
s$. 
\end{remark}

As a first application of Proposition~\ref{LiftingDecomp}, we will prove a generalization of Lemma~\ref{Lem:CharacterizeConvex} to $\~X$. Observe that a naive generalization of  Lemma~\ref{Lem:CharacterizeConvex} does not hold, because $X$ is (interval) convex in $\~X$ by Definition~\ref{Def:Convex}, but is not an intersection of extended half-spaces. However, the following lemma says that a \emph{closed}, convex set in $\~X$ is always an intersection of extended half-spaces.

\begin{lemma}\label{Lem:ConvexClosure}
Let $A \subset \~X$ be a convex set. Then the Roller closure $\overline{A}$ in $\~X$
 is the image of $\~X(\frakH_A)$ under the embedding of Proposition~\ref{LiftingDecomp}, and
 $$ \overline{A} = \bigcap_{h \in \frakH_A^+} h = \bigcap_{A \subset h} h.$$
\end{lemma}

\begin{proof}
Observe that $\frakH_A^+$ is a consistent set of half-spaces. Thus Proposition~\ref{LiftingDecomp}.\ref{ConsistentEmbedding} gives an embedding $\~X(\frakH_A) \hookrightarrow \~X$ whose image is
 $$ W =  \bigcap_{A \subset h} h = \bigcap_{h \in \frakH_A^+} h =  \bigcap_{h \in \frakH_{\~A}^+} h.$$
Here, the first equality is the definition of $W$, the second equality is the definition of $\frakH_A^+$, and the third equality holds because half-spaces are closed, meaning $\frakH_{A}^+=\frakH_{\~A}^+$.
 
It remains to show that $\overline{A} = W$. The inclusion $\overline{A} \subset W$ is obvious, because $\overline{A} \subset h$ for every $h \in \frakH_{\~A}^+$. To prove the reverse inclusion, suppose for a contradiction that there is a point  $y \in W \setminus \~A$. Since $\~X \setminus \~A$
is open, there is a basic open neighborhood of $y$ of the form $\bigcap_{i=1}^{n} h_i$, such that 
$$y \in \Bigg( \bigcap_{i=1}^{n} h_i \bigg) \subset \~X\setminus \~A.
$$
Therefore, $A \subset h_1^*\cup\cdots \cup h_n^*$. Without  loss of generality, assume that $n$ is minimal  in the sense that 
for each $j =1, \dots, n$ there is a point 
$$a_j \in A\setminus \Bigg(\bigcup_{j \neq i=1}^{n} h_i^* \Bigg) = A \cap \Bigg(\bigcap_{j \neq i=1}^n h_i \Bigg).$$
One crucial observation is that $n\geq 2$: otherwise, if $A \subset h_1^*$, then $h_1 \cap W = \emptyset$, but we know that $y \in h_1 \cap W$.

Now, define $p = m(y,a_1, a_2)$. Recall that $\I(a_1,a_2)\subset A$, because $A$ is convex. By the definition of medians, we have 
$$
p \in \I(a_1,a_2)\subset A, \qquad
p \in \I(y, a_1) \subset \Bigg(\bigcap_{j =2}^n h_i \Bigg),
\qquad 
p \in \I(y, a_2) \subset h_1.
$$
Consequently, we have $p \in A \cap  \big( \bigcap_{i =1}^{n} h_i \big)= \emptyset$, a contradiction.
Thus $\overline A = W$.
\end{proof}

\subsection{Gates and bridges}\label{subsec:gate_bridge}
A very useful property of a convex subset $A\subset X^{(0)}$ is the existence of a \emph{gate map} $\gate_A \from X\to A$; this is 
a well-known notion, see e.g.~\cite[Section 2.2]{Hagen:FacingTuples} and the citations therein.  Metrically, 
$\gate_A \from X^{(0)}\to A^{(0)}$ is just closest-point projection, but gates can be characterized entirely in terms of the median 
(see e.g.~\cite[Section 2.1]{Fioravanti:roller} and citations therein).  Since we saw that $\~X$ has a median satisfying \eqref{Eqn:Median}, convex 
subsets of $\~X$ do similarly admit gates.  This is the content of  
Proposition~\ref{Prop:Proj}, which has the advantage of showing an important relationship between 
gates and walls.

Recall that $\W(x,Y)$ is the set of hyperplanes separating $x$ from $Y$, where $Y\subset\~X$ is convex.

\begin{prop}[Gate projection in $\~X$]\label{Prop:Proj}
 Let $C\subset \~X$ be closed and convex. There is a unique projection $\gate_C\from \~X\onto C$ such that for any $x\in \~X$ we have 
 $$\W(x,\gate_C(x)) = \W(x,C).$$ 
\end{prop}

\begin{proof}
Let $\frakH = \frakH_C\sqcup \frakH_C^+ \sqcup \frakH_C^-$ be the lifting decomposition associated to $C$.  The map 
$2^\frakH\to 2^{\frakH_C}$ defined by $U\mapsto U\setminus (\frakH_C^+ \sqcup \frakH_C^-) $ restricts to a surjection $\~X \onto 
\~X(\frakH_C)$. 
Next, Proposition~\ref{LiftingDecomp} gives an embedding  $\~X(\frakH_C) \to \~X$, namely $V \mapsto V \sqcup \frakH_C^+$.
Composing the surjection $\~X \onto 
\~X(\frakH_C)$ with the embedding $\~X(\frakH_C) \to \~X$ yields a map $\gate_C \from \~X \to \~X$. The image $\gate_C(\~X)$ is $C$ since, by Lemma~\ref{Lem:ConvexClosure}, the image of the embedding $\~X(\frakH_C)\to\~X$ is exactly $C=\overline{C}$.

To prove the equality in the statement, first observe that $ \W(x,C) \subset \W(x,\gate_C(x))$ because $\gate_C(x) \in C$. For the opposite inclusion, consider a hyperplane $\hat h \in \W(x,\gate_C(x))$, and choose an orientation $h \in \frakH_x^+$. By the definition of $\gate_C$, we have
$$\frakH_{\gate_C(x)}^+ = (\frakH_x^+ \setminus \frakH_C^-) \cup \frakH_C^+.$$
Since $h \in \frakH_x^+ \setminus \frakH_{\gate_C(x)}^+$, we must have $h \in \frakH_C^-$, hence $\hat h \in \W(x,C)$. Thus $ \W(x,C) = \W(x,\gate_C(x))$.

Finally, for $x \in C$, we must have $x = \gate_C(x)$ because $\W(x,\gate_C(x)) = \W(x,C) = \emptyset$. Thus $\gate_C$ is indeed a projection to $C$.
\end{proof}

Note that $\gate_C$ restricts to the identity on $C$.  We call $\gate_C$ the \emph{gate map to $C$}, which is consistent 
with the terminology used in the theory of median spaces because of the following lemma:

\begin{lemma}\label{Lem:Gate_and_Median}
Let $\gate_C$ be the map from Proposition~\ref{Prop:Proj}.  Then for all $x\in\~X$ and $y\in C$, we have 
$m(x,y,\gate_C(x))=\gate_C(x)$.  In particular, $\gate_C(x) \in \bigcap_{y \in C} \I(x,y)$.
\end{lemma}

\begin{proof}
By Proposition~\ref{Prop:Proj} and the fact that $y\in C$, we have $\W(x,\gate_C(x))\subset\W(x,y)$, so 
$\I(x,\gate_C(x))\subset \I(x,y)$.  So, $\gate_C(x)\in \I(x,y)\cap \I(x,\gate_C(x))\cap \I(y,\gate_C(x))$, hence $m=\gate_C(x)$.
\end{proof}

The above lemma has two corollaries. First, we can characterize medians in terms of projections:

\begin{cor}\label{Cor:MedianViaProj}
Suppose that $x,y,z \in X^{(0)}$. Then $\gate_{\I(x,y)}(z)=m(x,y,z)$. 
\end{cor}

\begin{proof}
By
Lemma~\ref{Lem:Gate_and_Median}, we have $\gate_{\I(x,y)}(z)  \in \I(x,z)$ and $\gate_{\I(x,y)}(z) \in \I(y,z)$.  On the 
other hand, $\gate_{\I(x,y)}(z)\in \I(x,y)$ by definition.  So, by \eqref{Eqn:Median}, we get  
$\gate_{\I(x,y)}(z)=m(x,y,z)$.
\end{proof}

Second, Lemma~\ref{Lem:Gate_and_Median} allows us to project from $X$ to convex subcomplexes of $X$.

\begin{cor}\label{Cor:ProjectionInX}
Let $A \subset X$ be a convex subcomplex. Then, for all $x \in X^{(0)}$, we have $\gate_{\overline{A}} (x) \in A$. Consequently, we get a projection $\gate_A = \gate_{\overline{A}} \from X^{(0)} \to A^{(0)}$.
\end{cor}

\begin{proof}
Let $x \in X^{(0)}$ and $y \in A^{(0)}$. By Lemma~\ref{Lem:Gate_and_Median}, we have $\gate_{\overline A}(x) \in \I(x,y) \subset X^{(0)}$. But $\overline A \cap X^{(0)} = A^{(0)}$.
\end{proof}

\begin{remark}\label{Rem:GateExtension}
In fact, it is possible to extend $\gate_A$ over higher-dimensional cubes to 
get a map $\gate_A \from X\to A$. This is done as follows (see also \cite[Section 2.1]{BHS}).  First let $e$ be a $1$--cube of $X$ joining $0$--cubes $x,y$ and dual to a hyperplane $\hat h$.  From 
Proposition~\ref{Prop:Proj}, one has the following: if $\hat h$ does not cross $A$, then $\gate_A(x)=\gate_A(y)$, and, otherwise, $\gate_A(x)$ and $\gate_A(y)$ are joined by an edge $\bar e\subset A$ dual 
to $\hat h$.  
In the former case, $\gate_A$ extends over $e$ by sending every point to $\gate_A(x)=\gate_A(y)$; in the latter case, we extend by declaring $\gate_A \from e\to\bar e$ to be the obvious 
isometry. 

Now, if $c$ is a $d$--cube for $d\geq 2$, then for some $s\in\{0,\ldots,d\}$, there are $1$--cubes $e_1,\ldots,e_s\subset c$ that span an $s$--cube $c'$ and have the property that the 
hyperplanes $\hat h_1,\ldots,\hat h_s$ respectively dual to $e_1,\ldots,e_s$ are precisely the hyperplanes intersecting both $c$ and $A$.  (The case $s=0$ corresponds to the situation where no 
hyperplane intersecting $c$ intersects $A$, and $c'$ is an arbitrary $0$--cube of $c$.)  So for $1\leq i\leq s$, we have that $\bar e_i=\gate_A(e_i)$ is a $1$--cube of $A$ dual to $\hat h_i$, and the 
$1$--cubes $\bar e_1,\ldots,\bar e_s$ span an $s$--cube $\bar c'$.  We extend $\gate_A$ over $c'$ using the obvious cubical isometry $c'\to\bar c'$ extending the map $\gate_A$ on $1$--cubes.  By 
construction, for each $v\in c^{(0)}$, there is a unique $0$--cube $v'\in c'$ such that $\gate_A(v)=\gate_A(v')$, the assignment $v\mapsto v'$ extends to a cubical map $c\to c'$ (collapsing the 
hyperplanes not in $\{\hat h_1,\ldots,\hat h_s\}$), and composing this with $\gate_A:c'\to A$ gives the (extended) gate map $\gate_A:c\to A$.  Note that $c'\subset c$ is only unique up to parallelism, 
but any two allowable choices of $c'$ give the same map $c\to A$.  
\end{remark}

We conclude that gates enjoy the following properties: 

\begin{lemma}\label{Lem:Gate_Lipschitz}
Let $A\subset X$ be a convex subcomplex.  Then:
\begin{enumerate}[$(1)$ ]
     \item \label{Itm:GateHyperplanes} For all $x,y\in X$, we have $\W(\gate_A(x),\gate_A(y))=\W(x,y)\cap \W(A)$, and $\W(x,\gate_A(x))=\W(x,A).$
     \item\label{Itm:GateLipschitz} The map $\gate_A$ is $1$--lipschitz on $X^{(0)}$ with the $d_1$ metric and the CAT(0) 
metric $d_X$.
\end{enumerate}
\end{lemma}

\begin{proof}
For $0$--cubes, the first part of conclusion~\ref{Itm:GateHyperplanes} appears in various places in the literature; see e.g.~\cite[Lemma 2.5]{Hagen:FacingTuples}, or one can deduce it easily from 
Proposition~\ref{Prop:Proj}.  The second part of the first assertion (again for $0$--cubes) restates Proposition~\ref{Prop:Proj}.  The generalization of \ref{Itm:GateHyperplanes} to arbitrary points in $X$ 
follows immediately from the corresponding assertion for $0$--cubes, together with the construction in Remark~\ref{Rem:GateExtension}. 

Finally, \ref{Itm:GateHyperplanes} implies the first part of \ref{Itm:GateLipschitz}, since $\dist_1(x,y)=|\W(x,y)|$ for 
$0$--cubes $x,y$.  For the second part, first note that the restriction of $\pi_A$ to each cube is $1$--lipschitz for the 
CAT(0) metric, since it factors as the natural projection of the cube onto one of its faces (which is lipschitz for the 
Euclidean metric) composed with an isometric embedding.  Now, fix $x,y\in X$ and let $\gamma$ be the CAT(0) geodesic joining 
them, which decomposes as a finite concatenation of geodesics, each lying in a single cube.  Each such geodesic is mapped to 
a path whose length has not increased, so $\pi_A\circ\gamma$ is a path from $\pi_A(x)$ to $\pi_A(y)$ of length at most 
$\dist_X(x,y)$, as required.
\end{proof}

The following standard application of gates is often called the \emph{bridge lemma}.

\begin{lemma}[Bridge Lemma]\label{Lem:Bridge}
Let $I,J\subset X$ be convex subcomplexes.  Let $\gate_I \from X \to I$ and $\gate_J \from X\to J$ be the gate maps.  Then the following hold:
 \begin{enumerate}[$(1)$ ]
\item $\gate_I(J)$ and $\gate_J(I)$ are convex subcomplexes.
     \item $\W(\gate_I(J))=\W(\gate_J(I))=\W(I)\cap\W(J)$.
     \item The map $\gate_I \from \gate_J(I)\to\gate_I(J)$ is an isomorphism of CAT(0) cube complexes.
     \item The cubical convex hull  $\hull(\gate_I(J)\cup\gate_J(I))$ is a CAT(0) cube complex isomorphic to $\gate_J(I)\times 
\I(x,\gate_J(x))$ for any vertex $x\in \gate_I(J)$.
\item \label{Itm:WallSetIntersect} If $I\cap J\neq\emptyset$, then $\gate_I(J)=\gate_J(I)=I\cap J$, and $\W(I\cap J)=\W(I)\cap\mathcal 
W(J)$.
 \end{enumerate}
\end{lemma}

\begin{proof}
     This can be assembled from results in the literature in various ways.  For example, the first statement is part 
of~\cite[Lemma 2.2]{Fioravanti:roller}; the second and last statements follow from~\cite[Lemmas~2.1 and 2.6]{BHS}.  The 
third follows from~\cite[Lemma 2.4]{BHS}.  For the fourth, see~\cite[Lemma 2.18]{CFI} or~\cite[Lemmas 2.4 and 2.6]{BHS}.
\end{proof}

\subsection{Facing tuples and properness}
A \emph{facing $k$--tuple} of hyperplanes is a set $\{\hat h_1,\ldots,\hat h_k\}$ of hyperplanes such that, for each $i\leq 
k$, we can choose a half-space $h_i$ such that $h_i\cap h_j=\emptyset$ for $i\neq j$.  We are often particularly interested in 
facing triples, since many of the subcomplexes of $X$ considered later in the paper will be CAT(0) cube complexes that do 
not contain facing triples.

One useful application of the notion of a facing tuple is the following lemma, which in practice will be used to guarantee 
properness of certain subcomplexes of $X$.  See~\cite[Section 3]{Hagen:FacingTuples} for a more detailed discussion. See also \cite[Lemma 2.33]{CFI} for a related result,  where $X$ is required to  isometrically  embed into $\Z^D$ for some $D$.

\begin{lemma}\label{Lem:proper}
Let $X$ be a CAT(0) cube complex of dimension $D<\infty$.   Suppose that there exists $T$ such that $X$ does not contain a 
facing $T$--tuple.  Then  $(X,\dist_X)$ is a 
proper CAT(0) space.
\end{lemma}

\begin{proof}
Let $R\geq 0$ and let $B$ be a ball of radius $R$ in $X$ (with respect to the CAT(0) metric $\dist_X$).  Let $\W(B)$ be the set of hyperplanes intersecting $B$.

We say that hyperplanes $\hat h_1,\ldots,\hat h_n$ in $\W(B)$ form a \emph{chain} if (up to relabelling), each
$\hat h_i$ separates $\hat h_{i-1}$ from $\hat h_{i+1}$ for $2\leq i\leq n-1$.

We claim that there exists $N$, depending only on $D$ and $R$,  such that any chain in $\W(B)$ has cardinality at most 
$N$. Indeed, if $\hat h_1,\ldots,\hat h_n$ is a chain, then any edge dual to $\hat h_1$ lies at $\ell^1$ distance at least 
$n-2$ from any edge dual to $\hat h_n$. Now,  Lemma~\ref{Lem:WallQI} implies there exists $N$ (depending only on $D$ and $R$) such that $n>N$ implies that 
$\dist_X(\hat h_1,\hat h_n)>2R$.  This is impossible, since $\hat h_1,\hat h_n$ intersect a common $R$--ball.

On the other hand, \cite[Proposition~3.3]{Hagen:FacingTuples} provides a constant $K=K(D,T)$ such that any finite set 
$\mathcal F\subset \W(B)$ must contain a chain of length at least $|\mathcal F|/K$.  So, $|\W(B)|\leq KN$.

Let $\hull(B)$ be the cubical convex hull of $B$. 
  The set of hyperplanes of the CAT(0) cube complex $\hull(B)$ is $\mathcal 
W(B)$, which we have just shown is finite.  Hence $\hull(B)$ is a compact CAT(0) cube complex, and hence proper.  Since $\widehat 
B\hookrightarrow X$ is an isometric embedding (with respect to CAT(0) metrics), $B$ is a ball in $\hull(B)$, and is therefore 
compact.  This completes the proof.
\end{proof}

The use of the above lemma will be the following.  Given $x\in X^{(0)}$ and $y\in\roller X$, we can consider the interval 
$\I(x,y)$ in $\~X$.  The intersection $\I(x,y)\cap X$ is a vertex-convex set, which is the $0$--skeleton of a uniquely 
determined convex subcomplex $A$.  Note that $\W(A)=\W(x,y)$.  This 
infinite set of hyperplanes cannot contain a facing triple.  Hence, Lemma~\ref{Lem:proper} shows that $A$ is 
proper.  We will need this in the proof of Lemma~\ref{Lem:close_to_intersection}, to arrange for a sequence of CAT(0) geodesic segments in $A$ to converge uniformly on compact sets to a CAT(0) 
geodesic ray.  In this way, we avoid a blanket hypothesis that $X$ is proper. 

Note that we therefore only use a special case of the lemma, namely properness of cubical intervals.  Properness of intervals follows from a much stronger statement: intervals endowed with the $\ell_1$ metric  isometrically embed into $\R^{\dimension(X)}$ (i.e. are \emph{Euclidean}). It follows that intervals are proper in the CAT(0) metric, since the two metrics are bilipschitz.  This 
embedding  into $\R^{\dimension(X)}$ goes back at least to \cite[Theorem 1.16]{BCGNW}. This embedding is closely related to \cite[Lemma 2.33]{CFI}, which bounds the cardinality of facing tuples in Euclidean CAT(0) cube 
complexes, which include, but are more general than, cubical intervals.

\subsection{Combinatorial Geodesic Rays}

Next, we develop some basic facts about combinatorial geodesics in CAT(0) cube complexes that will be needed when we relate different types of boundaries.

\begin{definition}[Combinatorial geodesic rays]\label{Def:CombGeodesic}
 A map $\gamma \from [0, \infty)\to X^{(1)}$ is said to be a \emph{combinatorial geodesic ray} if $\gamma(\N) \subset X^{(0)}$ and for each 
$n,m \in \N$ we have  $d_1(\gamma(n), \gamma(m) ) = |n-m|$ and $\gamma$ is an isometry on each interval $[n, n+1]$. 
\end{definition}

\begin{lemma}\label{Lem:CombLimit}
 Let $\gamma \from [0,\infty)\to X$ be a combinatorial geodesic ray. There exists a unique point  $y\in \roller X$ such that $\gamma(n) \to  y$
in $\~X$ as $n \to \infty$. 
\end{lemma}

In the sequel, we will write $y = \gamma(\infty)$ to mean $\gamma(n) \to  y$.

\begin{proof}
Consider the set
$$
D(\gamma) = \{h\in \frakH: \gamma([N,\infty)\cap\N)\subset h\text{ for some } N\in \N\}.
$$ 
This is clearly a consistent choice of half-spaces. We claim that for each $h\in \frakH$, either $h$ or $h^*$ belongs to 
$D(\gamma)$.
Indeed, up to replacing $h$ by $h^*$, there must be a monotonic sequence $n_k\to \infty$ such that $\gamma(n_k) \in h$. Since $h$ is 
convex, it follows that $\gamma([n_1,\infty)\cap \N) \subset h$.

Now, define $C = 
\bigcap_{ h \in D(\gamma)} h$. We claim that $C$ consists of exactly one point. Note that $C \neq \emptyset$ because 
it is a nested intersection of compact sets in $\~X$. On the other hand, if $x,x' \in C$ are distinct points, then consider 
a 
half-space $h$ such that $x \in h$ and $x' \in h^*$. This half-space must satisfy both $h \in D(\gamma)$ and $h^* \in D(\gamma)$, a contradiction. Thus 
$C = \{y\}$, as desired.
\end{proof}

\begin{lemma}\label{Lem:CombGeodesicExists}
 For every $x\in X^{(0)}$ and every $y \in \roller X$, there is a combinatorial geodesic ray $\gamma \from [0,\infty) \to X$ such that $\gamma(0) = 
x$ and $\gamma(\infty) =  y$. 
\end{lemma}

The proof of Lemma~\ref{Lem:CombGeodesicExists} is modeled on the construction of normal cubed paths, introduced by Niblo and Reeves in~\cite{NibloReeves}.

\begin{proof}
Without loss of generality, let us assume that $\~X= \I(x,y)$, that is, $\frakH(X) = \frakH_x^+\triangle \frakH_y^+$. 
The collection of half-spaces $\frakH_x^+\setminus \frakH_y^+$ is partially ordered by set inclusion. Consider $h, k \in 
\frakH_x^+\setminus \frakH_y^+$. Since $x\in h\cap k$ and $y\in h^*\cap k^*$, the possible relations between $h$ and $k$ are: $h\subset k$ or $k\subset 
h$ or $h\pitchfork k$. In other words, incomparable elements are transverse. In particular, all minimal elements of 
$\frakH_x^+\setminus \frakH_y^+$ must be pairwise transverse.  By finite dimensionality, the set $M_1$ of minimal elements is 
finite.  Let $x_1$ be the vertex obtained by replacing every $h \in M_1$ by $h^*$, that is, $\frakH_{x_1}^+= 
(\frakH_{x}^+\setminus M_1)\cup M_1^*$, a clearly consistent and total choice of half-spaces. In other words, $x_1$ is 
diagonally across from $x$ in the unique maximal cube containing $x$ in the vertex interval $\I(x, y)$. 

Proceeding by induction, suppose that $x_{n-1}$ has been defined. 
Let $M_n$ be the set of minimal elements of  $\frakH_{x_{n-1}}^+ \setminus \frakH_y^+$.
Let $x_n$ be the vertex obtained by replacing every $h \in M_{n}$ by $h^*$, corresponding to the half-space interval $\frakH_{x_n}^+= (\frakH_{x_{n-1}}^+\setminus M_n)\cup M_n^*$.

Observing that $M_k^*\cap M_j = \emptyset $ for $j < k$,
we conclude that 
$$\frakH_{x_n}^+= \left( \frakH_{x}^+\setminus \Cup{j=1}{n-1} M_j \right) \cup \left( \Cup{j=1}{n-1} M_j^*\right).$$

Fix $k<m \in \N$. We claim that for any integer $ \ell \in  [k, m]$, $x_\ell$ lies on every geodesic path from $x_k$ to 
$x_m$.
Equivalently, we claim that
the median of $x_{k}$, $x_{\ell}$, and $x_{m}$ is $x_{\ell}$. Recall that this will be the case if every half-space 
containing $x_{k}$ and $x_{m}$ also contains $x_{\ell}$. This containment can be established as follows
\begin{align*}
\frakH_{x_{k}}^+ \cap \frakH_{x_{m}}^+ & = \left( \frakH_{x}^+ \setminus \Cup{j=1}{m-1} M_j \right)\cup \left( \Cup{j=1}{k-1}M_j^* \right) \\
& \subset \left( \frakH_{x}^+ \setminus \Cup{j=1}{\ell-1} M_j \right) \cup \left( \Cup{j=1}{\ell-1}M_j^* \right) = 
\frakH_{x_\ell}^+.
\end{align*}
This proves the claim.

Set $x_0=x$ and choose an edge-geodesic $\gamma_k \from [0, d(x_{k-1}, x_k)] \to \I(x_{k-1}, x_k)$ from $x_{k-1}$ to $x_k$. Finally, define $\gamma$ as the infinite
concatenation $\gamma_1  \gamma_2 \cdots$. Then, by construction, $\gamma(\infty) =  y$.
\end{proof}

We will discuss further properties of combinatorial geodesic rays in Section~\ref{Simplicial Boundary}.

\section{Cuboid complexes with modified CAT(0) metrics}\label{subSec:cuboid_defs}

A \emph{cuboid} is a box of the form $[0,a_1] \times \ldots \times [0,a_n] \subset \R^n$. 
In this section, we describe a way to modify the metric on a
finite-dimensional CAT(0) cube complex $X$ to produce a new CAT(0) space whose cells are cuboids rather than cubes.   Our 
results relating the Tits boundary and simplicial boundary of a CAT(0) cube 
complex will carry over to this context with small changes to some proofs, which we will indicate in the relevant places.

The reader only interested in CAT(0) cube complexes with the standard CAT(0) metric can safely skip this section, apart from the statements of Lemmas~\ref{Lem:CapraceSageevProduct} and~\ref{lem:combinatorial_fellow_travels_cat0}. 
Those lemmas are cuboid generalizations of standard results in the literature, namely~\cite[Lemma 2.5]{CapraceSageev} and~\cite[Proposition 2.8]{BeyrerFioravanti}. Thus, in the standard metric, one can substitute those results from the literature for the lemmas of this section.

Let $X$ be a finite-dimensional CAT(0) cube complex and let $G \to \Aut(X)$ be a group acting on $X$ by cubical automorphisms.  Let 
$\W$ be the set of hyperplanes in $X$. 

\begin{definition}[$G$--admissible hyperplane rescaling]\label{Def:hyp_rescaling}
A \emph{hyperplane rescaling} 
 of $X$ is a map $\rho \from \W\to (0,\infty)$. The rescaling is called \emph{$G$--admissible} 
if it is $G$--invariant, and in addition we have $m_\rho=\inf_{\hat h\in\W}\rho(\hat h)>0$ and $M_\rho=\sup_{\hat 
h\in\W}\rho(\hat h)<\infty$.
\end{definition}

Let $\rho$ be a $G$--admissible hyperplane rescaling of $X$.  Fix a cube $c$ of $X$.  For each hyperplane $\hat h_i$ 
intersecting $c$, let $\rho_i=\rho(\hat h_i)$.  Regarding $c$ as an isometric 
copy of $\prod_i[0,\rho_i]$, with the Euclidean metric, let $d_X^\rho$ be the resulting piecewise-Euclidean path metric on 
$X$.

  The following result is a special case of a theorem of Bowditch~\cite{Bowditch:Median}, but we are able to give a direct proof. 

\begin{lemma}\label{Lem:cat_0_rescaling}
Suppose that $G$ acts on $X$ by cubical automorphisms, and let $\rho$ be a $G$--admissible hyperplane rescaling of $X$.  Then 
$(X,d_X^\rho)$ is a CAT(0) space, and the 
action of $G$ on $(X,d_X^\rho)$ is by isometries.
\end{lemma}

\begin{proof}
Let $x \in X$ be an arbitrary point. Then the link of $x$ in $(X,\dist^\rho_X)$ 
is isometric to the link of $x$  in $(X,d_X)$, and is hence a CAT(1) space. Since the amount of rescaling in bounded, there is a small $\epsilon$ depending on $x$ and $\rho$ such that the 
$\epsilon$--balls about $x$ in the two metrics are (abstractly) isometric. Thus $(X, d_X^\rho)$ is locally CAT(0). Recalling that $X$ is contractible (the homeomorphism type has not changed), we 
conclude that $(X, d_X^\rho)$ is globally CAT(0) by Cartan--Hadamard theorem~\cite[Theorem II.4.1]{BridsonHaefliger}. 
 Since $\rho$ is $G$--equivariant, $G$ takes cuboids to $d_X^\rho$--isometric cuboids, 
hence $G$ acts by isometries.  
\end{proof}

One can also prove Lemma~\ref{Lem:cat_0_rescaling} using a result of Bridson and Haefliger~\cite[Theorem II.5.2]{BridsonHaefliger}. Our proof is essentially the same argument, but the cubical structure allows us to assume bounded rescaling rather than
finitely many shapes.

\begin{definition}[CAT(0) cuboid metric]\label{Def:cuboid_metric}
Consider a finite-dimensional CAT(0) cube complex $X$, a group $G$ acting on $X$ by cubical automorphisms, and a $G$--admissible hyperplane rescaling 
$\rho$ of $X$. We say that the CAT(0) metric 
$d_X^\rho$ from Lemma~\ref{Lem:cat_0_rescaling} is a \emph{CAT(0) cuboid metric} on $X$, and $(X,d_X^\rho)$ is a \emph{CAT(0) 
cuboid complex}.
\end{definition}

\begin{definition}[Automorphism group of cuboid complex]\label{Def:AutXrho}
Suppose that $G$ acts on $X$ by cubical automorphisms, and let $\rho$ be a $G$--admissible hyperplane rescaling of $X$. Define $\Aut(X^\rho) \subset \Aut(X)$ to be the automorphism group of the rescaled complex $(X,d_X^\rho)$. By construction, the representation $G \to \Aut(X)$ has image in 
$ \Aut(X^\rho)$. 
\end{definition}

\begin{cor}\label{Lem:CuboidallyConvex}
Every cubically convex subcomplex $Y \subset X$ is also geodesically convex in the rescaled metric $d_X^\rho$.
In particular, every CAT(0) half-space $\hull(k)$ is convex in $d_X^\rho$.
\end{cor}
 
\begin{proof}
By Definition~\ref{Def:ConvexInX}, a cubically convex subcomplex is an intersection of CAT(0) half-spaces. Since any intersection of $d_X^\rho$--convex sets will be $d_X^\rho$--convex, it suffices to prove that CAT(0) half-spaces are $d_X^\rho$--convex.

Now, let $k \in \frakH(X)$, and let $\hull(k)$ be the the corresponding CAT(0) half-space. Observe that $\hull(k)$ is a CAT(0) cube complex in its own right with half-space structure $\frakH(k) = \{h\in \frakH: 
h\pitchfork k\}$. Let $G_k= \stab_G(k)$ which then acts by isometries on $\Aut(\hull(k))$. Clearly the restriction 
$\rho|_{\frakH(k)}$ is $G_k$--admissible. By Lemma \ref{Lem:cat_0_rescaling} we have $(\hull(k), 
d^{\rho|{\frakH(k)}}_{\hull(k)})$ is a CAT(0) space, so that
$$d_X^\rho \big\vert_{\hull(k)} = d^{\rho \vert_{\frakH(k)}}_{\hull(k)}.$$

Therefore, the identity inclusion $\hull(h) \hookrightarrow X$ is isometric with respect to this new CAT(0) metric, and in 
particular $\hull(k)$ is $d_X^\rho$--convex. 
\end{proof}
 
 The following result is a generalization of~\cite[Lemma 2.5]{CapraceSageev} for CAT(0) cuboid complexes.

\begin{lemma}\label{Lem:CapraceSageevProduct}
Let $X$ be a CAT(0) cube complex and let $\rho$ be a $G$--admissible hyperplane rescaling.  Suppose that $\W(X) =\mathcal 
W_1\sqcup\W_2$, where $\hat h,\hat v$ cross whenever $\hat 
h\in\W_1,\hat v\in\W_2$.  

Then $X=X_1\times X_2$ (as a product cube complex), where $X_1$ and $X_2$ are CAT(0) cube complexes dual to $\W_1$ and $\W_2$.  

Letting $p_i \from X\to X_i$ denote the natural projection, 
for $i\in\{1,2\}$, the map $p_i$ induces a bijection on the sets of hyperplanes such that the preimages of the hyperplanes in 
$X_i$ are the hyperplanes in $\W_i$.  

Moreover, restricting $\rho$ 
 to $\W_i$, we obtain rescaled CAT(0) metrics $\dist^\rho_{X_i}$ such that $(X_i,\dist^\rho_{X_i})$ is a CAT(0) 
cuboid complex and we have
$$\dist^\rho_X(x,y)=\sqrt{\dist^{\rho}_{X_1}(p_1(x),p_1(y))^2+\dist^{\rho}_{X_2}(p_2(x),p_2(y))^2}$$ 
 for all $x,y\in X$. 
\end{lemma}

\begin{proof}[Sketch.]
The statement about $X$ decomposing as a product cube complex follows from~\cite[Lemma 2.5]{CapraceSageev} and the 
discussion in~\cite{CapraceSageev} preceding it.  So, it remains to prove the 
statement about rescaled CAT(0) metrics.

Admissibility of the rescalings on $X_1,X_2$ follows from admissibility of $\rho$ and the fact that $p_i$ is a 
\emph{restriction quotient} in the sense of~\cite[Section 2.3]{CapraceSageev}, meaning it sends 
hyperplanes to hyperplanes.  So, Lemma~\ref{Lem:cat_0_rescaling} makes each $(X_i,\dist^\rho_{X_i})$ a CAT(0) cuboid 
complex.

We now verify the metric statement.  Let $x,y\in X$.  Let $\alpha$ be the $\dist^\rho_X$--geodesic from $x=(p_1(x),p_2(x))$ 
to $y=(p_1(y),p_2(y))$.  Let $\beta$ be the $\dist^\rho_X$--geodesic from 
$(p_1(x),p_2(y))$ to $(p_1(x),p_2(x))$, and let $\gamma$ be the $\dist^\rho_X$--geodesic from $(p_1(y),p_2(y))$ to $(p_1(x),p_2(y))$.  So, 
$\alpha\gamma\beta$ is a geodesic triangle.  

Note that every point on $\beta$ has the form $(p_1(x),z)$ and every point on $\gamma$ has the form $(z,p_2(y))$, because the 
$p_i$ are restriction quotients and therefore have convex fibers by 
Lemma~\ref{Lem:CuboidallyConvex}.  In particular, $\|\beta\|=\dist^\rho_{X_2}(p_2(x),p_2(y))$ and 
$\|\gamma\|=\dist^\rho_{X_1}(p_1(x),p_1(y))$ and $p_1,p_2$ are respectively isometries on $\gamma$ 
and $\beta$.  Moreover, the Alexandrov angle (see Definition \ref{Def:tits_metric}) formed by $\beta$ and $\gamma$ at their common point is $\pi/2$.  Perform the 
same construction with $(p_1(x),p_2(y))$ replaced by $(p_1(y),p_2(x))$.  This 
yields a geodesic quadrilateral in $(X, \dist^\rho_X)$, with all angles $\pi/2$, two adjacent sides having length 
$\dist^\rho_{X_2}(p_2(x),p_2(y))$ and $\dist^\rho_{X_1}(p_1(x),p_1(y))$, and $x$ and $y$ as opposite 
corners.  The Flat Quadrilateral Theorem (\cite[Theorem II.2.11]{BridsonHaefliger}) now implies that 
$$\dist^\rho_X(x,y)=\sqrt{\dist^{\rho}_{X_1}(p_1(x),p_1(y))^2+\dist^{\rho}_{X_2}(p_2(x),p_2(y))^2},$$
as required.
\end{proof}

We also need the following standard lemma.  In the context of CAT(0) cube complexes, this is~\cite[Proposition 2.8]{BeyrerFioravanti}.  We state it here for cuboid complexes.  The proof 
from~\cite{BeyrerFioravanti} goes through with tiny changes that we indicate below.

\begin{lemma}\label{lem:combinatorial_fellow_travels_cat0}
 Let $X$ be a $D$--dimensional CAT(0) cube complex and let $\rho$ be a $G$--admissible hyperplane rescaling.  Let $(X,d^\rho_X)$ be the resulting CAT(0) cuboid complex.  Let 
$\gamma \from [0,\infty)\to X$ be a geodesic ray for the metric $d^\rho_X$, with $\gamma(0)\in X^{(0)}$.  Then $X$ contains a combinatorial geodesic ray $\alpha$ in $X^{(1)}$ such that 
$\alpha(0)=\gamma(0)$ and such that $\alpha$ fellow-travels with $\gamma$ at distance depending only on $D$ and $\rho$.
\end{lemma}

\begin{proof}
The proof from~\cite[Section 2.2]{BeyrerFioravanti} works with almost no change.  First, \cite[Lemma 2.9]{BeyrerFioravanti} is about the combinatorial structure of $X$ only, which does not change 
when we rescale edges to pass from the CAT(0) metric $d_X$ to the CAT(0) metric $d_X^\rho$.  The proof of~\cite[Proposition 2.8]{BeyrerFioravanti}  needs this lemma, plus CAT(0) convexity of 
half-spaces, which continues to hold in view of Corollary~\ref{Lem:CuboidallyConvex}.  The proof of~\cite[Proposition 2.8]{BeyrerFioravanti} produces a combinatorial ray $\alpha$ such that each 
point of $\alpha$ lies in a common cube as a point of $\gamma$, and vice versa.  Hence the two rays fellow travel at CAT(0) distance bounded by the diameter of a cube.  In the metric $d_X$, this is 
bounded in terms of $D$; in the metric $d_X^\rho$, this is bounded in terms of $D$ and the constant $M_\rho$ from Definition~\ref{Def:hyp_rescaling}, as required by the statement.
\end{proof}

Since $d_X$ and $d_X^\rho$ are path-metrics, and for each cube $c$ of $X$, the standard CAT(0) metric $d_X$ on $c$ is bilipschitz 
to the restriction of $d_X^\rho$ (with constant depending only on the rescaling constant), the identity map 
$(X,d_X)\to(X,d_X^\rho)$ is bilipschitz.  Combining this fact with Lemma~\ref{Lem:WallQI} yields a version of Lemma~\ref{Lem:WallQI} for the rescaled metric $d_X^\rho$:

\begin{lemma}\label{Lem:CuboidWallQI}
There are constants $\lambda_0^\rho \geq 1, \lambda_1^\rho \geq 0$, depending on $\dim X$ and $\rho$, such that the following holds. For any pair 
of points $x, y \in X$,
$$
\frac{1}{\lambda_0^\rho} \, d_X^\rho(x,y) - \lambda_1^\rho \leq |\W(x,y)| \leq \lambda_0^\rho \, d_X^\rho(x,y) + \lambda_1^\rho.
$$
\end{lemma}

In a similar fashion, by Remark~\ref{Rem:L1MetricOnX}, there is a bilipschitz relationship between $(X, d_X^\rho)$ and $(X, d_1)$.

\section{The  simplicial Roller boundary $\absimp X$}\label{Sec:SimplicialRoller}
In this section, we discuss one of the main objects in the paper: the \emph{simplicial Roller boundary}.

\begin{definition}[Roller class]\label{Def:RollerEquivalence}
Let $x,y \in \~ X$.  Recall that $\frakH(x,y) = \frakH_x^+ \triangle\,  \frakH_y^+$ is the set of half-spaces that separate $x$ from $y$. We say that $x$ is \emph{commensurate} to $y$, and write $x\sim y$, if  $|\frakH(x,y)|< \infty$.  
Observe that $X^{(0)}$ forms a single commensurability class in $\~ X$. Commensurability classes contained in  $\roller X = 
\~X \setminus X^{(0)}$ are also called \emph{Roller classes}.
\end{definition}

\begin{definition}[Guralnik quotient]\label{Def:Guralnik}
 The \emph{Guralnik quotient} of the Roller Boundary is $\guralnik X = \roller X / \!\! \sim$,  the set of commensurability 
classes of points in the Roller boundary~\cite{Guralnik}. We shall often implicitly consider elements $[x]\in \guralnik X$ as 
subsets $[x]\subset \~X$. 
\end{definition}

In most cases,  the quotient topology on $\guralnik X = \roller X / \!\sim$ is not Hausdorff (in  fact, it does not even 
satisfy the $T_1$ axiom). We think of $\guralnik X$ as only a set, without a topology. In 
Definition~\ref{Def:SimplicialRoller}, we will use a partial order on $\guralnik X$ to define a simplicial complex $\absimp 
X$ with far nicer topological properties.

Following \cite{NevoSageev},  we  say a set $S\subset \frakH$ is \emph{non-terminating} if $S$ contains no minimal elements: 
given any half-space $h\in S$, there is a half-space $k\subsetneq h$ with $k\in S$.

\begin{lemma}\label{ClassSetNonterminating}
Let $C \subset \roller X$ be a closed, convex set. The following are equivalent:
\begin{enumerate}[$(1)$ ]
\item\label{closureclass} There is a point $y\in \partial_R X$ such that $C= 
\~{[y]}$.
\item  \label{finitely many descending chains} There exists $k\leq \dim (X)$ and a family $(h^1_m)_{m\geq 0}$, $(h_m^2)_{m\geq 0},\dots, (h_m^k)_{m\geq 0}$ of descending chains of  half-spaces  such that 
$$C =\bigcap_{i=1}^k\bigcap_{m\geq 0}h_{m}^i.$$
\item \label{union of roller classes} $C$ is a union of Roller classes.
\item \label{half-spaces nonterminating} The set of half-spaces $\frakH_{C}^+$ is non-terminating.
\end{enumerate}
\end{lemma}

\begin{proof}
We will prove \ref{closureclass} $\Rightarrow$ \ref{finitely many descending chains} $\Rightarrow$ \ref{union of roller classes} $\Rightarrow$ \ref{closureclass} and $\neg$\ref{union of roller classes} $\Leftrightarrow \neg$\ref{half-spaces nonterminating}. 

The claim that \ref{closureclass} implies \ref{finitely many descending chains}  is a restatement of \cite[Lemma 6.17]{FLM:RandomWalks}.

\smallskip

To prove that \ref{finitely many descending chains} implies \ref{union of roller classes},
assume that $y \in C = \bigcap_{i=1}^k\bigcap_{m\geq 0}h_{m}^i$ and $y' \sim y$ 
meaning that  $|\frakH(y, y')|< \infty$. 
Since $  \frakH_y^+ \setminus \frakH_{y'}^+$ is finite, it follows that there exists an $N$ for which $h_{m}^i \in  \frakH_{y'}^+$ (i.e. $y'\in h_{m}^i$) for each $m\geq N$ and $i=1, \dots, k$.
That is to say, 
$$y'\in \bigcap_{i=1}^k\bigcap_{m\geq N}h_{m}^i= \bigcap_{i=1}^k\bigcap_{m\geq 0}h_{m}^i = C.$$

\smallskip

To prove
\ref{union of roller classes} implies \ref{closureclass},
suppose that $C$ is  a union of Roller classes. 
Let $\frakH_{C}$ be the collection of half-spaces that cut $C$.  By Theorem~\ref{Thm:SageevDuality}, there is an associated cube complex $X(\frakH_{C})$. 
Since $C$ is closed and convex, Proposition \ref{LiftingDecomp} and Lemma~\ref{Lem:ConvexClosure} give an embedding $i_C \from \~X(\frakH_{C})\hookrightarrow \~X$ whose image is $C$, and furthermore
$C= \bigcap_{h\in \frakH_C^+}h$.

Let $y_0 \in X(\frakH_{C})^{(0)}$ and note that $\~{[y_0]}= \~X(\frakH_{C})$.   Let  $y=i_C(y_0)\in C$ and note that $i_C([y_0]) \subset [y]$, hence $C=i_C(\~{[y_0]}) \subset \~{[y]}$.  Since $C$ is a union of equivalence classes, we have ${[y]}\subset C$. Since $C$ is closed, we have that $ \~{[y]}\subset C$ and hence $C= \~{[y]}$.

\smallskip

To prove $\neg$\ref{half-spaces nonterminating} implies $\neg$\ref{union of roller classes},
suppose  there exists a minimal half-space $h_0\in \frakH_{C}^+$. We claim that $\frakH_{C}^+\setminus\{h_0\}$ is consistent. Indeed, if $h\in \frakH_{C}^+\setminus\{h_0\}$ then $h^*\notin \frakH_{C}^+\setminus\{h_0\}$. If $h\subset k$ then $k\in \frakH_{C}^+$, and since $h_0$ is minimal, $h_0\neq k$ so $k\in \frakH_{C}^+\setminus\{h_0\}$. 

We may therefore consider the CAT(0) cube complex associated to $\frakH_C \sqcup\{h_0, h_0^*\}$. Fix $y, y' \in X(\frakH_C \sqcup\{h_0, h_0^*\})$ such that $y\in h_0$ and $y'\in h_0^*$. Then, by Proposition \ref{LiftingDecomp}  there is a unique point $\tilde  y\in C$ associated to the total and consistent choice of half-spaces $$\{h\in \frakH_C: y\in h\}\sqcup \frakH_C^+.$$ Similarly, there is a unique point $\tilde y'\in \roller X\setminus C$ associated to $$\{h\in \frakH_C: y'\in h\}\sqcup \{h_0^*\}\sqcup (\frakH_C^+\setminus\{h_0\}).$$
By construction, $|\frakH(\tilde y, \tilde y')| = 2$, hence $C$ is not a union of Roller classes. 

\smallskip

Finally, to prove $\neg$\ref{union of roller classes} implies $\neg$\ref{half-spaces nonterminating}, suppose that $C$ is not a union of Roller classes. Then there are commensurate elements 
$y\in C$ and  $y' \in \roller X \setminus C$. Convexity means 
$$C = \bigcap_{h\in \frakH_{C}^+}{}h,$$
hence there is a half-space $h\in \frakH_{C}^+$ with $y\in C \subset h$ and $y'\notin h$. Under the assumption that 
$|\frakH(y,y')|<\infty$, we may choose a minimal such $h$ and conclude that $\frakH_{C}^+$ fails to be non-terminating. 
\end{proof}

\begin{definition}[Principal class]\label{Def:PrincipalClass}
Given a closed, convex set $C$ satisfying 
one of the equivalent conditions of Lemma~\ref{ClassSetNonterminating}, the class $[y]$ whose closure is $C$ is called the \emph{principal class of $C$}. Compare Guralnik \cite[Section 3.1]{Guralnik}. 
\end{definition}

\begin{definition}[Partial order on $\guralnik X$]\label{Def:RollerOrder}
Let $[x],[y] \subset \roller X$ be Roller classes. 
We define a partial order on $\guralnik X$ by setting $[x] \leq [y]$ whenever $\frakH_{[x]}^+ \subset \frakH_{[y]}^+$, or equivalently $\frakH_{\~{[x]}}^+ \subset \frakH_{[y]}^+$.
\end{definition}

The following result is essentially a corollary of Lemma~\ref{ClassSetNonterminating}.

\begin{lemma}\label{POSET Roller}
 For Roller classes $[x],[y] \subset \roller X$, the following are equivalent:
 \begin{enumerate}[$(1)$ ]
\item\label{Itm:leq} $[x] \leq [y]$ 
\item\label{Itm:half-spaces} $\frakH_{[x]}^+ \subset \frakH_{[y]}^+$
\item\label{Itm:ContainedInClosure} $[y]\subset \~{[x]}$
\item\label{Itm:Intersection} $[y]\cap \~{[x]} \neq \emptyset$ 
\end{enumerate}

\end{lemma}

\begin{proof}
Conditions \ref{Itm:leq} and \ref{Itm:half-spaces} are equivalent by Definition~\ref{Def:RollerOrder}.  Since $\frakH_{[x]}^+ =\frakH_{\~{[x]}}^+ $, by Lemma~\ref{Lem:ConvexClosure}, Remark \ref{SubsetReverseInclusion} implies that \ref{Itm:half-spaces} and \ref{Itm:ContainedInClosure} are equivalent. Next, \ref{Itm:ContainedInClosure} trivially implies \ref{Itm:Intersection}. Finally, \ref{Itm:Intersection} implies \ref{Itm:ContainedInClosure} by Lemma~\ref{ClassSetNonterminating}.
\end{proof}

\begin{definition}[Simplicial Roller boundary]\label{Def:SimplicialRoller}
The \emph{simplicial Roller boundary} $\absimp X$ is the simplicial realization of the partial order of Definition~\ref{Def:RollerOrder}. That is: vertices of $\absimp X$ correspond to points of $\guralnik X$, or equivalently Roller classes in $\roller X$. Simplices in  $\absimp X$ correspond to totally ordered chains of Roller classes. 

Observe that the action of $\Aut(X)$ on $\roller X$ preserves commensurability classes and the order $\leq$, yielding a bijective action  on $\guralnik X$ and a simplicial action on $\absimp X$.
\end{definition}

\begin{remark}[Dimension of $\absimp X$]\label{Rem:AbsimpDimension}
The characterization of Lemma~\ref{ClassSetNonterminating}.\ref{finitely many descending chains} implies that any simplex in $\absimp X$ has $k \leq \dimension X$ vertices, hence $\dimension\absimp X\leq\dimension X-1$.
\end{remark}

\section{The simplicial boundary}\label{Simplicial Boundary}

In this section, we recall the definition of the simplicial boundary $\partial_\triangle X$  \cite{Hagen}; see Definition~\ref{Def:simplicial_boundary}. Then, we 
develop some connections between $\partial_\triangle X$, the Roller boundary $\roller X$, and the simplicial Roller boundary $\absimp X$. See Theorem~\ref{Thm:UBStoRoller} and Corollary~\ref{Cor:UBStoRollerSimp} for the most top-level statement. The lemmas used to prove those top-level results will be extensively used in later sections.

\subsection{Unidirectional boundary sets} 
Recall that $\frakH$ is the set of half-spaces of $X$ and $\W$ is the set of hyperplanes of $X$. In this section, we mainly focus on collections of hyperplanes.

\begin{definition}[Facing triple]\label{Def:FacingTriple}
A  triple  $h_1, h_2, h_3 \in \frakH$ is called a \emph{facing triple} if  their complements $h_1^*, h_2^*, h_3^*$ are 
pairwise disjoint. Similarly, a triple of hyperplanes $\hat  h_1, \hat  h_2, \hat  h_3 \in \W$ is called a \emph{facing 
triple} if there exists a choice of orientation such that $h_1, h_2, h_3$ form a facing triple.
\end{definition}

Caprace and Sageev  \cite{CapraceSageev} showed that  in a non-elementary
 essential CAT(0) cube complex, given any half-space $h\in \frakH$ there exist half-spaces $k, \ell \in \frakH$ such that $\{h,k,\ell\}$ is a facing triple.

\begin{definition}[Unidirectional, inseparable]\label{Def:Unidirectional}
A collection  of hyperplanes $\mathcal S \subset \W$ is called \emph{unidirectional} if for every $\hat h \in \mathcal S$, at most one side of $\hat h$ contains infinitely many other elements of $\mathcal S$. 
A collection $\mathcal S\subset \W$ is called \emph{inseparable} if whenever $\hat h, \hat k \in \mathcal S$ and $\hat \ell$ is a hyperplane that separates $\hat h$ and $\hat k$, then $\hat \ell \in \mathcal S$. This is closely related to the ``tightly nested'' condition in Proposition~\ref{LiftingDecomp}.
\end{definition}

\begin{definition}[UBS]\label{Def:UBS}
 A set $\U \subset \W$ is a \emph{unidirectional boundary set} (abbreviated \emph{UBS}) if it is infinite, unidirectional, inseparable, and does not contain a facing triple. 
 \end{definition}
 
\begin{definition}[Partial order on UBSes]\label{Def:PreceqMinimal} 
Let $\U$ and $\mathcal V$ be UBSes.  Define a relation $\preceq$ where $\U\preceq
\mathcal V$ if all but finitely many elements of $\U$ lie in
$\mathcal V$.  Say that $\U$ and $\V$ are \emph{commensurate} (denoted $\U \sim \V$)  if $\U \preceq \V$ and $\V \preceq \U$.

The relation $\preceq$ on UBSes descends to a
partial order on commensurability classes of UBSes. We say that $\V$ is \emph{minimal} if its commensurability class is minimal --- that is, if  $\U \preceq \V$ implies $\U \sim \V$. Similarly, we say that $\V$ is \emph{maximal} if $\V \preceq \V'$ implies $\V \sim \V'$.
\end{definition}

\begin{example}\label{intervalUBS}
For disjoint subsets $A, B \subset \~X$, let $\W(A, B)$ be the set of hyperplanes separating $A$ from $B$.
 It is straightforward to check that $\W(A, B)$ is inseparable and does not contain a facing triple. Furthermore, if $A\subset X$ then $\W(A, B)$ is unidirectional.  If, in addition, $B\subset \roller X$, then $\W(A, B)$ is infinite. Under these last two conditions, we conclude that $\W(A, B) = \W(A, \overline{B})$ is a UBS. Compare Lemma \ref{standardUBS}.
 \end{example}

\begin{definition}[Pruned UBS, canonical half-space]\label{Def:Pruned}
Let $\U$ be a UBS. We say that $\U$ is \emph{pruned} if every hyperplane $\hat h \in \U$ has exactly one associated half-space $h$ containing infinitely many elements of $\U$. In this situation, we call 
$h$ the  \emph{canonical orientation of $\hat h$} or the \emph{canonical half-space associated to $\hat h$}.
\end{definition}

\begin{lemma}\label{Lemma:PrunedUBS}
Let $\U$ be a UBS. Then $\U$ contains a canonical pruned sub-UBS $\U' \subset \U$ that is commensurate to $\U$.
\end{lemma}

\begin{proof}
Let $F_\U$ be the set the hyperplanes of $\U$ with the property that \emph{both} associated half-spaces contain finitely many elements of $\U$. We claim that $F_\U$ is finite.
To see this, let $\hat  h_1, \dots, \hat  h_D$ be a maximal family of pairwise transverse hyperplanes in $F_\U$. This family is finite because $D \leq \dim X$. For any $\hat  h \in F_\U\setminus\{\hat  h_1, \dots, \hat  h_D\}$ there is at least one $i$ for which $\hat h$ is parallel (i.e. not transverse) to $\hat  h_i$. By the definition of $F_\U$,  there are finitely many hyperplanes in $\U$ parallel to each $\hat  h_i$,
 hence $F_\U$ is finite. 
 
 Now, define $\U' = \U \setminus F_\U$. By construction, every $\hat h \in \U'$ has exactly one associated half-space containing infinitely many elements of $\U$, hence infinitely many elements of $\U'$ as well.
 
We claim that $\U'$ is a UBS. Since $F_\U$ is finite, it follows that $\U'$ is infinite. Unidirectionality and the lack of facing triples are properties that pass to subsets of $\U$, hence to $\U'$.  Next, suppose that $\hat  h, \hat  k \in \U'$ with $\hat  \ell $ in between. Since $\hat  h, \hat  k \in \U$ and $\U$ is inseparable, observe that $\hat \ell \in \U$. But
$\hat  \ell \notin F_\U$, because $\hat h, \hat k$ both bound half-spaces containing infinitely many elements of $\U$. We conclude that $\hat  \ell \in \U'$, verifying that $\U'$ is a UBS. Since $\U$ and $\U'$ differ by finitely many hyperplanes, they are commensurate.
\end{proof}

Hagen proved that every (commensurability class of) UBS $\U$ decomposes in a well-defined fashion into a disjoint union of minimal UBSes.
The following is a restatement of \cite[Theorem 3.10]{Hagen} and \cite[Theorem A]{Hagen:corr}.

\begin{lemma}\label{Lem:FinManyMinimals}
Let $\U$ be a UBS. Then there is a decomposition $\U \sim \bigsqcup_{i=1}^k \U_i$, where:
\begin{enumerate}[$(1)$ ]
  \item\label{Itm:Minimal} each $\U_i$ is a minimal UBS;
  \item\label{Itm:DisjointUniion} any UBS $\V \subset \U$ is commensurate with the disjoint union of some of the $\U_i$;
  \item\label{Itm:CrossAllBelow} if $1\leq j < i \leq k$, then every hyperplane in $\U_i$ crosses all but finitely many of the hyperplanes in $\U_j$.
 \end{enumerate}
Consequently, there are finitely many commensurability classes of UBSes $\V \preceq \U$.
\end{lemma}

Lemma~\ref{Lem:FinManyMinimals} enables the following definition.

\begin{definition}[Simplicial boundary]\label{Def:simplicial_boundary}
For a UBS $\U$, the \emph{dimension} of $\U$ is the number $k$ appearing in the decomposition of Lemma~\ref{Lem:FinManyMinimals}. This is an invariant of the commensurability class.
An immediate consequence of Lemma~\ref{Lem:FinManyMinimals}.\ref{Itm:CrossAllBelow} is that $\dim (\U) \leq \dim (X)$.

The \emph{simplicial boundary} $\simp X$ is the following simplicial complex.  For each commensurability class of a UBS 
$\U$, there is a $(k-1)$--simplex $\sigma_{\U}$, where $k$ is the dimension of the class of $\U$.  The face relation is defined 
as follows: $\sigma_{\U}$ is a face of $\sigma_{\V}$ if and only if $\U \preceq \V$.  
\end{definition}

Note that $\simp X$ has a vertex for each commensurability class of a minimal UBS. Furthermore, a finite collection of 
vertices spans a maximal simplex if and only if the union of the corresponding minimal UBSes is commensurate with a maximal UBS. However, there can exist non-maximal simplices that do not correspond 
to a UBS; see Example~\ref{exmp:weird_simplex}. For this reason, $\simp X$ is \emph{not} the simplicial realization of the partial order $\preceq$.

Observe that $\Aut(X)$ acts on UBSes in a way that preserves  the $\preceq$ relation, and hence preserves commensurability. Thus $\Aut(X)$ acts simplicially on $\simp X$.

\subsection{Minimal and dominant UBSes}\label{Subsec:MinimalDominant}
Next, we establish some structural results about minimal UBSes. 
Given a set $S$ consisting of hyperplanes, the \emph{inseparable closure} of $S$, denoted $\insep{S}$,
 is the intersection of all inseparable sets containing $S$. 

\begin{lemma}\label{MinimalChar}
Let $\U\subset\W(X)$ be a minimal UBS.  Then there exists an infinite descending chain $\{h_n : n \in \naturals \} \subset 
\frakH$ with $\hat h_n \in \U$, such that $\U\sim \insep{\{\hat h_n:n\in\N\}}$.

Conversely, for every descending chain $\{h_n : n \in \naturals \} \subset \frakH$, there exists $N\geq 0$ such that $\insep{\{\hat h_n:n\geq N\}}$ is a minimal UBS.
\end{lemma}

\begin{proof}
Let $\U$ be a minimal UBS.  Let  $\{h_n : n \in \naturals \} \subset \frakH$ be a descending chain such that ${\{\hat h_n\}}\subset\U$, which exists because $\U$ is infinite, unidirectional, contains no facing triple, and 
contains no infinite collection of pairwise-crossing hyperplanes.  Let $\U' = \insep { \{\hat h_n\} }$ be the inseparable closure of $\{\hat h_n\}$. Since $\U$ is 
inseparable, we have $\U'\subset\U$ by the definition of the inseparable closure.   Note that $\U'$ is a UBS (compare the proof of~\cite[Lemma 3.7]{Hagen} or~\cite[Caprace Lemma B.6]{CFI}).  Since $\U'\subset\U$ and $\U$ is minimal, we have $\U'\sim\U$, which 
proves the first assertion of the lemma.

Now we prove the second assertion.  For each $m\geq 0$, let $C_m=\{h_n:n\geq m\}$ and let $\insep{C_m}$ be its inseparable closure.  As before, each $\insep{C_m}$ is a UBS.  Note that $\insep{C_m} \subset 
\insep{C_{m'}}$ whenever $m\geq m'$.  In particular, $\insep{ C_m} \preceq \insep {C_{m'}}$ for $m' \geq m$.  So, by Lemma~\ref{Lem:FinManyMinimals}, there exists $M\geq 0$ such that $\insep{ C_m} \sim \insep{C_M}$ for all 
$m\geq M$.  Hence, if $m\geq M$ and $\{k_i\}$ is any descending chain of half-spaces corresponding to hyperplanes in $\insep{C_M}$, then there exists $I\geq 0$ such that $\{\hat k_i : i\geq I\}\subset 
\insep{ C_m} $.  This implies that $\insep{ C_M}$ is commensurate with $\insep{\{\hat k_i :i \geq 0\}}$, by, for example~\cite[Lemma 4.4]{Fioravanti}. See also the proof of~\cite[Proposition 4.7.(2)]{Fioravanti}.    

Now let $\V\subset \insep{ C_M} $ be a minimal UBS, which exists by~\cite[Lemma 3.7]{Hagen}.  By the first part of the lemma, $\V$ is commensurate with the inseparable closure of some chain contained in 
$\insep{ C_M}$.  By the preceding paragraph, $\insep{ C_M} \sim\V$, so $\insep{C_M}$ is minimal, as required.
\end{proof}

\begin{remark}
Lemma~\ref{MinimalChar} has been generalized to the setting of median spaces by Fioravanti~\cite[Proposition 4.7]{Fioravanti}. In fact, the same result of Fioravanti also generalizes 
many other basic facts about UBSes, such as~\cite[Lemma 3.7 and Theorem 3.10]{Hagen}.  

The second assertion of Lemma~\ref{MinimalChar} corrects a slight misstatement in~\cite[Caprace Lemma B.6]{CFI}: it is 
true that minimal UBSes are commensurate with inseparable closures of chains, but inseparable closures of chains need not be minimal; one might first need to pass to a (cofinite) sub-chain (see 
e.g.~\cite[Figure 1]{Fioravanti} and the paragraph following).  We emphasize that this does not affect the use 
of~\cite[Caprace Lemma B.6]{CFI} in~\cite{CFI}, because what's being used is the (correct) statement that, up to 
commensurability, minimal UBSes arise 
as inseparable closures of chains.

In this paper, we primarily use the first assertion of Lemma~\ref{MinimalChar}, in conjunction 
with Lemma~\ref{Lem:FinManyMinimals}, to decompose a UBS $\U$ as a disjoint union of minimal UBSes (up to commensurability), 
each of which has the additional property that it is the 
inseparable closure of a chain.
\end{remark}

The following consequence of Lemma~\ref{MinimalChar} neatly sums up the special properties of chains whose inseparable 
closures are minimal UBSes:

\begin{cor}\label{Cor:MinimalInseparable}
Let $\{h_n:n\geq 0\}$ be an infinite descending chain of half-spaces.  For each $m\geq 0$, let $\U_m$ be the inseparable closure of $\{\hat h_n:n\geq m\}$.  Then the following are equivalent:
\begin{enumerate}[$(1)$ ]
 \item The UBS $\U_m$ is minimal for all $m\geq 0$.
 \item $\U_m\sim\U_0$ for all $m\geq 0$.
\end{enumerate}
In particular, if $\U_0$ is minimal, then $\U_0\sim\U_m$ for all $m\geq 0$.
\end{cor}

\begin{proof}
Since $\U_m\preceq\U_0$ for all $m$, the first statement immediately implies the second.  For the reverse direction, Lemma~\ref{MinimalChar} provides $M\geq 0$ such that $\U_m$ is minimal for $m\geq 
M$.  If the second assertion holds, it then follows that the first must also.  
\end{proof}

\begin{definition}[Chain of hyperplanes]\label{Def:chain_of_hyperplanes}
Following Corollary~\ref{Cor:MinimalInseparable},  for the rest of this subsection we shall refer to a sequence of disjoint hyperplanes $\{ \hat h_n \}$ whose inseparable closure is a minimal UBS
 as a \emph{chain of hyperplanes}.
\end{definition}
 
 In particular, a chain of hyperplanes $\{ \hat h_n \}_{n\geq 0}$ defines a descending chain of half-spaces $ \{ h_n \}_{n\geq 0}$.  A descending chain of half-spaces $\{ h_n \}_{n\geq 0}$ has the 
property that for all sufficiently large $N$, the sequence $\{ \hat h_n \}_{n\geq N}$  is a chain of hyperplanes, in view of Lemma~\ref{MinimalChar}.

\begin{definition}[Dominant]\label{Def:dominant_minimal}
Let $\U$ be a UBS. A hyperplane $\hat h \in \U$ is called \emph{dominant for $\U$} if $\hat h = \hat h_0$ is the base of a chain of hyperplanes $\{ h_n \} \subset \U$ such that $\hat h$ 
crosses all but finitely many hyperplanes in $\U \setminus \insep{ \{ h_n \}   }$. 
\end{definition}

 Recall that by definition of a \emph{chain of hyperplanes}, $\insep{ \{h_n \}}$ is a minimal UBS. See Figure~\ref{fig:dominant} for an example. 
By Lemma~\ref{Lem:FinManyMinimals}, the inseparable closure $\insep{ \{ h_n \}   }$  is commensurate to some minimal UBS $\U_i$ in the decomposition $\U \sim \bigsqcup_{i=1}^k \U_i$ of that lemma. Furthermore, $\U_i$ 
is unique: since a dominant hyperplane $\hat h \in \U$ is the base of a chain in $\U_i$, then $\hat h$ must intersect all but finitely many hyperplanes of $\U_j$ (for $j \neq i$), hence $\hat h$ 
cannot be disjoint from a chain in $\U_j$.

We also observe that if $\U$ is minimal and pruned, then every hyperplane of $\U$ is vacuously dominant.

\begin{figure}[h]
 \includegraphics[width=0.6\textwidth]{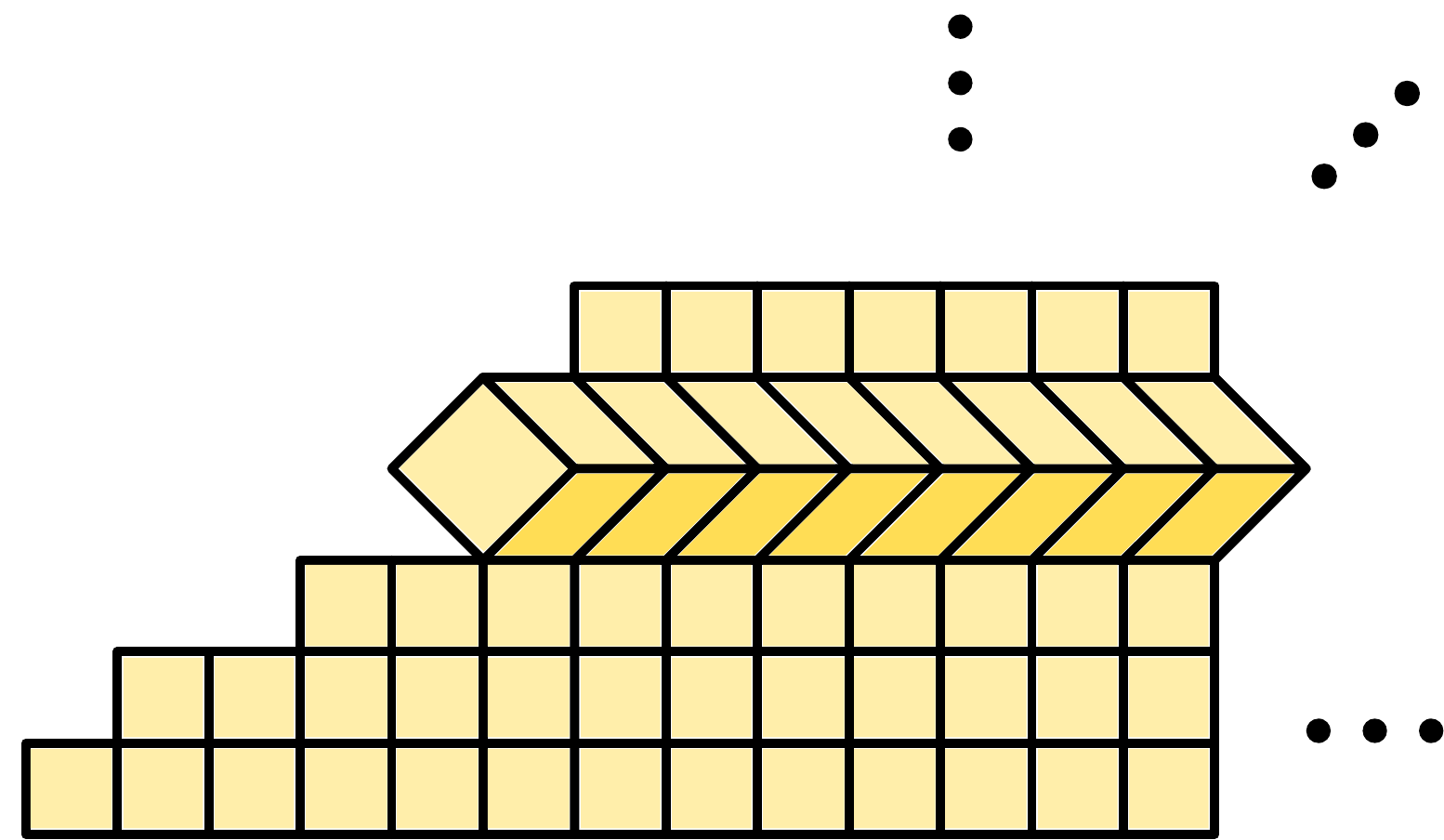}
 \caption{Part of a CAT(0) cube complex for which the set of hyperplanes is a UBS.  Each horizontal hyperplane is the base of a chain of horizontal hyperplanes whose inseparable closure consists of 
all but finitely many of the horizontal hyperplanes.  Each horizontal hyperplane crosses all but finitely many vertical ones.  So each horizontal hyperplane is dominant. The vertical hyperplanes are 
not dominant.  Indeed, if $v$ is vertical, it is the base of a chain of vertical hyperplanes whose inseparable closure consists of vertical hyperplanes, while $v$ fails to cross infinitely many 
horizontal hyperplanes.  On the other hand, $v$ is also the base of a sequence of hyperplanes consisting of $v$ and infinitely many horizontal hyperplanes, whose canonical half-spaces form a descending 
chain.  The inseparable closure of this set of hyperplanes contains all but finitely many of the hyperplanes, and therefore is not a minimal UBS, so this set is not a chain of hyperplanes.}
 \label{fig:dominant}
\end{figure}

The following lemma is only stated in the special case where $\W(X)$ is already a UBS. In practice, we will apply Lemma~\ref{Lem:DominantChain} to convex subcomplexes (such as the convex hull of a geodesic ray) where this hypothesis holds.

\begin{lemma}[Dominant hyperplane combinatorial properties]\label{Lem:DominantChain}
 Suppose that $\U = \W(X)$
  is a pruned UBS of dimension $k$.  Then $\U$ contains minimal UBSes $\mathcal D_d ,\dots, \mathcal D_k$ such that the following hold: 
 \begin{enumerate}[$(1)$ ]
  \item \label{item:dominant_hyperplanes}Each $\mathcal D_i$ is a set of dominant hyperplanes for $\U$, and all but finitely many dominant hyperplanes belong to $\bigcup_{i = d}^k \mathcal D_i$.
  \item \label{item:dominant_chain}For each $i\in \{d, \dots, k\}$ and each $\hat h_0 \in\mathcal D_i$, there is a chain $\{ \hat h_n 
\}_{n\ge0} \subset\mathcal D_i$ such that $\mathcal D_i$ is commensurate with $\insep{ \{ \hat h_n \}  }$,  the 
inseparable closure of the chain.
\item \label{item:dominant_sinks} For any such chain $\{ \hat h_n \}$, there is $N > 0$, such that for all $n \geq N$, we have 
$$\W(\hat h_{n+1})\setminus\mathcal D_i\subset\W(\hat h_n)\setminus\mathcal D_i.$$ Furthermore, for all $n 
\geq 0$, we have 
$$|\W(\hat h_n)-\W(\hat h_{n+1})|<\infty.$$
\item \label{item:dominants_cross} For every hyperplane $\hat u\in\U$, there
exists $i\in\{d,\ldots,k\}$ and a dominant hyperplane $\hat h\in\mathcal D_i$ 
such that $\hat h$ lies in the canonical 
half-space 
associated to $\hat u$. 
 \end{enumerate}
\end{lemma}

We will not use item \eqref{item:dominant_sinks} later in the paper, but we record it as a potentially useful fact about UBSes.

\begin{proof}[Proof of Lemma~\ref{Lem:DominantChain}]
By Lemma~\ref{Lem:FinManyMinimals}, there is a decomposition $\U \sim \bigsqcup_{i=1}^k\U_i$, with an ordering such that for $j < i$, every hyperplane in $\U_i$ crosses all but finitely many of the hyperplanes in $\U_j$. The ordering is only partially determined: if every hyperplane in $\U_i$ crosses all but finitely many hyperplanes in $\U_{i+1}$, then we can 
reverse the order of $\U_i,\U_{i+1}$ while keeping the conclusion of Lemma~\ref{Lem:FinManyMinimals}.\ref{Itm:CrossAllBelow}.  Let $d$ be the smallest index $i$ such that $\U_i$ can be placed at the top of the order in such a decomposition.

Now, fix an index $i \geq d$. 
By Lemma~\ref{MinimalChar}, $\U_i$ is commensurate with $\mathcal D_i = \insep{\{\hat h_n\}}$, where 
$\{\hat h_n\} \subset \U_i$ is a chain. Thus $\mathcal D_i \sim \U_i$ is a minimal UBS.

We claim that every $\hat h \in \mathcal D_i$ is dominant. 
 Given $\hat h\in\mathcal D_i$ and the canonical associated half-space $h$ (recall Definition~\ref{Def:Pruned}), there exists $N\geq 0$ such that $\hat h_N \subset h$. 
 Thus $\hat h$ is the base of the infinite chain $\{ \hat h,\hat h_N,\hat h_{N+1},\cdots\}$. Furthermore, $\hat h \in \mathcal D_i \subset \U_i$ intersects all but finitely many hyperplanes in $\U_j$ 
for $j \neq i$, because we have chosen an index $i \geq d$. Thus $\hat h$ is dominant by Definition~\ref{Def:dominant_minimal}.
 Thus we have found dominant minimal UBSes $\mathcal D_d,\ldots,\mathcal D_k$ satisfying conclusions~\ref{item:dominant_hyperplanes} and 
\ref{item:dominant_chain}.

Now we verify assertion \ref{item:dominant_sinks}. Fix $i \geq d$ and let $\hat h,\hat h'\in\mathcal D_i$.  Since the 
minimal UBS $\mathcal 
D_i$ is the inseparable closure of a chain, $\hat h,\hat h'$ both cross finitely many hyperplanes in $\mathcal 
D_i$.  Indeed, if $\mathcal W(\hat h)\cap\mathcal D_i$ is infinite, it must contain a chain of hyperplanes whose inseparable closure $\mathcal D'$ belongs to $\mathcal D_i$, and is contained in 
$\mathcal W(\hat h)$ since the latter is inseparable.  Since $\mathcal D_i$ is minimal, $\mathcal D'\sim\mathcal D_i$.  Hence $\mathcal W(\hat h)$ contains all but finitely many elements of $\mathcal 
D_i$;  since $\hat h\in\mathcal D_i$, this contradicts that $\mathcal D_i$ is the inseparable closure of a chain.  Thus $\mathcal W(\hat h)\cap\mathcal D_i$ is finite.  Moreover, since $\mathcal D_i$ 
is dominant, $\hat h,\hat h'$ both cross all but finitely many hyperplanes of $\U \setminus \mathcal D_i$, and we can therefore conclude that 
$$|\W(\hat h)\triangle\W(\hat 
h')|<\infty.$$

So,  it suffices to show that, for all sufficiently large $n$,  we have $\W(\hat 
h_{n+1})\setminus\mathcal D_i\subset\W(\hat 
h_n)\setminus\mathcal D_i$.
If this doesn't hold, then  for infinitely many $n$ we can find a hyperplane $\hat u_n\not\in\mathcal D_i$ that crosses 
$\hat h_{n+1}$ but not 
$\hat h_n$; note that $\hat u_n$ does not separate any two elements of our chain since $\hat u_n\not\in\mathcal 
D_i$.  Since there are no facing 
triples, $\hat u_n$ 
crosses $\hat h_m$ for 
$m\geq n+1$. 
   Since the collection of violating 
hyperplanes $\{ \hat u_n\}$ is infinite and contains no hyperplanes crossing $\hat h_0$, dominance of $\hat h_0$ is 
therefore 
contradicted.  Thus there exists $N$ such that $\W(\hat h_{n+1})\setminus\mathcal D_i\subset\W(\hat 
h_n)\setminus\mathcal D_i$ for $n\geq N$, establishing \ref{item:dominant_sinks}.

We now prove assertion~\ref{item:dominants_cross}.  Let $\hat u\in\U$ be a hyperplane, and let $u$ be the  canonical 
half-space associated to $\hat u$. Given the fixed decomposition $\U \sim \bigsqcup_{i=1}^k\U_i$, let $i$ be the 
largest index such that
$u$ contains a chain of hyperplanes in $\U$, whose inseparable closure $\U_i'$ is commensurate to some $\U_i$. Such an $i$ exists because $\U$ is pruned. We claim that $i \geq d$. This claim implies 
that $u$ contains infinitely many hyperplanes of $\mathcal D_i$, which will suffice to prove the lemma.

Suppose for a contradiction that $i < d$. Fix $\hat v \in \U_i'$, and consider a minimal UBS $\U_j$ for $j > i$. If any hyperplane in $\U_j$ fails to cross $\hat v$, then it fails to cross $\hat u$, 
hence $u$ contains a chain in $\U_j$, contradicting the maximality of $i$. See Figure~\ref{fig:maximal_i}.

\begin{figure}[h]
\begin{overpic}[width=0.3\textwidth]{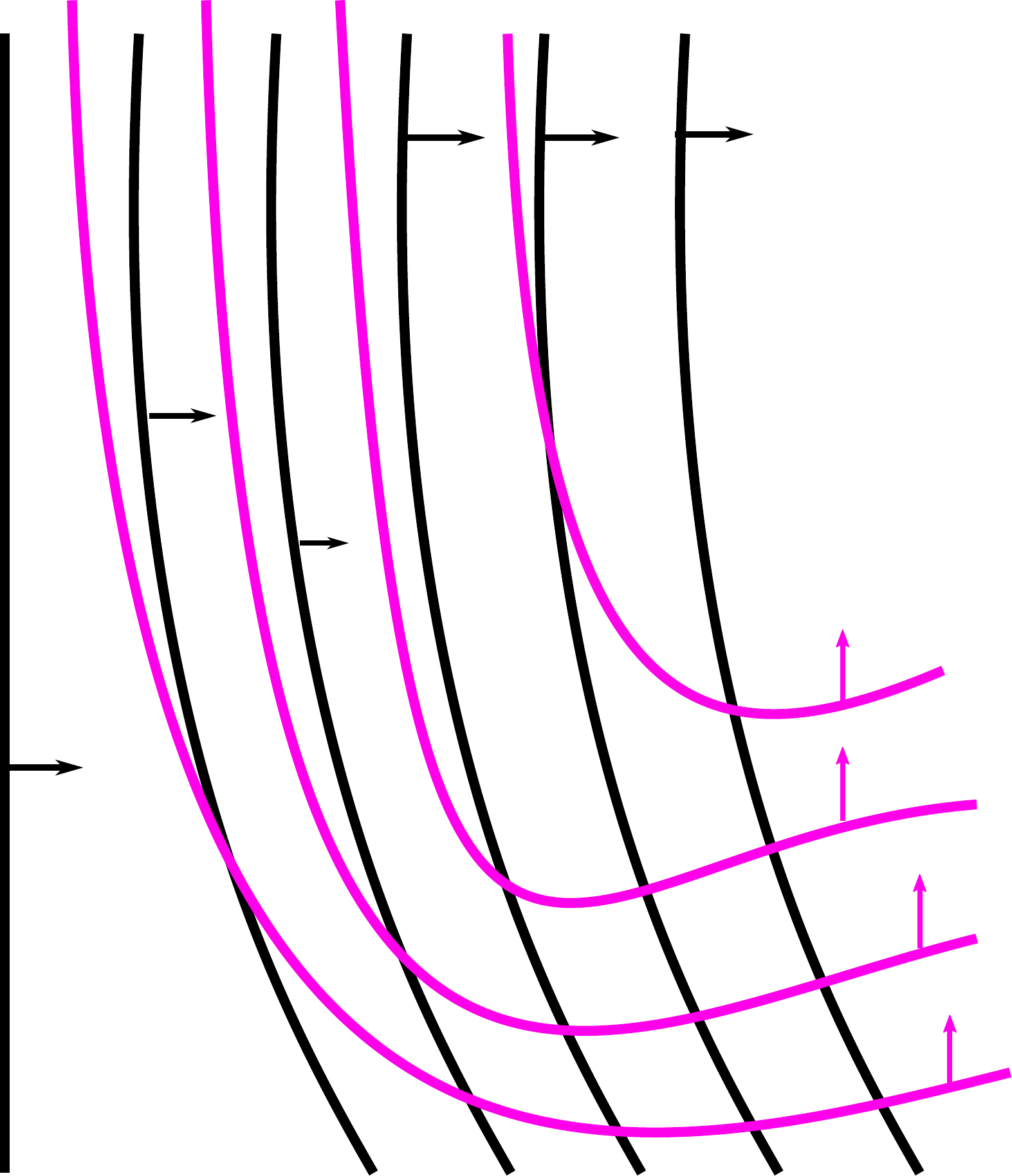}
\put(3,10){$\hat u$}
\put(50,-7){$\U_i'$}
\put(85,25){$\U_j$}
\end{overpic}

 \caption{The sets $\U'_i,\U_j$ in the proof of Lemma~\ref{Lem:DominantChain}.\ref{item:dominants_cross}.}\label{fig:maximal_i}
\end{figure}

Furthermore, since $\hat v \in \U_i$, it must cross all but finitely many hyperplanes of $\U_\ell$ for $\ell < i$ by 
Lemma~\ref{Lem:FinManyMinimals}.\ref{Itm:CrossAllBelow}. Thus $\hat v$ crosses all but finitely many hyperplanes of $\U \setminus \U_i'$, hence $\hat v$ is dominant. Since $\hat v \in \U_i'$ was 
arbitrary, this shows that $i \geq d$ by the definition of $d$.
\end{proof}

\subsection{Connections between $\simp X$ and $\roller X$}\label{Sec:RollerSimpConnections}
The next several lemmas develop tools for comparing the Roller boundary and the simplicial boundary.

\begin{definition}[Umbra of a UBS]\label{Def:Umbra}
Let $\U \subset \W(X)$ be a  UBS, and let $\U' \sim \U$ be a pruned UBS commensurate to $\U$. For every $\hat h \in \U'$, let 
$h \in \frakH(\~X)$ be the canonical orientation of $\hat h$,  that is,  the (extended, vertex) half-space  containing 
infinitely many hyperplanes 
of $\U'$. 
The \emph{umbra} $\~Y_\U$ associated to $\U$ is 
$$
\~Y_\U = \bigcap_{\hat  h \in  \U'}{} h \subset \~X.
$$
By Definition~\ref{Def:Convex}, the umbra $\~Y_\U$ is a convex subset of $\roller X$, one may think of it as the Roller realization of the UBS $\U$. In Lemma~\ref{Lem:UmbraClass}, we will show that  $\~Y_\U$ satisfies the equivalent conditions of Lemma~\ref{ClassSetNonterminating}. Then, $Y_\U$ will denote the principal class of $\~Y_\U$: that is, the unique Roller class whose closure is $\~Y_\U$.
\end{definition}

\begin{lemma}\label{UBScomplex}
Let $\U \subset \W(X)$ be a UBS. Then the umbra $\~Y_{\U}$ associated to $\U$ is a nonempty subset of $\roller X$.
Furthermore, 
 \begin{itemize}
\item If $\U \sim\V$ then $\~Y_\U = \~Y_\V$. 
\item If $\U  \sim \U_1\cup \cdots \cup \U_N$ for some UBSes $\U_1, \dots,\U_N$ then $$\~Y_\U = \bigcap_{i=1}^{N}\~Y_{\U_i}.$$
\end{itemize}
\end{lemma}

\begin{proof}
Let $\U'$ be the pruned UBS commensurate to $\U$. As in Definition~\ref{Def:Umbra}, let $h$ be the canonical orientation of every $\hat h \in \U'$, and define  $\mathfrak U' = \{h : \hat h \in \U' \} \subset \frakH$. Then
 $\~Y_\U:= \bigcap_{h\in \mathfrak U'}{}h$. 
  
 Since $\~X$ is compact, to show that $\~Y_\U$ is not empty, it is sufficient to show that $\mathfrak U'$ has the  finite 
intersection property. By the Helly property for half-spaces~\cite[Section 2.2]{Roller} this is equivalent 
to showing that every pair of half-spaces $h,k \in \mathfrak U'$ has nonempty intersection. But a disjoint pair 
$h, k\in \mathfrak U'$  contradicts the unidirectionality of $\U'$, as both sides of $\hat h$ would contain infinitely many elements  of $\U'$. 

Using the same argument, we observe that for every $h\in \mathfrak U'$ there is an infinite descending chain $\{ h_n \}$ with $h_0 = h$ and $h_n\in \mathfrak U'$. Indeed, fix $h\in \mathfrak U'$. Since $\U'$ is pruned, there are infinitely many hyperplanes in $\U'$ contained in $h$. Since $\U'$ is unidirectional, $\hat k \subset h$ and $\hat k \in \U'$ implies $k \subset h$. Applying induction, we find the infinite descending chain. Thus $\~Y_\U \subset \roller X$.

We now turn to the commensurability class of $\U$. Suppose that $\V$ is a UBS commensurate to $\U$. Since commensurability is an equivalence relation, we conclude that $\U'$ and $\V'$ are commensurate as well. Let $\mathfrak V' = \{h : \hat h \in \V' \} \subset \frakH$ be the set of canonical half-spaces associated to $\V'$.

We claim that $\mathfrak U'$ and $\mathfrak V'$ are also commensurate. Since the 
map taking half-spaces to hyperplanes is 2-to-1, it suffices to show that under 
this map, the set $\mathfrak U'\setminus \mathfrak V'$  has image in the finite 
set $\U'\setminus \V'$, and vice versa. Let $h\in  \mathfrak U'\setminus  
\mathfrak V'$. Then $\hat h\in \U'$. By definition of $\mathfrak U'$, we know 
that the set $\{\hat k \in \U: \hat k \subset h\}$ is infinite. The 
commensurability of $\U$ and $\V$ implies that $\{\hat k \in \V: \hat k \subset 
h\}$ is infinite as well. But $h\notin \mathfrak V'$, hence it must be the case 
that $\hat  h\notin \V'$.

To see that $\~Y_{\U}= \~Y_{\V}$ it is sufficient to  show that if $h\in 
\mathfrak U'$  then there is $k\in \mathfrak V'\cap  \mathfrak U'$ with 
$k\subset h$. But this is immediate. As was proved above, for any $h\in 
\mathfrak U'$ there is an infinite descending chain $\{ h_n \} \subset \mathfrak 
U'$ with $h_0 = h$. Since $\mathfrak U'$ and $\mathfrak V'$ are commensurate, 
for $N$ sufficiently large, we must have that $h_N\in \mathfrak V'$.

Finally, the construction above makes it clear that  if $\U$ is commesurate to $\U_1\cup \cdots \cup \U_N$ for some  UBSes 
$\U_1, \dots,\U_N$ then $\~Y_\U = \bigcap_{i=1}^{N}\~Y_{\U_i}$. 
\end{proof}

\begin{lemma}\label{Lem:UmbraClass}
For every UBS $\U$, there is a Roller class $Z \subset \roller X$ such that $\~Y_{\U} = \~Z$. Furthermore, the assignment $\U \mapsto Z = Y_\U$ is $\Aut(X)$--equivariant.
\end{lemma}

\begin{proof}
By Lemmas~\ref{Lem:FinManyMinimals} and~\ref{MinimalChar}, $\U$ is commensurate to a UBS $\V=\bigsqcup_{i=1}^k\U_i$, where $\U_i = \insep{\{\hat h_n^i\}}$ for a descending chain of half-spaces $\{ h_n^i \}$. Note that every $\U_i$ is pruned. Thus
\begin{equation}\label{Eqn:UmbraDecomposition}
  \~Y_{\U} =  \bigcap_{1\leq i \leq k} \~Y_{\U_i} = \bigcap_{1\leq i \leq k} \, \,  \bigcap_{\hat h \in \U_i} h = \bigcap_{1\leq i \leq k} \, \,  \bigcap_{n \in \naturals} h_n^i.
\end{equation}
Here, the first equality holds by Lemma~\ref{UBScomplex}, the second equality holds by Definition~\ref{Def:Umbra}, and the third equality holds because the hyperplanes added in the inseparable closure do not affect the intersection. By Lemma~\ref{ClassSetNonterminating}, a set of this form is the closure of a unique (principal) Roller class $Y_\U$. 

We observe that the above construction is independent of the choice of  descending chains with the desired properties, and hence the map $\U \mapsto Y_\U$ is $\Aut(X)$--equivariant.
\end{proof}

\begin{lemma}\label{standardUBS}
Let $x\in X^{(0)}$ and let $Y \subset \roller X$ be a convex set. Then the set of separating hyperplanes $\W(x,Y)= \W(x,\gate_Y(x))$ is a pruned UBS.  
If $x'\in X^{(0)}$, then 
$\W(x',Y) \sim \W(x,Y)$. Finally, $\W(x,[y])\sim\W(x,y)$ for every $y\in \roller X$. 
\end{lemma}

\begin{proof}
We begin by showing that  $\V:=\W(x,Y)$ is a UBS.  Recall that $\W(x,Y)$ is the set of hyperplanes associated to $\frakH_x^+\triangle\, \frakH_Y^+$. Since $x \in X^{(0)}$, it follows that $\frakH_x^+$ satisfies the descending chain condition and since $Y\subset \roller X$, it follows that $\frakH_{Y}^+=\frakH_{\~Y}^+$ 
has an infinite descending chain. Hence, $\V$ is unidirectional and infinite. Since $\V=\W(x,Y)$ is an interval, it follows that it is inseparable and does not contain a facing triple, hence $\V$ is a UBS. The fact that $\W(x,Y)= \W(x,\gate_Y(x))$ follows from Proposition~\ref{Prop:Proj} and shows that $\V$ is pruned.  

Let $x,x' \in X^{(0)}$. We claim that $\W(x,Y)\triangle\,\W(x',Y) \subset \W(x,x')$.  By symmetry, it suffices to observe that  a hyperplane $\hat  h$ separating $x$ from $Y$ but not $x'$ from $Y$ necessarily separates $x$ from $x'$. Since the cardinality of $\W(x,x')$  is the distance $d(x,x')<\infty$, we conclude that the two UBSes are commensurate. 
 
Finally, fix $y\in \partial X$. We show that $\W(x,[y])\sim\W(x,y)$. Proposition~\ref{Prop:Proj} says that $\W(x,[y]) = \W(x,\gate_{[y]}(x))$, which implies that $\W(x,[y]) \cup \W(\gate_{[y]}(x), y)= \W(x,y)$. But $\gate_{[y]}(x) \in [y]$ by definition, hence  $| \W(\gate_{[y]}(x),y) | < \infty$, completing the proof. 
\end{proof}

\begin{definition}[UBS of a Roller class]\label{Def:UBSofRoller}
Let  $Y \in \guralnik X$ be a Roller class, and fix a basepoint $x \in X^{(0)}$. Following Lemma~\ref{standardUBS}, we call $\W(x,Y)$ a \emph{UBS representing the class $Y$}, and denote it $\U_Y$. Since the commensurability class $[\U_Y]$ is independent of $x$, by Lemma~\ref{standardUBS}, we do not include the basepoint $x$ in the notation $\U_Y$.

It follows immediately that the assignment $Y \mapsto [ \U_Y ]$ is $\Aut(X)$--equivariant.
\end{definition}

With this notation in place, we  conclude:

\begin{cor}\label{CorOnto}
Fix $x\in X^{(0)}$ and $Y\in \guralnik X$. Then  $Y = Y_{\U_Y}$. In particular,  every Roller class arises as the principal class of the umbra of some UBS.
\end{cor}

\begin{proof}
Since $\U_Y = \W(x,Y)$ is a UBS, by Lemma~\ref{Lem:UmbraClass} it is commensurate to $\bigsqcup_{i=1}^k  \insep{\{\hat h_n^i\}}$ for descending chains $\{ h_n^i \}$, hence Equation~\eqref{Eqn:UmbraDecomposition} gives
$$
  \~Y_{\U_Y} =  \bigcap_{1\leq i \leq k} \, \,  \bigcap_{n \in \naturals} h_n^i.
$$
 By construction, $\~Y \subset h_n^i$ for every $\hat h_n^i \in \U_i$, hence $\~Y \subset \~Y_{\U_Y}$.

As in Lemma \ref{Lem:UmbraClass}, let  $Y_{\U_Y}$  denote  the principal Roller class of $\~Y_{\U_Y}$. 
 We claim that $Y = Y_{\U_Y}$. So far, we have shown that $Y_{\U_Y} \leq Y$ in the partial order of Lemma~\ref{POSET Roller} and so we are left to show that $Y\leq Y_{\U_Y}$. To this end, we fix a half-space $h\in \frakH$ with $Y \subset h$ and now show that  $Y_{\U_Y} \subset h$. 

Let $x'\in h^* \cap X^{(0)}$ and set $\U'_Y = \W(x',Y)$. Observe that $\hat h \in \U'_Y$ by construction. Then by Lemma \ref{standardUBS} we have that $\U'_Y$ is a pruned UBS commensurate to $\U_Y$ and so Lemma \ref{UBScomplex} implies $\~Y_{\U_Y}=\~Y_{\U'_Y}$. Finally, by Definition~\ref{Def:Umbra} we have that $Y_{\U_Y}\subset\~Y_{\U_Y}\subset h$. 
\end{proof}

\begin{definition}[$\ell^1$--visible]\label{Def:L1visible}
Let $\gamma \subset X$ be a geodesic ray, in either the CAT(0) or combinatorial metric. Let  $\W(\gamma)$ be the collection of hyperplanes intersecting $\gamma$, which is a UBS by  \cite[Section 3.5.1]{Hagen}.
 If $\U$ is a UBS and $\U\sim\mathcal 
W(\gamma)$ for some combinatorial geodesic ray $\gamma$, we say that the class of $\U$ is \emph{$\ell^1$--visible}. \end{definition}

\begin{remark}\label{Rem:UBSofRollerGeodesic}
Let $Y \in \guralnik X$. Then, for any $x \in X^{(0)}$ and any $y \in Y$, Lemma~\ref{Lem:CombGeodesicExists} provides a combinatorial geodesic ray $\gamma$ such that $\gamma(0)=x$ and $\gamma(\infty) =  y \in Y$. By Definition~\ref{Def:L1visible} and Lemma~\ref{Lem:ConvexClosure}, we have 
$$\W(\gamma) = \W(x, y) \sim \W(x, [y]) = \W(x, \~{[y]}).$$

Furthermore, by Lemma~\ref{standardUBS}, the commensurability class $[\W(\gamma)]= [\U_Y]$ does not depend on the choice of $x \in X^{(0)}$ and  $y \in Y$.  In particular, every UBS representing a class $Y$ is $\ell^1$--visible.
\end{remark}

\begin{lemma}\label{Lem:UBSclassSubset}
Let $\U$ be a UBS.  Let $y\in \~Y_{\U}$ and let $\gamma$ be a combinatorial geodesic in $X$ such that 
$\gamma(\infty) = y$. 
 Then $\U  \preceq \W(\gamma)$.  In particular, if $\U$ is maximal, then $\U  \sim \W(\gamma)$.
\end{lemma}

\begin{proof}
As in Lemma~\ref{Lem:UmbraClass}, $\U$ is commensurate to $\bigsqcup_{i=1}^k \U_i$, where the $\U_i$ are minimal and $\U_i = \insep {\{\hat h_n^i\}}$ for a descending chain of half-spaces $\{ h_n^i \}$. Thus 
Equation~\eqref{Eqn:UmbraDecomposition} gives
$$
  \~Y_{\U} =  \bigcap_{1\leq i \leq k} \~Y_{\U_i} = \bigcap_{1\leq i \leq k} \, \,  \bigcap_{\hat h \in \U_i} h = \bigcap_{1\leq i \leq k} \, \,  \bigcap_{n \in \naturals} h_n^i.
$$
Since $\gamma(m) \to y \in \~Y_{\U}$, we must have $y \in h_n^i$ for every $i$ and every $n$. At the same time, since the sequence $\{ h_n^i \}$ is a descending chain for each $i$, we have that $x = \gamma(0)
\notin h_n^i$ for sufficiently large $n$. Thus all but finitely many hyperplanes $\hat h_n^i$ belong to $\W(x,y)=\W(\gamma)$. The inseparable closure of this cofinite subset of $\{\hat h_n^i\}$ must belong 
to $\W(\gamma)$, and by minimality of $\U_i$ and Corollary~\ref{Cor:MinimalInseparable}, it follows that all but finitely many elements of each $\U_i$ belong to $W(\gamma)$.
\end{proof}

\begin{cor}\label{Cor:MaxIsVisible}
Let $\U$ be a maximal UBS. Then, for every $y \in  Y_\U$, there is a combinatorial geodesic $\gamma$ limiting to $y$, such that $\U \sim \W(\gamma)$. In particular, the commensurability class $[\U]$ is $\ell^1$--visible.
\end{cor}

\begin{proof}
Let $y \in Y_\U$. By Lemma~\ref{Lem:CombGeodesicExists}, there is a combinatorial geodesic ray $\gamma$ limiting to $y$.  By  
Lemma~\ref{Lem:UBSclassSubset}, we have $\U \sim \W(\gamma)$. Thus, by Definition~\ref{Def:L1visible}, $\U$ is $\ell^1$--visible.
\end{proof}

In Corollary~\ref{Cor:MaxIsVisible}, the conclusion that $[\U]$ is visible was previously shown by Hagen \cite[Theorem~3.19]{Hagen}. The above proof, using the Roller 
boundary viewpoint, is considerably simpler.

\subsection{Order-preserving maps}
Recall from Definitions~\ref{Def:RollerOrder} and~\ref{Def:PreceqMinimal} that the partial order $\leq$ on Roller classes is set containment of the associated sets of half-spaces, while the partial order $\preceq$ on commensurability classes of UBSes  is set containment up to commensurability. 
The following theorem says that these partial orders are closely related.

\begin{theorem}\label{Thm:UBStoRoller}
Let $\mathcal{UBS}(X)$ be the collection of commensurability classes of UBSes in $X$. Then there are $\Aut(X)$--equivariant functions 
$\UR \from \mathcal{UBS}(X) \to \guralnik X$ and $\RU \from \guralnik X \to \mathcal{UBS}(X)$, with the following properties.
\begin{enumerate}[$(1)$ ]
\item\label{Itm:URonto} The assignment $\UR \from [\U] \mapsto Y_\U$ is order-preserving and onto.
\item\label{Itm:RUsection} The assignment $\RU \from Y \mapsto [\U_Y]$ is an order-preserving section of $\UR$.
\item\label{Itm:URmax} $\UR$ restricts to a bijection between the set of  $\preceq$--maximal classes of UBSes and the set of $\leq$--maximal Roller classes, with inverse $\RU$.
\end{enumerate}
\end{theorem}

The notation $\RU$ stands for ``Roller to unidirectional,'' while $\UR$ stands for ``unidirectional to Roller.'' This is inspired from the classical notation denoting the collection of maps from a set $A$ to a set $B$ as $B^A$.

\begin{proof}
To prove \ref{Itm:URonto}, let $\U$ be a UBS. By Lemma~\ref{Lem:UmbraClass}, there is a well-defined Roller class $Y_\U$ such that the
 umbra $\~Y_\U$ is the Roller closure of $Y_\U$. By 
Lemma~\ref{UBScomplex}, $Y_\U$ only depends on the commensurability class $[\U]$. Hence the assignment $[\U] \mapsto Y_\U$ gives a well-defined function  
$\UR \from \mathcal{UBS}(X) \to \guralnik X$. This function $\UR$ is $\Aut(X)$--equivariant by Lemma~\ref{Lem:UmbraClass}, and onto by Corollary~\ref{CorOnto}. 

Next, we check that $\UR$ respects the ordering. Lemma~\ref{UBScomplex} implies that if
 $\U\preceq \mathcal V$, then
$\overline{Y}_{\mathcal V}\subset \overline{Y}_{\U}$, hence the
Roller classes satisfy
$Y_{\U} \leq Y_{\mathcal V}$ by Lemma~\ref{POSET Roller}. Thus $\UR$ is order-preserving, proving \ref{Itm:URonto}.

\smallskip

To prove \ref{Itm:RUsection}, let $Y \in \guralnik X$. By Definition~\ref{Def:UBSofRoller}, there is a well-defined commensurability class of UBS representing $Y$, namely $[\U_Y] = \UR(Y) \in \mathcal{UBS}(X)$. Thus $\RU$ is well-defined. Equivariance also follows from Definition~\ref{Def:UBSofRoller}. Next, observe that
$$
\UR \circ \RU(Y) = \UR([\U_Y]) = Y_{\U_Y} = Y,$$
 where the last equality follows from Corollary~\ref{CorOnto}. Thus $\RU$ is a section of $\UR$.

To check that $\RU$ is order-preserving, let $Y, Z$ be Roller classes with $Y \leq Z$. By  Lemma~\ref{POSET Roller}, this means $\overline{Z} \subset \overline{Y}$.  Fix a base vertex $x\in X^{(0)}$.  By Lemma~\ref{Lem:CombGeodesicExists}, there exist 
combinatorial geodesic rays $\gamma_y, \gamma_z$, originating at $x$, such that $\gamma_y(\infty) =  y\in Y$ and $\gamma_z(\infty) =  
z \in Z$.  Now, Remark~\ref{Rem:UBSofRollerGeodesic} says that 
$$ \U_{Y} \sim \W(x,Y) \subset \W(x,Z) \sim \U_{Z},$$
hence $[\U_{Y}] \preceq [\U_{Z}]$, as desired.

\smallskip

To prove \ref{Itm:URmax}, let $\U$ be a maximal UBS. Choose a point $y \in Y_\U = \UR([\U])$. By Corollary~\ref{Cor:MaxIsVisible}, there is a combinatorial geodesic ray $\gamma$ limiting to $y$, such that $\U \sim \W(\gamma)$. Thus, by Remark~\ref{Rem:UBSofRollerGeodesic}, we have
\begin{equation}\label{Eqn:MaxIsVisible}
[\U] = [\W(\gamma)] = [\U_{Y_\U}] = \RU(Y_\U) = \RU \circ \UR([\U]).
\end{equation}
Thus the restriction of $\UR$ to maximal classes in $\mathcal{UBS}(X)$ is invertible, and in particular bijective.

Equation~\eqref{Eqn:MaxIsVisible} also shows that $\UR$ sends maximal classes in $\mathcal{UBS}(X)$  to maximal classes in $\guralnik X$. Indeed, suppose that $[\U]$ is maximal but $\UR([\U]) = Y_{\U} < Z$ for some Roller class $Z$. Then we would have $[\U] = \RU(Y_\U) < \RU(Z)$, contradicting the maximality of $\U$. That $\RU$ sends maximal classes to maximal classes is checked in the same way.
\end{proof}

The following example shows that $\UR$ can fail to be injective.

\begin{example}
 Consider the standard $1/8$th-plane staircase, corresponding to squares whose  vertices $(x,y)$ satisfy $x\geq y\geq 0$. In 
this case, there are exactly 2 equivalence classes of minimal UBS, corresponding to vertical hyperplanes $\U_V$ and horizontal 
hyperplanes $\U_H$. These are almost transverse, hence $\U_H\nsim \U_V  \cup \U_H$. However, $\UR[\U_H] = \UR([\U_V  \cup  \U_H])$. Thus $\UR$ is not injective.
\end{example}

On the other hand, the following corollary shows that $\UR$ is finite-to-one.

\begin{cor}\label{Cor:FinitelyManyBelowRoller}
Let $Y \in \guralnik X$ be a Roller class. Then there are finitely many other Roller classes $Z$ such that $Z \leq Y$.
\end{cor}

\begin{proof}
For each $Z \leq Y$, Theorem~\ref{Thm:UBStoRoller}.\ref{Itm:RUsection} says that $\RU(Z) \preceq \RU(Y)$. Furthermore, $\RU$ is injective.
By Lemma~\ref{Lem:FinManyMinimals}, there are only finitely many commensurability classes  $[\W] \preceq \RU(Y)$. Thus there are finitely many $Z \leq Y$. \end{proof}

\begin{definition}\label{Def:UBSesimp}
Let $(\simp X)'$ be the barycentric subdivision of $\simp X$. Recall from Definition~\ref{Def:simplicial_boundary} that
$\simp X$ is the union of simplices corresponding to commensurability classes
of UBSes. However, a simplex of $\simp X$ may have proper faces that do not come from commensurability
classes, as illustrated in Example~\ref{exmp:weird_simplex} below.  Hence, a vertex of $(\simp X)'$ might fail to be 
associated to a UBS commensurability class. 

Let $\simpubs X$ be the simplicial realization of the partial order $\preceq$ on $\mathcal{UBS}(X)$. Then $\simpubs X$ naturally embeds into $(\simp X)'$. By the above paragraph, this embedding is not 
in general onto.
\end{definition}

The following example, constructed by Dan Guralnik and Alessandra Iozzi, illustrates a subtlety in the definition of $\simp 
X$ which shows the difference between $\simp X$ and $\simpubs X$.

\begin{example}[Weird faces of simplices of $\simp X$]\label{exmp:weird_simplex}
Let $X$ be the $3$--dimensional 
CAT(0) cube complex given by the following data:
\begin{itemize}
     \item The set of hyperplanes has the form $\W(X) = \{B_i\}_{i \ge0} \sqcup \{S_j\}_{j\ge0} \sqcup \{D_k\}_{k \ge0} $.  \item For each 
$A\in\{B,S,D\}$ and each $n\ge1$, the hyperplanes $A_{n\pm1}$ are separated by $A_n$.
\item $B_i$ crosses $S_j$ if and only if $i \geq j$.
\item $S_j$ crosses $D_k$ if and only if $j \geq k$.
\item $B_i$ crosses $D_k$ if and only if $i \geq k$.
\item If $i < j < k$, then $S_j$ separates $B_{i}$ from $D_k$.
\end{itemize}
Part of $X$ is shown in Figure~\ref{fig:weird}.  

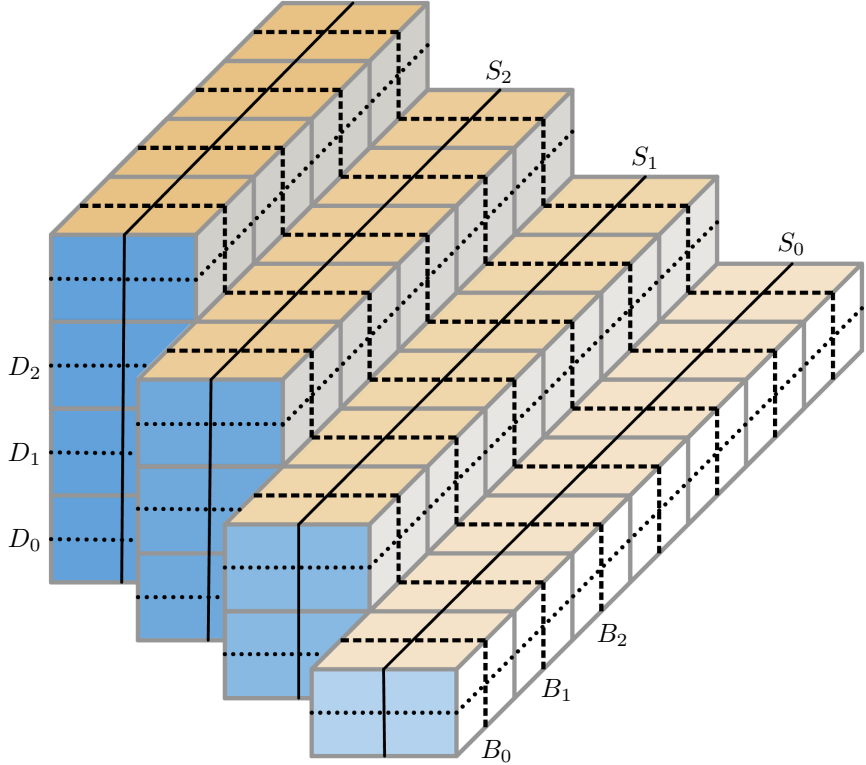
\begin{figure}[h]
\begin{overpic}[width=0.75\textwidth]{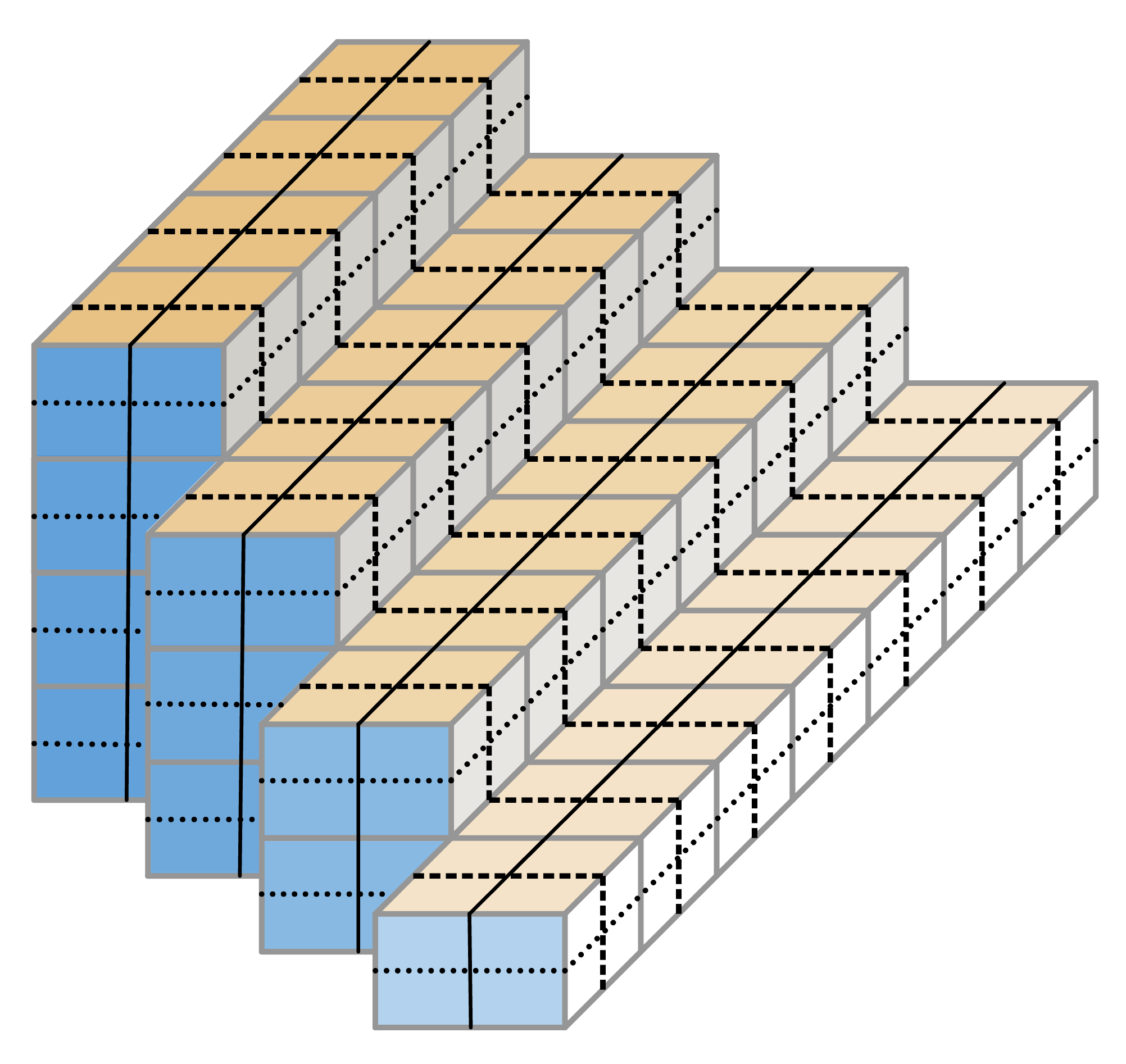}
\put(53,3){$B_0$}
\put(60,10){$B_1$}
\put(66.5,16.5){$B_2$}
\put(87.5,62){$S_0$}
\put(70.5,72){$S_1$}
\put(53.5,82.){$S_2$}
\put(-2,27.5){$D_0$}
\put(-2,37.5){$D_1$}
\put(-2,47.5){$D_2$}
\end{overpic}
\caption{The cube complex $X$ in Example~\ref{exmp:weird_simplex}, 
whose simplicial boundary has a ``weird'' $1$--simplex. The hyperplanes $B_i$ are shown in broken lines, the hyperplanes $S_j$ are solid, and the hyperplanes $D_k$ are dotted. This non-perspective drawing represents an embedding into $\R^3$ where the solid hyperplanes have also been rescaled by a factor of 2.}
\label{fig:weird}
\end{figure}
Then $\simp X$ is a single $2$--simplex represented by the UBS 
$\{B_i\}_{i \ge0} \cup \{S_j\}_{j\ge0} \cup \{D_k\}_{k \ge0} $.  The $0$--simplices are represented by the minimal UBSes 
$\{B_i\}, \{S_j\}, \{D_k\}$.  The sets $\{B_i\}\cup\{S_j\}$ and $\{S_j\}\cup\{D_k\}$ are UBSes representing $1$--simplices, 
but $\{B_i\}\cup\{D_k\}$ is not a UBS, because it is not inseparable (consider the last bullet).  The complex $\simpubs X$ is therefore obtained from 
the barycentric subdivision of the $2$--simplex by deleting all of the cells that contain a vertex at the barycenter of $\{B_i\}\cup\{D_k\}$.
\end{example}

 We remark that by \cite[Theorem 3.23]{Hagen}, a face of a simplex of $\simp X$ can only fail to correspond to a commensurability class of UBSes when one of the $0$--simplices of that face is
 $\ell^1$--invisible.

Having passed from $\simp X$ to $\simpubs  X$, Theorem~\ref{Thm:UBStoRoller} has the following corollary.

\begin{cor}\label{Cor:UBStoRollerSimp}
The order-preserving maps $\UR$ and $\RU$ of Theorem~\ref{Thm:UBStoRoller} induce $\Aut(X)$--equivariant simplicial maps $\URtri \from \simpubs  X\to \absimp X$ and $\RUtri \from  \absimp X \to \simpubs  X$, as follows:
\begin{enumerate}[$(1)$ ]
\item $\URtri \from \simpubs  X\to \absimp X$ is  surjective.
\item $\RUtri \from  \absimp X \to \simpubs  X$  is an injective section of $\URtri$.  
\item\label{Itm:RUhomotopy} $\URtri$ is a homotopy 
equivalence with homotopy inverse  $\URtri$.
\end{enumerate}
\end{cor}

\begin{proof}
 By definition, $\absimp X$ is the simplicial realization of the partial order $\leq$ on $\guralnik X$.  Similarly, $\simpubs  X$ is the simplicial realization of the partial order 
$\preceq$ on $\mathcal{UBS}(X)$.  So, the first two enumerated assertions follow from 
Theorem~\ref{Thm:UBStoRoller}.  

To prove \ref{Itm:RUhomotopy}, let $Y$ be a Roller class.  Let $[\mathcal V]\in (\UR)^{-1}(Y)$.  Then $Y_{\mathcal V}=Y_{\U_Y}$.  By 
Remark~\ref{Rem:UBSofRollerGeodesic} and Lemma~\ref{Lem:UBSclassSubset}, we have $\mathcal V\preceq\U_Y$.

Using the notation $Y$ for the $0$--simplex of $\absimp X$ corresponding to the Roller class $Y$, we have just shown that $\URtri^{-1}(Y)$ is a subcomplex of $\simpubs  X$ spanned by 
$0$--simplices adjacent or equal to the $0$--simplex $[\U_Y]$, so $\URtri^{-1}(Y)$ is topologically a 
cone.  In particular, $\URtri^{-1}(Y)$ is contractible.  Quillen's fiber theorem (see e.g.~\cite[Theorem 
10.5]{Bjorner:book}) implies that $\URtri$ is a homotopy equivalence, such that any section is a homotopy inverse. In particular, $\RUtri$ is a homotopy inverse of $\URtri$.
\end{proof}

\begin{remark}[Alternate strategy]\label{rem:alternate_strategy}
The preceding corollary provides an explicit $\Aut(X)$--equivariant homotopy equivalence $\simpubs  X\to\absimp X$.  Furthermore, it is possible, although somewhat tedious, to construct an explicit, 
$\Aut(X)$--equivariant homotopy equivalence $\simp X\to\simpubs  X$.  Composing these maps gives
a homotopy equivalence $\simp X \to \absimp X$, which is $\Aut(X)$--equivariant on the nose.
This is slightly stronger than Proposition~\ref{prop:simplicial_boundary}, which asserts that $\simp X$ and 
$\absimp X$ are \Authom homotopy equivalent.  We have chosen a different proof of Proposition~\ref{prop:simplicial_boundary} that is somewhat shorter, relying on the nerve theorem, and only yielding 
$\Aut(X)$--equivariance up to homotopy.  Our ultimate goal is to relate the homotopy type of $\simp X$ and $\absimp X$ to that of the Tits boundary $\tits X$, and in any case our homotopy equivalence $\absimp 
X\to\tits X$, provided by Theorem~\ref{Thm:first_HE}, is only equivariant up to homotopy. 
\end{remark}

\section{The Tits boundary}\label{Sec:Tits}

In this section, we recall the definition of the Tits boundary $\tits X$, and establish some 
connections between points on the Tits boundary and UBSes. The main result of this section, Proposition~\ref{Prop:build_CAT(0)}, is a technical result about combining UBSes that will be crucial in the next section.

In this section, we will work extensively with the CAT(0) metric $d_X$ on the cube complex $X$, hence most half-spaces considered here are CAT(0) half-spaces.  Recall from Section~\ref{Sec:half-spaces} that if $\hat h$ is a hyperplane with open carrier $N(\hat h)$,
the two vertex half-spaces associated to  $\hat h$ are denoted $h, h^*$. The two CAT(0) half-spaces associated to $\hat h$ are the cubical convex hulls $\hull(h)$ and $\hull(h^*)$, respectively. 
Observe that 
 $\hull(h)$ and $\hull(h^*)$ are exactly the components of $X \setminus N(\hat h)$.

\begin{definition}[The metric $\dist_T$]\label{Def:tits_metric}
The \emph{Tits boundary of $X$}, denoted $\tits X$, is the set of equivalence classes of CAT(0) geodesic rays  in 
$X$, where rays $\gamma,\gamma'$ are equivalent if the Hausdorff distance between them is finite.  As a set, $\tits X$ coincides with the visual boundary of $X$.

We endow $\tits X$ with a metric, as follows.
Given CAT(0) geodesic rays $\alpha,\beta \from [0,\infty)\to X$ with $\alpha(0)=\beta(0)=x$, the \emph{Alexandrov angle} formed by $\alpha,\beta$ at $x$ is 
\begin{equation}\label{Eqn:Alexandrov}
\angle_x(\alpha,\beta)=\cos^{-1}\left(\lim_{t,t'\to0}\frac{\dist_X(\alpha(t), \beta(t'))^2-t^2-(t')^2}{2tt'}
\right)
\end{equation}
In other words, $\angle_x(\alpha,\beta)$ is computed by applying the law of cosines to a triangle with sides along $\alpha$ and $\beta$, and then taking a limit as the triangle shrinks.
Let $a = [\alpha] \in \tits X$ and $b = [\beta] \in \tits X$ be the the equivalence classes of $\alpha,\beta$ respectively.  Following  \cite[Definition II.9.4]{BridsonHaefliger}, we define
\begin{equation}\label{Eqn:TitsAngle}
\angle_T(a,b) = \angle_T([\alpha],[\beta])
:=\sup_{x\in X}\angle_x(\alpha_x,\beta_x),
\end{equation}
 where $\alpha_x,\beta_x$ are the rays emanating from $x$ and representing 
$a,b$ respectively.  Then $\angle_T$ induces a length metric $d_T$ on $\tits X$ in the standard way, making it a CAT(1) 
space. See \cite[Theorem II.9.13 and Definition II.9.18]{BridsonHaefliger}. We always equip $\tits X$ with the metric $\dist_T$. By \cite[Proposition II.9.5]{BridsonHaefliger}, the automorphism group 
$\Aut(X)$ acts on $\tits X$ by isometries.
\end{definition}

We remark that the original definition of the Alexandrov angle \cite[Definition I.1.12]{BridsonHaefliger} is somewhat more involved. In our setting, since $X$ is a CAT(0) space, \cite[Proposition II.3.1]{BridsonHaefliger} says the limit in \eqref{Eqn:Alexandrov}  exists and is equal to the Alexandrov angle in the original definition. Compare \cite[Remark I.1.13.(4)]{BridsonHaefliger}.

Definition~\ref{Def:tits_metric} has the following generalization to the cuboid setting.

\begin{definition}[Cuboid Tits boundary]\label{Def:TitsMetricCuboid}
Fix a group $G$ acting on $X$ by cubical automorphisms, as well as a $G$--admissible hyperplane rescaling $\rho$, as in Definition~\ref{Def:cuboid_metric}.
Then the cuboid metric $d_X^\rho$ is still a CAT(0) metric by Lemma~\ref{Lem:cat_0_rescaling}. Applying Definition~\ref{Def:tits_metric} to the CAT(0) space $(X, d_X^\rho)$, we obtain the \emph{cuboid Tits boundary} $\tits^\rho X = \tits(X, d_X^\rho)$, with an associated CAT(1) metric $d_T^\rho$. 
The restricted automorphism group $\Aut(X^\rho)$ acts on $\tits^\rho X$ by isometries. Since $\rho$ is a $G$--admissible rescaling, we have $G \subset \Aut(X^\rho)$ by Definition~\ref{Def:AutXrho}, hence $\tits^\rho X$ still admits a $G$--action.
\end{definition}

All of the results of this section extend to a cuboid metric $d_X^\rho$. See Remarks~\ref{Rem:DaCuboid}, \ref{Rem:CuboidsIndistanceBound}, and~\ref{Rem:HuangCuboids} for details.

Throughout this section, we use Greek letters $\alpha, \beta, \gamma$ to denote geodesic rays, and the corresponding Roman letters $a,b,c$ to denote equivalence classes. We write $a = [\alpha] = \alpha(\infty)$ to denote the endpoint of $\alpha$ in $\tits X$.

\subsection{Deep sets}\label{subSec:Deep}
Every point of the Tits boundary has an associated \emph{deep set} of half-spaces.

\begin{definition}[Deep set of a Tits point]\label{Def:DeepSet}
Let $a\in \tits X$, and
let $\alpha \from [0,\infty)\to X$ be a CAT(0) geodesic representing $a$. We define the \emph{deep set} 
$$D_a = \{h\in \frakH : \text{ for all } R>0 \text{ there is } t_R>0 \text{ s.t. } \mathcal  N_R( \alpha[t_R,\infty)) \subset \hull(h) \}.$$
It is immediate to check that the definition does not depend on the representative geodesic $\alpha$. If $h \in D_a$, we say that $h$ is \emph{deep} for $a$.
\end{definition}

\begin{lemma}\label{Lem:DeepTransverse}
Let $\alpha \from [0,\infty)\to X$ be a CAT(0) geodesic ray with $\alpha(0) \in X^{(0)}$, and let $\hat h \in \W(\alpha)$. Then the following hold:
\begin{enumerate}[$(1)$ ]
\item\label{Itm:LinearDivergence} Define $f(t) = d_X(\alpha(t), \hat h)$, and let $t_0 = \alpha^{-1}(\hat h)$. Then, for 
$t_0 \leq s \leq t$, we have
$$f(t)-f(s) \geq A_0 \cdot (t-s),$$
where $A_0 > 0$,  and $A_0$ may depend on $\hat h$ and $\alpha$ but not on $s,t$.
\item\label{Itm:DeepSide} The hyperplane $\hat h$ has a deep half-space $h \in D_a$.
\end{enumerate}
Consequently, $\W(\alpha)$ is a pruned UBS. 
\end{lemma}

\begin{proof}
By~\cite[Corollary II.2.5]{BridsonHaefliger}, $f$ is a convex, $1$--Lipschitz function.   
By convexity, $f$ has left and right derivatives at every point in $(0,\infty)$, and these are non-decreasing.  Since $f$ is Lipschitz and convex,  the left and right 
derivatives coincide almost everywhere, hence $f'$ is well-defined almost everywhere, and
non-decreasing at all points where it is defined.

Let $t_0 = \alpha^{-1}(\hat h)$. Since the carrier $N(\hat h) \cong \hat h \times (-\frac{1}{2}, \frac{1}{2})$ is an $\ell^2$ product and $\alpha(t_0) \in \hat h$ lies on the midcube of a cube, there 
is an initial segment of $\alpha$ after crossing $\hat h$ on which $f(t)$ is an increasing linear function. Therefore, $f'(t) = f'(t_0) = A_0 > 0$ on this segment. (In fact, $f'(t_0) = \sin \theta_0$, 
where $\theta_0$ is the angle between $\alpha$ and $\hat h$ in the $\ell^2$ metric on a cube.) Hence $f'(t) \geq A_0$ for all $t\geq t_0$ at which $f'(t)$ is defined. Since $f$ is Lipschitz, it 
follows that for all $t_0\leq s\leq t$, we have 
$$f(t)-f(s)=\int^t_sf'(z) \,dz \geq A_0\cdot(t-s),$$
proving \ref{Itm:LinearDivergence}. 
Now, \ref{Itm:LinearDivergence} immediately implies \ref{Itm:DeepSide}.
The conclusion that $\W(\alpha)$ is pruned is then immediate from Definition~\ref{Def:Pruned}.
\end{proof}

The following lemma is a strengthening of \cite[Lemma 2.27]{CFI}. While \cite[Lemma 2.27]{CFI} proves that $D_a$ contains an infinite descending chain, we prove the stronger statement that $D_a$ is non-terminating.

\begin{lemma}\label{Lem:DeepSetProps}
For a point $a \in \tits X$, define $\~Y_{a} = \bigcap_{h\in D_a}{} h$. Then $D_a$ and $\~Y_a$ have the following properties:
\begin{enumerate}[$(1)$ ]
\item\label{Itm:DaConsist} $D_a$ is consistent, and therefore $D_a = \frakH^+_{\~Y_a}$. 
\item\label{Itm:StartOfChain} Every half-space $h_0 \in D_a$ belongs to a descending chain $\{h_0, h_1, \ldots \} \subset D_a$. Furthermore, if $\hat h_0 \in \W(\alpha)$ for a ray $\alpha \in a$, then $\hat h_n \in \W(\alpha)$ for all $n$.
\item\label{Itm:DaPrincipal} ${\~Y_a} \subset \roller X$ is the closure of a (unique, principal) Roller class $Y_a$.
\end{enumerate}
\end{lemma}

\begin{proof}
The consistency of $D_a$ is immediate from the definition. Thus, by Proposition~\ref{LiftingDecomp}.\ref{ConsistentEmbedding}, we obtain $D_a = \frakH^+_{\~Y_a}$. This proves \ref{Itm:DaConsist}.

To prove \ref{Itm:StartOfChain}, let $h_0 \in D_a$.  By \cite[Proposition II.8.2]{BridsonHaefliger},  there exists a geodesic 
ray $\alpha$ representing $a$ such that $\alpha(0) \in h_0^*$. Since $h_0 \in D_a$, Definition~\ref{Def:DeepSet} implies that 
$\alpha$ must deviate arbitrarily far from $\hat h_0$.
Thus, by Lemma~\ref{Lem:WallQI}, a point $\alpha(t)$ for large $t$ has the property that $|\W(\alpha(t), y)| \geq 1$ for 
every $y \in N(\hat h_0)$. Now,  Proposition~\ref{Prop:Proj} says that there is at least one hyperplane $\hat h_1$ separating 
$\alpha(t)$ from $N(\hat h_0)$. Observe that $\hat h_1 \in \W(\alpha)$ by construction, and orient $\hat h_1$ by choosing the 
half-space $h_1 \subset h_0$. Now, Lemma~\ref{Lem:DeepTransverse} says that $h_1 \in D_a$. Continuing inductively, we obtain 
a descending chain $\{h_0, h_1, \ldots \} \subset D_a$ where $\hat h_n \in \W(\alpha)$ for all $n$, proving 
\ref{Itm:StartOfChain}.

By \ref{Itm:DaConsist},  $D_a = \frakH^+_{\~Y_a}$ is consistent, and by  \ref{Itm:StartOfChain}, $D_a$ contains an infinite descending chain. Thus ${\~Y_a}  = \bigcap_{h\in D_a}{} h$ is a convex, nonempty subset of $\roller X$.  Since $D_a$ is non-terminating by  \ref{Itm:StartOfChain}, Lemma~\ref{ClassSetNonterminating} implies
 that $\~Y_{D_a}$ is the closure of a principal Roller class $Y_a \in \guralnik X$. (Compare Definition~\ref{Def:PrincipalClass}.) This proves \ref{Itm:DaPrincipal}.
 \end{proof}

\begin{remark}\label{Rem:DaCuboid}
The previous lemmas generalize readily to a cuboid metric $d_X^\rho$. Indeed, Lemma~\ref{Lem:DeepTransverse} generalizes because because hyperplane carriers are still $\ell^2$ products, and because \cite[Corollary II.2.5]{BridsonHaefliger} works in any CAT(0) space.
Similarly, Lemma~\ref{Lem:DeepSetProps} generalizes because \cite[Proposition II.8.2]{BridsonHaefliger} works in every CAT(0) space and, because $d_X^\rho$ is quasi-isometric to the $\ell^1$--metric on $X^{(0)}$. 
\end{remark}

\subsection{CAT(0) rays and their UBSes}\label{subSec:CAT_0_UBS}
Given a CAT(0) geodesic ray $\gamma$ in $X$, starting at a point of $X^{(0)}$, recall that  $\W(\gamma)$ denotes the set of 
hyperplanes $h$ that intersect $\gamma$.  By Lemma \ref{Lem:DeepTransverse}, 
$\W(\gamma)$ is a UBS.

Throughout this subsection, fix $a,b\in\tits X$ and let $\alpha, \beta$ be CAT(0) geodesic rays representing $a$ and $b$, respectively. We can assume that  $\alpha(0)=\beta(0) = x \in X^{(0)}$.

The following lemma says that when $\W(\alpha)\cup\W(\beta)$ has finite symmetric difference with a UBS, it actually is a UBS. 
Several subsequent lemmas build toward Proposition~\ref{Prop:build_CAT(0)}, which will show that $\W(\alpha)\cup\W(\beta)$ is in fact a UBS associated to a ray $\gamma$.

\begin{lemma}\label{Lem:SameBasepointUnionUBS}
Suppose that $\W(\alpha)\cup\W(\beta)$ is commensurate with a UBS. 
Define $\mathcal A=\W(\alpha) \setminus \W(\beta)$ and $\mathcal B=\W(\beta) \setminus \W(\alpha)$ and  $\mathcal C=\W(\alpha)\cap\W(\beta)$.
Then
\begin{enumerate}[$(1)$]
 \item\label{Itm:InfiniteOrEmpty} Each of $\mathcal A,\mathcal B$ is infinite or empty.
 \item\label{Itm:Crossers} Every element of $\mathcal A$ crosses every element of $\mathcal B$.
  \item\label{Itm:UBSDecomp} $\W(\alpha) \cup \W(\beta) = \mathcal A \sqcup \mathcal B \sqcup \mathcal C$ is a UBS.

\end{enumerate}
\end{lemma}

\begin{proof}
Let $\U = \W(\alpha) \cup \W(\beta)$. Since $\U$ is commensurate with a UBS by hypothesis, it must be infinite and unidirectional. We 
will check the other properties of a UBS at the end of the proof.

By Lemma~\ref{Lem:DeepTransverse}, every hyperplane $\hat h = \hat h_0 \in \W(\alpha)$ corresponds to a deep half-space $h_0 \in D_a$, and  by Lemma~\ref{Lem:DeepSetProps}, $h_0$ is the start of a descending chain $\{h_0, h_1, \ldots \} \subset D_a$, where $h_n \in \W(\alpha)$ for every $n$. Similarly, every hyperplane $k_0 \in \W(\beta)$ defines a descending chain $\{k_0, k_1, \ldots \} \subset D_b$, where $k_m \in \W(\beta)$ for every $m$.

 If $\mathcal A = \W(\alpha) \setminus \W(\beta) \neq \emptyset$, then any hyperplane $\hat h_0 \in \mathcal A$ defines a chain  $\{ \hat h_0, \hat h_1, \ldots \} \subset \W(\alpha)$, as above. Furthermore, $\hat h_0$ separates every $\hat h_n$ from $\beta$. Thus $\hat h_n \in \mathcal A$, and $\mathcal A$ is infinite. Similarly, $\mathcal B$ must be empty or infinite, proving  \ref{Itm:InfiniteOrEmpty}.

Next, suppose for a contradiction that $\hat h_0 \in \mathcal A $ is disjoint from $\hat k_0 \in \mathcal B$. Then the deep half-spaces $h_0 \in D_a$ and $k_0 \in D_b$ are also disjoint. But then the disjoint chains $\{\hat h_n\} \subset \W(\alpha)$ and $\{\hat k_m \} \subset \W(\beta)$ contradict the fact that $\U = \W(\alpha) \cup \W(\beta)$ is unidirectional. This proves  \ref{Itm:Crossers}.

Finally, we show that $\U$ is a UBS. We have already checked that $\U$ is infinite and unidirectional.
To check inseparability, suppose that $\hat \ell$ is a hyperplane of $X$ that separates $\hat h, \hat k \in \U$. Since we already know that $\W(\alpha)$ and $\W(\beta)$ are inseparable, it suffices to check the case that $\hat h \in \mathcal A$ and $\hat k \in \mathcal B$. But then $\hat h$ and $\hat k$ must cross by \ref{Itm:Crossers}, a contradiction. Similarly, since $\W(\alpha)$ and $\W(\beta)$ contain no facing triples, any potential facing triple in $\U$ must contain at least one hyperplane  $\hat h \in \mathcal A$ and at least one hyperplane  $\hat k \in \mathcal B$. But then $\hat h$ and $\hat k$ must cross, hence $\U$ cannot contain any facing triples. Thus $\U$ is a UBS, proving \ref{Itm:UBSDecomp}.
\end{proof}

Recall from Definition~\ref{Def:Umbra} that given a UBS $\U$, we have the umbra $\~Y_\U \subset \~X$.
Following Proposition~\ref{Prop:Proj}, let $\gate_{Y_\U} = \gate_{\~Y_\U} \from \~X \to Y_\U$ be the gate projection to $\~Y_\U$.

\begin{lemma}\label{Lem:UBSBoundaryPoint}
Suppose $\alpha$ and $\beta$ are geodesic rays starting at $x\in X^{(0)}$, and that $\U = \W(\alpha)\cup \W(\beta)$ is commensurate with a UBS. Then $\U = \W(x, \gate_{Y_\U}(x))$. 
\end{lemma}

\begin{proof}
Since the geodesic rays start at the vertex $x$, Lemma~\ref{Lem:DeepSetProps} says that  all the hyperplanes they cross are deep, i.e.  $\W(x,Y_a) =\W(\alpha)$ and $\W(x,Y_b)= \W(\beta)$, hence $\U = \W(x,Y_a)\cup \W(x,Y_b)$. In addition, Lemma~\ref{Lem:SameBasepointUnionUBS} says that $\U$ is a UBS.

Now, Lemma \ref{UBScomplex} gives  $\~Y_\U= \~Y_a  \cap \~Y_b$, and in particular $\frakH_{\~Y_\U}^+ = D_a\cup D_b$. By Lemma \ref{standardUBS} and
  Remark \ref{Rem:UBSofRollerGeodesic}, we have 
 $\W(x, \~Y_\U)=\W(x, Y_\U)= \W(x, \gate_{Y_\U}(x))$. 
We must show that this UBS equals  $\U = \W(x,Y_a)\cup \W(x,Y_b)$.
 To this end we have by Proposition \ref{LiftingDecomp}:
 \begin{eqnarray*}
\frakH_{\gate_{Y_a}(x)}^+ &=& [\frakH_x^+\setminus D_a^*]\cup D_a,\\
\frakH_{\gate_{Y_b}(x)}^+ &=& [\frakH_x^+\setminus D_b^*]\cup D_b,\\
\frakH_{\gate_{Y_\U}(x)}^+ &=& [\frakH_x^+\setminus (D_a^*\cup D_b^*)]\cup (D_a\cup D_b).
\end{eqnarray*}
Then a standard Venn diagram argument shows that 
\begin{eqnarray*}
\frakH_x^+\setminus \frakH_{\gate_{Y_\U}(x)}^+& = &\frakH_x^+\setminus \big( [\frakH_x^+\setminus (D_a^*\cup D_b^*)]\cup (D_a\cup D_b)
\big)
\\
&=&
 \big( \frakH_x^+\cap (D_a^*\cup D_b^*) \big) \setminus(D_a\cup D_b).
\end{eqnarray*}
The set of hyperplanes on the left-hand side is $\W(x, \gate_{Y_\U}(x))$, while the set of hyperplanes on the right-hand side is $\W(x,Y_a)\cup \W(x,Y_b)$.
\end{proof}

\begin{lemma}[Tits distance bound]\label{Lem:tits_distance_bound}
 Suppose that $\W(\alpha)\cup\W(\beta)$ is commensurate with a UBS.  Then $d_T(a,b)\leq\pi/2$.
\end{lemma}

\begin{proof}
We will show that $\pi/2$ is an upper bound on the Alexandrov angle between rays representing $a,b$ based at an arbitrary point, and then 
use \eqref{Eqn:TitsAngle}.  The bound will be produced by finding a convex subspace, 
isometric to a Euclidean cube, which has one corner at the basepoint and which contains nontrivial initial segments of the 
rays.

\smallskip
\textbf{Angle bound at a vertex:}  Let $y\in X$ be an arbitrary vertex.  Let $\alpha',\beta'$ be CAT(0) geodesic rays 
emanating from $y$ and representing $a,b$ respectively (so, $\alpha$ is asymptotic to $\alpha'$ and $\beta$ is asymptotic to 
$\beta'$).  Let $\U=\W(\alpha)\cup\W(\beta)$, and let $\U'=\W(\alpha')\cup\mathcal 
W(\beta')$.  Then $\U$ is commensurate to $\U'$, and hence commensurate to a UBS.  
Lemma~\ref{Lem:SameBasepointUnionUBS} then implies that $\U'$ is actually a UBS.

Now let $X_y\subset X$ be the cubical convex hull of $\alpha'\cup\beta'$.  Then the set of hyperplanes crossing $X_y$ is 
exactly $\U'$.  Moreover, by Lemma~\ref{Lem:UBSBoundaryPoint}, there exists a point $\bar y$ in the Roller boundary 
such that $\U'$ is precisely the set of hyperplanes separating $y$ from $\bar y$.

Let $\triangleleft$ be the partial order on $\U'$ defined by declaring that $\hat h\triangleleft \hat h'$ if the hyperplane $\hat h$ 
separates $y$ from $\hat h'$.  Since every hyperplane of $X_y$ separates $y$ from $\bar y$, any hyperplanes $\hat h,\hat 
h'\in\U'$ are either $\triangleleft$--comparable, or they cross.

Note that if $\hat h\in\U'$ is dual to a $1$--cube of $X_y$ incident to $y$, then $\hat h$ is $\triangleleft$--minimal.  
Hence the set $\U'_{\min}$ of $\triangleleft$--minimal hyperplanes is nonempty. Since $\triangleleft$--incomparable hyperplanes cross, $\U'_{\min}$ is a 
set of pairwise-crossing hyperplanes of $X$.  Since $X_y$ is convex and every hyperplane in $\U'$ crosses $X_y$, we 
have a cube $C \subset X_y$ such that the hyperplanes crossing $C$ are exactly those in $\U'_{\min}$.

Note $y\in C$.  Also note that any cube $C'$ of $X_y$ with $y\in C'$ satisfies $C'\subset C$.  Thus any nontrivial 
CAT(0) geodesic segment in $X_y$ emanating from $y$ has a nontrivial initial subsegment  lying in $C$.  In 
particular, $\alpha'$ and $\beta'$ have nontrivial initial segments lying in $C$.  Since the ambient CAT(0) metric restricts 
to the Euclidean metric on the cube $C$, the Alexandrov angle made by $\alpha',\beta'$ at $y$ satisfies
$$\angle_y(\alpha',\beta')=\lim_{t,t'\to0} \angle_y(\alpha'(t),\beta'(t))\leq \pi/2,$$
because any two segments in a cube emanating from a common corner make an angle at most $\pi/2$.

\smallskip
\textbf{Angle bound at arbitrary points:}  We now bound the angle made by rays representing $a,b$ and emanating from 
non-vertex points, by an identical argument taking place in an appropriate cubical subdivision of $X$.  Let $y\in X$ be an 
arbitrary point.  Perform a cubical subdivision of $X$ to obtain a CAT(0) cuboid complex $X'$ in which $y$ is a vertex, so that
the CAT(0) cuboid metric coincides with the original metric on $X$.  

More precisely, for each hyperplane $\hat h$, identify its carrier with $\hat h\times[-\frac12,\frac12]$.  For each $\hat h$ 
such that $y\in\hat h\times(-\frac12,\frac12)$, let $\epsilon_{\hat h}$ be such that $y\in \hat h\times\{\epsilon_{\hat 
h}\}$.  The (geodesically convex) subspace $\hat h\times\{\epsilon_{\hat h}\}$ has a natural cubical structure with an $n$--cube 
for each $(n+1)$--cube intersecting $\hat h$.  We subdivide $X$ so that the cubes of $\hat h\times\{\epsilon_{\hat h}\}$ are 
subdivided cubes whose edges are segments in $X$ whose lengths are inherited from $X$.  The result is a cuboid complex $X'$ such that 
the identity map $X\to X'$ is an isometry in the CAT(0) metric, preserves the median, and sends $y$ to a vertex.  Since the 
hyperplanes of $X'$ are parallel copies of hyperplanes of $X$, the set of hyperplanes of $X'$ that cross $\alpha\cup\beta$ 
continues to be a UBS, so Lemma~\ref{Lem:UBSBoundaryPoint} still applies. (Compare Remark~\ref{Rem:DaCuboid}.) We can thus argue exactly as before to see that 
$\angle_y(\alpha',\beta')\leq\pi/2$.

\smallskip
\textbf{Conclusion:}  We have shown that, for all $y\in X$, letting $\alpha',\beta'$ be the rays based at $y$ and 
representing $a,b$ respectively, we have $\angle_y(\alpha',\beta')\leq\pi/2$.  Taking the supremum over all $y\in X$, as in \eqref{Eqn:TitsAngle}, we see 
that the angular metric satisfies $\angle_T(a,b)\leq \pi/2$.  Now, 
Definition~\ref{Def:tits_metric} gives $\dist_T(a,b)\leq\pi/2$, as required.
\end{proof}

\begin{remark}\label{Rem:CuboidsIndistanceBound}
Lemmas~\ref{Lem:SameBasepointUnionUBS} and~\ref{Lem:UBSBoundaryPoint} extend immediately to the cuboid setting, because they use only hyperplane combinatorics and prior lemmas. The above proof of 
Lemma~\ref{Lem:tits_distance_bound} also extends to cuboids, because it uses prior lemmas combined with CAT(0) geometry. Indeed, the above argument uses cuboids in combination 
with Remark~\ref{Rem:DaCuboid} to prove the desired angle bound for arbitrary basepoints in $X$. The key conceptual reason why the argument works for cuboids is that every pair of geodesic segments 
in a cuboid,  emanating from a corner, meets at angle at most $\pi/2$. 

The reliance on cuboids to prove the desired result for cubes can be avoided, as follows. First, prove the angle bound at a vertex of $X$, exactly as above. Then, let $X'$ be the cubical subdivision 
of $X$, and observe that the CAT(0) geodesics in $X'$ are exactly the same as those in $X$. (See e.g.\ \cite[Section 2.1.6]{fioravanti2019deforming}.) Thus the same argument proves the angle bound at 
every vertex of $X'$. 

Continuing to subdivide by induction, we obtain a set $V_\infty\subset X$ of points that are vertices of the subdivision at some (and hence all subsequent) stages.  Note that $V_\infty$ intersects 
each of the original cubes of $X$ in a dense subset, and in particular contains all the original vertices.  The desired angle bound holds when the rays in question are based at any point in 
$V_\infty$.  Now, fixing $a,b\in\tits X$, we can consider the function $x\mapsto\angle_x(\alpha,\beta)$, where $\alpha,\beta$ are the rays representing $a,b$ and starting at $x$.  This function need 
not be continuous on $X$ (it is upper semicontinuous \cite[Proposition II.9.(2)]{BridsonHaefliger}), but it can be shown to be continuous on each open cube of $X$ (of any dimension).  Since it is 
bounded above by $\pi/2$ on a dense subset of each such cube, we conclude that $\sup_x\angle_x(\alpha,\beta)=\angle_T(a,b)\leq\pi/2$.

Finally, observe that the statement and proof of Lemma~\ref{Lem:tits_distance_bound} fail completely if we modify the angles of the cubes, precisely because two segments emanating from a vertex of a parallelogram might meet at a large angle.
\end{remark}

Our main goal in this section is to show that if $\alpha,\beta$ are geodesic rays with common initial point, and $\W(\alpha)\cup\W(\beta)$ is a UBS, then any 
geodesic ray $\gamma$ representing an interior point on the Tits geodesic from $\alpha(\infty)$ to $\beta(\infty)$ must cross all the hyperplanes in $\W(\alpha)\cup\W(\beta)$.  Before proving this in Proposition~\ref{Prop:build_CAT(0)}, we will need a few auxiliary results,  starting with a corollary of Lemma~\ref{Lem:DominantChain}.

\begin{cor}[Crossing all dominants implies crossing everyone]\label{Cor:Cross_Everybody}
Suppose that $\U = \W(\alpha)\cup \W(\beta)$ is a UBS. Let $Y$ be the cubical convex hull of $\alpha \cup \beta$. Let $\gamma \from [0,\infty)\to X$ be a CAT(0) geodesic ray with $\gamma(0)=\alpha(0)=\beta(0)$.  Suppose that $\gamma$ is contained in $Y$ and crosses every dominant hyperplane in $\U$.  Then $\gamma$ crosses every 
hyperplane in $\U$.
\end{cor}

\begin{proof}
This follows immediately from 
Lemma~\ref{Lem:DominantChain}.\ref{item:dominants_cross} applied to $\U = \W(Y)$.
\end{proof}

Next, we study angles at which rays cross dominant hyperplanes. 

\begin{lemma}[Lower angle bound for single rays]\label{Lem:SingleRayAngleBound}
 Let $\alpha$ be a CAT(0) geodesic ray with $\alpha(0)\in 
X^{(0)}$.  Then there exists $\kappa>0$ such that the following holds.  Let $\mathcal D_d,\cdots,\mathcal D_k\subset 
\W(\alpha)$ be the dominant minimal UBSes provided by applying Lemma~\ref{Lem:DominantChain} to $\mathcal 
W(\alpha)$, and let $i\in\{d,\ldots,k\}$.  Then $\mathcal D_i$ contains a chain $\{\hat u_n\}_{n\geq0}$ of hyperplanes such 
that $\mathcal D_i$ is commensurate with the inseparable closure of $\{\hat u_n\}_{n\geq 0}$, and 
$$ \angle_{y_n}(\alpha,\hat u_n)\geq\kappa $$
for all $n\geq 0$, where $y_n$ is the point $\alpha\cap\hat u_n$.
\end{lemma}

\begin{proof}
Recall from Lemma~\ref{Lem:DeepTransverse} that $\W(\alpha)$ is a pruned UBS. Thus Lemma~\ref{Lem:DominantChain} applies to $\W(\alpha)$.

Now, fix $i\in\{d,\ldots,k\}$.  
Lemma~\ref{Lem:DominantChain} says that $\mathcal D_i=\insep{\{\hat h_m\}_{m\geq0}}$, where
each $\hat h_m$ is dominant. This means each $\hat h_m$ crosses all but finitely many hyperplanes in $\mathcal 
W(\alpha) \setminus \mathcal D_i$.  

\smallskip
\textbf{The divergence $f(t)$ of $\alpha$ from $\hat h_0$:}  Define $f(t) = \dist_X(\alpha(t),\hat h_0)$, and 
let $T_0 = \alpha^{-1}(\hat h_0)$. By Lemma~\ref{Lem:DeepTransverse}.\ref{Itm:LinearDivergence}, there is a constant $A_0 > 0$ (depending on $\hat h_0$ but not $s,t$) such that
$$f(t)-f(s) \geq A_0\cdot(t-s)$$
for all $T_0\leq s\leq t$.

\smallskip
\textbf{Hyperplane count:}  We wish to produce a constant $C$ such that any length--$C$ subpath of $\alpha$ 
crosses an element of $\mathcal D_i$.

Let $s_0\geq T_0$ be sufficiently large that $\alpha([0,s_0))$ crosses all of the finitely many hyperplanes in $\mathcal 
W(\alpha) \setminus \mathcal D_i$ not crossing $\hat h_0$.  Suppose that $s_0\leq s\leq t$.  Then all hyperplanes crossing $\alpha([s,t])$ either cross $\hat h_0$ or belong to $\mathcal D_i$.  
Assume that $\alpha([s,t])$ does not cross any element of $\mathcal D_i$.  

Let $H_{st} = \hull(\alpha([s,t]))$ denote the cubical convex hull of $\alpha([s,t])$.  There may be hyperplanes crossing $H_{st}$ that do not cross $\alpha([s,t])$, because $\alpha(s),\alpha(t)$ need 
not be vertices. However, we can bound the number of such hyperplanes as follows.  Let $\mathcal W_{\operatorname{bad}}\subset\mathcal W(H_{st})\subset \mathcal W(\alpha)$ be the set of hyperplanes 
that cross $H_{st}$ but do not cross 
$\alpha([s,t])$.  Note that $\mathcal W(H_{st})$ inherits a partial order from $\mathcal W(\alpha)$, by restricting the partial order $\triangleleft$ from the proof of 
Lemma~\ref{Lem:tits_distance_bound}.  
Suppose that $\hat h\in\mathcal W(H_{st})$ is neither $\triangleleft$--maximal nor $\triangleleft$--minimal in $\mathcal W(H_{st})$.  Then there exist $\hat u,\hat v\in\mathcal W(H_{st})$ such that 
$\hat h$ separates $\hat u$ from $\hat v$.  By the definition of the convex hull, $\hat h$ cannot separate $\hat u$ or $\hat v$ from $\alpha([s,t])$, so $\hat h$ must cross $\alpha([s,t])$.  Thus 
every hyperplane in $\mathcal W_{\operatorname{bad}}$ is $\triangleleft$--minimal or $\triangleleft$--maximal.  Since incomparable hyperplanes have to cross, we conclude that $|\mathcal W_{\operatorname{bad}}|\leq 
2\dimension(X)$, hence there are at most $D=2\dimension X$ hyperplanes that cross $H_{st}$ but not $\alpha([s,t])$.
The rest of the 
hyperplanes crossing $H_{st}$ must cross $\hat h_0$, since they cross $\alpha([s,t])$, and we have chosen $s \geq s_0$ and assumed $\mathcal W(\alpha([s,t]))\cap\mathcal D_i=\emptyset$.

Consider the CAT(0) closest-point projections $p \from X\to \hat h_0$ and $q \from X\to H_{st}$.
By~\cite[Lemma 2.10]{Huang:quasiflats},  the $d_X$--convex hull of $p(H_{st})\cup 
q(\hat h_0)$ is isometric to $p(H_{st})\times [0,\dist_X(\hat h_0,H_{st})]$, where $p(H_{st})$  is identified with $p(H_{st}) \times \{ 0 \}$, and $q(\hat h_0)$ is identified with $p(H_{st}) \times \{\dist_X(\hat h_0,H_{st})\}$.
By~\cite[Lemma 2.14]{Huang:quasiflats}, the hyperplanes that cross $q(\hat h_0)$ are precisely those that 
cross $\hat h_0$ and $H_{st}$.  So, there are at most $D$ hyperplanes crossing $H_{st}$ but not $q(\hat h_0)$.  By 
Lemma~\ref{Lem:WallQI}, the points $\alpha(s),\alpha(t)$ are thus both within distance $\lambda_0D+\lambda_1$ of 
points in $q(\hat h_0)$, where $\lambda_0D+\lambda_1$ depends only on $D$. Thus
$$\dist_X(\alpha(t),\hat h_0) \: \leq \: \lambda_0D+\lambda_1+\dist_X(H_{st},\hat h_0) \: \leq \: \lambda_0 D+\lambda_1+\dist_X(\alpha(s),\hat 
h_0).$$  In other words, since  $s,t\geq s_0 \geq T_0$, we have shown
$$A_0(t-s) \leq f(t)-f(s)\leq \lambda_0D+\lambda_1.$$
Thus, for any $0\leq s\leq t$ (without assuming $s \geq s_0$), we get 
$$t-s\leq \frac{\lambda_0D+\lambda_1}{A_0}+ s_0.$$

Letting $C=  \frac{\lambda_0 D+\lambda_1}{A_0} + s_0+1$, we have shown that any subsegment of $\alpha$ of 
length at least $C$ crosses an element of $\mathcal D_i$.  
Hence, for any $0\leq s\leq t$, we have that $\alpha([s,t])$ 
crosses at least $\lfloor (t-s)/C\rfloor$ elements of $\mathcal D_i$.  

\smallskip
\textbf{Big angle hyperplanes:}  Let $\sigma$ be a subsegment of $\alpha$.  Consider the hyperplanes in $\mathcal D_i$ 
crossing $\sigma$.  Any two such hyperplanes are either disjoint or not, and any collection of pairwise crossing 
hyperplanes has size at most $\dimension X$.  So, if there are more than $\Ram(3,\dimension X+1)$ hyperplanes in $\mathcal D_i$ crossing $\sigma$, then there are 
three disjoint such hyperplanes, where $\Ram( \cdot, \cdot)$ denotes the Ramsey number.  Hence, if $\sigma$ has length $|\sigma|=C\cdot \Ram(3,\dimension X+1)+1$, 
then $\sigma$ crosses three disjoint hyperplanes $\hat u,\hat v,\hat w\in\mathcal D_i$.  Without loss of generality, 
say $\hat v$ separates $\hat u,\hat w$.  Hence $\alpha\cap N(\hat v)$ is a subpath of $\alpha$ lying between the points 
$\alpha\cap\hat u$ and $\alpha\cap\hat w\in\sigma$, so 
$$|\alpha\cap N(\hat v)|\leq C\cdot \Ram(3,\dimension X+1)+1=L.$$
Note that $L$ is independent of the hyperplane $\hat v$. Indeed, $L$ depends on $A_0,\lambda_0,\lambda_1,s_0$, and hence is determined by $X$, the hyperplane $\hat h_0$ and the geodesic $\alpha$.  
Now, using the fact that $N(\hat v)$ is isometric to a product
 of the form $\hat v\times(-\frac12,\frac12)$, we see that
$$\angle_{\alpha\cap\hat v}(\hat v,\alpha)\geq \sin^{-1}(\tfrac{1}{L}),$$
which we denote by $\kappa$.  

Since we can do the above procedure for infinitely many disjoint length--$L$ segments in $\alpha$, we find infinitely many 
hyperplanes in $\mathcal D_i$ making an angle at least $\kappa$ with $\alpha$.  Since any infinite subset of a UBS contains 
a chain, we thus have a chain $\{\hat u_n\}_{n\geq 0}\subset\mathcal D_i$ with this property.  Finally, the inseparable 
closure of $\{\hat u_n\}_{n\geq 0}$ is a UBS contained in $\mathcal D_i$, and hence commensurate with $\mathcal D_i$ since 
$\mathcal D_i$ is minimal.  This verifies the statement for the given $\mathcal D_i$, and we conclude by replacing $\kappa$ 
with the minimal $\kappa$ for the various $i\in\{d,\ldots,k\}$.
\end{proof}

\begin{remark}\label{Rem:HuangCuboids}
Extending Lemma~\ref{Lem:SingleRayAngleBound} to a cuboid metric $d_X^\rho$ requires a bit of care. The above proof relies on 
 some results of of Huang~\cite{Huang:quasiflats}, which are written in the context of a cube complex with finitely many isometry types of cells. This hypothesis may fail in $(X, d_X^\rho)$. 
Fortunately, we only need to use Huang's results in a \emph{finite} cuboid complex, namely the convex hull of $H_{st}\cup p(H_{st})$, where $H_{st}$ is itself finite and $p$ is a certain CAT(0) 
projection. Thus Huang's results~\cite[Lemmas 2.10 and  2.14]{Huang:quasiflats} apply to the subcomplex we need. 
 
Some constants in the proof would need to be adjusted for the cuboid metric.  The constants $\lambda_0, \lambda_1$ of Lemma~\ref{Lem:WallQI}, used in the definition of $C$, would have to be replaced 
by the constants $\lambda_0^\rho, \lambda_1^\rho$ of Lemma~\ref{Lem:CuboidWallQI}. 
In the constant $\sin^{-1}(\frac{1}{L})$, the numerator $1$ is the thickness of a hyperplane carrier, and would have to be replaced by the minimal thickness of a hyperplane carrier in  the metric 
$d_X^\rho$.
\end{remark}

Now we can prove the main result of this section.

\begin{prop}[Combining UBSes for interior points of Tits geodesic]
\label{Prop:build_CAT(0)}
Let $\alpha$, $\beta$ be CAT(0) geodesic rays with $\alpha(0) = \beta(0) \in X^{(0)}$.
Suppose that $\W(\alpha)\cup\W(\beta)$ is commensurate 
with a UBS.  Then $a=\alpha(\infty)$ and $b=\beta(\infty)$ are joined by a unique geodesic 
$g$ in $\tits X$. Furthermore, any interior point $c$ of $g$ is represented by a CAT(0) geodesic ray $\gamma$ 
such that $\W(\gamma)  = \W(\alpha)\cup\W(\beta)$.
\end{prop}

\begin{proof}
 By Lemma~\ref{Lem:tits_distance_bound}, we have 
$\dist_T(a,b)\leq\pi/2$, so the CAT(1) space $\tits X$ contains a 
unique geodesic $g$ from $a$ to $b$.  This proves the first assertion of the lemma. 

Let $c \in\tits X$ be an interior point of $g$.  Let $\gamma$ be the unique CAT(0) geodesic ray starting at $\alpha(0)$ and 
representing $c$.  We need to show that $\W(\gamma)=\mathcal 
W(\alpha)\cup\W(\beta)$. 

\smallskip
\textbf{Working in the convex hull of $\alpha\cup\beta$:}  Let $Y$ be the 
cubical convex hull of $\alpha\cup\beta$. We claim that $g$ lies in $\tits Y\subset\tits X$.  Indeed, 
applying Lemma~\ref{Lem:tits_distance_bound} inside of $Y$ shows that $a,b$ can be joined by a unique geodesic $g'$ in 
$\tits Y$.  Since the Tits distance in $\tits Y$ from $a$ to $b$ depends only on $\{\dist_Y(\alpha(t),\beta(t)):t\geq 0\}$ and 
$\dist_Y(\alpha(t),\beta(t))=\dist_X(\alpha(t),\beta(t))$ by convexity of $Y$, we have that $g'$ is a geodesic of $\tits X$.   Since 
$g'$ has length less than $\pi$, it is the unique geodesic in $\tits X$ from $a$ to $b$, i.e. $g=g'$ and $g$ lies in $\tits 
Y$. 

For the rest of the proof, we work entirely in $Y$.
For any $y\in Y$, we can choose rays $\alpha_y,\beta_y$ with $\alpha_y(0)=\beta_y(0)=y$ and $\alpha_y(\infty)=a,\beta_y(\infty)=b$, and note that the cubical convex hull $Y_y$ of 
$\alpha_y\cup\beta_y$ is contained in $Y$, and crosses all but finitely many of the hyperplanes crossing $Y$.  By convexity, $\tits Y_y=\tits Y$.  So, when convenient, we can move the basepoint.

\smallskip
\textbf{What we will actually verify:}  The plan is as follows.  Let $\hat h$ be a hyperplane of $Y$ that is dominant in the UBS $\U = \W(Y) = \W(\alpha) \cup \W(\beta)$.  We will show that $c\not\in\tits\hat h$.  Since the 
shallow side of any hyperplane in $Y$ is contained in a neighborhood of that hyperplane, it follows that $\gamma$ must cross $\hat h$.  Once we show that $\gamma$ crosses every dominant hyperplane, 
Corollary~\ref{Cor:Cross_Everybody} will then imply that $\gamma$ crosses every hyperplane, i.e. $\W(\gamma)=\W(\alpha)\cup\W(\beta)$, as required.  So, it remains to argue that 
$c\not\in\tits \hat h$ when $\hat h$ is a dominant hyperplane.

\smallskip
\textbf{Initial segments of $\alpha$ and $\beta$ inside a cube:}  For each $r\geq 0$, let $y_r=\alpha(r)$.  Let $\alpha_{y_r}$ and $\beta_{y_r}$ be geodesics asymptotic to $\alpha$ and $\beta$, defined as above.  Note that $Y_{y_r}\subset Y$ 
contains a ray based at $y_r$ and representing each element of $\tits Y$.  As in the proof of Lemma~\ref{Lem:tits_distance_bound}, there is a single cube $C_{y_r} \subset Y_{y_r}$, with $y_r\in 
C_{y_r}$, such that any ray in $Y_{y_r}$ emanating from $y_r$ has a nontrivial initial segment in $C_{y_r}$.  

Let $\alpha'_{y_r},\beta'_{y_r},\gamma'_{y_r}$ be the maximal (nontrivial) segments of $\alpha_{y_r},\beta_{y_r},\gamma_{y_r}$ that lie in $C_{y_r}$.  Then the Alexandrov angle between any two of the rays 
$\alpha_{y_r},\beta_{y_r},\gamma_{y_r}$ at $y_r$ is just the Euclidean angle in $C_{y_r}$ between the corresponding segments.

By~\cite[Proposition II.9.8.2]{BridsonHaefliger}, $\angle_{y_r}(\alpha_{y_r},\beta_{y_r})$ converges to $\angle_T(a,b) = \dist_T(a,b)$, where the equality follows from
Lemma~\ref{Lem:tits_distance_bound} because the distance is less than $\pi$.

\smallskip
\textbf{Angle computation:}  Let $\hat h$ be a dominant hyperplane in $\U$.  By Lemma~\ref{Lem:DominantChain},  we 
can assume that $\hat h\in\mathcal D$, where $\mathcal D = \mathcal D_j$ is a minimal UBS consisting of dominant hyperplanes. Since $\mathcal D \subset \W(\alpha) \cup \W(\beta)$ is a minimal UBS, we must have either $\mathcal D \preceq \W(\alpha)$ or $\mathcal D \preceq \W(\beta)$. We assume without loss of generality that $\mathcal D \preceq \W(\alpha)$.

Note that every hyperplane in $\mathcal D$ is dominant in $\W(\alpha)$, because any hyperplane crossing all but 
finitely many hyperplanes in $\U \setminus \mathcal D$ crosses all but finitely many hyperplanes in $\mathcal 
W(\alpha) \setminus \mathcal D$.  We now apply Lemma~\ref{Lem:SingleRayAngleBound} to $\W(\alpha)$ and $\mathcal D$ to produce 
a constant $\kappa$ and a chain  $\{ \hat h_n \} \subset \mathcal D$ whose inseparable closure is $\mathcal D$ and whose (necessarily dominant) 
hyperplanes all cross $\alpha$ at an angle at least $\kappa$:
$$\angle_{y_n}(\alpha_{y_n},\hat h_n)\geq\kappa,$$
where $y_n=\hat h_n\cap\alpha$ and 
$\alpha_{y_n}$ is the sub-ray of $\alpha$ emanating from $y_n$.  
Since all but finitely many of the $\hat h_n$ lie on the deep side of $\hat h$, it suffices to show that $\gamma$ crosses $\hat h_n$ for sufficiently large $n$. As explained above, it suffices to show that $c \notin \tits \hat h_n$ for sufficiently large $n$.

Let $\beta_{y_n}$ be the ray emanating from $y_n$ and 
asymptotic to $\beta$.  
Working in the cube $C_{y_n}$ constructed above, we have that the segment $\alpha_{y_n}'$ makes a Euclidean angle at least $\kappa$ with the midcube $M_n = C_{y_n} \cap\hat h_n$.  Meanwhile, 
the segment $\beta_{y_n}'$ makes some angle $\theta \geq 0$ with $M_n$.  
See Figure~\ref{fig:work_in_cube}.

Let $s\in(0,1)$ be such that $\dist_T(a,c)=s\cdot\dist_T(a,b)$ and $\dist_T(c,b)=(1-s)\cdot\dist_T(a,b)$.  Such an $s$ exists since $c$ is an interior point of the Tits geodesic $g$ from $a$ to $b$.

\begin{figure}[h]
 \begin{overpic}[width=0.6\textwidth]{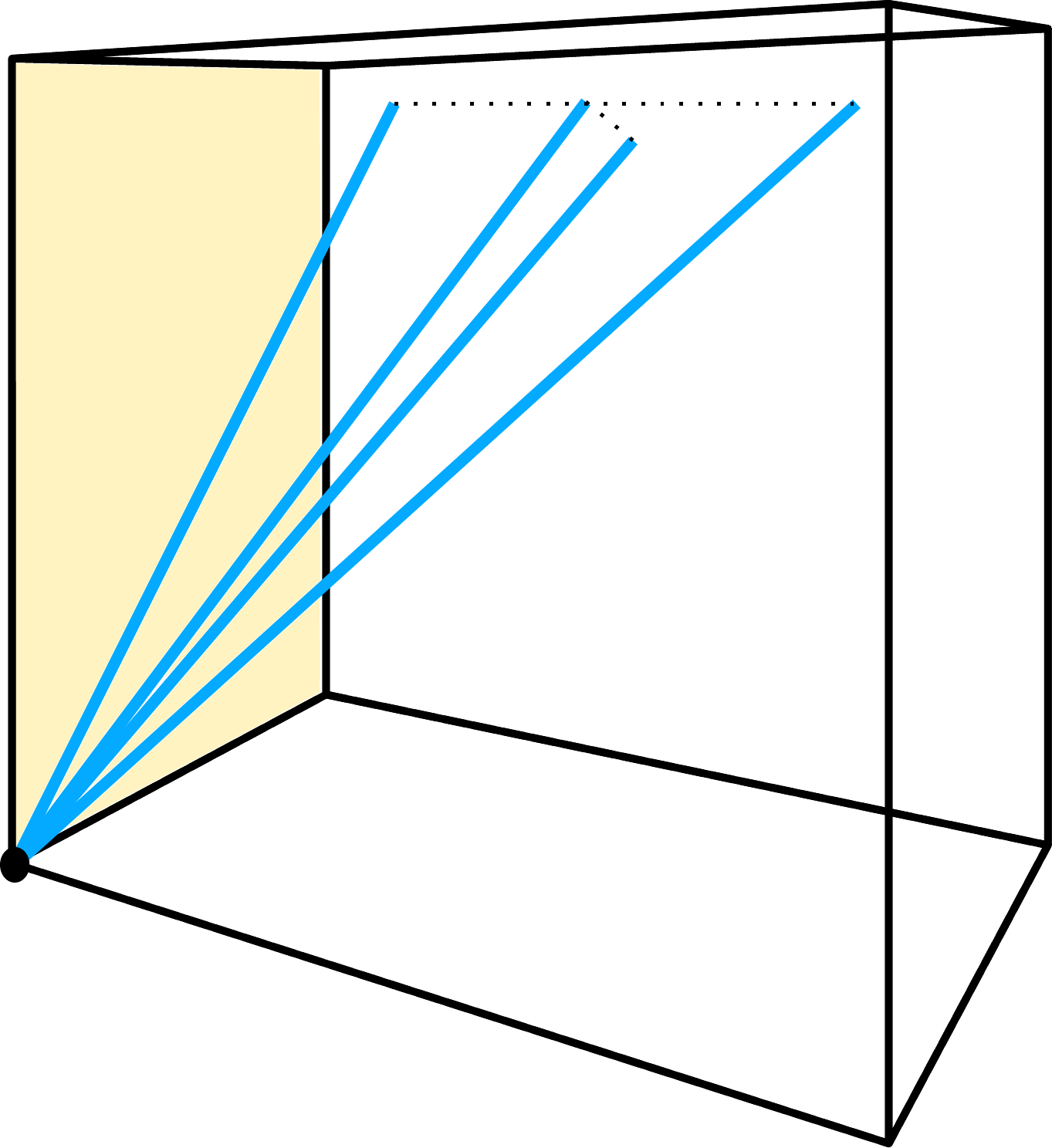}
  \put(-5,23){$y_n$}
  \put(5,80){$M_n$}
  \put(30,80){$\beta_{y_n}'$}
  \put(66,80){$\alpha_{y_n}'$}
  \put(51,80){$\gamma_{y_n}'$}
  \put(38,81){$\eta_n$}
  \put(38,91){$1-s$}
  \put(59,91){$s$}
 \end{overpic}

 \caption{The figure shows a portion of the cube $C_{y_n}$, with the back wall is $M_n = C_{y_n} \cap \hat h_n$.
The angle between $\alpha_{y_n}'$ and $M_n$ is at least $\kappa$, so the angle between $\eta_n$ and $M_n$ is bounded away from $0$ by $\kappa_1 = \kappa (1-s)$. But the angle between $\gamma_{y_n}'$ and $\eta_n$ can be made much smaller than $\kappa (1-s)$ by taking $n$ large.  So $\gamma_{y_n}'$ makes a positive angle with $\hat 
h_n$.}\label{fig:work_in_cube}
\end{figure}
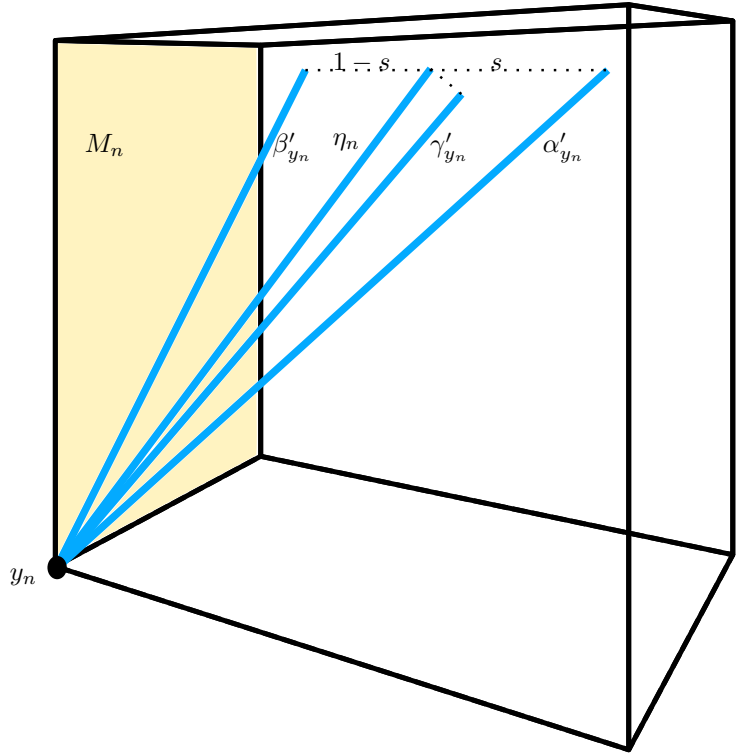

Let $\eta_n \subset C_{y_n}$ be the maximal segment that emanates from $y_n$ and makes an angle $s\cdot\angle_{y_n}(\alpha_{y_n}',\beta_{y_n}')$ with $\alpha_{y_n}'$ and and 
angle $(1-s)\cdot \angle_{y_n}(\alpha_{y_n}',\beta_{y_n}')$ with $\beta_{y_n}'$.  Then there exists $\kappa_1= (1-s) \kappa >0$ such that $\angle_{y_n}(\eta_n,
M_n) \geq \kappa_1$. Crucially, $\kappa_1$ is  independent of $n$.  On the other hand, as 
$n \to \infty$, we have convergent sequences
\begin{align*}
\angle_{y_n} (\alpha_{y_n}',\beta_{y_n}') & \to \dist_T(a,b), \\
\angle_{y_n} (\alpha_{y_n}',\gamma_{y_n}') &  \to \dist_T(a,c) = s \cdot \dist_T(a,b), \\
\angle_{y_n} (\beta_{y_n}',\gamma_{y_n}') & \to \dist_T(b,c) = (1-s) \cdot \dist_T(a,b),
\end{align*}
Thus,  for all sufficiently large $n$, 
$$\angle_{y_n}(\eta_n,\gamma_{y_n}')<\frac{\kappa_1}{2},$$
so by the triangle inequality for Alexandrov angles, $\angle_{y_n}(\gamma'_n,M_n) > \kappa_1/2 >0$.  Hence the ray $\gamma_{y_n}$ is not contained in $\hat h_n$ and, since its initial 
point is in $\hat h_n$, we thus have $c\not\in\tits\hat h_n$.  This completes the proof.
\end{proof}

We observe that the proof of Proposition~\ref{Prop:build_CAT(0)} extends to the cuboid setting with minimal effort. The proof of that lemma combines prior results and the local Euclidean geometry of a 
cube; all of the local arguments work equally well in a rescaled cuboid. 

\section{Connections between Tits and Roller boundaries}\label{Sec:TitsRollerConnections}

In this section, we establish some important connections between $\tits X$ and $\guralnik X$. Every Roller class $v \in \guralnik X$ is assigned a 
canonical CAT(1)--convex  \emph{Tits boundary realization} $Q(v) \subset \tits X$.
We define a pair of maps $\psi \from \tits X \to \guralnik X$ and $\varphi \from \guralnik X \to \tits X$  that will play a major role in the proof of our main theorem. See Proposition~\ref{Prop:diameter} for the properties of $Q(v)$ and Proposition~\ref{prop:phi_facts} for  the relationship between $\varphi$ and $\psi$. 

Toward the end of this section, we focus on $\ell^2$--visible Roller classes, namely all $v \in \guralnik X$
 such that $\psi \circ \varphi(v) = v$. 
In Section~\ref{Sec:covering_tits_boundary}, we will use the visible Roller classes to construct a closed covering of $\tits X$ that is compatible with a covering of a large part of $\absimp X$, denoted $\vispart X$.

All of the results of this section extend with minimal effort to a cuboid metric $d_X^\rho$ obtained via a $G$--admissible hyperplane rescaling (recall Definitions~\ref{Def:cuboid_metric} and~\ref{Def:TitsMetricCuboid}). This extension is described in Section~\ref{Sec:TitsCuboid}.

\begin{definition}[Map $\psi \from \tits X \to \guralnik X$] \label{Def:Psi}
Given $a \in\tits X$, define $\psi(a) = Y_a \in \guralnik X$ to be the Roller class $Y_a$ constructed in Lemma~\ref{Lem:DeepSetProps}. That is, $\psi(a)$ is the principal Roller class of the  intersection of half-spaces in the deep set  $D_a$. 
\end{definition}

Our definition of $\psi$ generalizes Guralnik's boundary decomposition map \cite[Definition 4.8]{Guralnik}, because $\tits X$ agrees as a set with the visual boundary $\bdy_\eye X$.

\begin{lemma}\label{Lem:PsiAlternate}
Let $a = [\alpha] \in\tits X$. Let $\W(\alpha)$ be the UBS consisting of hyperplanes crossing the geodesic $\alpha$. Then $\psi(a) = Y_a = Y_{\W(\alpha)}$.
\end{lemma}

\begin{proof}
By Lemma~\ref{Lem:DeepSetProps} we have that $D_a= \frakH_{\~Y_a}^+$ and $Y_a \in \guralnik X$. Fix $x\in X^{(0)}$. By definition,  $\widehat{D_a\setminus \frakH_x^+} = \W(x,Y_a) =  \W(x,[y])$ for any $y \in Y_a$. Finally by  Lemma  \ref{standardUBS}  $\W(x,[y])\sim\W(x,y)$. 

By Lemma~\ref{Lem:CombGeodesicExists}, there is a combinatorial geodesic $\alpha' \from [0,\infty) \to X$ with $\alpha'(0) = x$ and $\alpha'(\infty) =  y$.  By Remark~\ref{Rem:UBSofRollerGeodesic}, we have $\W(\alpha')= \W(x,y)$. 

Now, Lemma~\ref{lem:combinatorial_fellow_travels_cat0} implies that $\alpha$ and $\alpha'$ can be chosen to lie at finite Hausdorff distance. Thus $\alpha$ and $\alpha'$ cross the same hyperplanes, 
except possibly for 
finitely many. In symbols,
 $\W(\alpha') \sim \W(\alpha)$. By Lemma~\ref{UBScomplex}, we conclude that $Y_a = Y_{\W(\alpha)}$.
\end{proof}

\subsection{Tits boundary realizations}\label{Sec:BoundaryRealization}
Defining a map $\varphi \from \guralnik X \to \tits X$ that serves as a partial 
inverse to $\psi$ takes considerably more effort. As a first step, we will 
define a \emph{Tits boundary realization} $Q(v)$ associated to a Roller class 
$v$.

A family $\mathcal F$ of subsets of $X$ is called \emph{filtering} if for every $E,F\in \mathcal F$ there is a $D\in \mathcal F$ such that $D\subset E\cap F$. For example, for $y\in \roller X$,  the family 
$$
\mathcal F :=\{\hull(h_1)\cap \cdots \cap \hull(h_n) \; : \; h_1, \dots, h_n\in \frakH_{y}^+, n\in \N\}
$$ 
is a filtering family of closed convex subspaces. 

The following theorem combines results of Caprace--Lytchak \cite[Theorem 1.1]{CapraceLytchak} and Balser--Lytchak \cite[Proposition 1.4]{BalserLytchak}.

\begin{theorem}\label{thm:CapLytch}
Let $\{X_i\}_{i\in I}$ be a filtering family of closed convex subsets of a finite-dimensional CAT(0) space $X$. If the 
intersection $\bigcap_{i\in I} X_i$ is empty, then the intersection $\bigcap_{i\in I}{} \tits X_i$ of their 
boundaries is nonempty, and furthermore $\bigcap_{i\in I}{} \tits X_i$ has intrinsic radius at most $\pi/2$ and therefore a canonical circumcenter.  
\end{theorem}

Next, we present two definitions that will turn out to be equivalent (compare Lemma~\ref{Lem:Qy_defs}). The first of these definitions appears in~\cite[Corollary 6.2]{FLM:RandomWalks}.

\begin{definition}\label{Def:Qy}
Let $y\in \roller X$. Define $Q(y)=\bigcap_{h\in \frakH_{y}^+}{} \tits \hull(h)$. Note that $Q(y) \neq \emptyset$ by Theorem \ref{thm:CapLytch}.
We call $Q(y)$ the \emph{Tits boundary realization} of $y$.
\end{definition}

\begin{definition}\label{Def:Q'y}
Let $x\in X^{(0)}$ and let $y\in\roller X$.    Recall that $\I(x,y)\subset\~X$ is the vertex interval between $x$ and $y$, so 
that $\I(x,y)\cap X$ is a vertex--convex subset of $X^{(0)}$.  Generalizing Definition~\ref{Def:IntervalInX}, let $\J(x,y)$ be the union of all of the cubes in $X$ 
whose $0$--skeleta lie in $\I(x,y)\cap X$.
Then $\J(x,y)$ is a convex subcomplex of $X$, 
and hence has a well-defined Tits boundary naturally embedded in $\tits X$.  We define  $Q'(y)=\tits\J(x,y)$.
\end{definition}

In fact, the two definitions are equivalent: 

\begin{lemma}\label{Lem:Qy_defs}
Let $y\in\roller X$ and let $x,x'\in X^{(0)}$.  Then:
\begin{enumerate}[$(1)$ ]
 \item \label{item:q'_well_defined}$\tits \J(x,y)=\tits \J(x',y)$, 
hence $Q'(y)$ is well-defined.
 \item \label{item:q=q'} $Q(y)=Q'(y)$.
 \item \label{item:q'_depends_on_class} If $y\sim y'$, then $Q'(y)=Q'(y')$ and $Q(y)=Q(y')$.
\end{enumerate}
\end{lemma}

\begin{proof}
 To prove conclusion \ref{item:q'_well_defined}, observe that $\J(x,y)$ and $\J(x',y)$ lie at 
Hausdorff distance bounded by $\dist_1(x,x')$. Hence $\J(x,y)$ and $\J(x',y)$ have the same 
Tits boundary.

Next, we consider conclusion~\ref{item:q=q'}.  For any given vertex half-space $h\in \frakH_{y}^+$, 
let $\hull(h)$ be the 
associated CAT(0) half-space containing $y$.  By \ref{item:q'_well_defined}, we can assume that $x\in \hull(h)$.  
Hence $\J(x,y)\subset \hull(h)$, so $Q'(y)=\tits\J(x,y)\subset\tits \hull(h)$. Therefore $Q'(y)\subset Q(y)$.

For the reverse inclusion, suppose $q\notin Q'(y)$.  Then there is a CAT(0) half-space $\hull(h)$ such that $\mathcal 
J(x,y) \subset \hull(h)$ and $q \notin \tits \hull(h)$. (To find such a half-space, let $\gamma$ be any CAT(0) geodesic ray 
representing $q$. Since $\gamma$ leaves every finite neighborhood of $\J(x,y)$, it must cross a hyperplane $\hat h$ 
disjoint from $\J(x,y)$.) The associated vertex half-space $h$ satisfies $h\in \frakH_{y}^+$ but $q \notin \tits 
\hull(h)$. By Definition~\ref{Def:Qy}, it follows that $q \notin Q(y)$.  Hence $Q(y)\subset Q'(y)$.

Finally, consider conclusion~\ref{item:q'_depends_on_class}. If $y \sim y'$, the Hausdorff distance from 
$\J(x,y)$ to $\J(x,y')$ is bounded by $\dist_1(y,y') < \infty$. Using \ref{item:q=q'}, we obtain
$$Q(y) = Q'(y) = \tits \J(x,y) = \tits \J(x,y') = Q'(y') = Q(y'). \qedhere$$
\end{proof}

Following Theorem~\ref{thm:CapLytch} and Lemma~\ref{Lem:Qy_defs}, we can make the following definition.

\begin{definition}[Tits boundary realization]\label{Def:Qcircum}
Let $v \in \guralnik X$ be a Roller class. Define $Q(v) = Q(y)$ for any representative element $y \in v$.  
The reader can think of $Q(y)$ according to either Definition~\ref{Def:Qy} or Definition~\ref{Def:Q'y}.  We call
 $Q(v)$ the \emph{Tits boundary realization} of the Roller class $v$.
 
Observe that the collection of half-spaces $\mathfrak H_y^+$ containing $y$ has empty intersection in $X$. Thus, by  Theorem~\ref{thm:CapLytch}, 
$Q(v)$  has a canonical circumcenter. We define $\chi(v) \in \tits X$ to be the circumcenter of $Q(v)$. 
\end{definition}

\begin{cor}\label{Cor:Q_equivariant}
For every $v \in \guralnik X$ and every $g \in \Aut(X)$, we have $Q(gv) = gQ(v)$ and $\circum(gv) = g \circum(v)$. Furthermore, $Q(v)$ and $Q(gv)$ have the same intrinsic radius.
\end{cor}

\begin{proof}
Let $y\in \roller X$, and let $v=[y]$. Then, for every $h \in \frakH_{y}^+$, the map $g$ gives an isometry from $\tits \hull(h)$  to $\tits \hull(gh)$. Thus $g \from Q(v) \to Q(gv)$ is an isometry, hence $Q(v)$ and $Q(gv)$ have the same intrinsic radius. By Theorem~\ref{thm:CapLytch}, $\circum(v)$ and $\circum(gv)$ are uniquely determined by the geometry of $Q(v)$ and $Q(gv)$, respectively, hence $\circum(gv) = g \circum(v)$.
\end{proof}

Consider a Roller class $v\in\guralnik X$. We say that a combinatorial geodesic  ray $\gamma$ in $X$ \emph{represents} $v$ if 
$\gamma(\infty) = y \in v$, or equivalently if 
the UBS $\W(\gamma)$ represents the class $v$.   Setting $x = \gamma(0)$, observe that $\J(x,y)$ is the 
cubical convex hull of $\gamma$, and $\W(\gamma) = \W(x,y)$ is exactly the collection of hyperplanes crossing 
$\J(x,y)$. This leads to

\begin{lemma}\label{Lem:psiphi}
Let $v\in\guralnik X$ be a Roller class and let $a\in Q(v)$.  Then $\psi(a)\leq v$.  
\end{lemma}

\begin{proof}
Let $\gamma$ be a combinatorial geodesic ray representing $v$, with $\gamma(0) = x$ and  $\gamma(\infty) = y \in v$. Then  
the cubical convex hull of $\gamma$, namely $\J(x,y)$, is also CAT(0) geodesically convex. Since $a\in Q(v) = \tits \mathcal 
J(x,y)$ by Lemma~\ref{Lem:Qy_defs},
we may choose a CAT(0) geodesic $\alpha$  representing $a$ such that $\alpha \subset \J(x,y)$.  Then $\W(\alpha) 
\subset \W(\gamma) = \W(x,y)$.  Thus every hyperplane crossing $\alpha$ crosses $\J(x,y)$ and hence $\gamma$.   
Hence $\psi(a)\leq v$.
\end{proof}

\begin{remark}\label{rem:future_phi}
In Definition~\ref{Def:phi_map}, we will define a map $\varphi \from \guralnik X\to\tits X$ with the property that $\varphi(v)\in Q(v)$ for each Roller 
class $v$. This map will be defined by slightly perturbing the circumcenter $\circum(v)$.  In view of Lemma~\ref{Lem:psiphi}, we 
will have $\psi(\varphi(v))\leq v$.
\end{remark}

\begin{lemma}\label{Lem:containment}
Let $v,w\in\guralnik X$ satisfy $w\leq v$.  Then $Q(w)\subset Q(v)$.
\end{lemma}

\begin{proof}
Fix a basepoint $x_0\in X^{(0)}$ and let $\gamma_v,\gamma_w$ be combinatorial geodesic rays emanating from $x_0$ and 
representing $v,w$ respectively.  Since $w\leq v$, Theorem~\ref{Thm:UBStoRoller}.\ref{Itm:RUsection} says that $\mathcal 
W(\gamma_w)\preceq\W(\gamma_v)$.  Without moving $x_0$, we can replace $\gamma_w$ by its image under the gate map to the 
cubical convex hull of $\gamma_v$, 
ensuring that $\mathcal 
W(\gamma_w)\subset\W(\gamma_v)$.  Hence, setting $y_w = \gamma_w(\infty) \in w$ and $y_v = \gamma_v(\infty) \in v$, the 
cubical convex hulls of these geodesics satisfy $\J(x_0,y_w)\subset \J(x_0,y_v)$.  Taking 
boundaries and applying Lemma~\ref{Lem:Qy_defs} gives  $Q(w)\subset 
Q(v)$.
\end{proof}

\subsection{Diameter, intrinsic radius, and Tits-convexity of $Q(v)$}\label{subSec:Q_v_geometry}
Let $v \in \guralnik X$ be a Roller class. Recall that $Q(v)$ is the intersection of the Tits boundaries of the half-spaces corresponding to a representative element of $v$. We will need the following properties of  $Q(v)$:

\begin{prop}[Features of $Q(v)$]\label{Prop:diameter}
For any Roller class $v$, the Tits boundary realization $Q(v)$ has the following properties:
\begin{enumerate}[$(1)$ ]
\item \label{item:Q_diameter} $Q(v)$ has diameter at most $\pi/2$.
     \item \label{item:Q_convex} $Q(v)$ is Tits--convex.
     \item\label{item:Q_contractible} $Q(v)$ is contractible in the Tits metric topology and compact in the visual topology.
     \item \label{item:Q_radius} $Q(v)$ has intrinsic radius $r_v<\pi/2$. 
     \item \label{item:Q_center} $Q(v) \subset \~B_T(\circum(v), r_v)$.
\end{enumerate}
\end{prop}

\begin{proof}
 Let $v$ be an arbitrary Roller class and $x_0 \in X^{(0)}$ an arbitrary basepoint. Let $\gamma$ be a combinatorial geodesic ray based at $x_0$ and representing $v$. Let $H= H(x_0,v,\gamma)$ denote the cubical convex hull of $\gamma$. Then, for an arbitrary point $y \in [y] = v$, the cubical hull $H$ and the interval $\J(x_0, y)$ lie within bounded Hausdorff distance of each other.  Thus, by 
Lemma~\ref{Lem:Qy_defs}, $Q(v)$ is equal to the closed subset $\tits H  = \tits \J(x_0, y) \subset \tits X$.

 \smallskip
 \textbf{Diameter:} 
To bound $\diameter_T(Q(v))$, let $a,b\in 
Q(v)$.  Represent $a,b$ by CAT(0) rays $\alpha,\beta$ with initial point $x_0$. Note that $\alpha, \beta \subset H$ because $H$ is a convex subcomplex.  Then $\W(\alpha),\W(\beta)$ 
are both subsets of $\W(\gamma)$, which is a UBS representing $v$.  So, $\mathcal 
W(\alpha)\cup\W(\beta) \subset \W(\gamma)$.  Now, $\W(\alpha)\cup\W(\beta)$ is 
infinite.  Moreover, $\W(\alpha)\cup\W(\beta)$ is unidirectional and contains no facing triple, since $\W(\gamma)$ has those properties.  Thus $\mathcal 
W(\alpha)\cup\W(\beta)$ is a UBS, hence Lemma~\ref{Lem:tits_distance_bound} 
implies $\dist_T(a,b)\leq \pi/2$.

\smallskip
\textbf{Convexity:}  Since $H=H(x_0,v,\gamma)\hookrightarrow X$ is an isometric embedding in the CAT(0) 
metric, the 
inclusion $\tits H=Q(v)\hookrightarrow\tits X$ is a Tits-metric isometric embedding.  Indeed, let $a,b\in 
Q(v)$ be represented by CAT(0) rays $\alpha,\beta \subset H$ with initial point $x_0$. Since 
$Q(v)$ has diameter less than $\pi$, we have $\dist_T(a,b)=\angle(a,b)$ by  
\cite[Remark II.9.19.(2)]{BridsonHaefliger}.  
Since $\alpha, \beta \subset H$, the angle $\angle_T(a,b)$ is 
determined entirely by 
the set $\{\dist_X(\alpha(t),\beta(t))\}_{t\ge 0}$, by \cite[Proposition II.9.8.(4)]{BridsonHaefliger}).  By convexity of $H$, this set coincides with 
$\{\dist_H(\alpha(t),\beta(t))\}_{t\ge 0}$.  Hence the Tits distance from $a$ to $b$ measured in $\tits X$ 
is the same 
as the Tits distance measured in $\tits H=Q(v)$.

Now, $Q(v)$ contains a unique geodesic joining $a,b$.  Since this geodesic realizes 
the distance from $a$ to 
$b$ in $\tits 
X$, this shows that $Q(v)$ is Tits-convex.

\smallskip
\textbf{Intrinsic radius:}  Since $Q(v)$ is a CAT(1) space of 
diameter less than $\pi$, the intrinsic radius $r_v$ satisfies $r_v<\mathrm{diam}\ Q(v)$, by \cite[Proposition 1.2]{BalserLytchak}.  Since $\mathrm{diam}\ Q(v)\leq\pi/2$, we have $r_v<\pi/2$, proving 
assertion~\ref{item:Q_radius}.  
Now, observe that $Q(v) \subset \~B_T(\circum(v), r_v)$ by Definition~\ref{Def:Qcircum} and Theorem~\ref{thm:CapLytch}, proving~\ref{item:Q_center}.

\smallskip
\textbf{Topological properties:}  Since $Q(v)$ is a uniquely geodesic CAT(1) space, a standard straight-line homotopy allows one to deformation retract $Q(v)$ to 
a single point. Thus $Q(v)$ is contractible.

Finally, recall from Lemma~\ref{Lem:Qy_defs} (and from earlier in this proof) that $Q(v) = \tits \J(x_0,y)$ for an arbitrary point  
$y \in [y] = v$.  By Lemma~\ref{Lem:proper}, the subcomplex $\J(x_0, y)$ is proper, hence $Q(v)$ is compact in the visual topology.  (The visual topology is defined in \cite[Section 
II.8]{BridsonHaefliger}, where it is called the \emph{cone topology}.)
\end{proof}

\subsection{Visibility and perturbing the circumcenter}\label{subSec:visible_perturb}
Recall the $\Aut(X)$--equivariant circumcenter map 
 $\circum \from \guralnik X \to \tits X$, mentioned in Corollary~\ref{Cor:Q_equivariant}.
In this subsection, we perturb  $\circum$ to obtain a map $\varphi \from \guralnik X\to\tits X$ with slightly nicer 
properties.  First, like $\circum$, the map $\varphi$ will have the property that $\varphi(v)\in Q(v)$ for each Roller class 
$v$, and hence $\psi(\varphi(v))\leq v$.  Compare Remark~\ref{rem:future_phi}.

We will define $\varphi$ so that Roller classes $v$ for which $\psi(\varphi(v))=v$ are exactly those for which there is some 
CAT(0) geodesic ray representing $v$. (Note that a combinatorial geodesic exists for every $v$, but a CAT(0) geodesic is not guaranteed.) Later, we will work with only such Roller classes, which we term 
\emph{$\ell^2$--visible}; see Definition~\ref{Def:L2visible}. 
See also Lemma~\ref{Lem:L2Visibility} for several equivalent characterizations of visibility. While the map $\varphi$ is not guaranteed to be equivariant, the set of $\ell^2$--visible Roller classes 
will still be invariant.

The following example illustrates the point that $\circum(v)$ can lie in $Q(w)\subsetneq Q(v)$ for some $w<v$. 

\begin{example}\label{Ex:CircumcenterOnBoundary}
Consider the cone of $\mathbb{R}^3$ cut out by the planes $z=0$, $z=y$, and $z=x$, and let $X$ be the union of all cubes (in the standard cubical tiling of $\mathbb{R}^3$) contained in this cone. Then $\guralnik X$ has four classes, described in coordinates as $u=[(\infty, 0, 0)]$, $u'=[(0, \infty, 0)]$, $w = [(\infty, \infty, 0)]$, and $v = [(\infty, \infty, \infty)]$, all of which are represented by $\ell^2$ geodesic rays. Furthermore, $Q(v) = Q(\infty, \infty, \infty) = \tits X$ is isometric to an isosceles triangle of $S^2$ whose base has length $\pi/2$ and whose height is $\arctan(1/\sqrt{2})$. The base of the triangle is $Q(w) = Q(\infty, \infty, 0)$.
See Figure \ref{Fig:SphericalTriangle}.

\begin{figure}
\begin{overpic}[width=0.7\textwidth]{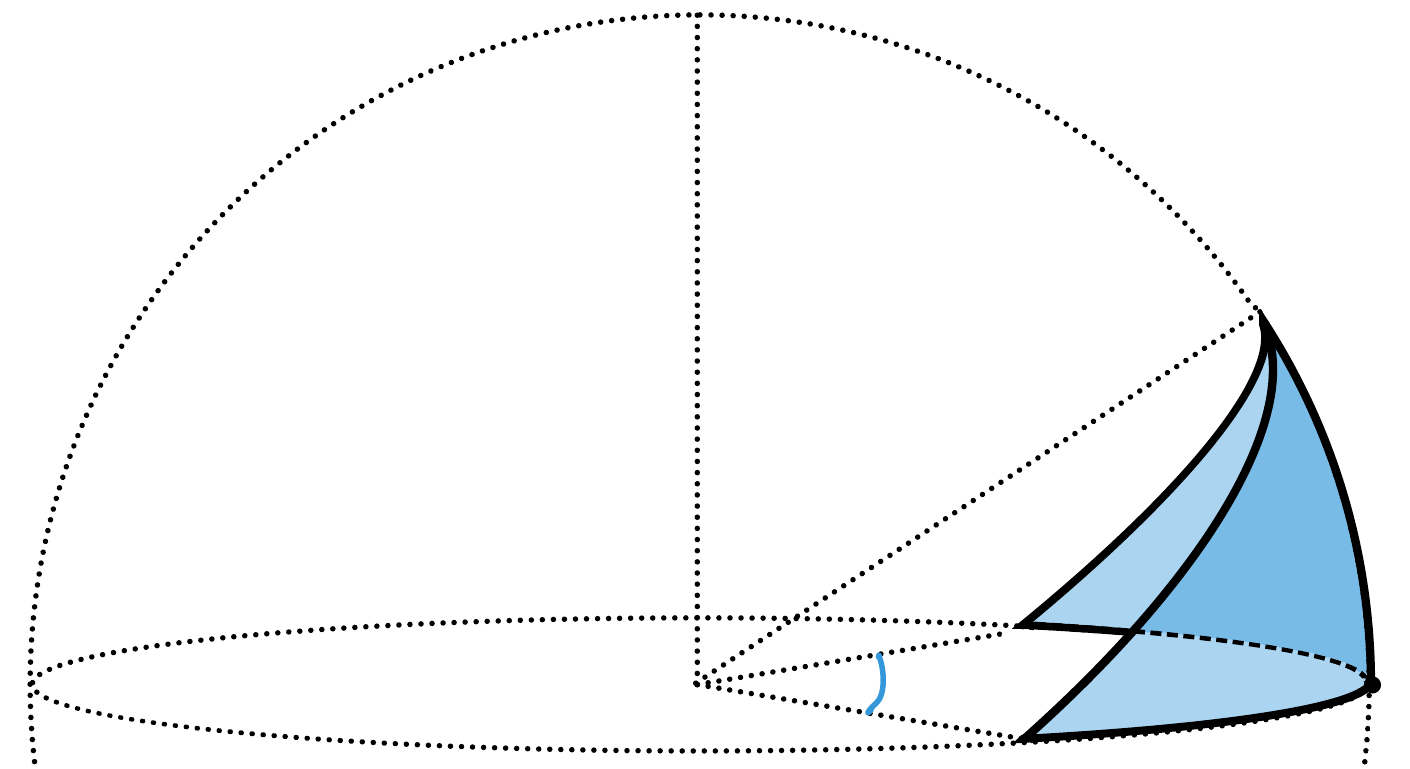}
\put(63.5,5){$\pi/2$}
\put(97,18){$\beta$}
\put(86,12){$Q(v)$}
\put(84,-1){$Q(w)$}
\put(99,5){$\circum(w) = \circum(v)$}
\end{overpic}
\caption{For the cube complex $X$ in Example~\ref{Ex:CircumcenterOnBoundary}, the Tits boundary $\tits X$ is the spherical triangle shown here. We have $\tits X = Q(v)$ for a single Roller class $v$, and $\circum(v) = \circum(w)$ for a Roller class $w < v$.}
\label{Fig:SphericalTriangle}
\end{figure}
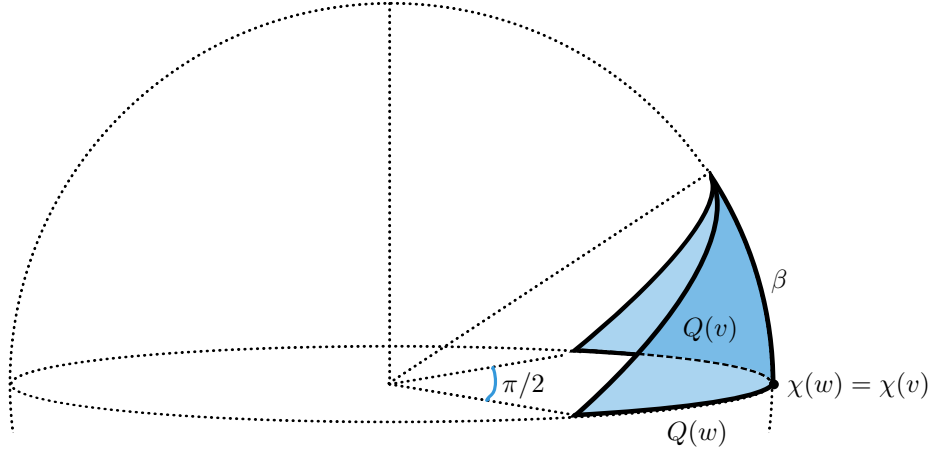

First, observe that $\circum(w)$ is the midpoint of the geodesic segment $Q(w)$. Next, we claim that $\circum(v) = \circum(w)$. To see this, let $\beta$ be the altitude from $\circum(w)$ to the the apex of $Q(v)$. Because $Q(v)$ has a reflective symmetry in $\beta$, its circumcenter must be contained in $\beta$. In addition, since $\operatorname{len}(\beta) = \arctan(1/\sqrt{2}) < \pi/4$, any point $b \in \beta$ must have distance less than $\pi/4$ to the apex but distance at least $\pi/4$ to the two endpoints of $Q(w)$. Thus $\circum(w) = \circum(v)$ even though $w<v$ and $Q(w)\subsetneq Q(v)$.
\end{example}

The
phenomenon of Example~\ref{Ex:CircumcenterOnBoundary} will be inconvenient later. To remedy this problem, we will define 
$\varphi(v)$ by perturbing $\circum(v)$ slightly, to get a point in the interior of $Q(v)$ that retains the property that 
all of $Q(v)$ is contained in the $r_v$--neighborhood of $\varphi(v)$ for some $r_v<\pi/2$.  (This latter property will 
also be necessary later.)

To achieve this, we need some preliminary discussion and lemmas.

\begin{lemma}\label{Lem:m_v}
For every Roller class $v$,  there is a unique Roller class $M_v \in \psi(Q(v))$ that is maximal among all Roller classes in $\psi(Q(v))$. Furthermore $Q(M_v)=Q(v)$.
\end{lemma}

\begin{proof}
If $w\in\psi(Q(v))$, then $w=\psi(a)$ for some $a\in Q(v)$, so $w\leq v$ by Lemma~\ref{Lem:psiphi}.  
Since there is a bound on the length of $\leq$--chains, it follows that $\psi(Q(v))$ contains $\leq$--maximal elements.

Suppose that $m,m'\in\psi(Q(v))$ are $\leq$--maximal.  Let $a,a'\in Q(v)$ be such that 
$\psi(a)=m$ and $\psi(a')=m'$.  Let $\alpha,\alpha'$ be CAT(0) geodesic rays in $X$ representing $a,a'$ 
respectively, chosen so that $\alpha(0)=\alpha'(0)$. Then $\U_m \sim \W(\alpha)$ and $\U_{m'} \sim \W(\alpha')$. 
 Since $m,m'\leq v$, Theorem~\ref{Thm:UBStoRoller} gives
 $\mathcal W(\alpha) \preceq \U_v$ and $\W(\alpha') \preceq \U_v$, hence
$W(\alpha)\cup\W(\alpha')$ is commensurate with a UBS by Lemma~\ref{Lem:FinManyMinimals}. Thus, by Proposition~\ref{Prop:build_CAT(0)}, there exists 
a CAT(0) geodesic ray $\beta$ with $\W(\beta)  = \W(\alpha)\cup\mathcal 
W(\alpha')$.  By construction, 
$$m \leq\psi(\beta(\infty))\leq v, \qquad m' \leq\psi(\beta(\infty))\leq v$$
 where the first inequality for $m$ or $m'$ is strict if 
$m\neq m'$. So, $\beta(\infty)\in Q(v)$ by Lemma~\ref{Lem:containment} and the second inequality.  Hence the 
first inequality contradicts the maximality of $m,m'$ in $\psi(Q(v))$, unless $m=m'$.  This proves 
the uniqueness of a maximal element $M_v$.

Next, $M_v\leq v$, so $Q(M_v)\subset Q(v)$ by Lemma~\ref{Lem:containment}.  On the other hand, if $a\in 
Q(v)$, then $\psi(a)\in\psi(Q(v))$, so $\psi(a)\leq M_v$.  Hence $a\in Q(\psi(a))\subset Q(M_v)$ by 
Lemma~\ref{Lem:containment}. Thus $Q(v)\subset Q(M_v)$, and we conclude  that $Q(v) = Q(M_v)$.
\end{proof}

For a Roller class $v$, define $Q_0(v) = Q(v) \setminus \bigcup_{u < M_v} Q(u)$.

\begin{lemma}\label{Lem:Q_0_characterize}
Let $v,w$ be Roller classes, and let $a \in Q(v)$. Then $a \in Q_0(v)$ if and only if $\psi(a)= M_v$.
Furthermore, if $Q(v) = Q(w)$, then $M_v=M_w$ and $Q_0(v)=Q_0(w)$.
\end{lemma}

\begin{proof}
First, suppose  that $a \in Q_0(v)$. Then $\psi(a) \leq M_v$ by the definition of $M_v$. If $\psi(a) = u < M_v$, then $a \in Q(u)$, contradicting the definition of $Q_0(v)$. Thus $\psi(a) = M_v$, as desired.

Conversely, suppose that $\psi(a)=M_v$.  If $a\in Q(u)$ for some $u<M_v$, then 
$\psi(a)\leq u$ by Lemma~\ref{Lem:psiphi}, so $\psi(a)<M_v$, a contradiction.  Hence $a\in Q_0(v)$.

Finally, suppose $Q(v) = Q(w)$. Then Lemma~\ref{Lem:m_v} says that $M_v = M_w$ is the unique maximal element of $\psi(Q(v)) = \psi(Q(w))$.
Thus, by the above equivalence, $Q_0(v)=Q_0(w)$.
\end{proof}

\begin{lemma}\label{Lem:Q_0_dense}
Let $v$ be a Roller class. Then  $Q_0(v)$ contains points arbitrarily close to $\circum(v)$ in the Tits metric. 
\end{lemma}

\begin{proof}
By Lemma~\ref{Lem:m_v}, there is a point $a \in Q(v)$  such that $\psi(a) = M_v$.
By Lemma~\ref{Lem:Q_0_characterize}, we have $a \in Q_0(v)$.
Let $g$ be the geodesic in $\tits X$ from $a$ to 
$b = \circum(v)$.  By Proposition~\ref{Prop:diameter}.\eqref{item:Q_convex}, $g$ lies in $Q(v)$.
 For any $\epsilon>0$, choose $c \in g$ so that $0 < \dist_T(c,\circum(v))<\epsilon$.  We will show $c \in Q_0(v)$.
 
By Proposition~\ref{Prop:build_CAT(0)}, the Tits points $a,b,c$ are represented by CAT(0) geodesic rays $\alpha, \beta, \gamma$ such that $\W(\gamma) = \W(\alpha) \cup \W(\beta) \supset \W(\alpha)$. Thus, by Lemma~\ref{Lem:PsiAlternate}, we have $\psi(c) = Y_{\W(\gamma)} \subset Y_{\W(\alpha)} = \psi(a)$.
By Lemma~\ref{POSET Roller}, we have $\psi(c) \geq \psi(a) = M_v$, hence $\psi(c) = M_v$ by the maximality of $M_v$.
Therefore, $c \in Q_0(v)$ by Lemma~\ref{Lem:Q_0_characterize}.
\end{proof}

\begin{definition}[The pseudocenter $\varphi$]\label{Def:phi_map}
Let $v$ be a Roller class, and fix $Q = Q(v)$. Let $r<\pi/2$ be the intrinsic radius of $Q$. Using Lemma~\ref{Lem:Q_0_dense}, choose $p_Q\in 
Q_0(v)$ such that $d_T(p_Q,\circum(v))\leq \pi/4 -r/2$. We call $p_Q$ the \emph{pseudocenter} of $Q$, and denote it $\varphi(v)$.

Using Lemma~\ref{Lem:Q_0_characterize}, we ensure that $\varphi(v) = \varphi(w) = p_Q$ for every Roller class $w$ such that $Q(w) = Q(v) = Q$.
\end{definition}

In the following proposition we collect the facts about the pseudocenter map 
$\varphi$ that will be needed later, in conjunction with Lemma~\ref{Lem:L2Visibility} and the definition of 
$\varphi$.

\begin{prop}[$\varphi$ and $\psi$ facts]\label{prop:phi_facts}
     The pseudocenter map $\varphi$ has the following properties.  Let $v\in\guralnik X$.  Then:
     \begin{enumerate}[$(1)$ ]
          \item $\varphi(v)\in Q(v)$.\label{item:phi_in_Q}
          \item $\psi(\varphi(v))\leq v$.\label{item:composition}
          \item $Q(v)=Q(\psi(\varphi(v)))$.\label{item:same_Q}
          \item If $Q(w)=Q(v)$, then $\varphi(w)=\varphi(v)$. \label{item:same_phi}
     \end{enumerate}
\end{prop}

\begin{proof}
     Assertion~\ref{item:phi_in_Q} holds by construction, and, together with Lemma~\ref{Lem:psiphi}, implies 
assertion~\ref{item:composition}.  Since $\varphi(v)\in Q_0(v)$, we have $\psi(\varphi(v))=M_v$ by Lemma~\ref{Lem:Q_0_characterize}, and $Q(v)=Q(M_v)$ by 
Lemma~\ref{Lem:m_v}, yielding assertion~\ref{item:same_Q}.  Assertion~\ref{item:same_phi} follows from the 
definition of $\varphi$ and Lemma~\ref{Lem:Q_0_characterize}.
\end{proof}

We can now define ($\ell^2$)--visible Roller classes:

\begin{definition}[$\ell_2$--visible]\label{Def:L2visible}
Let $v\in\guralnik X$.  Then $v$ is \emph{$\ell^2$--visible} if $v\in\psi(\tits X)$. Let 
$\visi(X) = \psi(\tits X)$ denote the set of visible Roller classes. Observe that $\Aut(X)$ stabilizes $\visi(X)$, because the map $\psi$ is $\Aut(X)$--equivariant.
\end{definition}

\begin{lemma}[Characterizing visibility]\label{Lem:L2Visibility}
Let $v\in\guralnik X$. Then the following are equivalent:
\begin{enumerate}[$(1)$ ]
\item\label{Itm:VisibleDef} $v$ is $\ell^2$--visible.
\item\label{Itm:VisibleRay} There exists a CAT(0) geodesic ray $\alpha$ such that $\W(\alpha)$ represents $v$.
\item\label{Itm:VisibleMv} $v = M_v$.
\item\label{Itm:VisiblePhi} $v = \psi(\varphi(v))$.
\end{enumerate}
\end{lemma}

\begin{proof}
To begin with, we have \ref{Itm:VisibleDef}$\Leftrightarrow$\ref{Itm:VisibleRay} by Lemma~\ref{Lem:PsiAlternate}.

For \ref{Itm:VisibleRay}$\Rightarrow$\ref{Itm:VisibleMv},  let $\alpha$ be a CAT(0) geodesic ray such that $\mathcal 
W(\alpha)$ represents $v$. Let $a = \alpha(\infty) \in \tits X$. 

Then, by Lemma~\ref{Lem:PsiAlternate}, we have $\psi(a) = v$.  Now, recall from Lemma~\ref{Lem:m_v} that $M_v$ is 
maximal among all Roller classes in $\psi(Q(v))$. By 
Lemma~\ref{Lem:psiphi}, every class $w \in \psi(Q(v))$ satisfies $w \leq v$, hence $M_v \leq v$.  On the other hand, $v = 
\psi(a) \in \psi(Q(v))$, hence $v \leq M_v$ by the maximality of $M_v$. Thus $v = M_v$.

For \ref{Itm:VisibleMv}$\Rightarrow$\ref{Itm:VisiblePhi}, suppose $v = M_v$.  By Definition~\ref{Def:phi_map}, we have 
$\varphi(v) \in Q_0(v)$. Thus, by Lemma~\ref{Lem:Q_0_characterize}, we have $\psi(\varphi(v)) = M_v = v$.

For \ref{Itm:VisiblePhi}$\Rightarrow$\ref{Itm:VisibleRay}, suppose $v = \psi(\varphi(v))$. 
Then $\varphi(v) \in Q_0(v) \subset Q(v)$. 
Choose a CAT(0) geodesic ray $\alpha$ representing $\varphi(v)$.   The associated UBS $\W(\alpha)$ represents the 
class $\psi(\varphi(v)) = v$.
\end{proof}

\begin{lemma}\label{Lem:maximal_visible}
Let $v\in\guralnik X$ be a $\leq$--minimal Roller class.  Then $v$ is $\ell^2$--visible.
\end{lemma}

\begin{proof}
Lemma~\ref{Lem:psiphi} gives $v\geq\psi(\varphi(v))$.  By 
minimality of $v$, we thus have $v=\psi(\varphi(v))$.
\end{proof}

\subsection{Cuboid generalization}\label{Sec:TitsCuboid}
Here, we explain how to generalize the results of this section to a cuboid metric $d_X^\rho$. A reader who is only interested in cube complexes with the standard $\ell^1$ and $\ell^2$ metrics is invited to skip ahead to Section~\ref{Sec:covering_tits_boundary}. 

For the duration of this subsection, fix a $G$--admissible hyperplane rescaling $\rho$ (Definition~\ref{Def:cuboid_metric}) and the resulting cuboid metric $d_X^\rho$. Recall that the cuboid Tits boundary $\tits^\rho X = \tits(X, d_X^\rho)$ was defined in Definition~\ref{Def:TitsMetricCuboid}. All of the constructions and results involving $\tits^\rho X$ will be invariant under the restricted automorphism group $\Aut(X^\rho)$ of Definition~\ref{Def:AutXrho}.

To start, we can define a map $\psi^\rho \from \tits^\rho X \to \guralnik X$ exactly as in Definition~\ref{Def:Psi}. Generalizing Lemma~\ref{Lem:PsiAlternate}, we can characterize $\psi^\rho(a) = \psi^\rho([\alpha])$ as the principal class of the umbra of a UBS representing a ray $\alpha$:
$$ \psi^\rho(a) = \psi^\rho([\alpha]) = Y_{\W(\alpha)}.$$
This characterization holds because the proof of Lemma~\ref{Lem:PsiAlternate} works for $\psi^\rho$: its proof combines previous lemmas and combinatorial properties of hyperplane sets. 

Now, consider a point $y \in \roller X$. Since $(X, d_X^\rho)$ is a CAT(0) space and the collection of half-spaces $\frakH_y^+$ is a filtering family, we can define the Tits boundary realization $Q^\rho(y)$ exactly as in Definition~\ref{Def:Qy}. Following Lemma~\ref{Lem:Qy_defs}, whose proof extends to cuboids because cubical half-spaces are convex in $d_X^\rho$, we learn that the boundary realization $Q^\rho(y)$ is also the boundary of a cubical interval, and depends only on the Roller class $v = [y]$. Thus we may define a Tits-convex set  $Q^\rho(v)$ and its circumcenter $\chi^\rho(v)$, as in Definition~\ref{Def:Qcircum}. Following Corollary~\ref{Cor:Q_equivariant}, both  $Q^\rho(v)$ and $\chi^\rho(v)$ are invariant under the restricted automorphism group  $\Aut(X^\rho)$.

The partial order properties of the Tits boundary realization $Q(v)$ that are proved in Lemmas~\ref{Lem:psiphi} and ~\ref{Lem:containment} still hold for $Q^\rho(v)$, because the proofs of those results are essentially combinatorial.
The features of $Q(v)$ described in  Proposition~\ref{Prop:diameter} also hold for $Q^\rho(v)$, because the proof of the Proposition uses prior results and the convexity of the cubical convex hull of a geodesic ray.

Next, Section~\ref{subSec:visible_perturb} contains several results about perturbing the circumcenter $\circum(v) \in Q(v)$ to a pseudocenter $\varphi(v)$. The definition of the pseudocenter $\varphi(v)$ is enabled by Lemmas~\ref{Lem:m_v}, \ref{Lem:Q_0_characterize}, and~\ref{Lem:Q_0_dense}. All of these lemmas generalize immediately to the cuboid setting, because their proofs are essentially a top-level assembly of prior results. Thus we may generalize Definition~\ref{Def:phi_map} to define a pseudocenter $\varphi^\rho(v) \in Q^\rho(v)$. This pseudocenter $\varphi^\rho(v)$ has all of the properties described in Proposition~\ref{prop:phi_facts}, because the proof of that proposition merely assembles previously established results.

Finally, Definition~\ref{Def:L2visible} generalizes immediately to define $\visi^\rho(X) = \psi^\rho (\tits^\rho X)$, the set of Roller classes that are visible after the rescaling $\rho$. Since $\psi^\rho$ is $\Aut^\rho(X)$--equivariant, the set $\visi^\rho(X)$ is $\Aut^\rho(X)$--equivariant as well. Lemmas~\ref{Lem:L2Visibility} and~\ref{Lem:maximal_visible} generalize immediately to the cuboid setting, because their proofs are top-level assemblies of previous lemmas.

\begin{remark}[The set $\visi(X)$ and cuboids]\label{Rem:VisCuboid}
In Definition~\ref{Def:L2visible}, the set $\visi(X) = \psi(\tits X)$ is defined in terms of the CAT(0) metric on $X$.  When 
we change the metric on $X$ using a $G$--admissible rescaling $\rho$, there is no \emph{a priori} reason to expect $\visi(X) 
= \psi(\tits X)$ to coincide with $\visi^\rho(X) = \psi^\rho (\tits^\rho X)$.  Instead, our whole argument simply goes 
through for whichever of $\visi(X)$ or $\visi^\rho(X)$ we are considering.

In fact, it turns out that $\visi(X)=\visi^\rho(X)$. Since we do not use this fact in any proofs, we only sketch a proof.  Let $v\in\visi(X)$ 
be a visible Roller class, and let $\alpha:[0,\infty)\to X$ be a CAT(0) geodesic ray for the metric $\dist_X$ with $\mathcal 
W(\alpha)$ representing $v$.  We will produce a geodesic ray $\beta$ for the CAT(0) metric $\dist_X^\rho$, satisfying 
$\mathcal W(\alpha)=\mathcal W(\beta)$, from which it follows that $v\in\visi^\rho(X)$.  This shows 
$\visi(X)\subseteq\visi^\rho(X)$, and a symmetric argument gives the reverse inclusion.

Since cubical convexity of subcomplexes is independent of which CAT(0) metric we consider, we assume for convenience that 
$X$ is the cubical convex hull of $\alpha$, i.e. $\mathcal W(\alpha)=\mathcal W(X)$.

First, use Lemma~\ref{Lem:SingleRayAngleBound} to produce chains $\{\hat h_n^1\}_n,\cdots,\{\hat h^k_n\}_n$ of dominant 
hyperplanes such that there exists a constant $C\geq 0$ with the property that, for all $j\leq k$, any subpath of $\alpha$ 
of length at least $C$ crosses $\hat h^j_n$ for some $n$.  The lemma allows us to choose these to be dominant hyperplanes in 
the UBS $\mathcal W(\alpha)$; therefore, if we construct a $\dist_X^\rho$--geodesic $\beta$ that crosses each $h_n^j$, then 
by Lemma~\ref{Lem:DominantChain}, we get $\mathcal W(\alpha)=\mathcal W(\beta)$.

By admissibility of the rescaling, there exists $m\geq 1$ such that the identity map $(X,\dist_X)\to(X,\dist_X^\rho)$ is 
$m$--bilipschitz.  Since $X$ is the convex hull of $\alpha$ and contains no facing triples, Lemma~\ref{Lem:proper} guarantees properness of 
$(X,\dist_X)$ and implies properness of $(X,\dist_X^\rho)$.  Thus, letting $\beta_t$ 
be the $\dist_X^\rho$--geodesic from $\alpha(0)$ to $\alpha(t)$, and letting $t$ tend to infinity, the $\beta_t$ subconverge 
uniformly on compact sets to a $\dist_X^\rho$--geodesic $\beta$ with $\beta(0)=\alpha(0)$.

Fix $j\leq k$.  Now, for any $t$, the hyperplanes crossing $\beta_t$ are exactly those crossing $\alpha([0,t])$.  Indeed, 
any hyperplane separating $\alpha(0),\alpha(t)$ crosses $\beta_t$, and any hyperplane crossing $\beta_t$ does so in a single 
point (and hence separates $\alpha(0),\alpha(t)$), since $\beta_t$ is a $\dist_X^\rho$--geodesic.  In particular, 
$\alpha([0,t])$ crosses $h_1^j,\ldots,h_{N_t}^j$, where $N_t=\lfloor t/C\rfloor$.  Hence $\beta_t$ crosses the same 
hyperplanes while having $\dist_X^\rho$--length at most $mt$.  From this one deduces the following: there exists $L$ such 
that for all $n$, and all sufficiently large $t$, the path $\beta_t$ crosses $h^1_1,\ldots,h^1_n$ and intersects each of 
those carriers in a subpath of length at most $L$.  In other words, there is a uniform bound on how long each $\beta_t$ can 
fellow-travel any of the $\hat h_n^j$, whence $\beta$ cannot be parallel to a ray in any $\hat h^n_j$.  This implies that 
$\beta$ must cross $\hat h^j_n$ for all $j,n$, and hence, as explained above, $\mathcal W(\beta)=\mathcal W(\alpha)$ and 
$v\in\visi^\rho(X)$.

Looking ahead to Definition~\ref{Def:MaxVis},  the above argument also shows that $\maxvis(X) = \maxvis^\rho(X)$ for any 
admissible rescaling $\rho$. In the next section, we will heavily use $\maxvis(X)$ to construct open and closed coverings of 
$\tits X$, so that the nerves of those coverings can be used to prove homotopy equivalence. In an analogous fashion, we will 
use $\maxvis^\rho(X)$ to construct open and closed coverings of $\tits^\rho X$.   One can then apply Lemma~\ref{Lem:psiphi} 
and Lemma~\ref{Lem:containment} to show that the nerves of the coverings of $\tits X$ coincide with the nerves of the 
corresponding coverings of $\tits^\rho X$, and one can use Lemma~\ref{Lem:omega_nerve} to see that the covering of 
$\vispart X$ coming from visible Roller classes of $X$ coincides with the corresponding covering constructed using 
visible classes of $X^\rho$.  However, as we mentioned, the cuboid version of our argument does not rely on this; one 
instead just substitutes $\maxvis^\rho(X)$ for $\maxvis(X)$ everywhere, and runs all the arguments.     
\end{remark}

\section{Open and closed coverings of $\tits X$}\label{Sec:covering_tits_boundary}
In this section, we study the closed covering of $\tits X$ by the sets $Q(v)$ corresponding to maximal  $\ell^2$--visible 
Roller classes (see Definition~\ref{Def:MaxVis}).  Our goal, achieved in Theorem~\ref{Thm:first_HE}, is to show that $\tits 
X$ is homotopy equivalent to the nerve $\mathcal N_T$ of this cover.

\begin{definition}[Maximal visible Roller classes]\label{Def:MaxVis}
Let $\maxvis(X)$ denote the set of Roller classes $v$ such that:
\begin{itemize}
     \item $v$ is $\ell^2$--visible;
     \item if $w$ is visible and $v\leq w$, then $w=v$.
\end{itemize}
\end{definition}

The plan for this section is as follows. We will first show that $\{Q(v):v\in\maxvis(X)\}$ is a closed covering of $\tits X$.  Then, we will thicken each closed set $Q(v)$ to an open set $U(v)$, in 
such a way that the intersection pattern of the open cover $\{U(v) :v\in\maxvis(X) \}$ is the same as that of the closed cover $\{Q(v):v\in\maxvis(X)\}$.  Then we will apply the Nerve Theorem to 
conclude that $\tits X$ is homotopy equivalent to the nerves of these covers.

The main result of this section is Theorem~\ref{Thm:first_HE}. Section~\ref{Sec:closed_cover_tits_boundary} is about the initial closed covering.  Section~\ref{Sec:open_cover_tits_boundary} describes the thickening procedure. Section~\ref{Sec:CuboidCover} describes how to generalize these results to a cuboid metric $d_X^\rho$.

\subsection{The closed covering}\label{Sec:closed_cover_tits_boundary}
The maximal visible Tits boundary realizations $Q(v)$ provide a closed covering of the Tits boundary. We use this covering to define a nerve, as follows.

\begin{definition}[Simplicial complex $\mathcal N_T$]\label{Def:Big_N_T}
Let $\mathcal N_T$ be the simplicial complex with vertex set  $\maxvis(X)$, where vertices $v_0,\ldots,v_n\in\maxvis(X)$ span an $n$--simplex if and only if $\bigcap_{i=0}^nQ(v_i)\neq\emptyset$.
\end{definition}

\begin{lemma}[Covering $\tits X$]\label{Lem:max_vis}
Let $v\in\guralnik X$.  Then  there exists $w\in\maxvis(X)$ with $\varphi(v)\in Q(w)$.  
Hence 
$\{Q(v):v\in\maxvis(X)\}$ covers $\tits X$.
\end{lemma}

\begin{proof}
Note that $\psi( \varphi(v)) \in \visi(X)$, by Definition~\ref{Def:L2visible}.
Hence there exists $w\in\maxvis(X)$ such that $\psi \varphi(v) \leq w$.  
By Lemma~\ref{Lem:containment}, $\varphi(v) \in Q(\psi \varphi(v))\subset Q(w)$.

Now, given a point $a\in\tits X$, we can apply the above argument to $v=\psi(a)$.  Then, for some $w\in\maxvis(X)$, Lemma~\ref{Lem:L2Visibility} implies
$$
a\in Q(\psi(a))=Q(\psi \varphi \psi(a))  \subset Q(w) .
 $$
Hence $\{Q(v):v\in\maxvis(X)\}$ covers 
$\tits X$.
\end{proof}

\begin{lemma}\label{Lem:visible_injective}
Let $v,w\in\visi(X)$.  Then $Q(v)=Q(w)$ if and only if $v=w$. In particular, the assignment $v \mapsto Q(v)$ gives a bijection from $\maxvis(X)$ to its image.
\end{lemma}

\begin{proof}
Suppose $Q(v) = Q(w)$. Then Proposition~\ref{prop:phi_facts}.\ref{item:same_phi} implies $\varphi(v)=\varphi(w)$. By visibility and Lemma~\ref{Lem:L2Visibility}, we have $\psi \varphi(v)=v$ and $\psi \varphi(w)=w$, hence $v=w$.
\end{proof}

Now, Lemmas~\ref{Lem:max_vis} and~\ref{Lem:visible_injective} combine to yield:

\begin{cor}\label{Cor:MtNerve}
$\mathcal N_T$ is the nerve of the covering of $\tits X$ by the collection of closed sets $\{Q(v):v\in\maxvis(X)\}$.
\end{cor}

\subsection{The open covering}\label{Sec:open_cover_tits_boundary}
The goal of this subsection is to thicken up the closed sets $Q(v)$ for $v \in \maxvis(X)$ to be open sets $U(v)$, such that the 
intersection pattern of the $U(v)$ is the same as that of the $Q(v)$. This is needed since the Nerve 
Theorem~\ref{Thm:EquivariantOpenNerve} works for open covers only.

\begin{lemma}\label{Lem:disjoint_Q}
Let $v,w\in\visi(X)$ be points such that $Q(v)\cap Q(w)=\emptyset$. Then
 $\dist_T(Q(v),Q(w))> \epsilon_0$, where $\epsilon_0 > 0$ is a constant depending only on $\dim X$.
 \end{lemma}

\begin{proof}
Let $a\in Q(v)$ and $b\in Q(w)$.  By definition, $a\in Q(\psi(a))$.  Since 
$v\geq\psi(a)$, we have $Q(\psi(a))\subset Q(v)$ and 
similarly $Q(\psi(b))\subset Q(w)$, by Lemma~\ref{Lem:containment}. Since we have assumed $Q(v)\cap Q(w)=\emptyset$, it follows that  $Q(\psi(a))\cap 
Q(\psi(b))=\emptyset$.

Choose CAT(0) geodesic rays $\alpha,\beta$ such that 
$\alpha(0)=\beta(0) \in X^{(0)}$ and $\alpha(\infty)=a$, $\beta(\infty)=b$.
Since $X$ is finite-dimensional, Lemma~\ref{lem:combinatorial_fellow_travels_cat0} provides  $\ell^1$ geodesic rays 
$\bar\alpha, \bar\beta$ with common basepoint $\alpha(0)=\beta(0)$, that lie at 
uniformly bounded  Hausdorff distance from $\alpha,\beta$ respectively.
Note that 
$\W(\alpha) \sim \W(\bar\alpha)$ and $\W(\beta) \sim \W(\bar\beta)$.

We claim that $\W(\bar \alpha)\cap\mathcal 
W(\bar \beta)$ is finite. Suppose not. By Lemma~\ref{Lem:Qy_defs}, $Q(\psi(a)) = \tits \J_{\bar \alpha}$, where $\J_{\bar 
\alpha} = \J(\bar \alpha(0), \bar \alpha(\infty))$  is the cubical convex hull of $\bar \alpha$. Similarly, 
$Q(\psi(b)) = \tits \J_{\bar \beta}$.
Since $C=\J_{\bar\alpha}\cap \J_{\bar\beta}\neq\emptyset$, we have
$\W(C) = \W(\bar\alpha)\cap\W(\bar\beta)$ by Lemma~\ref{Lem:Bridge}.\ref{Itm:WallSetIntersect}.
Now $\J_{\bar \alpha}$ is a proper CAT(0) space, by Lemma~\ref{Lem:proper}. Thus $C$ is also a proper CAT(0) space, hence $|\W(C)| = \infty$ implies $C$ is unbounded and 
$\tits C$ is a nonempty subspace of $Q(\psi(a))$.  Similarly, $\tits C$ is a nonempty subspace of $Q(\psi(b))$.  Thus $\tits 
C  \subset Q(\psi(a)) \cap Q(\psi(b)) = \emptyset$, a contradiction. This proves the claim.

Since $\W(\alpha) \sim \W(\bar\alpha)$ and $\W(\beta) \sim \W(\bar\beta)$, we may define
 $N = |\W(\alpha) \cap \W(\beta)| < \infty$. Recall that $\alpha(0) = \beta(0)$.
Then, for all $t \geq 0$, we have
$$| \W( \alpha(t) , \beta(t) ) | \geq | \W( \alpha(t) , \alpha(0))| + |\W(\beta(0), \beta(t) ) | - 2N.$$
Applying Lemma~\ref{Lem:WallQI} to both sides gives
$$
\lambda_0 \cdot d_X(\alpha(t), \beta(t) )+ \lambda_1 \geq \left(  \frac{t}{\lambda_0} - \lambda_1 \right) + \left(  \frac{t}{\lambda_0} - \lambda_1 \right) - 2N = \frac{2t}{\lambda_0} - 2(N + \lambda_1).
$$
Here, $\lambda_0 \geq 1$ and $\lambda_1$ are constants depending only on $\dim X$. Taking limits gives $$\lim_{t\to\infty}\frac{\dist_X(\alpha(t),\beta(t))}{2t}\geq \frac{1}{\lambda_0^2}.$$  

By~\cite[Proposition 
II.9.8.(4)]{BridsonHaefliger}, we thus have  $\sin(\angle_T(a,b)/2)\geq1/\lambda_0^2$.  Now,
$$\dist_T(a,b) \geq  \angle_T(a,b)  \geq \sin^{-1}(1/\lambda_0^2) >1/\lambda_0^2.$$  
Setting  $\epsilon_0 = 1/\lambda_0^2$  completes the proof.
\end{proof}

In the next lemma, we show that, given a collection of Roller classes, the intersection of the associated Tits boundary realizations 
coincides with the Tits boundary of the intersection of the convex hulls of the corresponding CAT(0) geodesic rays.

\begin{lemma}\label{Lem:QIntersection}
Let $v_1, \ldots, v_k \in \visi(X)$ be Roller classes. Then there exist convex subcomplexes 
 $\J_1, \ldots, \J_k \subset X$ such that the following holds. Each $\J_i$ is the cubical convex hull 
of a CAT(0) geodesic ray $\gamma_i$, with a common basepoint.
 Furthermore,
for every subset $J \subset \{1, \ldots, k\}$, we have $\bigcap_{i \in J} Q(v_i) = \tits \! \left(\bigcap_{i\in 
J}\J_i\right)$.
\end{lemma}

\begin{proof}
Fix a basepoint $x\in X^{(0)}$.  For $1\leq i\leq k$, let $\gate_{v_i} \! \from X\to v_i$ be the gate map of Proposition~\ref{Prop:Proj}, and let $y_i= \gate_{v_i} (x)$. 
Let $\J_i=\J(x,y_i)$ be the 
convex subcomplex of $X$ determined by the vertex 
interval between $x$ and $y_i$ as in Definition~\ref{Def:Q'y}. Then $\W(\J_i) = \W(x, y_i) = \W(x,v_i)$, where the last equality holds by Proposition~\ref{Prop:Proj}.
 By Lemma~\ref{Lem:Qy_defs}, $Q(v_i)=\tits \J_i$.

For each $i$, let $\gamma_i$ be a CAT(0) geodesic ray that emanates from $x$ and represents 
$\varphi(v_i)$. 
We want to show that $\J_i = \hull(\gamma_i)$, the cubical convex hull of $\gamma_i$.

First, we claim that $ \hull(\gamma_i) \subset \J_i$. Since $\J_i$ is convex in the CAT(0) metric $d_X$, it contains a CAT(0) geodesic from every point in the interior to every point on the boundary. In particular, it contains the unique CAT(0) geodesic $\gamma_i$ from $\gamma_i(0) = x\in \J_i$ to 
$\gamma_i (\infty)  \in Q(v_i) = \tits \J_i$. Thus $\gamma_i \subset \J_i$. Since $\J_i$ is cubically convex, it follows that $\hull(\gamma_i) \subset \J_i$.

Next, we claim that $ \J_i \subset \hull(\gamma_i)$, or equivalently $\W(\J_i) \subset \W(\gamma_i)$.  Since $\gamma_i(0)=x \in \J_i$, any hyperplane $\hat h \in \W(\mathcal 
J_i) \setminus \mathcal 
W(\gamma_i)$ would separate the entire ray $\gamma_i$ from 
the entire class $v_i$, contradicting that $\psi(\gamma_i(\infty))=\psi(\varphi(v_i))=v_i$, where the second equality  holds by visibility and Lemma~\ref{Lem:L2Visibility}.  Hence $\W(\J_i) \subset \W(\gamma_i)$, 
and we conclude that $\J_i = \hull(\gamma_i)$.

Finally, consider an arbitrary subset  $J \subset \{1, \ldots, k\}$. The inclusion $\tits \! 
\left(\bigcap_{i\in 
J}\J_i\right) \subset \bigcap_{i\in I}Q(v_i)$ is immediate. For the other inclusion, 
let 
$b \in \bigcap_{i\in J}Q(v_i)$. Then each convex set $\J_i$ contains the unique CAT(0) ray $\beta$ starting at 
$x$ and ending at $b$.
Thus $\beta \subset \bigcap_{i\in J} \J_i$, hence $b \in \tits \! \left(\bigcap_{i\in 
J}\J_i\right)$.
\end{proof}

\begin{remark}\label{Rem:L1Intersection}
The following related statement about $\ell^1$ geodesics can be proved by an easier analogue of the proof of Lemma~\ref{Lem:QIntersection}.
Let $v_1, \ldots, v_k$ be arbitrary Roller classes. Then there exist convex subcomplexes  $\J_1, \ldots, \mathcal 
J_k$ (defined in exactly the same way, namely  $\J_i=\J(x,  \gate_{v_i} (x)$) such that each  $\W(\J_i)$ is a UBS representing 
$v_i$. Furthermore, each $\J_i$ is the cubical 
convex hull of a combinatorial geodesic, with a common basepoint. The intersection $\bigcap_i \J_i$ itself has the 
property that $\W( \bigcap_i \J_i) = \bigcap_i \W( \J_i)$. In particular, if 
$\W( \bigcap_i \J_i)$ is infinite, then it is an $\ell^1$--visible UBS.
\end{remark}

The next lemma is crucial.  We will later replace the closed covering of $\tits X$ by Tits boundary realizations with an open covering, by 
slightly thickening each Tits realization.  The lemma says, roughly, that for a collection of Tits boundary realizations corresponding to a 
simplex in the nerve of the closed covering, any point in the intersection of small $\epsilon$--neighborhoods of the Tits boundary realizations is still in 
a $f(\epsilon)$--neighborhood of the intersection of the Tits boundary realizations.

\begin{lemma}\label{Lem:close_to_intersection}
There exists a constant $\epsilon_1>0$, depending on $\dimension(X)$, and a function $f \from (0,\epsilon_1)\to(0,\pi/2)$ such that $\lim_{\epsilon \to 0} f(\epsilon) = 0$ and 
such that the following holds.  Fix $\epsilon\in(0,\epsilon_1)$ and let $v_0,\ldots,v_k\in\visi(X)$ satisfy 
$\bigcap_{i=0}^kQ(v_i)\neq\emptyset$.  For every $a\in\tits X$ such that $\dist_T(a,Q(v_i)) \leq \epsilon$ for all 
$i$, there is a point $b \in \bigcap_{i=0}^kQ(v_i)$ such that
$\dist_T(a, b) \leq f(\epsilon)$.
\end{lemma}

\begin{proof}
We begin by defining the constant $\epsilon_1$ and the function $f \from (0,\epsilon_1)\to(0,\pi/2)$. Let 
$\lambda_0, \lambda_1$ be the constants of Lemma~\ref{Lem:WallQI}, which depend only on $\dimension(X)$. Define $K = 4 \lambda_0^2 \dimension (X)$. Then, we set 
$$
\epsilon_1 = \frac{4}{5K}, \qquad
f(\epsilon)  = \cos^{-1} \big(1 - K \epsilon \big) \leq \cos^{-1} (\tfrac{1}{5}) < \tfrac{\pi}{2}.$$
The bound $f(\epsilon) \leq \cos^{-1}(\frac{1}{5})$ holds because $K \epsilon < K \epsilon_1 = \frac{4}{5}$. The property $\lim_{\epsilon \to 0} f(\epsilon) = 0$ is now immediate.

\smallskip
\textbf{Plan of the proof.}
Let $v_0,\ldots,v_k\in\visi(X)$ be Roller classes such that 
$\bigcap_{i=0}^kQ(v_i)\neq\emptyset$.
By Lemma~\ref{Lem:QIntersection}, there exist convex subcomplexes $\J_0,\ldots,\J_k$ such that $\bigcap_{i = 0}^k Q(v_i)=\tits \! \left(\bigcap_{i= 0}^k \J_i\right)$. Choose  a basepoint $x \in \bigcap_{i= 0}^k \J_i$. 

Fix $\epsilon < \epsilon_1$, and let $a\in\tits X$ be a point such that $\dist_T(a,Q(v_i)) \leq \epsilon$ for all $i$. Let $\alpha\from[0,\infty)\to X$ be a CAT(0) geodesic ray, such that $\alpha(0) = x$ and $\alpha(\infty) = a$. To prove the lemma, we will find a point $b\in\bigcap_{i=0}^kQ(v_i)$ such that $d_T(a, b) \leq  f(\epsilon)$.

We locate the point $b$ using gate projections. Let $\gate \from X \to \bigcap_{i=1}^k \J_i$ be the gate map.
 For $t \geq 0$, let $x_t = \gate(\alpha(t))$. Since $\bigcap_{i=0}^k\J_i$ is a finite-dimensional CAT(0) 
cube complex with no facing triple of hyperplanes, it is  proper by 
Lemma~\ref{Lem:proper}. We will check that the sequence $\{x_t\}_{t\in\naturals}$ is unbounded.
  Thus the segments joining $x_0$ to the elements of $\{x_t\}_{t\in\naturals}$ 
subconverge uniformly on compact sets to a CAT(0) geodesic ray representing some point $b\in\bigcap_{i=0}^k Q(v_i)$.

To prove that $d_T(a, b) \leq f(\epsilon)$, we will estimate the Alexandrov angle at $x = \alpha(0)$ between $\alpha(t)$ and $x_t$. In order to do this, we will control the number of hyperplanes that separate $\alpha(t)$ from $\bigcap_{i=0}^k \J_i$. 

\smallskip
\textbf{Main hyperplane estimate.} For each $t \geq 0$, let $\U(t)$ be the set of hyperplanes that separate $\alpha(t)$ from 
$\bigcap_{i=0}^k\J_i$. We will prove the following double-sided estimate for all large $t$:
\begin{equation}\label{Eqn:UtCount}
\frac{1}{\lambda_0} \dist_X(\alpha(t),x_t) - \lambda_1 \leq| \U(t)|\leq  \dimension(X)  (\lambda_0 \cdot 2 t \epsilon + \lambda_1) + D,
\end{equation}
where $D$ is a constant independent of $t$.

The lower bound of \eqref{Eqn:UtCount} is straightforward.  By Lemma~\ref{Lem:Gate_Lipschitz}, the gate map $\gate  \from X\to\bigcap_{i = 0}^k \J_i$  is characterized by the property that a 
hyperplane $\hat h$ separates 
    $\alpha(t)$ from $x_t = \gate(\alpha(t))$ if and only if $\hat h$ separates $\alpha(t)$ from $\bigcap_{i=0}^k\J_i$.  
   In other words, $\U(t) = \W(\alpha(t), x_t)$.
     Hence Lemma~\ref{Lem:WallQI} gives
$$\frac{1}{\lambda_0} d_X(\alpha(t),x_t) - \lambda_1 \leq |\U(t)| ,$$
as desired. By contrast, the upper bound requires some hyperplane combinatorics.

\smallskip
\textbf{Hyperplane sets.} 
For each $i$, let $\mathcal B_i(t)$ be the set of hyperplanes that separate $\alpha(t)$ 
from $\J_i$.  Note that $\U(t)\supset\bigcup_{i=0}^k\mathcal B_i(t)$. We claim that $\U(t)=\bigcup_{i=0}^k\mathcal B_i(t)$.

Suppose that $\hat h\in\U(t) \setminus \bigcup_{i=0}^k\mathcal B_i(t)$.  Then $\hat h$ crosses each $\J_i$, so by a standard cubical 
convexity argument, $\hat h$ crosses $\bigcap_{i=0}^k\J_i$ and hence cannot separate any point from $\bigcap_{i=0}^k\J_i$.  This is 
a contradiction, and thus $\U(t)=\bigcup_{i=0}^k\mathcal B_i(t)$. 

Let $\mathcal A_0(t)=\mathcal B_0(t)$. For $i \geq 1$, let $\mathcal A_i(t)=\mathcal B_i(t) \setminus \bigcup_{j < i} \mathcal A_{j}(t)$.  Then the 
sets $\mathcal A_0(t),\ldots,\mathcal A_k(t)$ are pairwise disjoint, and their union is $\U(t)$.  For each $i$, let $\mathcal A_i = \bigcup_{t \geq 0} \mathcal A_i(t)$.
Observe that every $\mathcal A_i$ is inseparable, by definition. 
Furthermore, $\mathcal A_i$ is unidirectional. Indeed, $\mathcal A_i \subset \mathcal B_i =  \bigcup_{t \geq 0} \mathcal B_i(t)$, and $\mathcal B_i$ separates $\J_i$ from the tail of the ray 
$\alpha$. Unidirectionality passes to subsets, hence $\mathcal A_i$ is unidirectional.
 Consequently $\mathcal A_i$ is either finite or a UBS.

Let $J\subset\{0,\ldots,k\}$ be the set of $i$ such that $\mathcal A_i$ is a UBS. Observe that $\mathcal A_i  \cap \mathcal A_j = \emptyset$ for $i \neq j$ and $\bigsqcup_{i \in J} \mathcal A_i \subset \W(\alpha)$, which is itself a UBS. Thus, by Lemma~\ref{Lem:FinManyMinimals}, we have $|J| \leq \dimension(X)$.
 We define $D = \sum_{i \notin J} | \mathcal A_i |$. 

\smallskip
\textbf{The upper bound of \eqref{Eqn:UtCount}.} 
Next, we will use the decomposition $\U(t) = \bigsqcup_{i=1}^k \mathcal A_i(t)$ to prove the upper bound  of \eqref{Eqn:UtCount}.

For each $i \in \{0, \ldots, k\}$, let $\xi_i$ be a ray in $\J_i$ issuing from $x = \alpha(0)$, such that $\xi_i(\infty)$ is a closest point in 
$Q(v_i) = \tits \J_i $ to $a$.  
Now, for all $i$ and all $t$, any hyperplane $\hat h \in\mathcal A_i(t)$ separates $\alpha(t)$ from $\J_i$ and hence separates $\alpha(t)$ from 
$\xi_i(t)$.  Thus Lemma~\ref{Lem:WallQI} gives
$$|\mathcal A_i(t)| \leq |\W(\alpha(t),\xi_i(t))| \leq \lambda_0 \cdot d_X (\alpha(t), \xi_i(t)) + \lambda_1.$$ 

Now, by \cite[Proposition~II.9.8.(4)]{BridsonHaefliger}, for any $\delta>0$ there exists $t_0$ such that for all $t>t_0$, 
we have $\dist_X(\alpha(t),\xi_i(t))\leq t(2\sin\frac{\epsilon}{2}+\delta)$.  Since $\epsilon$ is independent of $t$, it follows that 
for all sufficiently large $t$,  we have 
$$\dist_X(\alpha(t),\xi_i(t))\leq 4t\sin \tfrac{\epsilon}{2} < 2 t \epsilon.$$

Summing $|\mathcal A_i(t)|$ over all $i$  gives
\begin{align*}
| \U(t) | &= \sum_{i\in J}|\mathcal A_i(t)| + \sum_{i\notin J}|\mathcal A_i(t)| \\
& \leq \sum_{i \in J} \Big( \lambda_0 \cdot d_X (\alpha(t), \xi_i(t)) + \lambda_1 \Big) + \sum_{i \notin J} | \mathcal A_i | \\
&\leq  |J| (\lambda_0 \cdot 2 t \epsilon + \lambda_1) + D,
\end{align*}
which proves \eqref{Eqn:UtCount} because $|J| \leq \dimension(X)$.

\smallskip
\textbf{Angle estimate.}
Let $x=\alpha(0)$.  Set $A = \dist_X(\alpha(0), \alpha(t))=t$ (latter equality since $\alpha$ is a geodesic), and $B = \dist_X(\alpha(0), x_t)$, and $C = \dist_X(\alpha(t), x_t)$.  Let $\bar\theta_t$ 
be the Euclidean angle at the vertex corresponding to $x$ in the Euclidean comparison triangle with sides of length $A,B,C$, i.e. the comparison angle between $\alpha(t)$ and $x_t$ at $x$.  Let 
$\theta_t$ be the Alexandrov angle at $x = \alpha(0)$ between $x_t$ and $\alpha(t)$, so that $\theta_t\leq\bar\theta_t$.   

We can estimate $\bar\theta_t$ using the law of cosines, starting with an estimate of $C$. Multiplying every term of \eqref{Eqn:UtCount} by $\lambda_0$ and choosing $t$ sufficiently large yields
\begin{equation}\label{Eqn:CEst}
C = d_X(\alpha(t), x_t) \leq 4 \lambda_0^2 \dimension(X) \cdot \epsilon t = K \epsilon t.
\end{equation}
Meanwhile, $B$ can be estimated as follows:
\begin{equation}\label{Eqn:BEst}
 t(1-K \epsilon) 
 \leq \dist_X(\alpha(0),\alpha(t))-\dist_X(\alpha(t),x_t))
\leq \dist_X(\alpha(0),x_t) = B \leq t.
\end{equation}
Here, the first inequality is by \eqref{Eqn:CEst}, the second inequality is the triangle inequality, and the final inequality 
holds because the projection $\gate$ is $1$--Lipschitz for the CAT(0) metric by Lemma~\ref{Lem:Gate_Lipschitz}.

By  the cosine law and the CAT(0) inequality, we have

$$ 2 AB \cos(\theta_t) \geq 2AB \cos(\bar\theta_t)=A^2 + B^2 - C^2.$$
The equality is the cosine law applied to the Euclidean comparison triangle, and the inequality holds for the corresponding CAT(0) triangle.  We are interested in the former, since we will need a 
bound on $\bar\theta_t$ (which will incidentally bound $\theta_t$).
Substituting the bounds \eqref{Eqn:CEst} and \eqref{Eqn:BEst} into the cosine law shows that for all large $t$, we have
$$2t^2\cos\bar\theta_t  \geq  t^2+t^2(1-K \epsilon)^2- t^2 (K\epsilon)^2,$$
which simplifies to
$\cos\bar\theta_t \geq 
1-K \epsilon$. Since we have chosen $\epsilon \in (0, \tfrac{4}{5K} )$, this ensures that $\bar\theta_t$ is a small angle; more precisely $\bar\theta_t \leq \cos^{-1}( 1- K \epsilon)$.

\smallskip
\textbf{Conclusion.} 
Since $K \epsilon < \frac{4}{5}$, equation~\eqref{Eqn:BEst} implies 
$d_X(\alpha(0), x_t) \geq (1-K \epsilon) t > \frac{1}{5} t.$
Thus the sequence $\{x_t\}_{t\in\naturals}$ is unbounded,
and some subsequence converges to a point $b\in\bigcap_{i=0}^kQ(v_i)$, as mentioned above.
By~\cite[Lemma II.9.16]{BridsonHaefliger}, which relates the Tits angle $\angle(a,b)$ to the Euclidean comparison angles $\bar\theta_t$, we have 
$$
\angle(a,b)\leq\lim\inf_{t\to\infty}\bar\theta_t \leq \cos^{-1} (1 - K \epsilon) = f(\epsilon).
$$
Since we have already checked that $f(\epsilon) < \frac{\pi}{2}$, this ensures  
 $\angle(a,b) \leq f(\epsilon)<\pi/2$.  In particular, $\dist_T(a,b)=\angle(a,b) \leq f(\epsilon)$.
\end{proof}

Recall from Proposition~\ref{Prop:diameter}.\eqref{item:Q_radius},\eqref{item:Q_center} that for each $v\in\visi(X)$, there is a radius $r_v<\pi/2$ such that $Q(v)$ is contained in the closed ball 
$\~B_T(\circum(v), r_v)$ in the Tits metric.

\begin{definition}[Thickening constant $\epsilon(w)$]\label{Def:EpsilonW}
Let $\epsilon_0$ be the constant of Lemma~\ref{Lem:disjoint_Q}. Let $f \from (0, \epsilon_1) \to \reals$ be the function from Lemma~\ref{Lem:close_to_intersection}.  Since $\lim_{\epsilon \to 0} f(\epsilon) = 0$, that lemma allows us to 
choose a constant $\epsilon_2 < \epsilon_1$, depending  only on $\dimension(X)$, so that $f(\epsilon)+\epsilon \leq \epsilon_0$ whenever 
$\epsilon \leq \epsilon_2$.  
For each Roller class $w$, define a \emph{thickening constant} 
$\epsilon(w) = \min\{\epsilon_0/2 ,\epsilon_2, r_w/4\} > 0$.

Since the constants $\epsilon_0, \epsilon_2$ depend only on $\dimension(X)$, and the intrinsic radius $r_w$ depends only on the $\Aut(X)$--orbit of $w$ by Corollary~\ref{Cor:Q_equivariant},
it follows that $\epsilon(w)$ itself depends only on the the $\Aut(X)$--orbit of $w$.
\end{definition}

\begin{prop}[Open neighborhood $U(v)$]\label{Prop:ThickNeighborhoodU}
For every $v\in\visi(X)$, there exists a subset $U(v)\subset\tits X$ such that the following 
hold:
\begin{itemize}
 \item $Q(v)\subset U(v)\subset\mathcal N_{\epsilon(v)}(Q(v))$;\label{item:small}
 \item $U(v)$ has diameter less than $3\pi/4$; 
 \item $U(v)$ is open;\label{item:open}
 \item $U(v)$ is convex and contractible;
 \item $U(gv) = gU(v)$ for all $g \in \Aut(X)$.
\end{itemize}
\end{prop}

\begin{proof}
Let $Y$ be the Euclidean cone on $\tits X$, with cone-point denoted $0$.  
By~\cite[Theorem II.3.14]{BridsonHaefliger}, the usual cone metric $\dist_Y$ on $Y$ has the 
property that $(Y,\dist_Y)$ is a CAT(0) geodesic space.  We use the notation $\nu_{\delta}(A)$ to denote the open $\delta$--neighborhood of a set $A \subset Y$ in this CAT(0) metric.

We identify $\tits X$ 
with the unit sphere about $0$ in $Y$. We also set $Q = Q(v)$ for simplicity of notation. Recall from Proposition~\ref{Prop:diameter}.\eqref{item:Q_radius} that the intrinsic radius of $Q$ is $r_v 
\in [0, \pi/2)$.

\smallskip
\textbf{The convex sub-cone $Z = Z(v)$:}  Let $Z\subset Y$ be the subspace arising as the union of all rays in $Y$ 
emanating from $0$ and passing through $Q\subset\tits X$.  We claim that 
$Z$ is convex in $(Y,\dist_Y)$.  Indeed, let $x,y\in Z$. Then $x,y$ 
respectively lie on rays $\alpha_x,\alpha_y$ emanating from $0$ and intersecting $Q$ 
in points $x',y'$.  By \cite[Proposition I.5.10(1)]{BridsonHaefliger}, the rays $\alpha_x,\alpha_y$ determine a convex sector $S \subset Y$ 
intersecting $\tits X$ in the unique Tits geodesic from $x'$ to $y'$, which lies in 
$Q$ by Tits convexity of $Q$.  Hence $S$ lies in $Z$, and thus so does 
the $\dist_Y$--geodesic from $x$ to $y$, since $x,y\in S$.

\smallskip
\textbf{The horoball $B = B(v)$:}  Let  $c=\circum(v)$ be the circumcenter of $Q = Q(v)$. Let $\alpha =\alpha_c$ be the parametrized geodesic ray in $Y$ such that $\alpha(0)=0$ and $\alpha(1)=c$.  Let $p \from Y\to\reals$ be the associated 
horofunction, defined by 
$$ p(y)=\lim_{s\to\infty}[\dist_Y(\alpha(s),y)-s] . $$ 
Let 
$t_0=\cos r_v$, which is positive since $r_v\in[0,\pi/2)$.

Let $B = B(v)=p^{-1}((-\infty,-t_0])$.  Recall from e.g.~\cite[Lemma 
3.88]{DrutuKapovich} that $B$ is a convex 
subset of $Y$.  Let $K(v)=B\cap Z$. See Figure~\ref{fig:U-prime}.

\smallskip
\textbf{$Q$ lies in $K(v)$:}  We claim that $Q\subset K(v)$.  Since $Q$ is the unit sphere in $Z$ by the definition of $Y$ and $Z$, 
it suffices to show that $Q\subset B$.  To that end, fix $q\in Q$, and observe that 
$\dist_Y(q,0)=1$.  Let $s\gg 0$ and consider the convex Euclidean triangle in $Y$ with vertices at $0,q,\alpha(s)$.  The angle in this triangle at $0$ is denoted 
$\theta$, and we note that $\theta\leq r_v$.  From the cosine law, we obtain 
$\dist_Y(q,\alpha(s))^2=s^2+1-2s\cos\theta$.  Hence 
$$p(q)=\lim_{s\to\infty}\left[\sqrt{s^2+1-2s\cos\theta}-s\right]=-\cos\theta \leq - \cos r_v = - t_0.$$
Thus, by the definition of $B$, we have $q\in B$, as 
required.  Since $Z\cap\tits X=Q$, we have in fact shown that $K(v)\cap\tits 
X=Q$. 

\smallskip
\textbf{$K(v)$ satisfies $d(0,K(v)) = t_0$:}  We claim that $K(v)$ is disjoint from the open $t_0$--ball in $Y$ about 
$0$.  Indeed, let $b\in K(v)$.  Then $p(b)\leq -t_0$.  Hence, for any fixed 
$\delta>0$, and any sufficiently large $s>0$, we have 
$\dist_Y(b,\alpha(s))<s-t_0+\delta$.  By the triangle inequality, $s\leq 
\dist_Y(0,b)+s-t_0+\delta$, i.e. $\dist_Y(0,b)>t_0-\delta$.  Hence 
$\dist_Y(b,0)\geq t_0$, as required.  At the same time, $\alpha = \alpha_c$ is a parametrized geodesic ray and $\alpha(t_0)$ is a point of $K(v)$ at distance exactly $t_0$ from $0$, hence $d(0,K_v) = t_0$.

\smallskip
\textbf{The projection $\varpi$ and the open set $U'(v)$:}  Let $\varpi:Y \setminus \{0\}\to \tits X$ be the radial projection. Recall that $\nu_\rho(0)$ is the open $\rho$--neighborhood of $0$. 
We claim that for every $0 < \rho \leq 1$, the restriction of $\varpi$ to  $Y \setminus \nu_\rho(0)$ is $1/\rho$--Lipschitz. Indeed,  by \cite[Proposition I.5.10(1)]{BridsonHaefliger}, it suffices to check this claim on a sector of $\reals^2$. In polar coordinates on $\reals^2$, the distance element $ds$ satisfies $ds^2 = dr^2 + r^2 d\theta^2$, hence the projection $\varpi$ rescales $ds$ by a factor of at most $r^{-1}$, hence $\varpi$ is $1/\rho$--Lipschitz when $r \geq \rho$. 

Recall the constant $\epsilon(v)$ defined in Definition~\ref{Def:EpsilonW}, and define a positive constant 
\begin{equation}\label{Eqn:EpsilonPrime}
\epsilon' = \epsilon'(v) = \frac{\epsilon(v) }{1 + \epsilon(v)} t_0 \in (0, t_0),
\quad \text{which satisfies} \quad
\frac{\epsilon'}{t_0 - \epsilon'} = \epsilon(v).
\end{equation}
Let $U'(v) = \nu_{\epsilon'} (K(v))$ be the open $\epsilon'$--neighborhood 
of $K(v)$ in $Y$.  Since $K(v) = B \cap Z$ is convex and $Y$ is CAT(0), it follows that $U'(v)$ is convex. Furthermore, $U'(v)$ is open and disjoint from the ball of radius $t_0 - \epsilon'$ about 
$0$.  See Figure~\ref{fig:U-prime}.

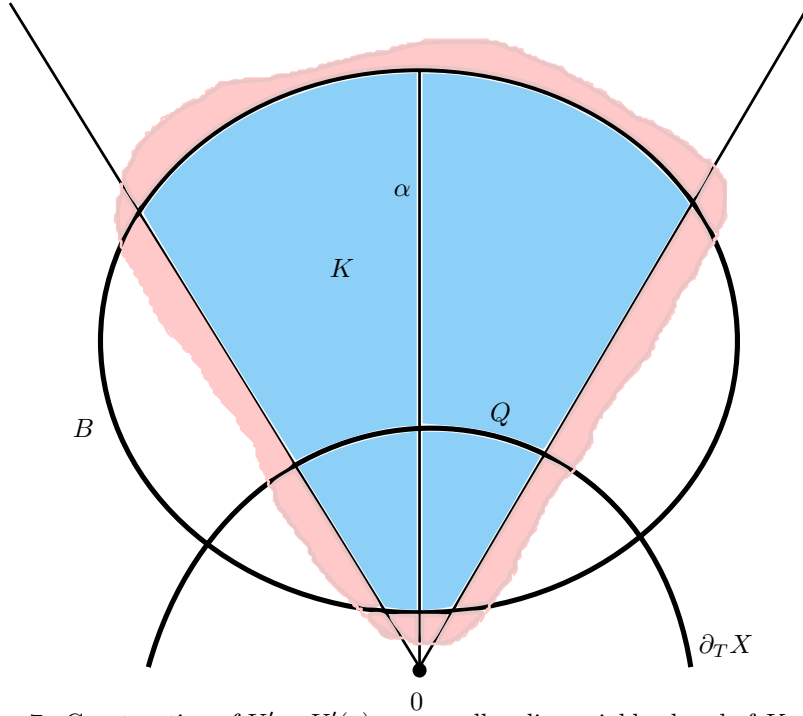
\begin{figure}[h]
\begin{overpic}[width=0.7\textwidth]{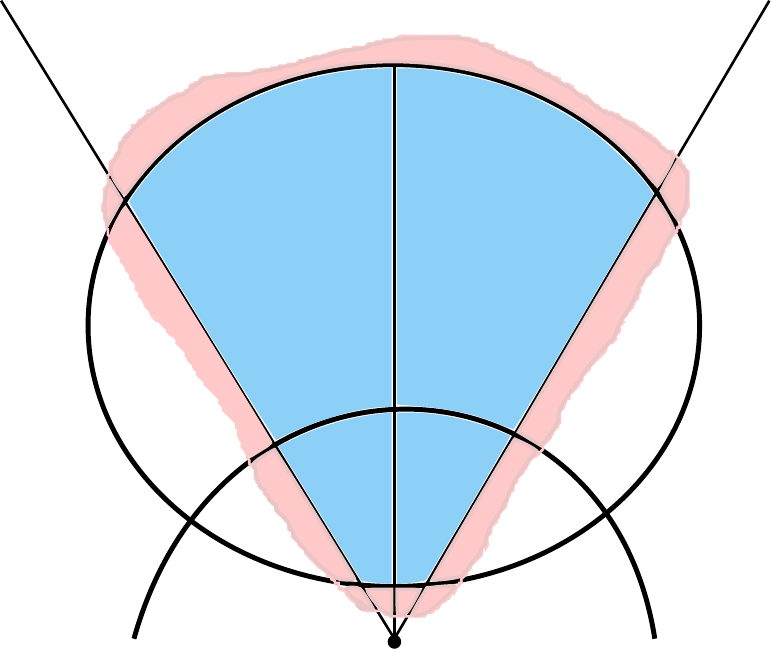}
 \put(50,-4){$0$}
 \put(48, 60){$\alpha$}
    \put(86,3){$\tits X$}
    \put(8,30){$B$}
    \put(60,32){$Q$}
    \put(40,50){$K$}
 \end{overpic}
\caption{Construction of $U'=U'(v)$ as a small-radius neighborhood of $K=B\cap Z$, where $Z\subset Y$ is the cone on $Q\subset\tits X$ and $B$ is an appropriately chosen horoball determined by 
$\alpha$.  The radius is chosen so that $U'$ avoids the $(t_0-\epsilon')$--ball around $0$.}
\label{fig:U-prime}
\end{figure}

\smallskip
\textbf{Definition and properties of $U(v)$:}
Let $U(v)=\varpi(U'(v))\subset \tits X$. We claim that this has all of the properties claimed in the lemma statement. First, since $Q(v) \subset K(v)$ by construction, and $\varpi$ is the identity map on $\tits X \subset Y$, it follows that $Q(v) = \varpi(K(v))\subset U(v)$.
On the other hand, since $d(0, U'(v)) \geq t_0 - \epsilon'$,  the projection $\varpi$ is $1/(t_0 - \epsilon')$--Lipschitz on $U'(v)$. Therefore, \eqref{Eqn:EpsilonPrime} implies
$$
U(v) = \varpi(\nu_{\epsilon'} (K(v)) ) \subset \nu_{\epsilon(v)} (\varpi(K(v))) =  \nu_{\epsilon(v)} (Q(v)),
$$
proving the first bullet of the lemma.

Second, recall from Definition~\ref{Def:EpsilonW} that $\epsilon(v) \leq r_v / 4$. Thus Proposition~\ref{Prop:diameter}.\eqref{item:Q_radius} implies
$$
\diameter(U(v)) \leq \diameter(Q(v)) + 2\epsilon(v) \leq \frac{\pi}{2} + \frac{r_v}{2} < \frac{\pi}{2} + \frac{\pi}{4} = \frac{3\pi}{4}.
$$

Third, observe that $\varpi$ is an open map.  Indeed, 
$Y \setminus \{0\}$ is homeomorphic to $\tits X\times(0,\infty)$, and $\varpi$ is 
projection to the first factor, which is an open map. Thus $U(v) = \varpi(U'(v))$ is open.

Fourth, it follows from \cite[Proposition I.5.10.(1)]{BridsonHaefliger} and the convexity of $U'(v)$  that $U(v)$ is convex in $\tits X$. Since $U(v)$ has diameter less than $\pi$, it is uniquely geodesic, hence contractible.

Fifth, observe that all of the ingredients in the definition of $U(v)$ are $\Aut(X)$--equivariant. Indeed, the intrinsic radius $r_v$ is invariant under $\Aut(X)$, hence the number $t_0 = \cos(r_v)$ is as well. Thus the definitions of the sub-cone $Z(v)$, the horoball $B(v)$, and the set $K(v) = Z(v) \cap B(v)$ are all $\Aut(X)$--equivariant. The constant $\epsilon'(v)$, defined in \eqref{Eqn:EpsilonPrime} is invariant under $\Aut(X)$, because both $t_0$ and $\epsilon(v)$ are. It follows that $U(v) = \varpi(\nu_{\epsilon'} (K(v)) )$ is also $\Aut(X)$--equivariant.
\end{proof}

Recall that in Section~\ref{Sec:closed_cover_tits_boundary}, we have defined a simplicial complex $\mathcal N_T$.
By Corollary~\ref{Cor:MtNerve}, $\mathcal{N}_T$ coincides with the nerve of the cover of $\tits X$ by the closed sets $Q(v)$  for $v \in \maxvis (X)$. By Lemmas~\ref{Lem:max_vis} and~\ref{Prop:ThickNeighborhoodU},  the open sets $U(v)$ for $v \in \maxvis (X)$ also cover $\tits X$. 

By analogy with Definition~\ref{Def:Big_N_T}, we define a simplicial complex $\mathcal C_T$ with vertex set $\maxvis X$, where vertices $v_0, \ldots, v_n \in \maxvis X$ span an $n$--simplex if and only if $\bigcap_{i=0}^n U(v_i)\neq\emptyset$.
The content of the following lemma is that $\mathcal C_T$ is isomorphic to $\mathcal N_T$.

\begin{lemma}\label{Lem:flag-ish}
For Roller classes $v_0, \ldots,v_n \in \maxvis(X)$, we have $\bigcap_{i=0}^n U(v_i)\neq\emptyset$ if and only if 
$\bigcap_{i=0}^n Q(v_i)\neq\emptyset$. Consequently, the identity map $\maxvis(X) \to \maxvis(X)$ induces an $\Aut(X)$--equivariant simplicial isomorphism $\mathcal C_T  \to \mathcal N_T$.
\end{lemma}

\begin{proof}
Observe that $\bigcap_{i=0}^n Q(v_i)\neq\emptyset$ trivially implies $\bigcap_{i=0}^n U(v_i)\neq\emptyset$. 
For the other direction, we assume $\bigcap_{i=0}^n U(v_i) \neq \emptyset$ and aim to show $\bigcap_{i=0}^n Q(v_i)\neq\emptyset$. 

We will argue by induction on $n$. 
 In the base case, $n=1$, suppose that $U(v_0) \cap U(v_1) \neq \emptyset$. Since $U(v_i) \subset \mathcal{N}_{\epsilon(v_i)}(Q(v_i))$, and $\epsilon(v_i) \leq \epsilon_0/2$ by Definition~\ref{Def:EpsilonW}, it follows that there are points $a_i \in Q(v_i)$ such that $d_T(a_0, a_1) < \epsilon_0$. But then Lemma~\ref{Lem:disjoint_Q} implies $Q(v_0) \cap Q(v_1) \neq \emptyset$, proving the base case.

For the inductive step, assume that $\bigcap_{i=0}^n U(v_i) \neq \emptyset$ and
$\bigcap_{i=0}^{n-1}Q(v_i)\neq\emptyset$. We claim that there is a Roller class $w \in \visi(X)$ such that $\bigcap_{i=0}^{n-1}Q(v_i)=Q(w)$.

By Lemma~\ref{Lem:QIntersection}, we have convex subcomplexes $\J_0, \ldots, \J_{n-1} \subset X$, such that $Q(v_i) = \tits \J_i$ 
and $\bigcap_{i=0}^{n-1}Q(v_i) = \tits \bigcap_{i=0}^{n-1}\J_i$. By Lemma~\ref{standardUBS}, the hyperplane 
collection $\W(\J_i)$ is a UBS. Furthermore, since each $\J_i$ is convex, we have $\W(\bigcap \J_i) 
= \bigcap \W(\J_i)$, which is a UBS because $\bigcap \J_i$ is unbounded. By Lemma~\ref{Lem:UmbraClass}, the 
UBS $\W(\bigcap \J_i)$ represents a Roller class $w'$, which means that  
$
Q(w') =  \tits \bigcap_{i=0}^{n-1}\J_i = \bigcap_{i=0}^{n-1}Q(v_i)$. Now, we let $w = \psi(\varphi(w'))$. Then 
Definition~\ref{Def:L2visible} says that $w \in \visi(X)$, and Proposition~\ref{prop:phi_facts}.\ref{item:same_Q} says that 
$Q(w) = Q(w') = \bigcap_{i=0}^{n-1}Q(v_i)$, proving the Claim.

Next, we claim that $Q(v_n)\cap Q(w) \neq \emptyset$. Since $Q(w) = \bigcap_{i=0}^{n-1}Q(v_i)$, this will complete the inductive step and prove the Lemma.

Suppose for a contradiction that $Q(v_n)\cap Q(w) = \emptyset$. Then, by Lemma~\ref{Lem:disjoint_Q}, we have
$\dist_T(Q(v_n),Q(w))>\epsilon_0$.  Let $\epsilon_2$ be as in Definition~\ref{Def:EpsilonW}.
Fix $a\in  \bigcap_{i=0}^n U(v_i) \subset \bigcap_{i=0}^n\mathcal N_{\epsilon(v_i)}(Q(v_i))$, which ensures that $\dist_T(a,Q(v_n))<\epsilon_2$.  Since $\epsilon(v_i) \leq \epsilon_2$ for each $v_i$,  
Lemma~\ref{Lem:close_to_intersection} implies $\dist_T(a,Q(w))<f(\epsilon_2)$.  
Thus $\dist_T(Q(v_n),Q(w))<\epsilon_2+f(\epsilon_2) \leq \epsilon_0$, a contradiction.  This proves the Claim and the equivalence $\bigcap_{i=0}^n Q(v_i)\neq\emptyset \Leftrightarrow \bigcap_{i=0}^n U(v_i)\neq\emptyset$.

This gives the simplicial isomorphism $\mathcal C_T \to \mathcal N_T$. This isomorphism is $\Aut(X)$--equivariant, by Corollary~\ref{Cor:Q_equivariant} 
and Proposition~\ref{Prop:ThickNeighborhoodU}.
\end{proof}

\subsection{Homotopy equivalence of $\tits X$ and $\mathcal N_T$}\label{Sec:first_homotopy}

\begin{theorem}\label{Thm:first_HE}
There is an \Authom homotopy equivalence from  the simplicial complex $\mathcal N_T$ to $\tits X$. \end{theorem}

\begin{proof}
By Corollary~\ref{Cor:MtNerve}, $\mathcal N_T$ is the nerve of the covering of $\tits X$ by the closed sets $Q(v)$, where $v$ varies over 
$\maxvis(X)$.

By Proposition~\ref{Prop:ThickNeighborhoodU}, $\tits X$ admits an open covering $\{ U(v) : v\in\maxvis(X) \}$, with each $U(v)$ a 
convex, uniquely geodesic subspace of $\tits X$.  Hence, if $v_0,\ldots,v_n\in\maxvis(X)$, then either 
$\bigcap_{i=0}^nU(v_i)=\emptyset$, or $\bigcap_{i=0}^nU(v_i)$ is contractible. Let $\mathcal L_T$ be the nerve of this open covering. By the Equivariant Open Nerve Theorem~\ref{Thm:EquivariantOpenNerve},
there is an \Authom homotopy equivalence $\mathcal L_T \to \tits X$.

As above, let $\mathcal C_T$ be the simplicial complex with one vertex for each $v\in\maxvis(X)$, with $v_0,\ldots,v_n$ spanning an 
$n$--simplex whenever $\bigcap_{i=0}^nU(v_i)\neq\emptyset$.   Note that $\mathcal C_T$ is not necessarily isomorphic to 
$\mathcal L_T$, since we may have $U(v_i)=U(v_j)$ for some pair $v_i,v_j$ with $i\neq j$.  But $\mathcal C_T$ is 
homotopy-equivalent to $\mathcal L_T$, which can be seen as follows.  The $\Aut(X)$--equivariant assignment 
$v\mapsto U(v)$ (i.e. the surjection $\mathcal C_T^{(0)}\to\mathcal L_T^{(0)}$) determines an equivariant simplicial map 
$\mathcal C_T \to\mathcal L_T$ such that the preimage of each $0$--simplex $U(v)$ is the full subcomplex of $\mathcal C_T$ 
spanned by the set of $w$ 
with 
$U(v)=U(w)$.  Since any finite set of such $w$ span a simplex of $\mathcal C_T$, this preimage is contractible: it is 
homeomorphic to the cone on the link of any of its vertices.  Now, let $\sigma$ be a simplex of $\mathcal C_T$.  Then by the 
definition of $\mathcal L_T$, the preimage of $\sigma$ is the join of the preimages of the vertices of $\sigma$, each of 
which we have just shown to be contractible, so the preimage of $\sigma$ is contractible.  Hence the map $\mathcal C_T 
\to\mathcal L_T$ is a homotopy equivalence, by Quillen's fiber theorem (see e.g.~\cite[Theorem 10.5]{Bjorner:book}).

To conclude, recall that by Lemma~\ref{Lem:flag-ish}, we have an equivariant simplicial isomorphism $\mathcal C_T \to \mathcal N_T$. Hence we have a chain of \Authom homotopy equivalences
$$
\mathcal N_T  \xrightarrow{ \: \cong \: }  \mathcal C_T \xrightarrow{ \: \sim \: }  \mathcal L_T \xrightarrow{\: \sim \:}  \tits X. \qedhere
$$
\end{proof}

\subsection{Cuboid generalization}\label{Sec:CuboidCover}

All of the results of this section have a straightforward generalization to the cuboid setting. In the following description, we refer to the notation introduced in Section~\ref{Sec:TitsCuboid}.

Definition~\ref{Def:MaxVis} generalizes immediately to give a set of maximal $\ell^2$ visible Roller classes, namely $\maxvis^\rho(X) \subset \visi^\rho(X)$. Following Definition~\ref{Def:Big_N_T}, we get a simplicial complex $\nerve_T^\rho$ with vertex set $\maxvis^\rho(X)$, where vertices $v_0, \ldots, v_n$ span an $n$--simplex if and only if $\bigcap_{i=1}^n Q^\rho(v_i) \neq \emptyset$. Then, Lemmas~\ref{Lem:max_vis} and~\ref{Lem:visible_injective} generalize immediately because their proofs are assembled from the results of Section~\ref{Sec:TitsRollerConnections}, and we have already checked that those results generalize to cuboids. Thus Corollary~\ref{Cor:MtNerve} generalizes as well, and we learn that $\nerve_T^\rho$ is the covering of $\tits^\rho X$ by the closed sets $\{Q^\rho(v) : v \in \maxvis^\rho(X) \}$.

(Recall that Remark~\ref{Rem:VisCuboid} outlined an argument that $\visi(X) = \visi^\rho(X)$ for every admissible rescaling 
$\rho$. It follows that $\maxvis(X) = \maxvis^\rho(X)$ as well.  The nerves $\nerve_T$ and $\nerve_T^\rho$ can then be 
shown to coincide, as mentioned in Remark~\ref{Rem:VisCuboid}. However, we do not use this.)

Moving ahead to Section~\ref{Sec:open_cover_tits_boundary}, the results generalize as follows. The existence statement of Lemma~\ref{Lem:disjoint_Q} extends, with the modification that the constant $\epsilon_0$ depends on $\dim(X)$ and $\rho$. This is because the proof of Lemma~\ref{Lem:disjoint_Q} uses the constants $\lambda_0, \lambda_1$ of Lemma~\ref{Lem:WallQI}, which needs to be replaced by the constants $\lambda_0^\rho, \lambda_1^\rho$ of Lemma~\ref{Lem:CuboidWallQI}. Lemma~\ref{Lem:QIntersection} extends to cuboids as well, because the proof of that lemma uses combinatorial arguments in combination with the $d_X^\rho$--convexity of cubical convex hulls.
Similarly, Lemma~\ref{Lem:close_to_intersection} extends: its proof works perfectly well in the cuboid metric $d_X^\rho$, once we substitute Lemma~\ref{Lem:CuboidWallQI} for Lemma~\ref{Lem:WallQI}. As above, the outputs of that lemma depend on both $\dim(X)$ and $\rho$. Combining these ingredients, we can follow Definition~\ref{Def:EpsilonW} to define a thickening constant $\epsilon^\rho(w) > 0$ for every Roller class $w$. This constant depends only $\rho$ and the $\Aut^\rho(X)$--orbit of $w$.

Proposition~\ref{Prop:ThickNeighborhoodU} extends verbatim to cuboids. Indeed, its proof is a CAT(0) argument that takes place in the Euclidean cone on $\tits X$, which works equally well in the Euclidean cone on $\tits^\rho X$. (The proof also uses Proposition~\ref{Prop:diameter}, which works for cuboids.) At the end of the lemma, we get a collection of open sets $U^\rho(v)$ for $v \in \visi^\rho(X)$, which is invariant under $\Aut^\rho(X)$.

By analogy with Definition~\ref{Def:Big_N_T}, we define a simplicial complex $\mathcal C_T^\rho$, where the simplices are defined by intersection patterns of the open sets $U(v_i)$ for $v_i \in \maxvis^\rho(X)$. Now, Lemma~\ref{Lem:flag-ish} extends verbatim to give an $\Aut^\rho(X)$--equivariant simplicial isomorphism $\mathcal C_T^\rho \to \nerve_T^\rho$. The proof of Theorem~\ref{Thm:first_HE} also extends verbatim, because all of its ingredients extend. Thus we obtain the following cuboid extension of 
Theorem~\ref{Thm:first_HE}:

\begin{theorem}\label{Thm:First_HE_Cuboid}
Let $\rho$ be a $G$--admissible hyperplane rescaling of $X$. Then there is a \Ghom homotopy homotopy equivalence $\nerve_T^\rho \to \tits^\rho X$. 
\end{theorem}

\section{Proof of the Main Theorems}\label{Sec:visible_part}

 In this section, we prove Theorem~\ref{Thm:HE-triangle}, which was stated in the Introduction. That is, we construct the following commutative diagram of \Authom homotopy equivalences:
\begin{diagram}
 \simp X &&\rTo^{\ST}  
 &     & \tits X  \\
&\rdTo_{\SR} & &\ruTo_{\RT}&\\
        && \Re_\triangle X&&
\end{diagram}

The homotopy equivalence $\SR \from \simp X \to \absimp X$ will be proved in Proposition~\ref{prop:simplicial_boundary}, while the homotopy equivalence $\RT \from \absimp X \to \tits X$ will be proved in Corollary~\ref{Cor:RollerSimpTitsHomotopy}. At the end of the section, we check that all of the arguments extend to cuboids, proving Theorem~\ref{Thm:MainCuboids}.

To prove Corollary~\ref{Cor:RollerSimpTitsHomotopy}, we first study a subcomplex of $\absimp X$.  

\begin{definition}[Visible subcomplexes of $\absimp  X$]\label{Def:VisiPart}
Let $\vispart X$ be the  subcomplex of $\absimp X$ consisting of all simplices corresponding to chains 
$v_0\leq\cdots\leq 
v_n$ such that $v_i\in\visi(X)$ for $0\leq i\leq n$. 

For each $v\in\visi(X)$, let $\Sigma_v$ be the connected subcomplex of $\vispart X$ consisting of all simplices 
corresponding to 
chains $v_0\leq\cdots\leq v_n$ for which:
\begin{itemize}
     \item each $v_i\in\visi(X)$;
     \item $v_n = v$.
\end{itemize}
\end{definition}

\subsection{Closed cover of $\vispart X$}\label{Sec:vispart_closed_cover}

Next, we show that the complexes $\Sigma_v$ yield a closed covering of $\vispart X$, and use this covering to build a nerve.

\begin{lemma}\label{Lem:visible_simplices_closed_cover}
The set $\{\Sigma_v:v\in\maxvis(X)\}$ is a closed covering of $\vispart X$.  Moreover, for all 
$v,v_0,\ldots,v_n\in\visi(X)$, we have the following:
\begin{enumerate}[$(1)$ ]
     \item \label{item:finite} $\Sigma_v$ is finite.
     \item \label{item:injective} $\Sigma_{v_i}=\Sigma_{v_j}$ only if $v_i=v_j$.
     \item \label{item:contractible}  $\bigcap_{i=0}^n\Sigma_{v_i}$ is either empty or contractible.
     \item \label{item:equi} For every $g \in \Aut(X)$, we have $g \Sigma_v = \Sigma_{gv}$.
\end{enumerate}
\end{lemma}

\begin{proof}
Let $w \in\visi(X)$.  Then there exists $v \in\maxvis(X)$ with $w \leq v$, and hence $w \in\Sigma_v^{(0)}$.  Thus 
$\{\Sigma_v:v\in\maxvis(X)\}$ is a closed covering of $\vispart X$.

Let $v$ be a (visible) Roller class. By Corollary~\ref{Cor:FinitelyManyBelowRoller}, there are finitely many other classes $w$ such that $w \leq v$. This proves 
assertion~\ref{item:finite}.

Now suppose that $v_i,v_j$ are (visible) Roller classes satisfying $\Sigma_{v_i}=\Sigma_{v_j}$.  Since $v_j \in\Sigma_{v_i}$, 
we have $v_j\leq v_i$.  Likewise, $v_i \leq v_j$.  Hence $v_i=v_j$, proving assertion~\ref{item:injective}.

Next, let $v_0,\ldots,v_n\in\visi(X)$ and suppose that $\bigcap_{i=0}^n\Sigma_{v_i}\neq\emptyset$.  First, note that 
$\Sigma_{v_i}=v_i\join L_{v_i}$, where $L_{v_i}$ (the \emph{link} of $v_i$ in $\Sigma_{v_i}$) is a subcomplex of $\vispart X$ and $\join$ denotes the simplicial 
join operation.   Hence $\Sigma_{v_i}$ is topologically a cone, and is thus contractible.

For each $i$, let $\mathcal V_i$ be a UBS representing $v_i$ (recall Definition~\ref{Def:UBSofRoller}).  Then 
$\bigcap_{i=0}^n\mathcal V_i$ is unidirectional and contains no facing triple, since those properties hold for each $\mathcal 
V_i$ and are inherited by subsets.  Second, $\bigcap_{i=0}^n\mathcal V_i$ is inseparable, because each $\mathcal V_i$ is.

Now, if $u \in\bigcap_{i=0}^n\Sigma_{v_i}$ is a $0$--simplex, then $u \leq v_i$ for all $i$, so 
$u$ is represented by a UBS of the form $\U \preceq \bigcap_{i=0}^n\mathcal V_i$.  
Hence $\bigcap_{i=0}^n\mathcal V_i$ is infinite. Combined with the above discussion, this shows that 
$\bigcap_{i=0}^n\mathcal V_i$ is a UBS.  Let $v$ be the corresponding Roller class. Since $\U \preceq \bigcap_{i=0}\mathcal V_i$, we have $u \leq v$.  

Since each such $u$ is $\ell_2$--visible, the set of visible Roller classes $w$ with $w\leq v$ is nonempty.  Hence there is a unique Roller class $m$ such that $m\leq v$, and $m$ 
is $\ell_2$--visible, and $m$ is $\leq$--maximal with those properties, by Lemma~\ref{Lem:m_v} and Definition~\ref{Def:L2visible}.  Since each $u\in\bigcap_{i=0}^n\Sigma_{v_i}$ satisfies $u\leq m\in 
\vispart X$, there is a subcomplex $L\subset\vispart X$ such that $\bigcap_{i=0}^n\Sigma_{v_i}=m\star L$, so $\bigcap_{i=0}^n\Sigma_{v_i}$ is contractible in $\vispart X$.

Finally, statement \ref{item:equi} follows from the equivariance of $\visi(X)$ and the partial order $\leq$.
\end{proof}

\begin{definition}[Simplicial complex $\mathcal N_\triangle$]\label{Def:SimplicialNerveVisi}
Let $\mathcal N_\triangle$ be the nerve of the closed covering $\{\Sigma_v:v\in\maxvis(X)\}$ of $\vispart X$.  By 
Lemma~\ref{Lem:visible_simplices_closed_cover}.\ref{item:injective}, $\mathcal N_\triangle$ has vertex-set $\maxvis(X)$, 
and $v_0,\ldots,v_n$ span an $n$--simplex if and only if 
$\bigcap_{i=0}^n\Sigma_{v_i}\neq\emptyset$. By 
Lemma~\ref{Lem:visible_simplices_closed_cover}.\ref{item:equi}, $\Aut(X)$ acts by simplicial automorphisms on $\mathcal N_\triangle$.
\end{definition}

\begin{lemma}\label{Lem:omega_nerve}
Let $v_0,\ldots,v_n\in\visi(X)$.  Then 
$$\bigcap_{i=0}^n\Sigma_{v_i}\neq\emptyset \: \Leftrightarrow \: \bigcap_{i=0}^nQ(v_i)\neq\emptyset.$$
Consequently, there is an $\Aut(X)$--equivariant simplicial isomorphism $\mathcal N_\triangle \to \mathcal N_T$.
\end{lemma}

\begin{proof}
Suppose that there exists $a\in\bigcap_{i=0}^nQ(v_i)$.  Then Lemma~\ref{Lem:psiphi} says $ \psi(a) \leq v_i $ for all $i$.  Hence, by Definition~\ref{Def:VisiPart}, we have $\psi(a)\in\Sigma_{v_i}$ for all $i$.  

Conversely, suppose that there exists a vertex (i.e. a Roller class) $w\in\bigcap_{i=0}^n\Sigma_{v_i}$.  Then $w \leq v_i $ 
for all $i$, by definition.  Hence, by Lemma~\ref{Lem:containment}, $Q(w)\subset Q(v_i)$ for all $i$, from which it 
follows that $\varphi(w)\in\bigcap_{i=0}^nQ(v_i)$. 

Finally, the simplicial isomorphism $\mathcal N_\triangle \to \mathcal N_T$ comes from identifying the $0$--skeleta of $\mathcal N_\triangle$  and $\mathcal N_T$ with $\maxvis (X)$.
\end{proof}

\subsection{Homotopy equivalence between $\absimp X$ and $\tits X$}\label{Sec:vis_tits_homotopic}
We can now assemble the proof that $\absimp X$ is homotopy equivalent to $\tits X$. We do this in two propositions:

\begin{prop}\label{prop:tits_visible_homotopic}
There is an \Authom homotopy equivalence  $\vispart X \to \tits X$.
\end{prop}

\begin{proof}
Consider the covering $\{\Sigma_v:v\in\maxvis(X)\}$ of $\vispart X$ by the subcomplexes $\Sigma_v$, coming from 
Lemma~\ref{Lem:visible_simplices_closed_cover}.  By Definition~\ref{Def:SimplicialNerveVisi}, the nerve of this cover is  
$\mathcal N_\triangle$.  By Lemma~\ref{Lem:visible_simplices_closed_cover}.\ref{item:contractible}, the intersection of any finite 
collection of the $\Sigma_v$ is either empty or contractible.  
Hence, by the Equivariant Simplicial Nerve Theorem~\ref{Thm:EquivariantSimplicialNerve}, there is an \Authom homotopy equivalence $\vispart X \to \mathcal N_\triangle$.

By Lemma~\ref{Lem:omega_nerve}, there is a simplicial isomorphism $\mathcal N_\triangle \to \mathcal N_T$. Finally, by Theorem~\ref{Thm:first_HE}, there is an \Authom homotopy equivalence $\mathcal N_T \to \tits X$. 
Putting it all together, we obtain a chain of \Authom homotopy equivalences
$$
\vispart X \xrightarrow{ \: \sim \: }   \mathcal N_\triangle \xrightarrow{ \: \cong \: } \mathcal N_T  \xrightarrow{ \: \sim \: } \tits X. \qedhere
$$
\end{proof}

\begin{prop}\label{prop:vis_all_equivalent}
The inclusion $\vispart 
X \hookrightarrow \absimp X$ is an $\Aut(X)$--equivariant homotopy equivalence. Its homotopy inverse is an \Authom deformation retraction $\absimp X \to \vispart X$.
\end{prop}

\begin{proof}
First, observe that since $\absimp X$ is $\Aut(X)$--invariant by Definition~\ref{Def:SimplicialRoller}, and the set $\visi (X)$ is $\Aut(X)$--invariant by Definition~\ref{Def:L2visible}, the inclusion  $\vispart 
X \hookrightarrow \absimp X$ is $\Aut(X)$--equivariant.
In the remainder of the proof, we will construct a deformation retraction $\absimp X \to \vispart 
X$ that serves as a homotopy inverse to the inclusion $\vispart 
X \hookrightarrow \absimp X$. 

Let $w$ be an invisible Roller class (viewed as a $0$--simplex of $\absimp X$). 
 By Lemma~\ref{Lem:maximal_visible}, $w$ is not minimal, so there exists a 
minimal Roller class $m$ with $m < w$.  Let $n(w)\ge1$ be the maximum length of a 
chain of the form $m < \cdots < w$ with $m$ a minimal Roller class. 

Let $D$ be the maximal number such that there exists an invisible Roller class $w$ with $n(w)=D$. By 
Remark~\ref{Rem:AbsimpDimension}, every simplex in $\absimp X$ has at most $\dimension X$ vertices, hence $D \leq \dimension X$.

Let $I_n$ be the set of invisible Roller classes $v$ with $n(v)=n$.
Then 
$$(\absimp X)^{(0)}= (\vispart X)^{(0)}\sqcup\bigsqcup_{n=1}^{ D}I_n.$$  
Let $(\thickvispart X)_0=\vispart X$.  For $1\leq N\leq  D$, let $(\thickvispart X)_N$ be the subcomplex spanned by 
$(\vispart X)^{(0)}\sqcup\bigsqcup_{n=1}^{N}I_n$, so
$$\vispart X = (\thickvispart X)_0 \subset(\thickvispart X)_1\subset\cdots\subset(\thickvispart X)_{D}=\absimp X.$$

Given a Roller class $w$ and a positive integer $n \leq D$, define the \emph{open star} $st_n(w)$ to be the union of $\{w\}$ and all of the 
open simplices of $(\thickvispart X)_n$ whose 
closures contain $w$.  Let $St_n(w)$ be the union of all (closed) simplices of 
$(\thickvispart X)_n$ containing $w$.  Define the  \emph{(downward) link} $L_n(w)=St_n(w) \setminus st_n(w)$.

\begin{claim}\label{claim:contained_in_lower}
For $w \in  I_n$, we have $L_n(w) \subset (\thickvispart X)_{n-1}$.
\end{claim}

Let $w\in I_n \subset (\thickvispart X)_n$.  Let $\sigma$ be a maximal simplex of $(\thickvispart X)_n$ 
containing $w$ and lying in $(\thickvispart X)_n$.  Let $v_0<\cdots<v_m$ be the 
corresponding chain, with $w=v_i$ for some $i$.  Since $n(w)=n$, the chain 
$v_0 <\cdots<v_i$ has length at most $n$. 

First, consider $v_j$ for $j<i$.  If $v_j$ is visible, then $v_j\in(\thickvispart X)_0\subset(\thickvispart X)_{n-1}$.  If 
$v_j$ is invisible, then since $v_0 <\cdots < v_j$ is a chain of length less than 
$n$, and $\sigma$ is maximal, $n(v_j)<n$, hence $v_j\in(\thickvispart X)_{n-1}$.  

Next, consider $v_j$ for $j>i$.  If $v_j$ is visible, then $v_j\in(\thickvispart X)_0\subset(\thickvispart X)_{n-1}$.  If 
$v_j$ is invisible, then $v_j\not\in I_n$, since there is a chain 
$v_0 <\cdots< v_j$ of length more than $n = n(v_i)$, which would contradict the containment $\sigma \subset (\thickvispart X)_n$.

Hence the simplex $\sigma'$ corresponding to the chain 
$v_0<\cdots<v_{i-1}<v_{i+1}<\cdots <v_m$ lies in $(\thickvispart X)_{n-1}$.  
This proves the claim.

\smallskip
\textbf{Contractibility of links:}  Now, fix $w \in  I_n$. We will show that $L_n(w)$ is 
contractible.  Establishing this involves several claims, leading up to Claim~\ref{claim:Ln_contractible}.

\begin{claim}\label{claim:full_link}
For a minimal Roller class with $m\in L_n(w)$, let $St(m) = \bigcup_n St_n(m)$ denote the star of $m$ in $\absimp X$.  Then:
\begin{enumerate}[$(1)$ ]
     \item $L_n(w)\cap St(m)$ is topologically a cone with cone-point $m$.
     \item $L_n(w)$ is the union of the subcomplexes $L_n(w)\cap St(m)$, as $m$ varies over  the finitely many minimal Roller classes with 
$m \leq w $.
\end{enumerate}
\end{claim}

Let $L(m)$ denote the link of $m$ in $\absimp X$. Then, by definition, $St(m)$ decomposes as a join: $St(m)=m \join L(m)$.  Moreover, $m\in L_n(w)$, so 
$$L_n(w)\cap St(m)=  L_n(w) \cap (m \join L(m)) = m \join(L_n(w)\cap L(m)),$$
 which proves the first assertion.

By Corollary~\ref{Cor:FinitelyManyBelowRoller}, there are finitely many minimal Roller classes $m \leq w$. For each such $m$, we clearly have $L_n(w)\cap St(m)\subset L_n(w)$.  Conversely, if $v\in L_n(w)$, 
then there exists a minimal class $m$ such that $m \leq v$ and $m \leq w$, hence $v\in L_n(w)\cap St(m)$.  This proves  the claim.

\begin{claim}\label{claim:intersection_is_cone}
Let $m_0, \ldots, m_k$ be distinct minimal Roller classes with $m_i < w$ for all $i$.  Then $\bigcap_{i=0}^k St(m_i) \cap L_n(w)$ is  topologically a cone, and is in particular nonempty.
\end{claim}

By 
Lemma~\ref{Lem:maximal_visible}, each $m_i$ is $\ell^2$--visible, so we can choose CAT(0) geodesic rays $\xi_i:[0,\infty)\to X$ such 
that $\xi_i(0)=\xi_j(0)$ for all $i$, and such that $\xi_i(\infty)=\varphi(m_i)$ for each $i$.  
For each $i$, define $\mathcal M_i=\W(\xi_i)$, a UBS representing $m_i$.
Since $m_i$ is minimal 
for all $i$, Lemma~\ref{Lem:FinManyMinimals} implies the intersection $\mathcal M_i\cap\mathcal M_j$ is finite whenever $i \neq j$.  
Moving the common basepoint $\xi_1(0) = \ldots = \xi_k(0)$ ensures that $\mathcal M_i\cap\mathcal M_j=\emptyset$ for all $i 
\neq j$, and that each $\mathcal M_i$ is the inseparable closure of an infinite descending chain.  Indeed, by moving the 
basepoint across each of the finitely many hyperplanes appearing in some $\mathcal M_i\cap\mathcal M_j$, we arrange the 
first property, and the inseparable closure property holds by Lemma~\ref{Lem:DeepTransverse}, Lemma~\ref{Lem:DeepSetProps}, 
and minimality of the $m_i$.

Let $\mathcal A$ be a UBS representing $w$. Since $m_i \leq w$ for all $i$, we have $\mathcal M_i \preceq \mathcal A$ for each $i$.
Thus it is readily checked that $\bigcup_{i=0}^k\mathcal M_i$ is a UBS. 

Consider a pair of indices $i\neq j$. Then since $\bigcup_{i=0}^k\mathcal M_i$ is unidirectional and $\mathcal M_i \cap \mathcal M_j = \emptyset$,  
any pair of hyperplanes  $\hat h_i \in\mathcal M_i$ and $\hat h_j \in\mathcal M_j$ must cross. Indeed, $\hat h_i \neq \hat h_j$, and $\hat h_i$ is the base of a 
chain in $\M_i$, and $\hat h_j$ is the base of a chain in $\M_j$.  If $\hat h_i, \hat h_j$ did not cross, then the union of this pair of chains 
would violate unidirectionality of $\M_i\cap\M_j$. 

Thus, by Lemma~\ref{Lem:CapraceSageevProduct}, the convex hull of 
$\bigcup_{i=0}^k\xi_i$ is isometric to $\prod_{i=0}^kY_i$, where $Y_i$ is the 
cubical convex hull of $\xi_i$.  For each $t>0$, let $x_t$ be the image of $(\xi_0(t),\ldots,\xi_k(t))$ under the isometric 
embedding $\prod_{i=0}^kY_i\to X$.  Then segments joining the basepoint to the points $x_t$ converge uniformly on compact sets to a geodesic ray representing a point $b\in\tits X$ with the property 
that $\psi(b)=u$ is 
represented by $\U = \bigsqcup_{i=0}^k\mathcal M_i$.  Hence $u$ is visible by Definition~\ref{Def:L2visible}.  Moreover, $m_i \leq u \leq w$ 
for all $i$. To complete the proof of the claim, we will show that $u$ is the promised cone point of $\bigcap_{i=0}^k St(m_i) \cap L_n(w)$.

For each $i$, the fact that  $m_i \leq u$ implies $u\in St(m_i)$.  On the other hand, $u\neq w$ since $u$ is visible but $w$ is 
not.  Thus $u < w$, so $u\in L_n(w)$.  Thus far, we have shown that $\bigcap_{i=0}^k St(m_i) \cap L_n(w) \neq\emptyset$. 
Now, suppose $v\in L_n(w)\cap St(m_i)$ for all $i$.  To complete the claim, we must show that $u$ and $v$ are $\leq$--comparable. By definition, $m_i \leq v $ for all $i$, hence $\mathcal M_i \preceq \V $ for all $i$, where $\V$ is a UBS representing $\V$. Therefore, $\U = \bigsqcup_{i=0}^k\mathcal M_i \preceq \V$, which implies $u \leq v$. Thus $\bigcap_{i=0}^k St(m_i) \cap L_n(w)$ is a cone with cone point $u$, completing Claim~\ref{claim:intersection_is_cone}.

\begin{claim}\label{claim:Ln_contractible}
$L_n(w)$ is contractible. 
\end{claim}

By Claim~\ref{claim:full_link}, $L_n(w)$ is a finite union of subcomplexes $L_n(w)\cap St(m)$, where $m$ varies over the minimal 
Roller classes satisfying $m \leq w$.  By Claim~\ref{claim:intersection_is_cone}, any collection of the $L_n(w)\cap St(m)$ 
intersect in a contractible subcomplex.  Hence, by the Simplicial Nerve Theorem~\ref{Thm:EquivariantSimplicialNerve}, 
$L_n(w)$ is homotopy equivalent to the nerve of this covering. (We do not need to check equivariance at this step.) But, since any collection of the $L_n(w)\cap St(m)$ 
have nonempty intersection, this nerve is a finite 
simplex and thus contractible.  Hence $L_n(w)$ is contractible.

\smallskip
\textbf{Conclusion:}  We have shown that the following hold for $1\le 
n\le D$:
\begin{itemize}
 \item Claim~\ref{claim:contained_in_lower} says that for each vertex $v\in(\thickvispart X)_n \setminus (\thickvispart X)_{n-1}$, the link in $(\thickvispart X)_n$ of 
$v$ is contained in $(\thickvispart X)_{n-1}$;
\item Claim~\ref{claim:Ln_contractible} says that for each vertex $v\in(\thickvispart X)_n \setminus (\thickvispart X)_{n-1}$, the link in $(\thickvispart X)_n$ of 
$v$ is contractible.
\end{itemize}
The first fact says that the open stars $st_n(v)$ of the vertices 
$v\in(\thickvispart X)_n \setminus (\thickvispart X)_{n-1}$ are pairwise disjoint.  Together with the second 
fact, this implies that $(\thickvispart X)_n$ is homotopy equivalent to 
$$
(\thickvispart X)_{n-1}  \: = \: (\thickvispart X)_n \setminus \Bigg( \bigcup_{v\in(\thickvispart X)_n^{(0)} \setminus (\thickvispart X)_{n-1}^{(0)}}st_n(v) \Bigg).
$$
It follows that we can independently deformation retract the various open stars $st_n(v)$ to the corresponding links $L_n(v)$ to get a deformation retraction $(\thickvispart X)_n\to(\thickvispart X)_{n-1}$ that is a homotopy inverse for the inclusion.
Composing these retractions (for $1\leq n\leq D$) gives the desired deformation retraction $\absimp X \to \vispart X$. 

Finally, recall from the discussion after Definition~\ref{Def:G-homotopy} that every homotopy inverse of a $\Aut(X)$--equivariant inclusion $\vispart X \to \absimp X$  is itself an \Authom homotopy equivalence. Thus we have an \Authom deformation retraction $\absimp X \to \vispart X$.
\end{proof}

Combining Propositions~\ref{prop:tits_visible_homotopic} and~\ref{prop:vis_all_equivalent} gives a proof of half of Theorem~\ref{Thm:HE-triangle}.

\begin{cor}\label{Cor:RollerSimpTitsHomotopy}
There is an \Authom homotopy equivalence $\RT \from \absimp X \to \tits X$.
\end{cor}

\subsection{Homotopy equivalence between $\simp  X$ and $\absimp X$}\label{Sec:simp-roller}

The following proposition proves the homotopy equivalence of the simplicial boundary and the simplicial Roller boundary, completing the proof of Theorem~\ref{Thm:HE-triangle}.

\begin{prop}\label{prop:simplicial_boundary}
There is an \Authom homotopy equivalence $\SR \from \simp  X \to \absimp X$.
\end{prop}

\begin{proof}
Let $\sigma$ be a maximal simplex of $\simp X$, corresponding to a class $[\V_\sigma] \in \mathcal{UBS}(X)$. Then, by Theorem~\ref{Thm:UBStoRoller}, $v_\sigma = \UR([\V_\sigma])$ is a maximal Roller class.

Let $\mathcal A$ be the nerve of the covering of $\simp  X$ by  maximal simplices.  Now, since any 
collection of maximal simplices intersect in $\emptyset$ or a simplex, and simplices are contractible, 
the Equivariant Simplicial Nerve Theorem~\ref{Thm:EquivariantSimplicialNerve}
provides an \Authom homotopy equivalence $\simp X \to \mathcal A$.

Now, for each maximal Roller class $v$, consider the subcomplex $\Upsilon_v$ of $\absimp X$ consisting of all simplices 
corresponding to chains in which $v$ is the maximal element.  Note that the set of such $\Upsilon_v$ is a cover of $\absimp 
X$.  

A simpler version of the proof of Lemma~\ref{Lem:visible_simplices_closed_cover} implies that for all finite collections 
$\{ v_1,\ldots, v_n \}$ of maximal Roller classes, $\bigcap_{i=1}^n\Upsilon_{v_i}$ is empty or contractible.  Indeed, suppose this 
intersection is nonempty.  Exactly as in the proof of Lemma~\ref{Lem:visible_simplices_closed_cover}, each $v_i$ corresponds 
to a commensurability class $[\mathcal V_i]$ of UBSes whose intersection is a UBS $\mathcal V$, which in turn determines a 
$\leq$--maximal Roller class $v$ such that $v\leq v_i$ for all $i$.  As in Lemma~\ref{Lem:visible_simplices_closed_cover}, 
this implies that $\bigcap_{i=1}^n\Upsilon_{v_i}$ is a cone with cone-point $v$.  (In the context of 
Lemma~\ref{Lem:visible_simplices_closed_cover}, there is an additional step to check that $v$ is visible, but that is 
unnecessary here since we are working in all of $\absimp X$ rather than in $\absimp^\sphericalangle X$.)

Let $\mathcal B$ be the nerve of the covering of 
$\absimp X$ by the subcomplexes $\Upsilon_v$, as $v$ varies over the maximal Roller classes. Observe that the assignment $v \mapsto \Upsilon_v$ is a bijection from the set of maximal Roller classes to $\mathcal B^{(0)}$. By Theorem~\ref{Thm:EquivariantSimplicialNerve}, as above, there is an \Authom homotopy equivalence $\absimp X \to \mathcal B$.

To conclude, we will show that $\mathcal A$ and $\mathcal B$ are equivariantly isomorphic.  The vertex set of $\mathcal A$ is the set of maximal simplices 
$\sigma$ of $\simp X$, so we may define a function $f \from \mathcal A^{(0)}\to\mathcal B^{(0)}$ via the composition
$$
\sigma \mapsto [\V_\sigma] \mapsto v_\sigma \mapsto \Upsilon_{v_\sigma},
$$
where the middle arrow is the equivariant map $\UR$ of Theorem~\ref{Thm:UBStoRoller}.  Since both $\sigma$ and $v_\sigma$ are maximal by definition, the map $f$ is a bijection by Theorem~\ref{Thm:UBStoRoller}.\ref{Itm:URmax}. Every arrow is $\Aut(X)$--equivariant by construction.

To extend $f$ to a simplicial isomorphism $f \from \mathcal A\to\mathcal B$, it suffices to check that for all maximal simplices 
$\sigma_0,\ldots,\sigma_n$ of $\simp X$, we have $\bigcap_{i=0}^n\sigma_i\neq \emptyset$ if and only if 
$\bigcap_{i=0}^n\Upsilon_{v_{\sigma_i}}\neq\emptyset$; that this is sufficient follows since $\mathcal A$ is the nerve of the covering of $\simp X$ by maximal simplices and $\mathcal B$ is the nerve 
of the covering of $\absimp X$ by the subcomplexes $\Upsilon_v$.  

Suppose that $\bigcap_{i=0}^n\sigma_i\neq \emptyset$.  Let $\mathcal V_i$ be a UBS representing $\sigma_i$.  So, $\mathcal V_i$ also represents the corresponding Roller class $f(\sigma_i)$.  Let 
$\mathcal V=\bigcap_{i=0}^n\mathcal V_i$, and note that $\mathcal V$ is infinite.  Then by Remark~\ref{Rem:L1Intersection}, $\mathcal V$ is an $\ell^1$--visible UBS. Hence 
$\mathcal V$ represents a Roller class $w$ such that $w\leq v_{\sigma_i}$ for all $i$.  Thus 
$\bigcap_{i=0}^n\Upsilon_{v_{\sigma_{i}}}$ contains $w$.  The converse is similar.  Hence $f$ extends to an isomorphism, and 
it follows that 
$$
\SR \from \simp X \xrightarrow{ \: \sim \: }  \mathcal A \xrightarrow{\: f \:} \mathcal B \xrightarrow{ \: \sim \: }  \absimp X
$$
is an \Authom homotopy equivalence.
\end{proof}

\subsection{Cuboid generalization}\label{Subsec:FinalProofCuboids} We can now conclude the proof of 
Theorem~\ref{Thm:MainCuboids}.

\begin{proof}[Proof of Theorem~\ref{Thm:MainCuboids}]
Let $G$ be a group acting on $X$ by cubical automorphisms, and let $\rho$ be a $G$--admissible rescaling of $X$, with rescaled metric $d_X^\rho$. Recall from Definition~\ref{Def:AutXrho} that the action of $G$ factors through $\Aut^\rho(X)$. In Section~\ref{Sec:TitsCuboid}, we have checked that all of the constructions and results about the Tits boundary $\tits X$ also apply to $\tits^\rho X$, the Tits boundary of the rescaled metric $(X, d_X^\rho)$, in a $G$--equivariant way.

In Section~\ref{Sec:CuboidCover}, we have checked that the constructions of open and closed nerves for $\tits X$ also work for $\tits^\rho X$. In particular, by Theorem~\ref{Thm:First_HE_Cuboid}, there is a \Ghom homotopy equivalence between $\tits^\rho X$ and the nerve
 $\nerve_T^\rho$ of the closed cover  $\{Q^\rho(v) : v \in \maxvis^\rho(X)\}$.
 
Now, we inspect the results of this section. In Lemma~\ref{Lem:visible_simplices_closed_cover}, one simply needs to replace $\maxvis(X)$ by $\maxvis^\rho(X)$ and $\Aut(X)$ by $\Aut^\rho(X)$; the same exact proof then applies. 
In Definition~\ref{Def:SimplicialNerveVisi}, replacing $\maxvis(X)$ by $\maxvis^\rho(X)$ yields a nerve $\nerve_\triangle^\rho$. Then the proof of Lemma~\ref{Lem:omega_nerve} extends verbatim to give an $\Aut^\rho(X)$--equivariant simplicial isomorphism $\nerve_\triangle^\rho \to \nerve_T^\rho$. Proposition~\ref{prop:tits_visible_homotopic} extends verbatim, because its proof is a top-level assembly of previous results. Similarly, Proposition~\ref{prop:vis_all_equivalent} extends immediately to cuboids, because its proof is a topological argument about simplicial complexes. (The proof of Proposition~\ref{prop:vis_all_equivalent} does use several lemmas from Sections~\ref{Sec:Tits} and~\ref{Sec:TitsRollerConnections}, particularly in Claim~\ref{claim:intersection_is_cone}, but all of those lemmas have been extended to cuboids. Compare Remark~\ref{Rem:DaCuboid}.) Combining the cuboid versions of Propositions~\ref{prop:tits_visible_homotopic} and~\ref{prop:vis_all_equivalent} gives a \Ghom homotopy equivalence $\absimp X \to \tits^\rho X$.

Finally, note that the \Authom homotopy equivalence $\simp X \to \absimp X$ established in Proposition~\ref{prop:simplicial_boundary} is also a \Ghom homotopy equivalence, because $G$ acts by cubical automorphisms. Thus we have both of the \Ghom homotopy equivalences claimed in the theorem.
\end{proof}

\bibliographystyle{amsalpha}
\bibliography{guralnik_tits_refs}
\end{document}